\documentclass[a4paper, twoside,10pt]{article}
\usepackage{amsmath}
\usepackage{amssymb}
\usepackage{amsfonts}
\usepackage{graphicx}
\usepackage{subfigure}
\usepackage{eucal}
\usepackage{amsthm}
\usepackage{verbatim}
\usepackage{epsfig}
\usepackage{hyperref}

\newcommand{\brk}[1]{\left( #1 \right)}
\newcommand{\abrk}[1]{\left| #1 \right|}

\def\Rc{\mathbb{R}}
\def\Cc{\mathbb{C}}

\newcommand{\F}{\mathcal{F}}
\newtheorem{thm}{Theorem}
\newtheorem{prop}{Proposition}
\newtheorem{cor}{Corollary}
\newtheorem{definition}{Definition}

\newtheorem*{goal}{Principal goal of this paper}
\newtheorem{remark}{Remark}
\newtheorem{conjecture}{Conjecture}

\begin{document}

\title{
Detailed analysis of prolate quadratures and interpolation
formulas}
\author{Andrei Osipov\footnote{This author's research was supported in part
 by the AFOSR grant \#FA9550-09-1-0241.}
\footnote{Yale University, 51 Prospect st, New Haven, CT 06511.
Email: andrei.osipov@yale.edu.
},
Vladimir Rokhlin\footnote{
This author's research was supported in part 
by the ONR grants \#N00014-10-1-0570, \#N00014-11-1-0718,
the AFOSR grant \#FA9550-09-1-0241,
and the ONR / Telcordia grant
\#N00014-10-C-0176 / PO\#20013660.}
\footnote{
This author has a significant financial interest
in the Fast Mathematical Algorithms and Hardware corporation
(FMAHc) of Connecticut.}
}

\maketitle

\begin{abstract}
As demonstrated by Slepian et. al. in a sequence of
classical papers (see \cite{SlepianComments},
\cite{ProlateSlepian1},
\cite{ProlateLandau1},
\cite{ProlateSlepian2},
\cite{SlepianAsymptotic}),
prolate spheroidal wave functions (PSWFs)
provide a natural and efficient tool
for computing with bandlimited functions defined on an interval.
As a result, PSWFs are becoming increasing popular
in various areas in which such function occur - this
includes
physics (e.g. wave phenomena, fluid dynamics),
engineering (e.g. signal processing, filter design), etc.

To use PSWFs as a computational tool, one needs
fast and accurate numerical algorithms
for the evaluation of PSWFs and related quantities,
as well as for the construction of quadratures,
interpolation formulas, etc. 
For the last half a century, substantial progress has been made in
design of such algorithms
- this includes both classical results 
(see e.g. \cite{Bouwkamp}) as well as 
more recent developments (see e.g. \cite{RokhlinXiaoProlate}). 

% ALTERNATIVE - start
% The existing algorithms are quite satisfactory
% for moderate values of band limit $c$ (e.g. $c \leq 10^3$),
% in terms of both speed and accuracy.
% Still, in the modern computational environment
% one might encounter problems with relatively large values of $c$
% (e.g. $c \geq 10^4$).
% ALTERNATIVE - end

% MAIN - start
The complexity of many of the existing algorithms, however,
is at least quadratic in the band limit $c$.
For example, the evaluation of the $n$th
eigenvalue of the prolate integral operator requires at least
$O(c^2)$ operations (see e.g. \cite{RokhlinXiaoProlate}); 
also, the construction of accurate quadrature
rules for the integration of bandlimited functions of band limit $c$
requires $O(c^3)$ operations (see e.g. \cite{ChengRokhlin}).
Therefore, while 
the existing algorithms are
quite satisfactory for moderate values of $c$
(e.g. $c \leq 10^3$), 
they tend to be relatively slow
when $c$ is large
(e.g. $c \geq 10^4$).
% MAIN - end

In this paper, we describe several numerical algorithms
for the evaluation of PSWFs and related quantities,
and design a class of PSWF-based quadratures for the integration of 
bandlimited functions. Also, we 
perform detailed analysis of the related properties of PSWFs.
While the analysis is somewhat involved, the resulting
numerical algorithms are quite simple and efficient in practice.
For example, the evaluation of the $n$th
eigenvalue of the prolate integral operator requires
$O(n+c)$ operations; also,  
the construction of accurate quadrature
rules for the integration of bandlimited functions of band limit $c$
requires $O(c)$ operations.

Our results are illustrated via several numerical experiments.
\end{abstract}


\noindent
{\bf Math subject classification:} {
33E10, 34L15, 35S30, 42C10, 45C05, 54P05,
65D05, 65D15, 65D30, 65D32}
% 65D05 interpolation
% 65D15 algorithms for function approximation
% 65D30 numerical integration
% 65D32 quadrature and cubature formulas

%%%%%%%%%%%%%%%%%%%%%%%%%%%%%%%%%%%%%%%%%%%
% User defined macros
%%%%%%%%%%%%%%%%%%%%%%%%%%%%%%%%%%%%%%%%%%%

%%%%%%%%%%%%%%%%%%%%%%%%%%%%%%%%%%%%%%%%%%%%%
%%%%%%%%%%%%%%%%%%%%%%%%%%%%%%%%%%%%%%%%%%%%%
\section{Outline}
\subsection{Quadratures for Bandlimited Functions}
\label{sec_outline}
The principal goal of this paper is a quadrature
designed for the integration of bandlimited functions
of a specified band limit $c > 0$.

A function $f: \Rc \to \Rc$ is bandlimited of band limit $c>0$, if there
exists a function $\sigma \in L^2\left[-1,1\right]$ such that
\begin{align}
f(x) = \int_{-1}^1 \sigma(t) \cdot e^{icxt} \; dt.
\label{eq_intro_f}
\end{align}
In other words, the Fourier transform of a bandlimited function
is compactly supported.
While \eqref{eq_intro_f} defines $f$ for all real $x$, 
one is often interested in bandlimited functions, whose 
argument is confined to an interval, e.g. $-1 \leq x \leq 1$.
Such functions are encountered in physics (wave phenomena,
fluid dynamics), engineering (signal processing), etc.
(see e.g. \cite{SlepianComments}, \cite{Flammer}, \cite{Papoulis}).

By quadrature we mean a set of nodes
\begin{align}
\label{eq_quadrature_nodes}
-1 < t_1^{(n)} < \dots < t_n^{(n)} < 1
\end{align}
and weights
\begin{align}
\label {eq_quadrature_weights}
W_1^{(n)}, \dots, W_n^{(n)}.
\end{align}
If $f:(-1,1) \to \Rc$ is a bandlimited function, we use the quadrature to
approximate the integral of $f$ over the interval $(-1,1)$ by
a finite sum; more specifically,
\begin{align}
\int_{-1}^1 f(t) \; dt \approx \sum_{j=1}^n W_j^{(n)} f\brk{t_j^{(n)}} .
\label{eq_quadrature_used}
\end{align}
About half a century ago it was observed that the eigenfunctions 
of
the integral operator 
$F_c: L^2\left[-1,1\right] \to L^2\left[-1,1\right]$,
defined via the formula
\begin{align}
F_c\left[\varphi\right] (x) = \int_{-1}^1 \varphi(t) e^{icxt} \; dt,
\label{eq_intro_fc}
\end{align}
provide a natural tool for dealing with bandlimited functions, defined
on the interval $\left[-1,1\right]$. Moreover, it
was observed 
(see \cite{ProlateSlepian1}, \cite{ProlateLandau1}, \cite{ProlateSlepian2})
that the eigenfunctions of $F_c$
are precisely the prolate spheroidal wave functions (PSWFs)
of band limit $c$,
well known from the mathematical physics (see, for example,
\cite{PhysicsMorse}, \cite{Flammer}).
Therefore, when designing a quadrature for the integration
of bandlimited functions of band limit $c>0$,
it is natural to require that this quadrature integrate several
first PSWFs of band limit $c$ with high accuracy.

We formulate the principal objective of this paper
in a more precise manner, as follows.
%%%%%%%%%%%%%%%
\begin{goal}
Suppose that $c>0$ is a real number.
For every integer $n > 0$, 
we define a quadrature of order $n$ 
(for the integration of bandlimited functions
of band limit $c$ over $(-1,1)$) by specifying
$n$ nodes and 
$n$ weights (see \eqref{eq_quadrature_nodes}, \eqref{eq_quadrature_weights}).
Suppose also that $\varepsilon>0$.
We require that, for sufficiently large $n$,
the quadrature of order $n$ integrate
the first $n$ PSWFs of band limit $c$ up to the error $\varepsilon$.
More specifically, we find the integer $M=M(c,\varepsilon)$
such that, for every integer $n \geq M$
and all integer $m=0, 1, \dots, n-1$,
\begin{align}
\abrk{ \int_{-1}^1 \psi_m(t) \; dt - 
       \sum_{j=1}^n W_j^{(n)} \psi_m\brk{t_j^{(n)}} } \leq \varepsilon,
\label{eq_main_goal}
\end{align}
where $\psi_m : (-1,1) \to \Rc$ is the $m$th PSWF
of band limit $c$
(see Section~\ref{sec_pswf}).
\end{goal}
%%%%%%%%%%%%%
Quadratures for the integration of
bandlimited functions which satisfy \eqref{eq_main_goal}
have already been discussed in the literature, for example:

{\it \bf{Quadrature 1.}}
Suppose that $n>0$ is an integer.
The existence and uniqueness of $n$ nodes and weights, such that 
\eqref{eq_main_goal} holds for $\varepsilon = 0$ and 
all $m = 0, 1, \dots,
2n - 1$, was first observed more than 100 years ago 
(see, for example, \cite{Karlin}, \cite{Krein}, \cite{Markov1},
\cite{Markov2}) for all Chebyshev systems, of which PSWFs are a special
case (see \cite{RokhlinXiaoProlate}). 
Although numerical algorithms for the design
of this optimal quadrature were recently constructed
(see \cite{ChengRokhlin},
\cite{MaRokhlin},
\cite{YarvinRokhlin}), they tend to be rather expensive
(require order $n^3$ operations with a large proportionality constant).

{\it \bf{Quadrature 2.}}
Another quadrature was suggested in \cite{RokhlinXiaoProlate}.
The PSWF $\psi_n$ has $n$ roots $t_1, \dots, t_n$ in
the interval $(-1,1)$ (see Theorem~\ref{thm_pswf_main}
in Section~\ref{sec_pswf});
the idea is to use these roots as the quadrature nodes,
solve the linear system
of $n$ equations
\begin{align}
\left\{ 
\sum_{j=1}^n \psi_m(t_j) W_j = \int_{-1}^1 \psi_m(t) \; dt 
\right\}_{m=0}^{n-1}
\label{eq_weights_lin_system}
\end{align}
for the unknowns $W_1, \dots, W_n$, 
and use the resulting weights and nodes to define a quadrature for 
the integration of functions
of band limit $2c$. This approach is justified by the generalization
of the Euclid's division algorithm for PSWFs 
(see \cite{RokhlinXiaoProlate}),
and is less expensive computationally than the previous
one (its cost is dominated by the cost of solving the linear system
\eqref{eq_weights_lin_system}).
The same quadrature can be used 
to integrate functions of band limit $c$,
since \eqref{eq_weights_lin_system} implies that \eqref{eq_main_goal}
holds with $\varepsilon = 0$, for all
$m = 0, \dots, n-1$. 
% It has been observed that,
% for a ``typical'' bandlimited function, the error in
% \eqref{eq_quadrature_used} is expected
% to be proportional to $|\lambda_n|$,
% and thus $n$ should be chosen to be at least greater than $2c/\pi$
% (see \cite{Yoel}, \cite{RokhlinXiaoProlate}
% and Section~\ref{sec_pswf}
% below). 

In this paper, we describe another quadrature whose nodes
are the $n$ roots of $\psi_n$ in $\brk{-1, 1}$. However, its weights 
differ from the solution of \eqref{eq_weights_lin_system}, and 
can be evaluated in $O(n)$ operations (see Section~\ref{sec_quad} and 
Section~\ref{sec_numerical} below).

Thus, the quadratures of this paper are much faster to 
evaluate than those described above.
Moreover, \eqref{eq_main_goal}
ensures that their accuracy is similar to that of Quadrature 2.
Also, their nodes and weights can be used as starting points
for the scheme that computes the optimal Quadrature 1.

In order to define the weights, to make sure that \eqref{eq_main_goal}
holds and to be able to compute them efficiently, we need to analyze
the PSWFs in a somewhat detailed manner. This analysis
will be preceded by a heuristic explanation, which
provides some intuition as well as prevents one from the danger
of not seeing the forest for the trees (see
Section~\ref{sec_intuition} below).
Section~\ref{sec_overview} contains a short overview of the analysis.
Section~\ref{sec_prel} contains mathematical and numerical
preliminaries, to be used in the rest of the paper.
In Section~\ref{sec_summary}, we summarize the principal analytical
results of the paper.
Section~\ref{sec_analytical} contains the corresponding theorems
and proofs.
Section~\ref{sec_numerical} contains the description and analysis
of the numerical algorithms for the evaluation of the quadrature
and some related quantities.
In Section~\ref{sec_num_res}, we report the results
of several numerical experiments.

%%%%%%%%%%%%%%%%%%%%%%%%%%%%%%%%%%%%%%%%%%%%%%%%
\subsection{Intuition Behind Quadrature Weights}
\label{sec_intuition}

We recall the following classical interpolation problem.
Suppose that $t_1, \dots, t_n$ are $n$ distinct points
in the interval $(-1,1)$. We need to find the real numbers
$W_1, \dots, W_n$ such that 
\begin{align}
\int_{-1}^1 p(t) \; dt = \sum_{i=1}^n W_i \cdot p(t_i),
\label{eq_quad_polynomial}
\end{align}
for all the polynomials $p$ of degree at most $n-1$.
In other words, the quadrature with nodes $t_1,\dots,t_n$
and weights $W_1,\dots,W_n$ integrates all the polynomials
of degree up to $n-1$ exactly
(see \eqref{eq_quadrature_nodes}, \eqref{eq_quadrature_weights},
\eqref{eq_quadrature_used}).

To solve the problem, one constructs $n$ polynomials
$l_1, \dots, l_n$ of degree $n-1$ with the property
\begin{align}
\l_j(t_i) = \begin{cases}
0 & i \neq j, \\
1 & i = j
\end{cases}
\label{eq_pol_quad_l_prop}
\end{align}
for every integer $i, j = 1, \dots, n$ (see, for example,\cite{Isaacson}). 
It is easy to verify that, for every $j=1,\dots,n$, the polynomial $l_j$
is defined via the formula
\begin{align}
l_j(t) = \frac{ w_n(t) }{ w_n'(t_j) \cdot (t - t_j) },
\label{eq_pol_quad_l_formula}
\end{align}
where $w_n$ is the polynomial of degree $n$ whose roots are
precisely $t_1, \dots, t_n$. The weights $W_1, \dots, W_n$
are defined via the formula
\begin{align}
W_j = \int_{-1}^1 l_j(t) \; dt
    = \frac{1}{w_n'(t_j)} \int_{-1}^1 \frac{w_n(t) \; dt}{t - t_j}, 
\label{eq_pol_quad_weights}
\end{align}
for every integer $j=1, \dots, n$.
We observe that 
a single function $w_n$ is used to define all the $n$ weights; also, 
$w_n$ is a polynomial of degree $n$, and hence does not belong to the space
of the polynomials of degree up to $n-1$.

In our case, the basis functions are the PSWFs and not the polynomials.
Suppose that the roots $t_1, \dots, t_n$ of $\psi_n$ in the interval
$(-1,1)$ are chosen to be the nodes of the quadrature.
If we choose the weights $W_1,\dots,W_n$
such that the resulting quadrature integrates the 
first $n$ PSWFs exactly, this will lead to
the linear system \eqref{eq_weights_lin_system}, and hence 
to Quadrature 2 from
Section~\ref{sec_outline}. Instead, we define the weights via
using $\psi_n$ in the same way we used $w_n$ in \eqref{eq_pol_quad_weights}.
More specifically, similar to 
\eqref{eq_pol_quad_l_formula}, for every integer $j = 1, \dots, n$,
we define the function $\varphi_j : (-1,1) \to \Rc$ via
the formula
\begin{align}
\varphi_j(t) = \frac{ \psi_n(t) }{ \psi_n'(t_j) \cdot ( t-t_j ) }.
\label{eq_quad_phi_first}
\end{align}
We observe that, for every integer $i, j = 1, \dots, n$,
\begin{align}
\varphi_j(t_i) = 
\begin{cases}
0 & i \neq j, \\
1 & i = j,
\end{cases}
\label{eq_varphi_ij}
\end{align}
analogous to \eqref{eq_pol_quad_l_prop}. Viewed as a function
on the whole real line, each $\varphi_j$ is bandlimited
with the same band limit $c$ (see, for example,
Theorem~\ref{thm_quad_coef} in Section~\ref{sec_phi} or
Theorem 19.3 in \cite{Rudin}). 
On the other hand, $\varphi_j$
does not belong to the span of $\psi_0, \psi_1, \dots, \psi_{n-1}$
(see Theorem~\ref{thm_quad_coef} in Section~\ref{sec_phi}).
We define the weights $W_1, \dots, W_n$ via the formula
\begin{align}
W_j = \int_{-1}^1 \varphi_j(t) \; dt,
\label{eq_quad_w_first}
\end{align}
for $j = 1, 2, \dots, n$. The weights $W_1, \dots, W_n$, defined
via \eqref{eq_quad_w_first}, are different
from the solution of the linear system \eqref{eq_weights_lin_system}.
Nevertheless,
the resulting quadrature is expected to satisfy
\eqref{eq_main_goal} with $\varepsilon$ of order $|\lambda_n|$
(see Theorem~\ref{thm_quad} in Section~\ref{sec_quad_error}),
since the reciprocal of $\psi_n$ can be approximated well
by a rational function with $n$ poles.
Making the latter
statement precise is the principal purpose of Section~\ref{sec_analytical}
of this paper. While the analysis of the issue is somewhat detailed,
the principal idea is simple enough to be presented in the next few
sentences.

If $P$ is a polynomial with $m$ simple roots $z_1, \dots, z_m$ in $(-1, 1)$,
then the function $z \to P(z)^{-1}$ is meromorphic in the complex plane;
moreover,
\begin{align}
\frac{ 1 }{ P(z) } = \sum_{j = 1}^m \frac{ 1 }{ P'(z_j) \cdot (z - z_j) },
\label{eq_one_over_p}
\end{align}
for all complex $z$ different from $z_1, \dots, z_m$ 
(see Theorem~\ref{thm_cauchy} in Section~\ref{sec_misc_tools}).
The right-hand side of \eqref{eq_one_over_p}
is referred to as ``partial fractions expansion of $P^{-1}$''.
Similarly, the function $z \to \psi_n(z)^{-1}$ is meromorphic;
however, it has
infinitely many poles, all of which are real and simple
(see Corollary~\ref{cor_infinite} in Section~\ref{sec_first_order}),
and exactly
$n$ of which lie in $(-1, 1)$
(see Theorem~\ref{thm_pswf_main} in Section~\ref{sec_pswf}).
Suppose that the roots of $\psi_n$ in $(-1,1)$
are denoted by $t_1 < \dots < t_n$.
Motivated by
\eqref{eq_one_over_p}, we analyze the
partial fractions expansion of $\psi_n^{-1}$.
It turns out that 
\begin{align}
\frac{1}{\psi_n(t)} =
\sum_{j=1}^n \frac{ 1 }{ \psi_n'(t_j) \cdot (t-t_j) } + O(|\lambda_n|),
\label{eq_full_expansion}
\end{align}
for real $-1 < t < 1$
(see Section~\ref{sec_one_over_psi} and Theorem~\ref{thm_cauchy}
in Section~\ref{sec_misc_tools}).
In other words, \eqref{eq_full_expansion} means that 
the reciprocal of $\psi_n$
differs from a certain rational function with $n$ poles by a function,
whose magnitude in the interval $(-1,1)$ is of order
$|\lambda_n|$.

A rigorous version of \eqref{eq_full_expansion}
is established and proven in Section~\ref{sec_one_over_psi}.
The relation between \eqref{eq_main_goal}, \eqref{eq_quad_phi_first},
\eqref{eq_quad_w_first} and \eqref{eq_full_expansion} is studied
in Section~\ref{sec_quad}. The results of these two sections
rely on the machinery, 
developed in Sections~\ref{sec_oscillation},~\ref{sec_growth}.

%%%%%%%%%%%%%%%%%%%%%%%%%%%%%%%%%%%%%%%%%%%%%%%%%%%%%%%
\subsection{Overview of the Analysis}
\label{sec_overview}
%%%%%%%%%%%%%%%%%%%%%%%%%%%%%%%%%%%%%%%%%%
\subsubsection{Partial Fractions Expansion of $1 / \psi_n$}
\label{sec_pf}
%%%%%%%%%%%%%%%%%%%%%%%%%%%%%%%%%%%%%%%%%%
To establish \eqref{eq_full_expansion}, we proceed as follows.
Suppose that $x_1 < x_2 < \dots$ are the roots of $\psi_n$ in
$(1, \infty)$
(see Corollary~\ref{cor_infinite} in Section~\ref{sec_first_order}).
Suppose also that $M>1$, and $R>1$ is a point between $x_M$ and $x_{M+1}$.
In other words,
\begin{align}
1 < x_1 < x_2 < \dots < x_M < R < x_{M+1} < \dots.
\label{eq_pf_xr}
\end{align}
Then, for all real $-1 < t < 1$,
\begin{align}
& \frac{1}{\psi_n(t)} - \sum_{j=1}^n \frac{1}{\psi_n'(t_j)\cdot(t-t_j)} =
\nonumber \\
& \sum_{k=1}^M \left(
\frac{1}{\psi_n'(x_k) \cdot (t-x_k)} + 
\frac{1}{\psi_n'(-x_k) \cdot (t+x_k)}
\right) + \frac{1}{2\pi i}\oint_{\Gamma_R} \frac{dz}{\psi_n(z)\cdot(z-t)},
\label{eq_pf_cauchy}
\end{align}
where $\Gamma_R$ is the boundary of the square 
$\left[-R,R\right] \times \left[-iR,iR\right]$, traversed
in the counterclockwise direction
(see Theorem~\ref{thm_cauchy}
in Section~\ref{sec_misc_tools}).

Suppose now that $x>1$ is a root of $\psi_n$.
We observe that $\psi_n$ is a holomorphic function defined
in the entire complex plane.
We use the integral equation \eqref{eq_prolate_integral}
in Section~\ref{sec_pswf} and Theorem~\ref{thm_watson_lemma}
in Section~\ref{sec_misc_tools} to show that
\begin{align}
& \sqrt{
|\psi_n(x+it)|^2 + 
|\psi_n'(x+it)|^2 \cdot \frac{|(x+it)^2-1|}{|c^2\cdot(x+it)^2-\chi_n|}
} \sim \nonumber \\
& \frac{e^{ct} \cdot |\psi_n(1)| \cdot \sqrt{2}}{ct \cdot |\lambda_n|},
\quad t \to \infty
\label{eq_pf_watson}
\end{align}
(see Theorem~\ref{lem_q_at_infty} in Section~\ref{sec_upper}). On the 
other hand, we use the differential equation 
\eqref{eq_prolate_ode} in Section~\ref{sec_pswf} and
Theorem~\ref{thm_lewis} in Section~\ref{sec_growth_ode}
to show that
\begin{align}
& \sqrt{
|\psi_n(x+it)|^2 + 
|\psi_n'(x+it)|^2 \cdot \frac{|(x+it)^2-1|}{|c^2\cdot(x+it)^2-\chi_n|}
} \leq \nonumber \\
& \frac{e^{1/4} \cdot e^{ct} \cdot |\psi_n'(x)| \cdot (x^2-1)^{3/4}}
{ct \cdot \left(x^2 - (\chi_n/c^2)\right)^{1/4}}
\label{eq_pf_der_ineq}
\end{align}
(see Theorems~\ref{lem_harmonic},
\ref{lem_drr},
\ref{lem_big_int}, 
\ref{lem_q_sharp},
\ref{thm_bc_big} in
Section~\ref{sec_upper}). 
We combine \eqref{eq_pf_watson} and \eqref{eq_pf_der_ineq}
to establish the inequality
\begin{align}
\frac{1}{\abrk{\psi_n'(x)}} \leq 
e^{1/4} \cdot \abrk{\lambda_n} \cdot 
\frac{ (x^2-1)^{\frac{3}{4}} }{ (x^2-(\chi_n/c^2))^{\frac{1}{4}} }
\label{eq_pf_sharp}
\end{align}
(see Theorem~\ref{thm_sharp_simple} in Section~\ref{sec_upper}).
Then, we use \eqref{eq_pf_sharp} to show that,
for every integer $M>1$,
\begin{align}
\left|
\sum_{k=1}^M \frac{1}{(t-x_k) \cdot \psi_n'(x_k)}
\right| < 5 \cdot |\lambda_n| \cdot 
\left( \log(2 \cdot x_M) + (\chi_n)^{1/4} \right)
\label{eq_pf_head}
\end{align}
(see Theorems~\ref{lem_6_10},~\ref{lem_head_main} in
Section~\ref{sec_head} for a more precise statement).

We observe that \eqref{eq_pf_head} provides an upper
bound on the first summand in right-hand side
of \eqref{eq_pf_cauchy}. While this bound is of order $|\lambda_n|$
for $x_M < O(|\lambda_n|^{-1})$, it diverges if we let
$M$ go to infinity (see, however, \eqref{eq_pf_two} below).

To overcome this obstacle, we use the integral equation
\eqref{eq_prolate_integral2} in Section~\ref{sec_pswf}
to analyze the behavior of $\psi_n(x)$ and
$\psi_n'(x)$ for $x > |\lambda_n|^{-2}$ (see Section~\ref{sec_tail}).
In particular, if $x > 1$ is a root
of $\psi_n$ and if $x > |\lambda_n|^{-2}$,
then
\begin{align}
\left| \psi_n'(x) \right| = 
\left| \frac{2 \psi_n(1) }{ \lambda_n x } \right| \cdot
\left[1 + O\left(\left|x \cdot \lambda_n \right|^{-1} \right)\right]
\label{eq_pf_dpsi}
\end{align}
(see Theorem~\ref{lem_tail_dpsi} in Section~\ref{sec_tail}
for a more precise statement). More detailed analysis reveals that,
if $y > x > |\lambda_n|^{-2}$ are two consecutive roots of $\psi_n$
and $-1 < t < 1$ is a real number, then
\begin{align}
\left| \frac{1}{\psi_n'(x) \cdot (x-t)} +
       \frac{1}{\psi_n'(y) \cdot (y-t)} \right| \leq
20 \cdot c \cdot \int_x^y \frac{ds}{s^2}
\label{eq_pf_two}
\end{align}
(see Theorem~\ref{lem_two_bound} in Section~\ref{sec_tail}). 

In Theorem~\ref{thm_head_tail} of Section~\ref{sec_head_tail},
we establish, for all real $-1 < t < 1$, the inequality of the form
\begin{align}
\left|
\sum_{k=1}^{\infty}
 \frac{1}{\psi_n'(x_k) \cdot (x_k-t)}
\right| \leq \text{const} \cdot |\lambda_n| \cdot 
\left( \log\left( \frac{1}{|\lambda_n|} \right)+ 
      (\chi_n)^{1/4} \right),
\label{eq_pf_series}
\end{align}
where \eqref{eq_pf_head}, \eqref{eq_pf_two} are used
to bound the head and the tail of the infinite sum,
respectively.

Eventually, we analyze the behavior of $\psi_n$ of the complex
argument to demonstrate that, for all real $-1 < t < 1$,
\begin{align}
\limsup_{k \to \infty} \left|
\frac{1}{2\pi i} \oint_{\Gamma_{R_k}} \frac{dz}{\psi_n(z) \cdot (z-t)}
\right| < 2\sqrt{2}\cdot |\lambda_n|,
\label{eq_pf_contour}
\end{align}
where $\left\{ R_k \right\}$ is a certain sequence that tends
to infinity, and 
the contours $\Gamma_{R_k}$ are as in \eqref{eq_pf_cauchy} (see
Theorems~\ref{lem_vertical_bound},~\ref{lem_contour}
in Section~\ref{sec_head_tail} for more details).
We substitute \eqref{eq_pf_series} and \eqref{eq_pf_contour}
into \eqref{eq_pf_cauchy} to obtain, for all real $-1 < t < 1$,
an inequality of the form
\begin{align}
\left| \frac{1}{\psi_n(t)} - \sum_{j=1}^n \frac{1}{\psi_n'(t_j)\cdot(t-t_j)}
\right| \leq 
\text{const} \cdot |\lambda_n| \cdot
\left( \log\left( \frac{1}{|\lambda_n|} \right)+ 
      (\chi_n)^{1/4} \right)
\label{eq_pf_final}
\end{align}
(see Theorems~\ref{thm_complex},
\ref{thm_complex_summary} in Section~\ref{sec_head_tail}).
In the next subsection, we overview
the implications of \eqref{eq_pf_final} to the analysis of
the quadrature, discussed in Section~\ref{sec_intuition}.

%%%%%%%%%%%%%%%%%%%%%%%%%%%%%%%%%%
\subsubsection{Quadrature Weights}
%%%%%%%%%%%%%%%%%%%%%%%%%%%%%%%%%%
Roughly speaking, \eqref{eq_pf_final} asserts that, for all real $-1<t<1$,
\begin{align}
\frac{1}{\psi_n(t)} - \sum_{j=1}^n \frac{1}{\psi_n'(t_j)\cdot(t-t_j)} =
O\left(|\lambda_n|\right).
\label{eq_qw_basic}
\end{align}
In other words, the left-hand side of \eqref{eq_qw_basic} is uniformly
bounded in $(-1,1)$, and its magnitude is of order $|\lambda_n|$.
If we multiply both sides of \eqref{eq_qw_basic} by $\psi_n(t)$ and
use \eqref{eq_quad_phi_first}, we obtain
\begin{align}
1 = \varphi_1(t) + \dots + \varphi_n(t) + \psi_n(t) \cdot 
O\left(|\lambda_n|\right)
\label{eq_qw_1}
\end{align}
In other words, $\varphi_1, \dots, \varphi_n$ constitute a partition
of unity in the interval $(-1,1)$, up to an error of order
$|\lambda_n|$.
We integrate both sides of \eqref{eq_qw_1} over $(-1,1)$ 
and use Theorem~\ref{thm_pswf_main} in Section~\ref{sec_pswf} 
and \eqref{eq_quad_w_first} in Section~\ref{sec_intuition}
to obtain
\begin{align}
2 = W_1 + \dots + W_n + O\left(|\lambda_n|\right)
\label{eq_qw_ws}
\end{align}
(see Section~\ref{sec_weights} for more details).

Suppose now that $m \neq n$ is an integer. We multiply both sides of
\eqref{eq_qw_1} by $\psi_m$ to obtain
\begin{align}
\psi_m(t) = \sum_{j=1}^n \psi_m(t) \cdot \varphi_j(t) +
\psi_m(t) \cdot \psi_n(t) \cdot
O\left(|\lambda_n|\right).
\label{eq_qw_psim}
\end{align}
On the other hand, for every integer $j=1,\dots,n$, we use
integration by parts to evaluate
\begin{align}
& \int_{-1}^1 \varphi_j(t) \cdot \psi_m(t) \; dt = \nonumber \\
& \frac{ |\lambda_m|^2 \cdot \psi_m(t_j) }{ |\lambda_m|^2-|\lambda_n|^2} \cdot
\left[
W_j + \frac{ic\lambda_n}{\psi_n'(t_j)} 
\int_0^1 \psi_n(x) \cdot e^{-icxt_j} \; dx
\right]
\label{eq_qw_coef}
\end{align}
(see Theorem~\ref{thm_quad_coef} in Section~\ref{sec_phi}).
We combine \eqref{eq_pf_final},
\eqref{eq_qw_psim} and \eqref{eq_qw_coef} with some
additional analysis to conclude that,
for all integer $0 \leq m < n$,
\begin{align}
\left|
\int_{-1}^1 \psi_m(t) \; dt - 
\sum_{j=1}^n \psi_m(t_j) \cdot W_j 
\right| \leq
\text{const} \cdot |\lambda_n| \cdot \left(
\log \frac{1}{|\lambda_n|} + \chi_n
\right)
\label{eq_qw_quad_error}
\end{align}
(see Theorems~\ref{thm_quad},~\ref{thm_quad_simple} in 
Section~\ref{sec_quad_error}).

According to \eqref{eq_qw_quad_error}, the quadrature
error \eqref{eq_main_goal} in Section~\ref{sec_outline} is roughly
of order $|\lambda_n|$. It remains to establish for what
values of $n$ this error is smaller than the predefined
accuracy parameter $\varepsilon > 0$.
In Section~\ref{sec_main_result}, we 
combine
Theorems~\ref{thm_n_khi_simple},~\ref{thm_khi_1},~\ref{thm_lambda_khi}
with \eqref{eq_qw_quad_error}
to achieve that goal.
Namely, we show that, if
\begin{align}
n > \frac{2c}{\pi} + \text{const} \cdot
\log(c) \cdot \left( \log(c) + \log\frac{1}{\varepsilon} \right),
\label{eq_qw_main_n}
\end{align}
then
\begin{align}
\left|
\int_{-1}^1 \psi_m(t) \; dt - 
\sum_{j=1}^n \psi_m(t_j) \cdot W_j 
\right| \leq \varepsilon,
\label{eq_qw_main_error}
\end{align}
for all integer $0 \leq m < n$
(see Theorem~\ref{thm_quad_eps_simple}).

Numerical experiments seem to indicate that 
the situation is even better
in practice: namely, to achieve the desired accuracy 
it suffices to pick the minimal $n$ such that
$\abrk{\lambda_n} < \varepsilon$, which occurs for
$n = 2c/\pi + O((\log c) \cdot (-\log \varepsilon))$
(see Section~\ref{sec_num_res}, in particular,
Conjecture~\ref{conj_quad_error}
and Experiment 14 in Section~\ref{sec_exp14}).

%%%%%%%%%%%%%%%%%%%%%%%%%%%%%%%%%%%%%%%%%%%%%
%%%%%%%%%%%%%%%%%%%%%%%%%%%%%%%%%%%%%%%%%%%%%
\section{Mathematical and Numerical Preliminaries}
\label{sec_prel}
In this section, we introduce notation and summarize
several facts to be used in the rest of the paper.

%%%%%%%%%%%%%%%%%%%%%%%%%%%%%%%%%%%%%%%%%%%%%
\subsection{Prolate Spheroidal Wave Functions}
\label{sec_pswf}

In this subsection, we summarize several facts about
the PSWFs. Unless stated otherwise, all these facts can be 
found in \cite{RokhlinXiaoProlate}, 
\cite{RokhlinXiaoApprox},
\cite{LandauWidom},
\cite{ProlateSlepian1},
\cite{ProlateLandau1},
\cite{Report},
\cite{ReportArxiv}.

Given a real number $c > 0$, we define the operator
$F_c: L^2\left[-1, 1\right] \to L^2\left[-1, 1\right]$ 
via the formula
\begin{align}
F_c\left[\varphi\right] (x) = \int_{-1}^1 \varphi(t) e^{icxt} \; dt.
\label{eq_pswf_fc}
\end{align}
Obviously, $F_c$ is compact. We denote its eigenvalues by
$\lambda_0, \lambda_1, \dots, \lambda_n, \dots$ and assume that
they are ordered such that $\abrk{\lambda_n} \geq \abrk{\lambda_{n+1}}$
for all natural $n \geq 0$. We denote by $\psi_n$ the eigenfunction
corresponding to $\lambda_n$. In other words, the following
identity holds for all integer $n \geq 0$ and all real $-1 \leq x \leq 1$:
\begin{align}
\label{eq_prolate_integral}
\lambda_n \psi_n\brk{x} = \int_{-1}^1 \psi_n(t) e^{icxt} \; dt.
\end{align}
We adopt the convention\footnote{
This convention agrees with that of \cite{RokhlinXiaoProlate},
\cite{RokhlinXiaoApprox} and differs from that of \cite{ProlateSlepian1}.
}
that $\| \psi_n \|_{L^2\left[-1,1\right]} = 1$.
The following theorem describes the eigenvalues and eigenfunctions
of $F_c$.
%%%%%%%%%%%
\begin{thm}
Suppose that $c>0$ is a real number, and that the operator $F_c$
is defined via \eqref{eq_pswf_fc} above. Then,
the eigenfunctions $\psi_0, \psi_1, \dots$ of $F_c$ are purely real,
are orthonormal and are complete in $L^2\left[-1, 1\right]$.
The even-numbered functions are even, the odd-numbered ones are odd.
Each function $\psi_n$ has exactly $n$ simple roots in $\brk{-1, 1}$.
All eigenvalues $\lambda_n$ of $F_c$ are non-zero and simple;
the even-numbered ones are purely real and the odd-numbered ones
are purely imaginary; in particular, $\lambda_n = i^n \abrk{\lambda_n}$.
\label{thm_pswf_main}
\end{thm}
%%%%%%%%%
\noindent
We define the self-adjoint operator
$Q_c: L^2\left[-1, 1\right] \to L^2\left[-1, 1\right]$ via the formula
\begin{align}
Q_c\left[\varphi\right] (x) =
\frac{1}{\pi} \int_{-1}^1 
\frac{ \sin \left(c\left(x-t\right)\right) }{x-t} \; \varphi(t) \; dt.
\label{eq_pswf_qc}
\end{align}
Clearly, if we denote by $\F:L^2(\Rc) \to L^2(\Rc)$ 
the unitary Fourier transform,
then
\begin{align}
Q_c\left[\varphi\right] (x) = 
\chi_{\left[-1,1\right]}(x) \cdot
\F^{-1} \left[ 
  \chi_{\left[-c,c\right]}(\xi) \cdot \F\left[\varphi\right](\xi)
\right](x),
\label{eq_pswf_fourier}
\end{align}
where $\chi_{\left[-a,a\right]} : \Rc \to \Rc$ is the characteristic
function of the interval $\left[-a,a\right]$, defined via the formula
\begin{align}
\chi_{\left[-a,a\right]}(x) = 
\begin{cases}
1 & -a \leq x \leq a, \\
0 & \text{otherwise},
\end{cases}
\label{eq_char_function}
\end{align}
for all real $x$.
In other words,
$Q_c$ represents low-passing followed by time-limiting.
$Q_c$ relates to $F_c$, defined via \eqref{eq_pswf_fc}, by 
\begin{align}
Q_c = \frac{ c }{ 2 \pi } \cdot F_c^{\ast} \cdot F_c,
\label{eq_pswf_qc_fc}
\end{align}
and the eigenvalues $\mu_n$ of $Q_n$ satisfy the identity
\begin{align}
\mu_n = \frac{c}{2\pi} \cdot \abrk{\lambda_n}^2,
\label{eq_prolate_mu}
\end{align}
for all integer $n \geq 0$. Obviously,
\begin{align}
\mu_n < 1,
\label{eq_mu_leg_1}
\end{align}
for all integer $n \geq 0$, due to \eqref{eq_pswf_fourier}.
Moreover, $Q_c$ has the same eigenfunctions $\psi_n$ as $F_c$.
In other words,
\begin{align}
\mu_n \psi_n(x) = \frac{1}{\pi} 
      \int_{-1}^1 \frac{ \sin\left(c\left(x-t\right)\right) }{ x - t } 
  \; \psi_n(t) \; dt,
\label{eq_prolate_integral2}
\end{align}
for all integer $n \geq 0$ and all $-1 \leq x \leq 1$.
Also,  $Q_c$ is closely related to the operator
$P_c: L^2(\Rc) \to L^2(\Rc)$,
defined via the formula
\begin{align}
P_c\left[\varphi\right] (x) =
\frac{1}{\pi} \int_{-\infty}^{\infty}
\frac{ \sin \left(c\left(x-t\right)\right) }{x-t} \; \varphi(t) \; dt,
\label{eq_pswf_pc}
\end{align}
which is a widely known orthogonal projection onto the space
of functions of band limit $c > 0$ on the real
line $\Rc$.

The following theorem about the eigenvalues $\mu_n$ of the operator $Q_c$,
defined via
\eqref{eq_pswf_qc},
can be traced back to \cite{LandauWidom}:
%%%%%%%%%%%
\begin{thm}
Suppose that $c>0$ and $0<\alpha<1$ are positive real numbers,
and that the operator $Q_c: L^2\left[-1,1\right] \to L^2\left[-1,1\right]$
is defined via \eqref{eq_pswf_qc} above.
Suppose also that the integer $N(c,\alpha)$ is the number of 
the eigenvalues $\mu_n$ of $Q_c$ that are greater than $\alpha$. In
other words,
\begin{align}
N(c,\alpha) = \max\left\{ k = 1,2,\dots \; : \; \mu_{k-1} > \alpha\right\}.
\end{align}
Then,
\begin{align}
N(c,\alpha)
= \frac{2c}{\pi} + \brk{ \frac{1}{\pi^2} \log \frac{1-\alpha}{\alpha} }
    \log c + O\brk{ \log c }.
\label{eq_mu_spectrum}
\end{align}
\label{thm_mu_spectrum}
\end{thm}
%%%%%%%%%
\noindent
According to \eqref{eq_mu_spectrum}, there are about $2c/\pi$
eigenvalues whose absolute value is close to one, order of $\log c$
eigenvalues that decay exponentially, and the rest of them are
very close to zero. 

The eigenfunctions $\psi_n$ of $Q_c$ turn out to be the PSWFs, well
known from classical mathematical physics \cite{PhysicsMorse}.
The following theorem, proved in a more general form in
\cite{ProlateSlepian2},
formalizes this statement.
%%%%%%%%%%%
\begin{thm}
For any $c > 0$, there exists a strictly increasing unbounded sequence
of positive numbers $\chi_0 <  \chi_1 <  \dots$ such that, for
each integer $n \geq 0$, the differential equation
\begin{align}
\brk{1 - x^2} \psi''(x) - 2 x \cdot \psi'(x) 
+ \brk{\chi_n - c^2 x^2} \psi(x) = 0
\label{eq_prolate_ode}
\end{align}
has a solution that is continuous on $\left[-1, 1\right]$.
Moreover, all such solutions are constant multiples of 
the eigenfunction $\psi_n$ of $F_c$,
defined via \eqref{eq_pswf_fc} above.
\label{thm_prolate_ode}
\end{thm}
%%%%%%%%%
\begin{remark}
\label{rem_continuation}
For all real $c > 0$ and all integer $n \geq 0$,
\eqref{eq_prolate_integral} defines an analytic continuation
of $\psi_n$ onto the entire complex plane.
All the roots of $\psi_n$ are simple and real.
In addition, the ODE \eqref{eq_prolate_ode} is satisfied
for all complex $x$.
\end{remark}
%%%%%%%%%%%%
%%%%%%%%%%%%%%%%%
Many properties of the PSWF $\psi_n$ depend on
whether the eigenvalue $\chi_n$ of the ODE \eqref{eq_prolate_ode}
is greater than or less than $c^2$. 
In the following theorem from \cite{Report}, \cite{ReportArxiv}, we describe
a simple relationship 
between $c, n$ and $\chi_n$.
%%%%%%%%%%%%
\begin{thm}
Suppose that $n \geq 2$ is a non-negative integer.
\begin{itemize}
\item If $n \leq (2c/\pi)-1$, then $\chi_n < c^2$.
\item If $n \geq (2c/\pi)$, then $\chi_n > c^2$.
\item If $(2c/\pi)-1 < n < (2c/\pi)$, then either inequality is possible.
\end{itemize}
\label{thm_n_and_khi}
\end{thm}
%%%%%%%%%

In the following theorem, upper and lower bounds on $\chi_n$
in terms of $c$ and $n$ are provided.
\begin{thm}
Suppose that $c > 0$ is a real number,
and $n \geq 0$ is an integer. Then,
\begin{align}
\label{eq_khi_crude}
n \brk{n + 1} < \chi_n < n \brk{n + 1} + c^2.
\end{align}
\label{thm_khi_crude}
\end{thm}
%%%%%%%%%%%
It turns out that, for the purposes of this paper,
the inequality \eqref{eq_khi_crude} is insufficiently sharp.
More accurate bounds on $\chi_n$ are described in
the following three theorems (see \cite{Report}, 
\cite{ReportArxiv}, \cite{Report2}, \cite{Report2Arxiv}).
%%%%%%%%%%%%%%%%%%
\begin{thm}
Suppose that $n \geq 2$ is a positive integer, and that $\chi_n > c^2$. Then,
\begin{align}
n < & \; 
\frac{2}{\pi} \int_0^1 \sqrt{ \frac{\chi_n - c^2 t^2}{1 - t^2} } \; dt
= \nonumber \\
& \; \frac{2}{\pi} \sqrt{\chi_n} \cdot E \brk{ \frac{c}{\sqrt{\chi_n}} }
< n+3,
\label{eq_both_large_simple_prop}
\end{align}
where the function $E:\left[0,1\right] \to \Rc$ is defined via
\eqref{eq_E} in Section~\ref{sec_elliptic}.
\label{thm_n_khi_simple}
\end{thm}
%%%%%%%%%%%
%%%%%%%%%%%%%%%%
\begin{thm}
Suppose that $n$ is a positive integer, and that
\begin{align}
n > \frac{2c}{\pi} + \frac{2}{\pi^2} \cdot \delta
   \cdot \log \left( \frac{4e\pi c}{\delta} \right),
%\label{eq_n_greater}
\end{align}
for some
\begin{align}
0 < \delta < \frac{5\pi}{4} \cdot c.
\end{align}
Then,
\begin{align}
\chi_n > c^2 + \frac{4}{\pi} \cdot \delta \cdot c.
\end{align}
\label{thm_khi_1}
\end{thm}
%%%%%%%%%%%%%%%%
\begin{thm}
Suppose that $n$ is a positive integer, and that
\begin{align}
\frac{2c}{\pi} \leq 
n \leq \frac{2c}{\pi} + \frac{2}{\pi^2} \cdot \delta
   \cdot \log \left( \frac{4e\pi c}{\delta} \right) - 3,
\label{eq_n_greater}
\end{align}
for some
\begin{align}
3 < \delta < \frac{5\pi}{4} \cdot c.
\label{eq_delta_khi_2}
\end{align}
Then,
\begin{align}
\chi_n < c^2 + \frac{8 }{\pi} \cdot \delta \cdot c.
\end{align}
\label{thm_khi_2}
\end{thm}
%%%%%%%%%%%%%%%%
The following theorem is a direct consequence of 
Theorem~\ref{thm_n_khi_simple}.
\begin{thm}
Suppose that $n > 0$ is a positive integer, and that
\begin{align}
n > \frac{2c}{\pi} + 1.
\label{eq_n_elem}
\end{align}
Then,
\begin{align}
\chi_n > c^2 + 1.
\label{eq_khi_elem}
\end{align}
\label{thm_khi_elem}
\end{thm}
\begin{proof}
It follows from \eqref{eq_both_large_simple_prop} of
Theorem~\ref{thm_n_khi_simple} that
\begin{align}
n & \; < \frac{2c}{\pi} \int_0^1 
\sqrt{1 + \frac{\chi_n-c^2}{c^2} \cdot \frac{1}{1-t^2}} \; dt \nonumber \\
& \; <
\frac{2c}{\pi} + \frac{2}{\pi} \cdot \sqrt{\chi_n-c^2} 
\cdot \int_0^1 \frac{dt}{\sqrt{1-t^2}}
= \frac{2c}{\pi} + \sqrt{\chi_n-c^2}.
\label{eq_elem_1}
\end{align}
We combine \eqref{eq_elem_1} with \eqref{eq_n_elem} to obtain 
\eqref{eq_khi_elem}.
\end{proof}
%%%%%%%%%%%%%%%%
In the following theorem from \cite{Report2}, \cite{Report2Arxiv}, 
we provide an upper bound on $|\lambda_n|$ in terms of $n$ and $c$.
%%%%%%%%%%%%
\begin{thm}
Suppose that $c>0$ is a real number, and that
\begin{align}
c > 22.
\label{eq_c22}
\end{align}
Suppose also that $\delta>0$ is a real number, and that
\begin{align}
3 < \delta < \frac{\pi c}{16}.
\label{eq_delta_crude}
\end{align}
Suppose, in addition, that $n$ is a positive integer, and that
\begin{align}
n > \frac{2c}{\pi} + \frac{2}{\pi^2} \cdot \delta \cdot
    \log\left( \frac{4e\pi c}{\delta} \right).
\label{eq_n_crude}
\end{align}
Suppose furthermore that the real number $\xi(n,c)$ is defined
via the formula
\begin{align}
\xi(n,c) = 7056 \cdot c \cdot 
\exp\left[-\delta\left(1 - \frac{\delta}{2\pi c}\right) \right].
\label{eq_xi_n_c}
\end{align}
Then, 
\begin{align}
\abrk{ \lambda_n } < \xi(n,c).
\label{eq_crude_inequality}
\end{align}
\label{thm_crude_inequality}
\end{thm}
%%%%%%%%%%%
In the following theorem from \cite{Report2}, \cite{Report2Arxiv},
we provide another upper bound on $|\lambda_n|$.
\begin{thm}
Suppose that $n>0$ is a positive integer, and that 
\begin{align}
n > \frac{2c}{\pi} + \sqrt{42}.
\label{eq_lambda_khi_0}
\end{align}
Suppose also that the real number $x_n$ is defined via the formula
\begin{align}
x_n = \frac{\chi_n}{c^2}.
\label{eq_lambda_khi_1}
\end{align}
Then,
\begin{align}
& |\lambda_n| < \nonumber \\
& 1195 \cdot c \cdot (x_n)^{\frac{3}{4}} \cdot (x_n-1)^{\frac{1}{4}} \cdot
\left(x_n-\frac{1}{2}\right)^3 \cdot
\exp\left[ -\frac{\pi}{4} \cdot \left(\sqrt{x_n}-\frac{1}{\sqrt{x_n}} \right)
\cdot c \right].
\label{eq_lambda_khi}
\end{align}
\label{thm_lambda_khi}
\end{thm}
%%%%%%%%%%%
The following theorem is a combination of certain results
from \cite{RokhlinXiaoApprox} and \cite{Report}, \cite{ReportArxiv}.
\begin{thm}
Suppose that $c > 0$ is a real number, and that $\chi_n > c^2$.
Then,
\begin{align}
\frac{1}{2} < \psi_n^2(1) < n + \frac{1}{2}.
\label{eq_psi1_bound}
\end{align}
\label{thm_psi1_bound}
\end{thm}
%%%%%%%%%%%%%%%%%
The following theorem appears in \cite{Report}, \cite{ReportArxiv}.
\begin{thm}
Suppose that $n \geq 0$ is a non-negative integer, and that
$x, y$ are two arbitrary extremum points of $\psi_n$ in $(-1,1)$.
If $\abrk{x} < \abrk{y}$, then
\begin{align}
\abrk{\psi_n(x)} < \abrk{\psi_n(y)}.
\label{eq_extremum_general}
\end{align}
If, in addition, $\chi_n > c^2$, then
\begin{align}
\abrk{\psi_n(x)} < \abrk{\psi_n(y)} < \abrk{\psi_n(1)}.
\label{eq_extremum_special}
\end{align}
\label{thm_extrema}
\end{thm}
%%%%%%%%%%%%%%%%%
The following theorem appears in \cite{Yoel}.
\begin{thm}
For all real $c > 0$ and all natural $n \geq 1$,
\begin{align}
\max_{m \leq n+1} \; \max_{\abrk{t} \leq 1} |\psi_m(t)| \leq 2\sqrt{n}.
\end{align}
\label{thm_all_psi_upper_bound}
\end{thm}
%%%%%%%%%%%%%%%%%%%%%
In the following theorem, we provide a recurrence relation
between the derivatives of $\psi_n$ of arbitrary order
(see Lemma 9.1 in \cite{RokhlinXiaoProlate}).
\begin{thm}
Suppose that $c>0$ is a real number, and that $n \geq 0$
is an integer. Then, 
\begin{align}
& \brk{1 - t^2} \psi_n'''(t) - 4t \psi_n''(t) +
\brk{\chi_n - c^2 t^2 - 2} \psi_n'(t) - 2c^2 t \psi_n(t) = 0
\label{eq_dpsi_rec3}
\end{align}
for all real $t$. Moreover,
for all integer $k \geq 2$ and all real $t$,
\begin{align}
& \brk{1 - t^2} \psi_n^{(k+2)}(t) 
  - 2 \brk{k+1} t \psi_n^{(k+1)}(t)
  + \brk{\chi_n - k\brk{k+1} - c^2 t^2} \psi_n^{(k)}(t) \nonumber \\
& \quad \quad 
  -c^2 k t \psi_n^{(k-1)}(t) - c^2 k \brk{k-1} \psi_n^{(k-2)}(t) = 0.
\label{eq_dpsi_reck}
\end{align}
\label{thm_dpsi_reck}
\end{thm}
%%%%%%%%%%%%%%%%%%%%%%
We refer to the roots of $\psi_n$, the roots of $\psi_n'$ and
the turning points of the ODE \eqref{eq_prolate_ode} as 
''special points''.
In the following theorem from \cite{Report}, \cite{ReportArxiv}, we describe
the location of some of the special points.
%%%%%%%%%%%%%%%%%%%%%%%%%%%
\begin{thm}[Special points]
Suppose that $n \geq 2$ is a positive integer. Suppose
also that $t_1<t_2<\dots$ are the roots of $\psi_n$ in $(-1,1)$,
and that $s_1<s_2<\dots$ are the roots of $\psi_n'$ in $(-1,1)$.
If $\chi_n < c^2$, then
\begin{align}
-1 < -\frac{\sqrt{\chi_n}}{c} < s_1 < t_1 < s_2 < \dots
< t_{n-1} < s_n < t_n < s_{n+1} < \frac{\sqrt{\chi_n}}{c} < 1
\label{eq_all_tx_small}
\end{align}
In particular, $\psi_n$ has $n$ roots in $(-1,1)$, and $\psi_n'$ has
$n+1$ roots in $(-1,1)$. 
On the other hand,
if $\chi_n > c^2$, then
\begin{align}
-\frac{\sqrt{\chi_n}}{c} < -1 < t_1 < s_1 < t_2 < \dots
< t_{n-1} < s_{n-1} < t_n < 1 < \frac{\sqrt{\chi_n}}{c}.
\label{eq_all_tx_large}
\end{align}
In particular, $\psi_n$ has $n$ roots in $(-1,1)$, and $\psi_n'$ has
$n-1$ roots in $(-1,1)$.
\label{thm_five}
\end{thm}
%%%%%%%%%%%%%%%%%%%%%%%%%%%%%%%%%%%%
In the following theorem, proven in \cite{Report}, \cite{ReportArxiv}, 
we describe
a relation between the magnitude of $\psi_n$ and $\psi_n'$
in the interval $(-1,1)$.
%%%%%%%%%%%
\begin{thm}
Suppose that $n \geq 0$ is a non-negative integer, and that
the functions $p, q: \Rc \to \Rc$ are defined via 
\eqref{eq_prufer_1} in Section~\ref{sec_prufer}. 
Suppose also that the functions 
$Q, \tilde{Q} : (0,\min\left\{\sqrt{\chi_n}/c,1\right\}) \to \Rc$ are defined,
respectively, via the formulae
\begin{align}
Q(t) 
& \; = \psi_n^2(t) + \frac{ p(t) }{ q(t) } \cdot \brk{\psi'_n(t)}^2 
 = \psi_n^2(t) + 
       \frac{ \brk{1-t^2} \cdot \brk{\psi'_n(t)}^2 }{ \chi_n-c^2 t^2}
\label{eq_q_old}
\end{align}
and
\begin{align}
\tilde{Q}(t) 
& \; = p(t) \cdot q(t) \cdot Q(t)  \nonumber \\
& \; = \brk{1 - t^2} \cdot \brk{ \brk{\chi_n - c^2 t^2} \cdot \psi_n^2(t) +
                           \brk{1 - t^2} \cdot \brk{\psi'_n(t)}^2 }.
\label{eq_qtilde_old}
\end{align}
Then, $Q$ is increasing in the interval
$\left(0, \min\left\{\sqrt{\chi_n}/c, 1\right\}\right)$,
and $\tilde{Q}$ is decreasing in the interval 
$\left(0, \min\left\{\sqrt{\chi_n}/c, 1\right\}\right)$.
\label{thm_Q_Q_tilde}
\end{thm}

%%%%%%%%%%%%%%%%%%%%%%%%%%%%%%%%%%%%%%%%%
\subsection{Legendre Polynomials and PSWFs}
\label{sec_legendre}
In this subsection, we list several well known
facts about Legendre polynomials and the relationship between
Legendre polynomials and PSWFs.
All of these facts can be found, for example, in 
\cite{Ryzhik},
\cite{RokhlinXiaoProlate},
\cite{Abramovitz}.

The Legendre polynomials $P_0, P_1, P_2, \dots$ are defined via
the formulae
\begin{align}
& P_0(t) = 1, \nonumber \\
& P_1(t) = t,
\label{eq_legendre_pol_0_1}
\end{align}
and the recurrence relation
\begin{align}
\brk{k+1} P_{k+1}(t) = \brk{2k+1} t P_k(t) - k P_{k-1}(t),
\label{eq_legendre_pol_rec}
\end{align}
for all $k = 1, 2, \dots$. The even-indexed Legendre
polynomials are even functions,
and the odd-indexed Legendre polynomials are odd functions.
The Legendre polynomials $\left\{P_k\right\}_{k=0}^{\infty}$ constitute a 
complete orthogonal system in $L^2\left[-1, 1\right]$. The normalized
Legendre polynomials are defined via the formula
\begin{align}
\overline{P_k}(t) = P_k(t) \cdot \sqrt{k + 1/2}, 
\label{eq_legendre_normalized}
\end{align}
for all $k=0,1,2,\dots$. The $L^2\left[-1,1\right]$-norm of each
normalized Legendre polynomial equals to one, i.e.
\begin{align}
\int_{-1}^1 \left( \overline{P_k}(t) \right)^2 \; dt = 1.
\label{eq_legendre_normalized_norm}
\end{align}
Therefore, the normalized Legendre polynomials constitute
an orthonormal basis for $L^2\left[-1, 1\right]$. 
In particular, for every real $c>0$
and every integer $n \geq 0$, the prolate spheroidal
wave function $\psi_n$, corresponding
to the band limit $c$, can be expanded into the series
\begin{align}
\psi_n(x) = \sum_{k = 0}^{\infty} \beta_k^{(n)} \cdot \overline{P_k}(x)
          = \sum_{k = 0}^{\infty} \alpha_k^{(n)} \cdot P_k(x),
\label{eq_num_leg_exp}
\end{align}
for all $-1 \leq x \leq 1$,
where $\beta_0^{(n)}, \beta_1^{(n)}, \dots$ are defined via
the formula
\begin{align}
\beta_k^{(n)} = \int_{-1}^1 \psi_n(x) \cdot \overline{P_k}(x) \; dx,
\label{eq_num_leg_beta_knc}
\end{align}
and $\alpha_0^{(n)}, \alpha_1^{(n)}, \dots$ are defined via
the formula
\begin{align}
\alpha_k^{(n)} = \beta_k^{(n)} \cdot \sqrt{k + 1/2},
\label{eq_num_leg_alpha_knc}
\end{align}
for all $k=0, 1, 2, \dots$. Due to the combination of
Theorem~\ref{thm_pswf_main} in Section~\ref{sec_pswf} with
\eqref{eq_legendre_normalized_norm},
\eqref{eq_num_leg_exp}, 
\eqref{eq_num_leg_beta_knc},
\begin{align}
\left( \beta^{(n)}_0 \right)^2 + 
\left( \beta^{(n)}_1 \right)^2 +
\left( \beta^{(n)}_2 \right)^2 + \dots = 1.
\label{eq_beta_unit_length}
\end{align}
The sequence $\beta_0^{(n)}, \beta_1^{(n)}, \dots$ satisfies
the recurrence relation
\begin{align}
A_{0,0} \cdot \beta_0^{(n)} + 
A_{0,2} \cdot \beta_2^{(n)} & = \chi_n \cdot \beta_0^{(n)}, 
\nonumber \\
A_{1,1} \cdot \beta_1^{(n)} + 
A_{1,3} \cdot \beta_3^{(n)} & = \chi_n \cdot \beta_1^{(n)}, 
\nonumber \\
A_{k,k-2} \cdot \beta_{k-2}^{(n)} + 
A_{k,k} \cdot \beta_k^{(n)} +
A_{k,k+2} \cdot \beta_{k+2}^{(n)} & = \chi_n \cdot \beta_k^{(n)}, 
\label{eq_num_a_rec}
\end{align}
for all $k=2,3,\dots$, where $A_{k,k}$, $A_{k+2,k}$, $A_{k,k+2}$ are
defined via the formulae
\begin{align}
& A_{k, k} = k(k+1) + \frac{ 2k(k+1) - 1 }{ (2k+3)(2k-1) } \cdot c^2, 
  \nonumber \\
& A_{k, k+2} = A_{k+2, k} =
\frac{ (k+2)(k+1) }{ (2k+3) \sqrt{(2k+1)(2k+5)} } \cdot c^2,
\label{eq_num_a_matrix}
\end{align}
for all $k=0,1,2,\dots$.
In other words,
the infinite vector $\left( \beta_0^{(n)}, \beta_1^{(n)}, \dots \right)$
satisfies the identity
\begin{align}
\brk{A - \chi_n I} \cdot 
\left( \beta_0^{(n)}, \beta_1^{(n)}, \dots \right)^T = 0,
\label{eq_num_beta}
\end{align}
where $I$ is the infinite identity matrix,
and the non-zero entries of the infinite symmetric matrix $A$ are given via
\eqref{eq_num_a_matrix}. 
 
The matrix $A$ naturally splits into two infinite 
symmetric tridiagonal matrices,
$A^{even}$ and $A^{odd}$, the former consisting of the elements of $A$
with even-indexed rows and columns, and the latter consisting of
the elements of $A$ with odd-indexed rows and columns.
Moreover, 
for every pair of integers $n, k \geq 0$,
\begin{align}
\beta^{(n)}_k = 0, \quad \text{if } k+n \text{ is odd},
\label{eq_beta_zero_parity}
\end{align}
due to the combination of
Theorem~\ref{thm_pswf_main} in Section~\ref{sec_pswf} and 
\eqref{eq_num_leg_beta_knc}. In the following theorem (that appears
in \cite{RokhlinXiaoProlate} in a slightly different form), we summarize
the implications of these observations to the identity \eqref{eq_num_beta},
that lead to numerical algorithms for the evaluation of PSWFs.
\begin{thm}
Suppose that $c>0$ is a real number, and that the infinite
tridiagonal symmetric matrices $A^{even}$ and $A^{odd}$
are defined, respectively, via
\begin{align}
A^{even} =
\begin{pmatrix}
A_{0,0} & A_{0,2} &  &  &  \\
A_{2,0} & A_{2,2} & A_{2,4} &  &  \\
        & A_{4,2} & A_{4,4} & A_{4,6} & \\
        &         & \ddots  & \ddots & \ddots \\
\end{pmatrix}
\label{eq_a_even}
\end{align}
and
\begin{align}
A^{odd} =
\begin{pmatrix}
A_{1,1} & A_{1,3} &  &  &  \\
A_{3,1} & A_{3,3} & A_{3,5} &  &  \\
        & A_{5,3} & A_{5,5} & A_{5,7} & \\
        &         & \ddots  & \ddots & \ddots \\
\end{pmatrix},
\label{eq_a_odd}
\end{align}
where the entries $A_{k,j}$ are defined via \eqref{eq_num_a_matrix}. Suppose
also that the unit length
infinite vector $\beta^{(n)} \in l^2$ is defined via 
the formula
\begin{align}
\beta^{(n)} =
\begin{cases}
\left( \beta^{(n)}_0, \beta^{(n)}_2, \dots \right)^T & $n$ \text{ is even}, \\
\left( \beta^{(n)}_1, \beta^{(n)}_3, \dots \right)^T & $n$ \text{ is odd},
\end{cases}
\label{eq_beta_n}
\end{align}
where $\beta^{(n)}_0, \beta^{(n)}_1, \dots$ are defined via
\eqref{eq_num_leg_beta_knc}. If $n$ is even, then
\begin{align}
A^{even} \cdot \beta^{(n)} = \chi_n \cdot \beta^{(n)}.
\label{eq_a_even_eig}
\end{align}
If $n$ is odd, then
\begin{align}
A^{odd} \cdot \beta^{(n)} = \chi_n \cdot \beta^{(n)}.
\label{eq_a_odd_eig}
\end{align}
\label{thm_tridiagonal}
\end{thm}
\begin{remark}
While the matrices $A^{even}$ and $A^{odd}$ are infinite, and their entries
do not decay with increasing row or column number, the coordinates
of each eigenvector $\beta^{(n)}$ decay superexponentially fast
(see e.g. \cite{RokhlinXiaoProlate} for estimates of this decay).
In particular, suppose that we need to evaluate 
the first $n+1$ eigenvalues $\chi_0, \dots,
\chi_{n}$ and the corresponding eigenvectors 
$\beta^{(0)}, \dots, \beta^{(n)}$ numerically. Then, we can replace
the matrices $A^{even}, A^{odd}$ in \eqref{eq_a_even_eig}, 
\eqref{eq_a_odd_eig}, respectively, with their $N \times N$ upper left
square submatrices, where $N$ is of order $n$,
and solve the resulting symmetric tridiagonal
eigenproblem by any standard technique (see, for example,
\cite{Wilkinson}, \cite{Dahlquist}; see also \cite{RokhlinXiaoProlate}
for more details about this numerical algorithm).
The cost of this algorithm is $O(n^2)$ operations.
\label{rem_tridiagonal}
\end{remark}

The Legendre functions of the second kind $Q_0, Q_1, Q_2, \dots$ 
are defined via the formulae
\begin{align}
&  Q_0(t) = \frac{1}{2} \log \frac{1+t}{1-t}, \nonumber \\
& Q_1(t) = \frac{t}{2} \log \frac{1+t}{1-t} - 1,
\label{eq_legendre_fun_0_1}
\end{align}
and the recurrence relation
\begin{align}
\brk{k+1} Q_{k+1}(t) = \brk{2k+1} t Q_k(t) - k Q_{k-1}(t),
\label{eq_legendre_fun_rec}
\end{align}
for all $k = 1, 2, \dots$. 
In particular,
\begin{align}
& Q_2(t) = \frac{3t^2 - 1}{4} \log \frac{1+t}{1-t} - \frac{3}{2} t, 
\nonumber \\
& Q_3(t) = \frac{5t^3 - 3t}{4} \log \frac{1+t}{1-t} - \frac{5}{2} t^2 + 
  \frac{2}{3}.
\label{eq_legendre_fun_2_3}
\end{align}
We observe that the recurrence relation
\eqref{eq_legendre_fun_rec} is the same as the recurrence relation
\eqref{eq_legendre_pol_rec}, satisfied by the Legendre polynomials.
It follows from \eqref{eq_legendre_pol_rec}, 
\eqref{eq_legendre_fun_rec}, that both the Legendre polynomials
$P_0, P_1, \dots$ and the Legendre functions of the second kind
$Q_0, Q_1, \dots$ satisfy another
recurrence relation, namely
\begin{align}
& t^2 P_k(t) = A_{k-2} P_{k-2}(t) + B_k P_k(t) + C_{k+2} P_{k+2}(t),
\nonumber \\
& t^2 Q_k(t) = A_{k-2} Q_{k-2}(t) + B_k Q_k(t) + C_{k+2} Q_{k+2}(t),
\label{eq_num_rk_rec}
\end{align}
for all $k = 2, 3, \dots$,
where
\begin{align}
\label{eq_a_k_rec}
A_k & \; = \frac{ (k+1)(k+2) }{ (2k+3)(2k+5) }, \\
\label{eq_b_k_rec}
B_k & \; = \frac{ 2k(k+1) - 1 }{ (2k+3)(2k-1) }, \\
\label{eq_c_k_rec}
C_k & \; = \frac{ k(k-1) }{ (2k-3)(2k-1) }.
\end{align}
In addition, for every integer $k = 0, 1, 2, \dots$, the $k$th Legendre
polynomial $P_k$ and the $k$th Legendre function of the second
kind $Q_k$ are two independent solutions of the second order Legendre
differential equation
\begin{align}
(1 - t^2) \cdot y''(t) - 2t \cdot y'(t) + k(k+1) \cdot y(t) = 0.
\label{eq_legendre_ode}
\end{align}
Also, for every integer $k=0, 1, \dots$ and all 
complex $z$ such that $\arg \brk{z - 1} < \pi$,
\begin{align}
Q_k(z) = \frac{1}{2} \int_{-1}^1 \frac{ P_k(t) }{z - t} \; dt
\label{eq_qn_pn}
\end{align}
(see, for example, Section 8.82 of \cite{Ryzhik}).
\begin{remark}
For any real number $-1 < x < 1$ and integer $n \geq 0$, 
we can use the three-term recurrences
\eqref{eq_legendre_pol_rec}, \eqref{eq_legendre_fun_rec}
to evaluate numerically
$P_0(x), \dots, P_n(x)$ and
$Q_0(x), \dots, Q_n(x)$ with high precision, in
$O(n)$ operations (see, for example,
\cite{Dahlquist} for more details).
\label{rem_legendre_evaluate}
\end{remark}

%%%%%%%%%%%%%%%%%%%%%%%%%%%%%%%%%%%%%%%%%%%%%
\subsection{Elliptic Integrals}
\label{sec_elliptic}
In this subsection, we summarize several facts about
elliptic integrals. These facts can be found, for example,
in section 8.1 in \cite{Ryzhik}, and in \cite{Abramovitz}.

The incomplete elliptic integrals of the first and second kind
are defined, respectively, by the formulae
\begin{align}
\label{eq_F_y}
& F(y, k) =  \int_0^y \frac{dt}{\sqrt{1 - k^2 \sin^2 t}}, \\
& E(y, k) = \int_0^y \sqrt{1 - k^2 \sin^2 t} \; dt,
\label{eq_E_y}
\end{align}
where $0 \leq y \leq \pi/2$ and $0 \leq k \leq 1$.
By performing the substitution $x = \sin t$, we can write 
\eqref{eq_F_y} and \eqref{eq_E_y} as
\begin{align}
& F(y, k) = \int_0^{\sin(y)}
  \frac{ dx }{ \sqrt{\brk{1 - x^2} \brk{1 - k^2 x^2} } },
\label{eq_F_y_2} \\
\nonumber \\
& E(y, k) = \int_0^{\sin(y)}
\sqrt{ \frac{1 - k^2 x^2}{1 - x^2} } \; dx.
\label{eq_E_y_2}
\end{align}
The complete elliptic integrals of the first and second kind are
defined, respectively, by the formulae
\begin{align}
\label{eq_F}
& F(k) = F\brk{\frac{\pi}{2}, k} = 
\int_0^{\pi/2} \frac{dt}{\sqrt{1 - k^2 \sin^2 t}}, \\
& E(k) = E\brk{\frac{\pi}{2}, k} =
\int_0^{\pi/2} \sqrt{1 - k^2 \sin^2 t} \; dt,
\label{eq_E}
\end{align}
for all $0 \leq k \leq 1$. Moreover,
\begin{align}
E\left( \sqrt{1-k^2} \right) =
1 + \left(-\frac{1}{4}+\log(2)-\frac{\log(k)}{2}\right) \cdot k^2 +
    O\left( k^4 \cdot \log(k) \right).
\label{eq_E_exp}
\end{align}

%%%%%%%%%%%%%%%%%%%%%%%%%%%%%%%%%%%%%%%%%%%%%
\subsection{Oscillation Properties of Second Order ODEs}
\label{sec_oscillation_ode}
In this subsection, we state several well known facts from the general theory
of second order ordinary differential equations (see e.g. \cite{Miller}).

The following two theorems appear in Section 3.6 of \cite{Miller}
in a slightly different form.
%%%%%%%%%%%
\begin{thm}[distance between roots]
Suppose that $h(t)$ is a solution of the ODE
\begin{align}
y''(t) + Q(t) \cdot y(t) = 0.
\label{eq_25_09_ode}
\end{align}
Suppose also that $x < y$ are two consecutive roots of $h(t)$, and
that
\begin{align}
A^2 \leq Q(t) \leq B^2, 
\label{eq_25_09_cond}
\end{align}
for all $x \leq t \leq y$.
Then,
\begin{align}
\frac{\pi}{B} < y - x < \frac{\pi}{A}.
\label{eq_25_09_thm}
\end{align}
\label{thm_25_09}
\end{thm}
%%%%%%%%%
%%%%%%%%%%%
\begin{thm}
Suppose that $a<b$ are real numbers, and that $g:(a,b) \to \Rc$
is a continuous monotone function.
Suppose also that
$y(t)$ is a solution of the ODE
\begin{align}
 y''(t) + g(t) \cdot y(t) = 0,
\label{eq_25_09_ode2}
\end{align}
in the interval $(a,b)$.
Suppose furthermore that
\begin{align}
t_1 < t_2 < t_3 < \dots
\end{align}
are consecutive roots of $y(t)$. 
If $g$ is non-decreasing, then
\begin{align}
t_2 - t_1 \geq t_3 - t_2 \geq t_4 - t_3 \geq \dots.
\label{eq_shrink}
\end{align}
If $g$ is non-increasing, then
\begin{align}
t_2-t_1 \leq t_3-t_2 \leq t_4-t_3 \leq \dots.
\label{eq_stretch}
\end{align}
\label{thm_25_09_2}
\end{thm}
%%%%%%%%%
\noindent
The following theorem is a special case
of Theorem 6.2 from Section 3.6 in \cite{Miller}:
\begin{thm}
Suppose that $g_1, g_2$ are continuous
functions, and that, for all real $t$ in the interval $\brk{a, b}$,
the inequality $g_1(t) < g_2(t)$ holds. 
Suppose also that the function $\phi_1, \phi_2$ satisfy,
for all $a < t < b$,
\begin{align}
& \phi_1''(t) + g_1(t) \cdot \phi_1(t) = 0, \nonumber \\
& \phi_2''(t) + g_2(t) \cdot \phi_2(t) = 0.
\label{eq_08_12_g12}
\end{align}
Then, $\phi_2$ has a root 
between every two consecutive roots of $\phi_1$.
\label{thm_08_12_zeros}
\end{thm}
\begin{cor}
Suppose that the functions $\phi_1, \phi_2$
are those of Theorem~\ref{thm_08_12_zeros} above.
Suppose also that
\begin{align}
\phi_1(t_0) = \phi_2(t_0), \quad \phi_1'(t_0) = \phi_2'(t_0),
\end{align}
for some $a < t_0 < b$.
Then, $\phi_2$ has at least as many roots in $\brk{t_0, b}$ as $\phi_1$.
\label{cor_08_12_zeros}
\end{cor}
\begin{proof}
By Theorem~\ref{thm_08_12_zeros}, we only need to show
that if $t_1$ is the minimal root of $\phi_1$ in $\brk{t_0, b}$,
then there exists a root of $\phi_2$ in $\brk{t_0, t_1}$.
By contradiction, suppose that this is not the case. In addition, without
loss of generality, suppose that $\phi_1(t), \phi_2(t)$ are positive
in $\brk{t_0, t_1}$. Then, due to \eqref{eq_08_12_g12},
\begin{align}
\phi_1'' \phi_2 - \phi_2'' \phi_1 = \brk{g_2 - g_1} \phi_1 \phi_2,
\end{align}
and hence
\begin{align}
0 & \; <
\int_{t_0}^{t_1} \brk{g_2(s) - g_1(s)} \phi_1(s) \phi_2(s) ds 
\nonumber \\
& \; = 
\left[\phi_1'(s) \phi_2(s) - \phi_1(s) \phi_2'(s)\right]_{t_0}^{t_1} 
\nonumber \\
& \; =
\phi_1'(t_1) \phi_2(t_1) \leq 0,
\end{align}
which is a contradiction.
\end{proof}
%%%%%%%%%%%%

%%%%%%%%%%%%%%%%%%%%%%%%%%%%%%%%%%%%%%%%%%%%%
\subsection{Growth Properties of Second Order ODEs}
\label{sec_growth_ode}
The following theorem appears in \cite{Le52} in a more general form.
We provide a proof for the sake of completeness.

%%%%%%%%%%%
\begin{thm}
\label{thm_lewis}
Suppose that $a<b$ are real numbers, and that the functions
$w, u, \beta, \gamma: (a,b) \to \Cc$ are continuously differentiable.
Suppose also that, for all real $a < t < b$,
\begin{align}
\label{eq_wu_ode}
\begin{pmatrix} w'(t) \\ u'(t) \end{pmatrix} =
\begin{pmatrix} 0 & \beta(t) \\ \gamma(t) & 0 \end{pmatrix}
\begin{pmatrix} w(t) \\ u(t) \end{pmatrix},
\end{align}
and that
\begin{align}
\beta(t) \neq 0, \quad \gamma(t) \neq 0,
\end{align}
for all $a < t < b$. Suppose furthermore that the functions
$R, Q: (a,b) \to \Rc$ are defined, respectively, via the formulae
\begin{align}
R(t) = \frac{ \abrk{ \beta(t) } }{ \abrk{ \gamma(t) } }
\label{eq_r_lewis}
\end{align}
and
\begin{align}
\label{eq_wu_q}
Q(t) = \abrk{ w\brk{t} }^2 + R(t) \cdot \abrk{ u\brk{t} }^2.
\end{align}
Then, for all real $a < t_0, t < b$,
\begin{align}
& \brk{ \frac{R(t)}{R(t_0)} }^{\frac{1}{4}}
\exp\left[ -\int_{t_0}^t 
    \brk{ \brk{\frac{R'(s)}{4R(s)}}^2 + 
          \frac{ \abrk{\beta(s)}\abrk{\gamma(s)} + 
                 \Re\left(\beta(s)\gamma(s)\right) }{ 2 } }^{\frac{1}{2}} ds 
 \right]
\nonumber \\
& \leq \sqrt{ \frac{Q(t)}{Q(t_0)} } \leq \nonumber \\
& \brk{ \frac{R(t)}{R(t_0)} }^{\frac{1}{4}}
\exp\left[ \int_{t_0}^t 
    \brk{ \brk{\frac{R'(s)}{4R(s)}}^2 + 
          \frac{ \abrk{\beta(s)}\abrk{\gamma(s)} + 
                 \Re\left(\beta(s)\gamma(s)\right) }{ 2 } }^{\frac{1}{2}} ds 
\right].
\label{eq_wu_main_ineq}
\end{align}
\end{thm}
\begin{proof}
We note that, for a each fixed $t$, the formula \eqref{eq_wu_q} 
can be written in the matrix notation as
\begin{align}
Q(t) =
\begin{pmatrix} \bar{w}(t) & \bar{u}(t) \end{pmatrix}
\begin{pmatrix} 1 & 0 \\ 0 & R(t) \end{pmatrix}
\begin{pmatrix} w(t) \\ u(t) \end{pmatrix}.
\label{eq_q_lewis}
\end{align}
We differentiate $Q(t)$ with respect to $t$ to obtain, 
by using \eqref{eq_wu_ode},
\begin{align}
Q'(t) & \; = w'(t)\bar{w}(t) + w(t)\bar{w}'(t) + 
             R(t)\bar{u}(t)u'(t) + R(t)\bar{u}'(t)u(t) + R'(t)u(t)\bar{u}(t)
\nonumber \\
   & \; = \beta(t) u(t) \bar{w}(t) + \bar{\beta}(t) \bar{u}(t) w(t) +
        R(t) \gamma(t) w(t) \bar{u}(t) + R(t) \bar{\gamma}(t) \bar{w}(t) u(t) + 
          R'(t)u(t)\bar{u}(t)
\nonumber \\
   & \; = \begin{pmatrix} \bar{w}(t) & \bar{u}(t) \end{pmatrix}
          \begin{pmatrix} 0 & \beta(t) + R(t) \bar{\gamma}(t) \\
                          \bar{\beta}(t) + R(t) \gamma(t) & R'(t) \end{pmatrix}
          \begin{pmatrix} w(t) \\ u(t) \end{pmatrix}.
\label{eq_dq_lewis}
\end{align}
Then, we define the functions $x, y: (a,b) \to \Rc$ via the formulae
\begin{align}
& x(t) = w(t), \\
& y(t) = u(t) \cdot \sqrt{R(t)}.
\label{eq_xy_lewis}
\end{align}
We substitute \eqref{eq_xy_lewis} into \eqref{eq_q_lewis},
\eqref{eq_dq_lewis} to obtain
\begin{align}
\label{eq_dQ_Q}
& \frac{Q'(t)}{Q(t)} = \nonumber \\
& \begin{pmatrix} \bar{x}(t) & \bar{y}(t) \end{pmatrix}
\begin{pmatrix} 0 & \frac{\beta(t) + R(t)\bar{\gamma}(t)}{\sqrt{R(t)}} \\
                \frac{\bar{\beta}(t) + R(t)\gamma(t)}{\sqrt{R(t)}} & 
                \frac{R'(t)}{R(t)}
\end{pmatrix} \begin{pmatrix} x(t) \\ y(t) \end{pmatrix} \cdot
\frac{1}{|x(t)|^2 + |y(t)|^2}.
\end{align}
To find the eigenvalues of the matrix in \eqref{eq_dQ_Q},
we solve, for each $a < t < b$, the quadratic equation
\begin{align}
\lambda^2 - \frac{R'(t)}{R(t)} \cdot \lambda -
\left( \frac{\beta(t) + R(t) \bar{\gamma}(t)}{\sqrt{R(t)}} \right) \cdot
\left( \frac{\bar{\beta}(t) + R(t) \gamma(t)}{\sqrt{R(t)}} \right) = 0,
\label{eq_quadr_lewis}
\end{align}
in the unknown $\lambda$. Suppose that $\lambda_1(t) < \lambda_2(t)$
are the roots of \eqref{eq_quadr_lewis} for a fixed $a < t < b$.
We use \eqref{eq_r_lewis} to obtain
\begin{align}
& \lambda_1(t) = \frac{R'(t)}{2R(t)} -
\left[ \brk{\frac{R'(t)}{2R(t)}}^2 + 2\brk{\abrk{\beta(t)}\abrk{\gamma(t)} + 
        \Re\left(\beta(t) \gamma(t)\right) } \right]^{\frac{1}{2}}, \nonumber \\
& \lambda_2(t) = \frac{R'(t)}{2R(t)} +
\left[ \brk{\frac{R'(t)}{2R(t)}}^2 + 2\brk{\abrk{\beta(t)}\abrk{\gamma(t)} + 
       \Re\left(\beta(t) \gamma(t)\right)} \right]^{\frac{1}{2}}.
\label{eq_lambdas_lewis}
\end{align}
Due to \eqref{eq_dQ_Q}, for all $a < t < b$,
\begin{align}
\lambda_1(t) \leq \frac{Q'(t)}{Q(t)} \leq \lambda_2(t).
\label{eq_ineq_lewis}
\end{align}
We substitute \eqref{eq_lambdas_lewis} into \eqref{eq_ineq_lewis},
integrate it from $t_0$ to $t$ and exponentiate the result
to obtain \eqref{eq_wu_main_ineq}.
\end{proof}

%%%%%%%%%%%%%%%%%%%%%%%%%%%%%%%%%%%%%
\subsection{Pr\"ufer Transformations}
\label{sec_prufer}

In this subsection, we describe the classical Pr\"ufer
transformation of a second order ODE 
(see e.g. \cite{Miller},\cite{Fedoryuk}).
Also, we describe a modification of Pr\"ufer transformation,
introduced in \cite{Glaser} and used in the rest of the paper.

Suppose that we are given the second order ODE
\begin{align}
\frac{ d }{ dt } \brk{ p(t) u'(t) } + q(t) u(t) = 0,
\label{eq_prufer_0}
\end{align}
where $t$ varies over some interval $I$ in which $p$ and $q$
are continuously differentiable and have no roots. 
We define the function 
$\theta: I \to \Rc$ via 
\begin{align}
\frac{ p(t) u'(t) }{ u(t) } =
\gamma(t) \tan \theta(t),
\label{eq_general_prufer}
\end{align}
where $\gamma: I \to \Rc$ is an arbitrary 
positive continuously differentiable function.
The function $\theta(t)$ satisfies, for all $t$ in $I$,
\begin{align}
\theta'(t) = -\frac{\gamma(t)}{p(t)} \sin^2 \theta(t)
             -\frac{q(t)}{\gamma(t)} \cos^2 \theta(t)
             -\brk{ \frac{\gamma'(t)}{\gamma(t)}} 
              \frac{\sin\brk{2\theta(t)}}{2}.
\label{eq_general_dtheta}
\end{align}
One can observe that if $u'(\tilde{t}) = 0$ for $\tilde{t} \in I$, then 
by \eqref{eq_general_prufer} 
\begin{align}
\theta(\tilde{t}) = k \pi, \quad
k \text{ is integer}.
\label{eq_prufer_du_zero}
\end{align}
Similarly, if $u(\tilde{t}) = 0$ for $\tilde{t} \in I$, then
\begin{align}
\theta(\tilde{t}) = \brk{k + 1/2} \pi, \quad
k \text{ is integer}.
\label{eq_prufer_u_zero}
\end{align}
The choice $\gamma(t) = 1$ in \eqref{eq_general_prufer}
gives rise to the classical Pr\"ufer transformation (see e.g. 
section 4.2 in \cite{Miller}).

In \cite{Glaser}, the choice $\gamma(t) = \sqrt{q(t) p(t)}$ is
suggested and shown to be more convenient numerically in 
several applications. In this paper, this choice also leads to a
more convenient analytical tool than the classical Pr\"ufer transformation.

Writing \eqref{eq_prolate_ode} in the form of \eqref{eq_prufer_0} yields
\begin{align}
p(t) = t^2-1, \quad q(t) = c^2 t^2 - \chi_n,
\label{eq_prufer_1}
\end{align}
for all real $t > \max\left\{\sqrt{\chi_n}/c, 1\right\}$.
The equation \eqref{eq_general_prufer} admits the form
\begin{align}
\frac{ p(t) \psi_n'(t) }{ \psi_n(t) } =
\sqrt{ p(t) q(t) } \tan \theta(t),
\label{eq_modified_prufer}
\end{align}
which implies that
\begin{align}
\theta(t) = \text{atan} \brk{ \sqrt{\frac{p(t)}{q(t)}} 
                        \frac{\psi_n'(t)}{\psi_n(t)} } + \pi m(t),
\label{eq_modified_prufer_theta}
\end{align}
where $m(t)$ is an integer determined for all $t$ by an arbitrary choice
at some $t = t_0$ (the role of $\pi m(t)$ in
\eqref{eq_modified_prufer_theta} is to enforce the continuity
of $\theta$ at the roots of $\psi_n$).
The first order ODE \eqref{eq_general_dtheta} admits the form
(see \cite{Glaser}, \cite{Fedoryuk})
\begin{align}
\theta'(t) = -f(t) - \sin \brk{2 \theta(t)} v(t),
\label{eq_jan_prufer_fv}
\end{align}
where the functions $f, v$ are defined, respectively, via the formulae
\begin{align}
f(t) = \sqrt{ \frac{ q(t) }{ p(t) } }
     = \sqrt{ \frac{ c^2 t^2 - \chi_n}{ t^2 - 1 } }
\label{eq_jan_f}
\end{align}
and
\begin{align}
v(t) =
\frac{1}{4} \cdot \frac{ p(t)q'(t) + q(t)p'(t) }{ p(t) q(t) } 
= 
\frac{1}{2} \brk{ \frac{t}{t^2-1} + \frac{c^2 t}{c^2 t^2 - \chi_n} }.
\label{eq_jan_v}
\end{align}
%%%%%%%%%%%%%%%%%%%%%%%%%%%%%%%%%%%%
\begin{remark}
In this paper, the variable $t$ in \eqref{eq_modified_prufer},
\eqref{eq_modified_prufer_theta}, \eqref{eq_jan_prufer_fv}
will be confined to the open ray
\begin{align}
\left( \max\left\{1,\sqrt{\chi_n}/c\right\}, \infty \right).
\end{align}
Nevertheless, a similar analysis is possible for $t$ in the interval
\begin{align}
\left( -\min\left\{1,\sqrt{\chi_n}/c\right\}, 
\min\left\{1,\sqrt{\chi_n}/c\right\}
\right).
\end{align}
The following theorem from \cite{Report}, \cite{ReportArxiv}, 
summarizes such analysis
for the case $\chi_n > c^2$.
\label{rem_prufer}
\end{remark}
%%%%%%%%%%%
\begin{thm}
Suppose that $n \geq 2$ is a positive integer, and that
$\chi_n > c^2$. Suppose
also that $t_1,\dots,t_n$ are the roots of $\psi_n$ in $(-1,1)$,
and $s_1,\dots,s_{n-1}$ are the roots of $\psi_n'$ in $(-1,1)$
(see Theorem~\ref{thm_five} in Section~\ref{sec_pswf}).
Suppose furthermore that
the function $\theta: \left[-1,1\right] \to \Rc$ is
defined via the formula
\begin{align}
\theta(t) = 
\begin{cases}
\left(i-\frac{1}{2}\right) \cdot \pi, & \text{ if } t = t_i
\text{ for some } 1 \leq i \leq n, \\
\\
\text{atan}\left( -\sqrt{\frac{1-t^2}{\chi_n-c^2t^2}} \cdot
                        \frac{\psi_n'(t)}{\psi_n(t)} \right)
          + m(t) \cdot \pi, & \text{ if } \psi_n(t) \neq 0, \\
\end{cases}
\label{eq_prufer_theta_old}
\end{align}
where $m(t)$ is the number of the roots
of $\psi_n$ in the interval $(-1,t)$.
Then, $\theta$ has
the following properties:
\begin{itemize}
\item $\theta$ is continuously differentiable in 
the interval $\left[-1,1\right]$.
\item $\theta$ satisfies, for all $-1 < t < 1$, the
differential equation
\begin{align}
\theta'(t) = f(t) - v(t) \cdot \sin( 2 \theta(t) ),
\label{eq_prufer_theta_ode_old}
\end{align}
where the functions $f,v$ are defined, respectively, 
via \eqref{eq_jan_f}, \eqref{eq_jan_v} in
Section~\ref{sec_prufer}.
\item for each integer $0 \leq k \leq 2n$, there is a unique solution
to the equation
\begin{align}
\theta(t) = k \cdot \frac{\pi}{2},
\label{eq_prufer_sm2_equation_old}
\end{align}
for the unknown $t$ in $\left[-1,1\right]$.
More specifically,
\begin{align}
\label{eq_theta_at_xn_old}
& \theta(-1) = 0, \\
\label{eq_theta_at_t_old}
& \theta(t_i) = \left(i-\frac{1}{2}\right) \cdot \pi, \\
\label{eq_theta_at_x_old}
& \theta(s_j) = j \cdot \pi, \\
\label{eq_theta_at_1_old}
& \theta(1) = n \cdot \pi,
\end{align}
for each $i = 1, \dots, n$ and each $j = 1, \dots, n-1$. 
\item For all real $-1 < t < 1$,
\begin{align}
\theta'(t) > 0.
\label{eq_dtheta_pos_old}
\end{align}
In other words, $\theta$ is monotonically increasing.
\end{itemize}
\label{thm_prufer_old}
\end{thm}
%%%%%%%%%%%

%%%%%%%%%%%%%%%%%%%%%%%%%%%%
\subsection{Numerical Tools}
In this subsection, we summarize several numerical techniques
to be used in this paper.

%%%%%%%%%%%%%%%%%%%%%%%%%%%%%%%
\subsubsection{Newton's Method}
\label{sec_newton}
Newton's method solves the equation $f(x) = 0$ iteratively given
an initial approximation $x_0$ of the root $\tilde{x}$.
The $n$th iteration is defined by
\begin{align}
x_n = x_{n-1} - \frac{ f(x_{n-1}) }{ f'(x_{n-1}) }.
\label{eq_newton}
\end{align}
The convergence is quadratic provided that $\tilde{x}$ is a simple
root and $x_0$ is close enough to $\tilde{x}$. More details
can be found e.g. in \cite{Dahlquist}.

%%%%%%%%%%%%%%%%%%%%%%%%%%%%%%%%%%%%%%%%%%%%%%%%%%%%%%%%%%%%%%%%%
\subsubsection{The Taylor Series Method for the Solution of ODEs}
\label{sec_taylor}
The Taylor series method for the solution of a linear second order
differential equation is based on the Taylor formula
\begin{align}
u(x+h) = \sum_{j = 0}^k \frac{ u^{(j)}(x) }{ j! } h^j + O(h^{k+1}).
\label{eq_taylor}
\end{align}
This method evaluates $u(x+h)$ and $u'(x+h)$ by using \eqref{eq_taylor}
and depends on the ability to compute $u^{(j)}(x)$ for $j = 0, \dots, k$.
When the latter satisfy a simple recurrence relation like \eqref{eq_dpsi_reck}
and hence can be computed in $O(k)$ operations, this method is particularly
useful. The reader is referred to \cite{Glaser} for further details.

%%%%%%%%%%%%%%%%%%%%%%%%%%%%%%%%%%%%%%%%%%%%%%%%%
\subsubsection{A Second Order Runge-Kutta Method}
\label{sec_runge_kutta}
We use the following second order Runge-Kutta Method, which
can be found, for example, in \cite{Dahlquist}. It solves
the initial value problem
\begin{align}
y(t_0) = y_0, \quad y'(t) = f(t, y)
\label{eq_rk_ivp}
\end{align}
on the interval $t_0 \leq t \leq t_0 + L$ by computing
\begin{align}
& t_{i+1} = t_i + h, \nonumber \\
& k_{i+1} = h f\brk{t_{i+1}, y_i + k_i}, \nonumber \\
& y_{i+1} = y_i + \brk{k_i + k_{i+1}}/2
\label{eq_runge_kutta}
\end{align}
with $i = 0, \dots, n$ and
\begin{align}
h = \frac{ L }{ n }, \quad k_0 = f(t_0, y_0).
\label{eq_rk_h_k0}
\end{align}
Exactly $n+1$ evaluations of $f$ are required for this algorithm,
which results in the total cost being $O(n)$. The global truncation
error is $O(h^2)$.

%%%%%%%%%%%%%%%%%%%%%%%%%%%%%%%%%%%%%%%%%%%%%%%
\subsubsection{Power and Inverse Power Methods}
\label{sec_power}
The methods described in this subsection are widely known and can be found,
for example, in \cite{Dahlquist}. Suppose
that $A$ is an $n \times n$
real symmetric matrix, whose eigenvalues satisfy
\begin{align}
\abrk{ \sigma_1 } > \abrk{ \sigma_2 } \geq \abrk{ \sigma_3 } \geq \dots
\geq \abrk{ \sigma_n }.
\label{eq_power_sigma}
\end{align}
The Power Method approximates $\sigma_1$ and the corresponding unit eigenvector
in the following way.
\begin{itemize}
\item Set $v_0$ to be a random vector in $\Rc^n$ such that 
$\| v_0 \| = \sqrt{v_0^T v_0} = 1$.
\item Set $j = 1$ and $\eta_0 = 0$.
\item Compute $\hat{v}_j = A v_{j-1}$.
\item Set $\eta_j = v_{j-1}^T \hat{v}_j$.
\item Set $v_j = \hat{v}_j / \| \hat{v}_j \|$.
\item If $\abrk{\eta_j - \eta_{j-1}}$ is ``sufficiently small'', stop.
\item Otherwise, set $j = j + 1$ and repeat the iteration.
\end{itemize}
The output value $\eta_j$ approximates $\sigma_1$, and $v_j$
approximates a unit eigenvector corresponding to $\sigma_1$. The cost
of each iteration is dominated by the cost of evaluating $A v_{j-1}$.
The rate of convergence of the algorithm is linear and equals
to $\abrk{\sigma_2} / \abrk{\sigma_1}$, that is, the error after $j$
iterations is of order $\brk{ \abrk{\sigma_2} / \abrk{\sigma_1} }^j$.
\begin{remark}
A modification of the algorithm used in this paper defines $\eta_j$ by
\begin{align}
i = \mbox{\text{\rm{argmax}}} 
\left\{ \abrk{v_{j-1}(k)} \; : \; k = 1, \dots, n \right\}, 
\quad
\eta_j = \frac{ \hat{v}_j(i) }{ v_{j-1}(i) }.
\end{align}
\end{remark}
The Inverse Power Method finds the eigenvalue $\sigma_k$ of $A$ and a
corresponding unit eigenvector provided that an approximation $\sigma$
of $\sigma_k$ is known such that 
\begin{align}
\abrk{ \sigma - \sigma_k } < \max \left\{
   \abrk{ \sigma - \sigma_j } \; : \; j \neq k \right\}.
\label{eq_inverse_power_sigma}
\end{align}
Conceptually, the Inverse Power Method is an application of the Power
Method on the matrix $B = \brk{A - \sigma I}^{-1}$. In practice,
$B$ need not be evaluated explicitly and it suffices to be able to solve
the linear system of equations
\begin{align}
\brk{A - \sigma I} \hat{v}_j = v_{j-1}
\label{eq_inverse_power_lin}
\end{align}
for the unknown $\hat{v}_j$ on each iteration of the algorithm. 
\begin{remark}
\label{rem_power}
If
the matrix $A$ is tridiagonal, the system \eqref{eq_inverse_power_lin}
can be solved in $O(n)$ operations, for example, by means of
Gaussian elimination or QR decomposition (see e.g
\cite{Wilkinson}, \cite{Dahlquist}).
\end{remark}

%%%%%%%%%%%%%%%%%%%%%%%%%%%%%%
\subsubsection{Sturm Sequence}
\label{sec_sturm}
The following theorem can be found, for example, in \cite{Wilkinson}
(see also \cite{Wilkinson2}).
It provides the basis for an algorithm of evaluating the $k$th smallest
eigenvalue of a symmetric tridiagonal matrix.
%%%%%%%%%%%
\begin{thm}[Sturm sequence]
Suppose that
\begin{align}
C = 
\begin{pmatrix} 
a_1 & b_2 & 0 & \cdots & \cdots & 0 \\
b_2 & a_2 & b_3 & 0 & \cdots & 0 \\
\vdots & \ddots & \ddots & \ddots & \ddots & \vdots \\
0 & \cdots & 0 & b_{n-1} & a_{n-1} & b_n \\
0 & \cdots & \cdots & 0 & b_n & a_n
\end{pmatrix}
\label{eq_sturm_c}
\end{align}
is a symmetric tridiagonal matrix such that none of $b_2, \dots, b_n$
is zero.
Then, its $n$ eigenvalues satisfy
\begin{align}
\sigma_1(C) < \dots < \sigma_n(C).
\label{eq_sturm_lambda}
\end{align}
Suppose also that
$C_k$ is the $k \times k$ leading principal submatrix
of $C$,
for every integer $k=1,\dots,n$.
We define the polynomials $p_{-1}, p_0, \dots, p_n$ via
the formulae
\begin{align}
p_{-1}(x) = 0, \quad p_0(x) = 1
\label{eq_sturm_p}
\end{align}
and
\begin{align}
p_k(x) = \text{det} \brk{C_k - x I_k},
\label{eq_sturm_pk}
\end{align}
for $k=2, \dots, n$. In other words, $p_k$ is
the characteristic polynomials of $C_k$. Then,
\begin{align}
p_k(x) = \brk{a_k - x} p_{k-1}(x) - b^2_k p_{k-2}(x), 
\label{eq_sturm_rec}
\end{align}
for every integer $k = 1, 2, \dots, n$.
Suppose furthermore, that, for any real number $\sigma$,
the integer
$A(\sigma)$ is defined to be
the number of agreements of sign
of consecutive elements of the sequence 
\begin{align}
p_0(\sigma), p_1(\sigma), \dots, p_n(\sigma),
\label{eq_sturm_p_seq}
\end{align}
where the sign of $p_k(\sigma)$ is taken to be opposite to the sign
of $p_{k-1}(\sigma)$ if $p_k(\sigma)$ is zero. Then,
the number of eigenvalues of $C$ that are strictly larger than $\sigma$
is precisely $A(\sigma)$.
\label{thm_sturm}
\end{thm}
%%%%%%%%%%%
\begin{cor}[Sturm bisection]
The eigenvalue $\sigma_k(C)$ of \eqref{eq_sturm_c} can be found
by means of bisection, each iteration of which costs
$O(n)$ operations.
\label{cor_sturm}
\end{cor}
\begin{proof}
We initialize the bisection by choosing $x_0 < \sigma_k(C) < y_0$.
Then we set $j = 0$ and iterate as follows.
\begin{itemize}
\item Set $z_j = (x_j + y_j) / 2$.
\item If $y_j - x_j$ is small enough, stop and return $z_j$.
\item Compute $A_j = A(z_j)$ using \eqref{eq_sturm_rec} and
\eqref{eq_sturm_p_seq}.
\item If $A_j \geq k$, set $x_{j+1} = z_j$ and $y_{j+1} = y_j$.
\item If $A_j < k$, set $x_{j+1} = x_j$ and $y_{j+1} = z_j$.
\item Increase $j$ by one and go to the first step.
\end{itemize}
In the end $\abrk{\sigma_k(C) - z_j}$ is at most $y_j - x_j$. The
cost of the algorithm is due to \eqref{eq_sturm_rec} and
the definition of $A(\sigma)$.
\end{proof}

%%%%%%%%%%%%%%%%%%%%%%%%%%%%%%%%
\subsection{Miscellaneous tools}
\label{sec_misc_tools}

In this subsection, we list some widely know theorems of real analysis.
%%%%%%%%%%%

The following theorem can be found in section 6.4 of \cite{Orszag}
in a more general form. In this theorem, we use the following 
widely used
notation. Suppose that $g, h: (0, \infty) \to \Cc$ are complex-valued functions.
The expression
\begin{align}
g(t) \sim h(t), \quad t \to \infty,
\label{eq_asymptotic}
\end{align}
means that
\begin{align}
\lim_{t \to \infty} \frac{h(t)}{g(t)} = 1.
\label{eq_asymptotic_meaning}
\end{align}
%%%%%%%%%%%%%%%%%%%%%%%%%%%
\begin{thm}[Watson's Lemma]
Suppose that $b>0$, and that
the function $f: \left[0, b\right] \to \Rc$ is twice
continuously differentiable.
Then,
\begin{align}
\label{eq_watson}
\int_0^b f(s) \cdot e^{-st} ds \sim \frac{f(0)}{t}, \quad t \to \infty,
\end{align}
in the sense of \eqref{eq_asymptotic}.
In other words,
\begin{align}
\lim_{t \to \infty} \frac{t}{f(0)} \cdot \int_0^b f(s) \cdot e^{-st} ds = 1.
\label{eq_watson_meaning}
\end{align}
\label{thm_watson_lemma}
\end{thm}
%%%%%%%%%
The following theorem appears, for example, in \cite{Evans}
in a more general form.
%%%%%%%%%%%
\begin{thm}
Suppose that $x_0$ is a real number, and $u : \Rc^2 \to \Rc$
is a function of two real variables $(t,x)$, defined in the shifted upper
half-plane
\begin{align}
\bar{H}_{x_0} = \left\{ (t,x) \; : \; -\infty < t < \infty, \quad
   x_0 \leq x < \infty \right\}.
\end{align}
Suppose also, that $u$ is bounded in $\bar{H}_{x_0}$ and is harmonic
in the interior of $\bar{H}_{x_0}$. Suppose furthermore,
that 
\begin{align}
\int_{-\infty}^{\infty} \abrk{ u(t,x_0) } \; dt < \infty.
\label{eq_harm_l1}
\end{align}
Then, for all real $t$ and $x > x_0$, the value $u(t,x)$ is given
by the formula
\begin{align}
u(t,x) = \frac{1}{\pi} \int_{-\infty}^{\infty}
   u(s,x_0) \cdot \frac{ x-x_0 }{ (t-s)^2 + (x-x_0)^2 } \; ds,
\end{align}
and, moreover, for all $x > x_0$,
\begin{align}
\int_{-\infty}^{\infty} u(t,x_0) \; dt = \int_{-\infty}^{\infty} u(t,x) \; dt.
\end{align}
\label{thm_harmonic}
\end{thm}
%%%%%%%%%%%%%%%%%%%%%%%%%%
The following theorem is a special case of the well known 
Cauchy's integral formula (see, for example, \cite{Rudin}).
\begin{thm}
Suppose that $D \subseteq \Cc$ is an open bounded simply connected
subset of the complex plane, and that the boundary $\Gamma$
of $D$ is piecewise continuously differentiable. Suppose also
that the function $g: \Cc \to \Cc$ is holomorphic in a neighborhood
of $D$, and that none of the roots of $g$ lies on $\Gamma$.
Suppose furthermore that $z_1, z_2, \dots, z_m \in D$ are the roots of $g$
in $D$, all of which are simple, and that $z \in D$ is a complex number
such that $g(z) \neq 0$. In other words,
\begin{align}
z \in D \setminus \left\{ z_1, z_2, \dots, z_m \right\}.
\end{align}
Then,
\begin{align}
\frac{1}{g(z)} = \sum_{j=1}^m \frac{1}{g'(z_j) \cdot (z-z_j)} +
\frac{1}{2\pi i} \oint_{\Gamma} \frac{d\zeta}{g(\zeta) \cdot (\zeta-z)},
\label{eq_cauchy}
\end{align}
where $\oint_{\Gamma}$ denotes the contour integral over $\Gamma$
in the counterclockwise
direction.
\label{thm_cauchy}
\end{thm}

%%%%%%%%%%%%%%%%%%%%%%%%%%%%%%%%%%%%%%%%%%%%%
%%%%%%%%%%%%%%%%%%%%%%%%%%%%%%%%%%%%%%%%%%%%%
\section{Summary}
\label{sec_summary}
In this section, we summarize some of the properties of
prolate spheroidal wave functions (PSWFs),
proved in the rest of the paper, mainly in Section~\ref{sec_analytical}.
The PSWFs and the related notation were introduced in 
Section~\ref{sec_pswf}.
Throughout this section, the band limit $c > 0$ is
assumed to be a positive real number.

In the following proposition, we describe the location of ``special points''
(roots of $\psi_n$, roots of $\psi_n'$, turning points of the
ODE \eqref{eq_prolate_ode}), in the case $\chi_n > c^2$.
This proposition is proven in Theorem~\ref{thm_four}
and Corollary~\ref{cor_infinite} in Section~\ref{sec_first_order}
(see also Theorem~\ref{thm_five} in Section~\ref{sec_pswf}).
It is illustrated in
Figures~\ref{fig:test75a}, \ref{fig:test75b}
(see Experiment 1 in Section~\ref{sec_exp1}).
%%%%%%%%%%%%
\begin{prop}
Suppose that $n \geq 0$ is a positive integer, and that $\chi_n > c^2$. 
Suppose also that 
$x_1 < x_2 < \dots$ are the roots of $\psi_n$ in $\left(1,\infty\right)$,
and $y_1 < y_2 < \dots$ are the roots of $\psi_n'$ in
$\left(1,\infty\right)$. Then,
\begin{align}
1 < \frac{\sqrt{\chi_n}}{c} < y_1 < x_1 < y_2 < x_2 < \dots.
\label{eq_four_prop}
\end{align}
Also, $\psi_n$ has infinitely many roots in $\brk{1, \infty}$;
all of these roots are simple. 
\label{prop_four}
\end{prop}

%%%%%%%%%%
The following proposition summarizes the statements of
Theorems~\ref{thm_x1_x0_good}, \ref{thm_spacing}
in Section~\ref{sec_oscillation}.
It is illustrated in
Tables~\ref{t:test76},
\ref{t:test81a}, \ref{t:test81b}.
%%%%%%%%%%%%
\begin{prop}
Suppose that $n \geq 0$ is an integer, and that $\chi_n > c^2$.
Suppose also that $x_1 < x_2 < \dots$ are the roots of $\psi_n$
in $(1, \infty)$.
\begin{itemize}
\item For each integer $k = 1, 2, \dots$,
\begin{align}
\frac{ \pi }{ c }
\sqrt{ 1 - \frac{ 1 }{ 1 + c^2 \left(x_k^2 - 1\right)^2 } } \leq
x_{k+1} - x_k \leq \frac{ \pi }{ c }
   \sqrt{ \frac{x_k^2 - 1}{ x_k^2 - (\chi_n/c^2) } }.
\label{eq_25_09_h1_prop}
\end{align}
\item If, in addition, $c > 1/5$ and
\begin{align}
n > \frac{2c}{\pi} + \frac{1}{2\pi} \cdot 
\left(\log c + \log(16 \cdot e)\right),
\label{eq_jan_20_prop}
\end{align}
then
\begin{align}
x_2 - x_1 \geq 
x_3 - x_2 \geq \dots \geq x_{k+1} - x_k \geq \dots \geq
\frac{ \pi }{ c }.
\label{eq_25_09_h3_prop}
\end{align}
\item Also,
\begin{align}
x_1 - \frac{\sqrt{\chi_n}}{c} > \frac{ \pi }{ 2c }.
\label{eq_x1_x0_good_prop}
\end{align}
\item Moreover,
\begin{align}
\sqrt{ \frac{x_1^2-1}{x_1^2-(\chi_n/c^2)} } < 
\frac{2c}{\pi} \cdot \left( 
x_1 - \frac{\sqrt{\chi_n}}{c}
\right).
\label{eq_x1_x0_smart_prop}
\end{align}
\end{itemize}
\label{prop_spacing_big}
\end{prop}

%%%%%%%%%%
The following proposition is an analogue of
Proposition~\ref{prop_spacing_big} in the case
$\chi_n < c^2$. Its proof can be found in
Theorem~\ref{thm_spacing_2}
in Section~\ref{sec_oscillation}.
%%%%%%%%%%%%
\begin{prop}
Suppose that 
$n \geq 0$ is an integer, and that
$\chi_n < c^2$.
Suppose also that $x_1 < x_2 < \dots$ are the roots of $\psi_n$
in $(1, \infty)$.
Then,
\begin{align}
x_2 - x_1 \leq 
x_3 - x_2 \leq \dots \leq x_{k+1} - x_k \leq \dots \leq
\frac{ \pi }{ c }.
\label{eq_25_09_h3_reverse_prop}
\end{align}
\label{prop_spacing_small}
\end{prop}
%%%%%%%%%%
%%%%%%%%%
The following inequality is proved in
Theorem~\ref{thm_consec_der} in Section~\ref{sec_two_by_two}
and is illustrated in
Tables~\ref{t:test83a}, \ref{t:test83b}
(see Experiment 5 in Section~\ref{sec_exp5}).
%%%%%%%%%%%
\begin{prop}
Suppose that $n \geq 0$ is an integer, and that $\chi_n > c^2$. 
Suppose also that $x < y$ are two roots of $\psi_n$ in $(1, \infty)$.
Then,
\begin{align}
\label{eq_consec_der_prop}
\abrk{ \psi'_n(x) } \cdot \frac{x^2 - 1}{y^2 - 1} \leq
\abrk{ \psi'_n(y) } \leq
\abrk{ \psi'_n(x) } 
\sqrt{ \frac{ x^2 - 1 }{ c^2 x^2 - \chi_n} \cdot 
       \frac{ c^2 y^2 - \chi_n }{ y^2 - 1 } }.
\end{align}
\label{prop_consec_der}
\end{prop}
%%%%%%%%%%

The following proposition summarizes Theorem~\ref{thm_sharp_simple}
in Section~\ref{sec_upper}.
%%%%%%%%%%%%%%
\begin{prop}
Suppose that $n\geq 0$ is an integer, and that
$\chi_n > c^2$. Suppose also that $x$ is 
a root of $\psi_n$ in $(1,\infty)$. Then,
\begin{align}
\frac{1}{\abrk{\psi_n'(x)}} \leq 
e^{1/4} \cdot \abrk{\lambda_n} \cdot 
\frac{ (x^2-1)^{\frac{3}{4}} }{ (x^2-(\chi_n/c^2))^{\frac{1}{4}} }.
\label{eq_sharp_simple_prop}
\end{align}
\label{prop_sharp_simple}
\end{prop}
%%%%%%%%%%
%%%%%%%%%%
The following two estimates are proven,
in a more precise form, in Theorem~\ref{lem_psi_for_large_x}
in Section~\ref{sec_tail}.
They describe the behavior of $\psi_n(x)$  for
$x > 1$ and are meaningful only when $x$ is large compared to
$\abrk{\lambda_n}^{-1}$.
%%%%%%%%%%%
\begin{prop} 
Suppose that $n \geq 0$ is a non-negative integer, and
that $x > 1$ is a real number. 
If $n$ is even, then
\begin{align}
\psi_n(x) = \frac{2 \psi_n(1)}{cx\lambda_n}
\left[\sin(cx) + O\brk{\frac{1}{x \abrk{\lambda_n} \psi_n(1) }} \right].
\label{eq_2_10_even1_prop}
\end{align}
If $n$ is odd, then
\begin{align}
\psi_n(x) = \frac{2 \psi_n(1)}{icx\lambda_n}
\left[\cos(cx) + O\brk{\frac{1}{x \abrk{\lambda_n} \psi_n(1) }} \right].
\label{eq_2_10_odd1_prop}
\end{align}
\label{prop_psi_for_large_x}
\end{prop}
%%%%%%%%%%
The following proposition asserts that, in the interval $(-1,1)$,
the difference between the reciprocal of $\psi_n$ and a certain
rational function with $n$ poles is of order $|\lambda_n|$.
This is an immediate consequence of
Theorem~\ref{thm_complex_summary} in Section~\ref{sec_head_tail}
and the proof of Theorem~\ref{thm_sum_w} in Section~\ref{sec_weights}.
\begin{prop}
Suppose that $c>30$, and that $n>0$ is an even positive integer.
Suppose also that
\begin{align}
n > \frac{2c}{\pi} + 7.
\label{eq_complex_1_prop}
\end{align}
Suppose furthermore that $-1 < t_1 < \dots < t_n < 1$ are the roots
of $\psi_n$ in $(-1,1)$, and that the function $I:(-1,1) \to \Rc$
is defined via the formula
\begin{align}
I(t) = \frac{1}{\psi_n(t)} - 
\sum_{k=1}^n \frac{1}{\psi_n'(t_j) \cdot (t-t_j)},
\label{eq_complex_it_prop}
\end{align}
for $-1 < t < 1$.
Then,
\begin{align}
|I(t)| \leq |\lambda_n| \cdot \left(
24 \cdot \log\left( \frac{1}{|\lambda_n|} \right) +
130 \cdot (\chi_n)^{1/4}
\right),
\label{eq_complex_prop}
\end{align}
for all real $-1 < t < 1$.
\end{prop}
%%%%%%%%%%%%%%
%%%%%%%%%%%%%%
The following proposition is the principal analytical result of
the paper. It is proven in Theorem~\ref{thm_quad_eps_simple}
in Section~\ref{sec_main_result}.
It is illustrated in
Table~\ref{t:test91} and
Figures~\ref{fig:test92}, \ref{fig:test93a}, \ref{fig:test93b}.
\begin{prop}
Suppose that $c>0$ is a positive real number, and that
\begin{align}
c > 30.
\label{eq_quad_eps_simple_c30_prop}
\end{align}
Suppose also that $\varepsilon > 0$ is a positive real number, and that
\begin{align}
\exp\left[ -\frac{3}{2}\cdot(c-20) \right] < \varepsilon < 1.
\label{eq_quad_eps_simple_eps_prop}
\end{align}
Suppose furthermore that $n>0$ and $0 \leq m < n$ are positive integers, 
and that
\begin{align}
n > \frac{2c}{\pi} +
\left(10 + \frac{3}{2} \cdot \log(c) + 
   \frac{1}{2} \cdot \log\frac{1}{\varepsilon}
\right) \cdot \log\left( \frac{c}{2} \right).
\label{eq_quad_eps_simple_n_prop}
\end{align}
Then,
\begin{align}
\left| \int_{-1}^1 \psi_m(s) \; ds - \sum_{j=1}^n \psi_m(t_j) W_j \right|
<
\varepsilon,
\label{eq_quad_eps_simple_prop}
\end{align}
where $t_1 < \dots < t_n$ are the roots of $\psi_n$ in $(-1,1)$,
and $W_1, \dots, W_n$ are defined via 
\eqref{eq_quad_w_first} in Section~\ref{sec_intuition}.
\label{prop_quad_eps_simple}
\end{prop}
In Proposition~\ref{prop_quad_eps_simple}, we address the
accuracy of the quadrature, discussed in Section~\ref{sec_intuition}.  
More specifically, it asserts that to achieve the prescribed
absolute accuracy $\varepsilon$
(in the sense of \eqref{eq_main_goal}), it suffices
to take $n$ of the order 
$2c/\pi + O\left(\log(c)\cdot(\log(c)-\log(\varepsilon))\right)$.

The assumptions of Proposition~\ref{prop_quad_eps_simple}, however,
have a minor drawback: namely, $\varepsilon$ is assumed not to be
``too small'', in the sense of \eqref{eq_quad_eps_simple_eps_prop}.
In the following proposition, 
proven in Theorem~\ref{thm_quad_eps_weak} in Section~\ref{sec_main_result},
we eliminate this inconvenience.
On the other hand, the resulting lower bound on $n$ is considerably
weaker than that of Proposition~\ref{prop_quad_eps_simple}.
\begin{prop}
Suppose that $c>0$ is a positive real number, and that
\begin{align}
c > 30.
\label{eq_quad_eps_weak_c30_prop}
\end{align}
Suppose also that $\varepsilon > 0$ is a positive real number, and that
\begin{align}
0 < \varepsilon < 1.
\label{eq_quad_eps_weak_eps_prop}
\end{align}
Suppose furthermore that $n>0$ and $0 \leq m < n$ are positive integers, 
and that
\begin{align}
n \cdot \left(1 - \frac{40}{\pi c}\right) > 
c + \frac{12}{\pi} \cdot \log(c) + \frac{4}{\pi} \cdot
\log \frac{1}{\varepsilon}.
\label{eq_quad_eps_weak_n_prop}
\end{align}
Then,
\begin{align}
\left| \int_{-1}^1 \psi_m(s) \; ds - \sum_{j=1}^n \psi_m(t_j) W_j \right|
<
\varepsilon,
\label{eq_quad_eps_weak_prop}
\end{align}
where $t_1 < \dots < t_n$ are the roots of $\psi_n$ in $(-1,1)$,
and $W_1, \dots, W_n$ are defined via 
\eqref{eq_quad_w_first} in Section~\ref{sec_intuition}.
\label{prop_quad_eps_weak}
\end{prop}
%
%%%%%%%%%%%%%%%
In the following proposition, we assert that the quadrature weights
$W_1, \dots, W_n$ are positive, provided that $n$ is large enough.
It is proven in Theorem~\ref{thm_positive_w}
in Section~\ref{sec_weights}.
\begin{prop}
Suppose that $c>0$ is a positive real number, and that
\begin{align}
c > 30.
\label{eq_positive_w_c30_prop}
\end{align}
Suppose also that $n>0$ is a positive odd integer, and that
\begin{align}
n > \frac{2c}{\pi} + 5 \cdot \log(c) \cdot \log\left(\frac{c}{2}\right).
\label{eq_positive_w_n_prop}
\end{align}
Suppose furthermore that
$W_1, \dots, W_n$ are defined via 
\eqref{eq_quad_w_first} in Section~\ref{sec_intuition}.
Then, for all integer $j=1,\dots,n$,
\begin{align}
W_j > 0.
\label{eq_positive_w_prop}
\end{align}
\label{prop_positive_w}
\end{prop}
Numerical experiments seem to indicate that
the assumptions \eqref{eq_positive_w_n_prop} and that $n$ be odd
are unnecessary (see Remarks~\ref{rem_w_even_n},~\ref{rem_w_always_pos}
in Section~\ref{sec_weights}).

%%%%%%%%%%%%%%%%%%%%%%%%%%%%%%%%%%%%%%%%%%%%%
%%%%%%%%%%%%%%%%%%%%%%%%%%%%%%%%%%%%%%%%%%%%%
\section{Analytical Apparatus}
\label{sec_analytical}
The purpose of this section is to provide
the analytical apparatus to be used in the
rest of the paper.

%%%%%%%%%%%%%%%%%%%%%%%%%%%%%%%%%%%%%%%%%%%%%
\subsection{Oscillation Properties of PSWFs}
\label{sec_oscillation}
In this subsection, we prove several facts
about the distance between consecutive roots of PSWFs 
and find a more
subtle relation between $n$ and $\chi_n$ 
(see \eqref{eq_prolate_ode} 
in Section~\ref{sec_pswf})
than the inequality \eqref{eq_khi_crude}.
Throughout this subsection,
$c > 0$ is a positive real number 
and $n$ is a non-negative integer.
The principal results of this subsection are
Theorems~\ref{thm_x1_x0_good},
\ref{thm_spacing}.

%%%%%%%%%%%
\subsubsection{Elimination of the First-Order Term of the Prolate ODE}
\label{sec_first_order}
In this subsection, we analyze the oscillation properties of $\psi_n$
via 
transforming the ODE \eqref{eq_prolate_ode} into
a second-order linear ODE without the first-order term.
The following theorem is the principal technical tool of this subsection.
\begin{thm}
Suppose that $n \geq 0$ is a non-negative integer. Suppose also that
that the functions $\Psi_n, Q_n: (1,\infty) \to \Rc$
are defined, respectively, via the formulae
\begin{align}
\Psi_n(t) = \psi_n(t) \cdot \sqrt{t^2-1}
\label{eq_big_psi_n}
\end{align}
and
\begin{align}
Q_n(t) = 
\frac{ c^2\cdot t^2 - \chi_n  }{ t^2-1 } +
       \frac{ 1 }{ \brk{t^2-1}^2 },
\label{eq_big_q_n}
\end{align}
for $t > 1$. Then,
\begin{align}
\Psi_n''(t) + Q_n(t) \cdot \Psi_n(t) = 0,
\label{eq_big_psi_ode}
\end{align}
for all $t > 1$.
\label{lem_trans}
\end{thm}
\begin{proof}
We differentiate $\Psi_n$ with respect to $t$ to obtain
\begin{align}
\Psi_n'(t) = \psi_n'(t) \sqrt{t^2-1} +
             \psi_n(t) \cdot \frac{t}{\sqrt{t^2-1}}.
\label{eq_d_big_psi}
\end{align}
Then, using \eqref{eq_d_big_psi}, we differentiate $\Psi_n'$ with 
respect to $t$ to obtain
\begin{align}
\Psi_n''(t)
& \; = 
\psi_n''(t) \sqrt{t^2-1} +
\psi_n'(t) \cdot \frac{ 2t }{ \sqrt{t^2 - 1} } +
\psi_n(t) \cdot
   \frac{ \sqrt{t^2-1} - t^2 / \sqrt{t^2-1} }{ t^2-1 }
\nonumber \\
& \; =
\psi_n''(t) \sqrt{t^2-1} +
\psi_n'(t) \cdot \frac{ 2t }{ \sqrt{t^2-1} } -
\psi_n(t) \brk{t^2-1}^{-\frac{3}{2}}
\nonumber \\
& \; = 
\frac{1}{\sqrt{t^2-1}}
\left[
  \brk{t^2-1} \cdot \psi_n''(t) + 2t \cdot \psi_n'(t) - 
  \frac{ \psi_n(t) }{ t^2-1 }
\right]
\nonumber \\
& \; = 
\frac{1}{\sqrt{t^2-1}}
\left[ \psi_n(t) \cdot \left(\chi_n - c^2 \cdot t^2 \right) -
       \frac{ \psi_n(t) }{ t^2-1 } \right]
\nonumber \\
& \; = 
-\Psi_n(t) \cdot
\brk{ \frac{ c^2 \cdot t^2 - \chi_n}{t^2-1} + \frac{1}{\brk{t^2-1}^2}}.
\label{eq_dd_big_psi}
\end{align}
To conclude the proof,
we observe that \eqref{eq_big_psi_ode} follows from \eqref{eq_dd_big_psi}.
\end{proof}
%%%%%%%%%%%
\begin{cor}
Suppose that $n \geq 0$ is an integer. Then, $\psi_n$ has infinitely
many roots in $\left(1,\infty\right)$.
\label{cor_infinite}
\end{cor}
\begin{proof}
Suppose that $Q_n: (1,\infty) \to \Rc$ is defined via \eqref{eq_big_q_n}.
Then,
\begin{align}
\lim_{t \to \infty} Q_n(t) = c^2.
\label{eq_q_n_lim}
\end{align}
We conclude by combining \eqref{eq_q_n_lim} with 
\eqref{eq_big_psi_ode} of Theorem~\ref{lem_trans} above and
Theorem~\ref{thm_25_09} in Section~\ref{sec_oscillation_ode}.
\end{proof}
%%%%%%%%%%%%
The following theorem is a counterpart of Theorem~\ref{thm_five}
in Section~\ref{sec_pswf}.
\begin{thm}
Suppose that $n \geq 0$ is a positive integer, and that $\chi_n > c^2$. 
Suppose also that 
$x_1 < x_2 < \dots$ are the roots of $\psi_n$ in $\left(1,\infty\right)$,
and $y_1 < y_2 < \dots$ are the roots of $\psi_n'$ in
$\left(1,\infty\right)$. Then,
\begin{align}
1 < \frac{\sqrt{\chi_n}}{c} < y_1 < x_1 < y_2 < x_2 < \dots.
\label{eq_four}
\end{align}
\label{thm_four}
\end{thm}
\begin{proof}
Without loss of generality, we assume that 
\begin{align}
\psi_n(1) > 0.
\label{eq_psi1_pos}
\end{align}
We combine \eqref{eq_psi1_pos} with the assumption that $\chi_n > c^2$
and the ODE \eqref{eq_prolate_ode}
to obtain
\begin{align}
\psi_n'(1) = \frac{\chi_n - c^2}{2} \cdot \psi_n(1) > 0.
\label{eq_dpsi1_pos}
\end{align}
If, by contradiction to \eqref{eq_four}, 
\begin{align}
1 < y_1 < \frac{\sqrt{\chi_n}}{c},
\label{eq_four_contra1}
\end{align}
then, due to \eqref{eq_prolate_ode},
\begin{align}
\psi_n''(y_1) = - \frac{\chi_n-c^2 \cdot y_1^2}{1 - y_1^2} \cdot \psi_n(y_1)
>0,
\end{align}
in contradiction to \eqref{eq_dpsi1_pos}.
Therefore, $\psi_n'$ is positive in the interval 
$\left(1, \sqrt{\chi_n}/c\right)$; in particular,
\begin{align}
x_1 > \frac{\sqrt{\chi_n}}{c}
\label{eq_four_x1}
\end{align}
and
\begin{align}
\psi_n\left( \frac{\sqrt{\chi_n}}{c} \right) > 0, \quad
\psi_n'\left( \frac{\sqrt{\chi_n}}{c} \right) > 0.
\label{eq_four_both}
\end{align}
We combine \eqref{eq_four_x1} and \eqref{eq_four_both} to conclude that
\begin{align}
\frac{\sqrt{\chi_n}}{c} < y_1 < x_1.
\label{eq_four_triple}
\end{align}
Suppose now that $k$ is a positive integer, and $y$ is a root of $\psi_n'$
in the interval $(x_k, x_{k+1})$. Due to \eqref{eq_prolate_ode},
\begin{align}
\psi_n''(y) = -\frac{c^2 \cdot y^2 - \chi_n}{y^2-1} \cdot \psi_n(y).
\label{eq_four_ddpsi_y}
\end{align}
It follows from \eqref{eq_four_ddpsi_y} that $\psi_n'$ has exactly
one root between two consecutive roots of $\psi_n$. We combine
this observation with \eqref{eq_four_triple} to obtain \eqref{eq_four}.
\end{proof}
%%%%%%%%%%%%
In the following theorem, we describe several properties
of the modified Pr\"ufer transformation (see Section~\ref{sec_prufer})
applied to the prolate differential equation \eqref{eq_prolate_ode}.
\begin{thm}
Suppose that $n \geq 0$ is a positive integer, and that
$\chi_n > c^2$. Suppose also that
$x_1<x_2<\dots$ are the roots of $\psi_n$ in $(1,\infty)$,
and $y_1<y_2<\dots$ are the roots of $\psi_n'$ in $(1,\infty)$
(see Theorem~\ref{thm_four}).
Suppose furthermore that
the function $\theta: \left[\sqrt{\chi_n}/c,\infty\right) \to \Rc$ is
defined via the formula
\begin{align}
\theta(t) = 
\begin{cases}
-\frac{\pi}{2}, & \text{ if } t = \frac{\sqrt{\chi_n}}{c}, \\
\left(i-\frac{1}{2}\right) \cdot \pi, & \text{ if } t = x_i
\text{ for some } i = 1, 2, \dots, \\
\\
\text{atan}\left( -\sqrt{\frac{1-t^2}{\chi_n-c^2t^2}} \cdot
                        \frac{\psi_n'(t)}{\psi_n(t)} \right)
          + m(t) \cdot \pi, & \text{ otherwise }, \\
\end{cases}
\label{eq_prufer_theta}
\end{align}
where $m(t)$ is the number of the roots
of $\psi_n$ in the interval $(1,t)$.
Then, $\theta$ has
the following properties:
\begin{itemize}
\item $\theta$ is continuously differentiable in 
$\left[\sqrt{\chi_n}/c,\infty\right)$.
\item $\theta$ satisfies, for all $t > \sqrt{\chi_n}/c$, the
differential equation
\begin{align}
\theta'(t) = f(t) - v(t) \cdot \sin( 2 \theta(t) ),
\label{eq_prufer_theta_ode}
\end{align}
where the functions $f,v$ are defined, respectively, 
via \eqref{eq_jan_f}, \eqref{eq_jan_v} in
Section~\ref{sec_prufer}.
\item for each integer $k \geq -1$, there is a unique solution
to the equation
\begin{align}
\theta(t) = k \cdot \frac{\pi}{2},
\label{eq_prufer_sm2_equation}
\end{align}
for the unknown $t$ in $\left[\sqrt{\chi_n}/c,\infty\right)$.
More specifically,
\begin{align}
\label{eq_theta_at_x0}
& \theta\left( \frac{\sqrt{\chi_n}}{c} \right) = -\frac{\pi}{2}, \\
\label{eq_theta_at_x}
& \theta(x_i) = \left(i-\frac{1}{2}\right) \cdot \pi, \\
\label{eq_theta_at_y}
& \theta(y_i) = (i-1) \cdot \pi,
\end{align}
for each integer $i \geq 1$.
\end{itemize}
\label{thm_prufer}
\end{thm}
\begin{proof}
We combine \eqref{eq_four} in 
Theorem~\ref{thm_four} with \eqref{eq_prufer_theta} to conclude that
$\theta$
is well defined for all $t \geq \sqrt{\chi_n}/c$.
Obviously, $\theta$ is continuous, and the identities
\eqref{eq_theta_at_x0}, \eqref{eq_theta_at_x}, \eqref{eq_theta_at_y}
follow immediately from the combination
of Theorem~\ref{thm_four} and \eqref{eq_prufer_theta}.
In addition, $\theta$ satisfies the ODE
\eqref{eq_prufer_theta_ode} in $\left(\sqrt{\chi_n}/c,\infty\right)$ due to
\eqref{eq_general_dtheta}, \eqref{eq_modified_prufer},
\eqref{eq_jan_prufer_fv} in Section~\ref{sec_prufer}.

Finally, to establish the uniqueness of
the solution to the equation \eqref{eq_prufer_sm2_equation},
we make the following observation. Due to \eqref{eq_prufer_theta},
for any point $t > \sqrt{\chi_n}/c$, the value $\theta(t)$ is an integer
multiple of $\pi/2$ if and only if $t$ is either a root of $\psi_n$
or a root of $\psi_n'$. 
We conclude the proof by combining this observation with 
\eqref{eq_theta_at_x0},
\eqref{eq_theta_at_x} and \eqref{eq_theta_at_y}.
\end{proof}
%%%%%%%%%%%%%%%%%%%
The following theorem is illustrated in Table~\ref{t:test76}
(see Experiment 2 in Section~\ref{sec_exp2}).
%
% verified in test76
%
%%%%%%%%%%%
\begin{thm}
Suppose that $n \geq 0$ is an integer, and that
$\chi_n > c^2$. Suppose also that $x_1$ is the minimal root
of $\psi_n$ in $(1,\infty)$. Then,
\begin{align}
x_1 - \frac{\sqrt{\chi_n}}{c} > \frac{ \pi }{ 2c }.
\label{eq_x1_x0_good}
\end{align}
Moreover,
\begin{align}
\sqrt{ \frac{x_1^2-1}{x_1^2-(\chi_n/c^2)} } < 
\frac{2}{\pi} \cdot c \cdot \left( 
x_1 - \frac{\sqrt{\chi_n}}{c}
\right).
\label{eq_x1_x0_smart}
\end{align}
\label{thm_x1_x0_good}
\end{thm}
\begin{proof}
Suppose that $y_1$ is the minimal root of $\psi_n'$ in $(1,\infty)$.
Due to Theorem~\ref{thm_four},
\begin{align}
\frac{\sqrt{\chi_n}}{c} < y_1 < x_1.
\label{eq_x1_x0_triple}
\end{align}
Moreover, due to \eqref{eq_prufer_theta} in Theorem~\ref{thm_prufer}
and \eqref{eq_four_both} in the proof of Theorem~\ref{thm_four},
\begin{align}
\sin(2 \theta(t)) > 0,
\label{eq_sin_pos}
\end{align}
for all real $y_1 < t < x_1$,
where $\theta$ is defined via \eqref{eq_prufer_theta}.
We combine \eqref{eq_sin_pos} with
\eqref{eq_prufer_theta_ode}, \eqref{eq_theta_at_x}, \eqref{eq_theta_at_y}
to obtain
\begin{align}
\frac{\pi}{2}
& \; = \int_{y_1}^{x_1} \theta'(t) \; dt
     = \int_{y_1}^{x_1} \left( f(t) - v(t) \cdot \sin(2\theta(t)) \right)\;dt
     \nonumber \\
& \; < \int_{y_1}^{x_1} f(t) \; dt\
     = \int_{y_1}^{x_1} \sqrt{ c^2 - \frac{\chi_n-c^2}{t^2-1}} \; dt
     < c \cdot (x_1-y_1).
\label{eq_x1_y1}
\end{align}
We combine \eqref{eq_x1_x0_triple} with \eqref{eq_x1_y1}
to obtain \eqref{eq_x1_x0_good}. It also follows from \eqref{eq_x1_y1}
that
\begin{align}
\frac{\pi}{2} < \int_{\sqrt{\chi_n}/c}^{x_1}
\sqrt{ c^2 - \frac{\chi_n-c^2}{t^2-1}} \; dt
<
\left(x_1 - \frac{\sqrt{\chi_n}}{c}\right) \cdot
\sqrt{c^2 - \frac{\chi_n-c^2}{x_1^2-1}},
\end{align}
which implies \eqref{eq_x1_x0_smart}.
\end{proof}
%%%%%%%%%%%%%%
The following theorem is a consequence of Theorems~\ref{lem_trans},
\ref{thm_x1_x0_good}.
The results of the corresponding numerical experiments
are reported in
Tables~\ref{t:test81a}, \ref{t:test81b}
(see Experiment 3 in Section~\ref{sec_exp3}).
%%%%%%%%%%%
\begin{thm}
Suppose that $n \geq 0$ is an integer, and that $\chi_n > c^2$.
Suppose also that $x_1 < x_2 < \dots$ are the roots of $\psi_n$
in $(1, \infty)$ (see Theorem~\ref{thm_four}).
Then, 
\begin{align}
\frac{ \pi }{ c }
\sqrt{ 1 - \frac{ 1 }{ 1 + c^2 \left(x_k^2 - 1\right)^2 } } \leq
x_{k+1} - x_k \leq \frac{ \pi }{ c }
   \sqrt{ \frac{x_k^2 - 1}{ x_k^2 - (\chi_n/c^2) } },
\label{eq_25_09_h1}
\end{align}
for each integer $k = 1, 2, \dots$.
If, in addition, $c > 1/5$ and
\begin{align}
n > \frac{2}{\pi} c + \frac{1}{2\pi} \cdot 
\left(\log c + \log(16 \cdot e)\right),
\label{eq_jan_20}
\end{align}
then 
\begin{align}
x_2 - x_1 \geq 
x_3 - x_2 \geq \dots \geq x_{k+1} - x_k \geq \dots \geq
\frac{ \pi }{ c }.
\label{eq_25_09_h3}
\end{align}
\label{thm_spacing}
\end{thm}
\begin{proof}
Suppose that the functions $\Psi_n, Q_n: (1,\infty) \to \Rc$
are those of Theorem~\ref{lem_trans} above. Suppose also
that $k \geq 1$ is a positive integer. Then, due
to \eqref{eq_big_q_n},
\begin{align}
& c^2 \cdot \frac{ x_k^2 - (\chi_n/c^2) }{ x_k^2 - 1 } < 
c^2 \cdot \frac{ t^2 - (\chi_n/c^2) }{ t^2 - 1 } < \nonumber \\
& Q_n(t) < 
c^2 + \frac{1}{\brk{t^2-1}^2} < c^2 + \frac{1}{\brk{x_k^2-1}^2},
\label{eq_qn_upper}
\end{align}
for all real $x_k < t < x_{k+1}$.
We observe
that $\psi_n$ and $\Psi_n$ have the same roots in $(1,\infty)$ due
to \eqref{eq_big_psi_n}, and combine this observation
with \eqref{eq_big_psi_ode} of Theorem~\ref{lem_trans}
and Theorem~\ref{thm_25_09} of Section~\ref{sec_oscillation_ode}
to obtain \eqref{eq_25_09_h1}.

Now we assume that $c > 1/5$ and that $n$ satisfies \eqref{eq_jan_20}.
Also, we define the real number $\delta$ via the formula
\begin{align}
\delta = \frac{\pi}{4}.
\label{eq_25_09_delta}
\end{align}
We recall that $c>1/5$ and 
combine \eqref{eq_jan_20},
\eqref{eq_25_09_delta} and Theorem~\ref{thm_khi_1} in Section~\ref{sec_pswf}
to conclude that
\begin{align}
\frac{\chi_n - c^2}{c} > 1.
\label{eq_25_09_gc1}
\end{align}
Next, we differentiate $Q_n$ with respect to $t$ to obtain
\begin{align}
Q_n'(t) = 
\frac{ 2 \brk{\chi_n - c^2} t }{ \brk{t^2-1}^2 } -
\frac{ 4t }{ \brk{t^2 - 1}^3 }
= 
\frac{ 2t }{ \brk{t^2 - 1}^3 } \brk{ \brk{\chi_n-c^2}\brk{t^2-1} - 2 }.
\label{eq_25_09_dq}
\end{align}
We combine \eqref{eq_25_09_gc1} with Theorem~\ref{thm_x1_x0_good}
to obtain
\begin{align}
& \brk{\chi_n-c^2}\brk{x_1^2-1} - 2 >
  \left(\frac{ \chi_n-c^2 }{ c }\right)^2 + \pi \cdot \frac{\chi_n-c^2}{c} - 2
> 0,
\label{eq_25_09_inter}
\end{align}
and substitute \eqref{eq_25_09_inter} into \eqref{eq_25_09_dq} to conclude
that
\begin{align}
Q_n'(t) > 0,
\label{eq_dqn_pos}
\end{align}
for all $t > x_1$. Thus \eqref{eq_25_09_h3} follows
from the combination of \eqref{eq_dqn_pos} and
Theorem~\ref{thm_25_09_2} in Section~\ref{sec_oscillation_ode}.
\end{proof}
%%%%%%%%%%%%
\begin{remark}
Extensive numerical experiments seem to indicate that,
if $\chi_n > c^2$, then \eqref{eq_25_09_h3}
always holds. In other words, the assumption \eqref{eq_jan_20}
is unnecessary.
\label{rem_big_spacing}
\end{remark}
%%%%%%%%%%%%%%%%%%%%
The following theorem is a counterpart of
Theorem~\ref{thm_spacing} in the case $\chi_n < c^2$.
%%%%%%%%%%%
\begin{thm}
Suppose that 
$n \geq 0$ is an integer, and that
$\chi_n < c^2$.
Suppose also that $x_1 < x_2 < \dots$ are the roots of $\psi_n$
in $(1, \infty)$.
Then,
\begin{align}
x_2 - x_1 \leq 
x_3 - x_2 \leq \dots \leq x_{k+1} - x_k \leq \dots \leq
\frac{ \pi }{ c }.
\label{eq_25_09_h3_reverse}
\end{align}
\label{thm_spacing_2}
\end{thm}
\begin{proof}
Suppose that the functions $\Psi_n, Q_n: (1,\infty) \to \Rc$
are those of Theorem~\ref{lem_trans} above.
We observe that $Q_n$ is monotonically decreasing in $(1,\infty)$,
due to \eqref{eq_big_q_n}. Also, we observe that $\Psi_n$
and $\psi_n$ have the same zeros in $(1,\infty)$, due
to \eqref{eq_big_psi_n}. We combine these observations
with \eqref{eq_big_psi_ode} and
Theorem~\ref{thm_25_09_2} in Section~\ref{sec_oscillation_ode}
to obtain \eqref{eq_25_09_h3_reverse}.
\end{proof}

%%%%%%%%%%%%%%%%%%%%%%%%%%%%%%%%%%%%%%%%%%%%%
\subsection{Growth Properties of PSWFs}
\label{sec_growth}
In this subsection, we find several bounds
on $\left|\psi_n\right|$ and $\left|\psi_n'\right|$.
Throughout this subsection, $c > 0$ is 
a positive real number
and $n$ is a non-negative integer. The principal result of this subsection
is Theorem~\ref{thm_consec_der}.
%
% It follows by Theorem~\ref{thm_good_n_good_khi} that $\chi_n > c^2$.
%
%%%%%%%%%%%
\subsubsection{Transformation of the Prolate ODE into 
a 2$\times$2 System}
\label{sec_two_by_two}
The ODE \eqref{eq_prolate_ode} can be transformed into a linear
two-dimensional
first-order system of the form 
\begin{align}
Y'(t) = A(t) Y(t),
\label{eq_prolate_ode_2D_again}
\end{align}
where the diagonal entries of $A(t)$ vanish. The application of
Theorem~\ref{thm_lewis} in Section~\ref{sec_growth_ode}
to \eqref{eq_prolate_ode_2D_again}
yields somewhat 
crude but useful estimates on the magnitude of $\psi_n$ and $\psi_n'$.
The following theorem is a technical tool to be used in
the rest of this subsection.
This theorem is illustrated in
Figures~\ref{fig:test82c}, \ref{fig:test82d} (see
Experiment 4 in Section~\ref{sec_exp4}).
%%%%%%%%%%%%%%%
\begin{thm}
Suppose that $n \geq 0$ is a non-negative integer, and that
the functions $p, q: \Rc \to \Rc$ are defined via 
\eqref{eq_prufer_1} in Section~\ref{sec_prufer}. 
Suppose also that the functions 
$Q, \tilde{Q} : (\max\left\{\sqrt{\chi_n}/c,1\right\}, \infty) 
\to \Rc$ are defined,
respectively, via the formulae
\begin{align}
Q(t) 
& \; = \psi_n^2(t) + \frac{ p(t) }{ q(t) } \cdot \brk{\psi'_n(t)}^2 
 = \psi_n^2(t) + 
       \frac{ \brk{t^2-1} \cdot \brk{\psi'_n(t)}^2 }{ c^2 t^2 - \chi_n}
\label{eq_psi1_q}
\end{align}
and
\begin{align}
\tilde{Q}(t) 
& \; = p(t) \cdot q(t) \cdot Q(t)  \nonumber \\
& \; = \brk{t^2 - 1} \cdot \brk{ \brk{c^2 t^2-\chi_n} \cdot \psi_n^2(t) +
                           \brk{t^2-1} \cdot \brk{\psi'_n(t)}^2 }.
\label{eq_psi1_qtilde}
\end{align}
Then, $Q$ is decreasing in
$(\max\left\{\sqrt{\chi_n}/c,1\right\}, \infty)$,
and $\tilde{Q}$ is increasing in
$(\max\left\{\sqrt{\chi_n}/c,1\right\}, \infty)$.
\label{lem_Q_Q_tilde}
\end{thm}
\begin{proof}
%
% see 12_01_10
%
We differentiate $Q$, defined via \eqref{eq_psi1_q}, with respect to $t$
to obtain
\begin{align}
Q'(t) = & \; 2 \cdot \psi_n(t) \cdot \psi_n'(t) + \left(
        \frac{2c^2 t \cdot (1-t^2)}{(\chi_n-c^2t^2)^2}-\frac{2t}{\chi_n-c^2t^2}
        \right) \cdot \left(\psi_n'(t)\right)^2 + \nonumber \\
  & \; \frac{2\cdot(1-t^2)}{\chi_n-c^2t^2} \cdot \psi_n'(t) \cdot \psi_n''(t).
\label{eq_dq_long}
\end{align}
Due to \eqref{eq_prolate_ode} in Section~\ref{sec_pswf},
\begin{align}
\psi_n''(t) = \frac{2t}{1-t^2} \cdot \psi_n'(t) - 
\frac{\chi_n-c^2 t^2}{1-t^2} \cdot \psi_n(t),
\label{eq_ddpsi}
\end{align}
for all $-1 < t < 1$. We substitute \eqref{eq_ddpsi} into \eqref{eq_dq_long}
and carry out straightforward algebraic manipulations to obtain
\begin{align}
Q'(t) = \frac{2t}{(\chi_n-c^2 t^2)^2} \cdot \left(\chi_n+c^2-2c^2 t^2\right)
        \cdot \left( \psi_n'(t) \right)^2.
\label{eq_dq_short}
\end{align}
Obviously, for all real $t > \max\left\{\sqrt{\chi_n}/c, 1\right\}$,
\begin{align}
\chi_n+c^2-2c^2 t^2 < 0.
\label{eq_term_pos}
\end{align}
We combine \eqref{eq_dq_short} with \eqref{eq_term_pos} to conclude that
\begin{align}
Q'(t) < 0,
\label{eq_dq_positive}
\end{align}
for all real $t > \max\left\{\sqrt{\chi_n}/c, 1\right\}$. Then,
we differentiate $\tilde{Q}$, defined via \eqref{eq_psi1_qtilde},
with respect to $t$ to obtain
\begin{align}
\tilde{Q}'(t) = 
& \; -2t \cdot \left( (\chi_n-c^2 t^2) \cdot \psi_n^2(t) + 
                      (1-t^2) \cdot \left(\psi'_n(t)\right)^2 \right) 
     \nonumber \\
& \; + (1-t^2) \cdot \left( -2c^2 t \cdot \psi_n^2(t) +
                  2 \cdot (\chi_n-c^2 t^2) \cdot \psi_n(t) \cdot \psi_n'(t)
     \right. \nonumber \\
& \; \left. \quad -2t \cdot \left(\psi_n'(t)\right)^2 
                  +2 \cdot (1-t^2) \cdot \psi_n'(t) \cdot \psi_n''(t)
                     \right).
\label{eq_dqtilde_long}
\end{align}
We substitute \eqref{eq_ddpsi} into \eqref{eq_dqtilde_long}
and carry out straightforward algebraic manipulations to obtain
\begin{align}
\tilde{Q}'(t) = 2t \cdot (2 c^2 t^2 - \chi_n - c^2) \cdot \psi_n^2(t).
\label{eq_dqtilde_short}
\end{align}
We combine \eqref{eq_term_pos} with \eqref{eq_dqtilde_short} to conclude
that
\begin{align}
\tilde{Q}'(t) > 0,
\label{eq_dqtilde_negative}
\end{align}
for all real $t > \max\left\{\sqrt{\chi_n}/c, 1\right\}$.
We combine \eqref{eq_dq_positive} and \eqref{eq_dqtilde_negative}
to finish the proof.
\end{proof}
%%%%%%%%%%%%%%
\begin{remark}
We observe that the statement of Theorem~\ref{lem_Q_Q_tilde}
is similar to that of Theorem~\ref{thm_Q_Q_tilde} in Section~\ref{sec_pswf}.
However, while in Theorem~\ref{thm_Q_Q_tilde} the behavior of $\psi_n$
and $\psi_n'$ inside the interval $(-1,1)$ is described,
Theorem~\ref{lem_Q_Q_tilde} deals with $(1,\infty)$ instead.
\label{rem_q}
\end{remark}
The following 
theorem follows directly from Theorem~\ref{lem_Q_Q_tilde}.
It is illustrated in
Tables~\ref{t:test83a}, \ref{t:test83b}
(see Experiment 5 in Section~\ref{sec_exp5}).
%%%%%%%%%%%
%%%%%%%%%%%
\begin{thm}
\label{thm_consec_der}
Suppose that $n \geq 0$ is an integer, and that $\chi_n > c^2$. 
Suppose also that $x < y$ are two roots of $\psi_n$ in $(1, \infty)$.
Then,
\begin{align}
\label{eq_consec_der}
\abrk{ \psi'_n(x) } \cdot \frac{x^2 - 1}{y^2 - 1} \leq
\abrk{ \psi'_n(y) } \leq
\abrk{ \psi'_n(x) } 
\sqrt{ \frac{ x^2 - 1 }{ x^2 - (\chi_n/c^2)} \cdot 
       \frac{ y^2 - (\chi_n/c^2) }{ y^2 - 1 } }.
\end{align}
\end{thm}
\begin{proof}
Due to Theorem~\ref{thm_four},
\begin{align}
\frac{\sqrt{\chi_n}}{c} < x < y.
\label{eq_consec_triple}
\end{align}
Due to Theorem~\ref{lem_Q_Q_tilde}, 
the function $Q:\left(\sqrt{\chi_n}/c, \infty\right) \to \Rc$,
defined
via \eqref{eq_psi1_q}, is monotonically decreasing.
We combine this observation with \eqref{eq_consec_triple} to obtain
\begin{align}
\sqrt{Q(x)} = \frac{ \abrk{ \psi_n'(x) } }{ c } 
              \sqrt{ \frac{x^2 - 1}{x^2 - (\chi_n/c^2)} } \geq
              \frac{ \abrk{ \psi_n'(y) } }{ c }
              \sqrt{ \frac{y^2 - 1}{y^2 - (\chi_n/c^2)} } = \sqrt{Q(y)}.
\label{eq_consec_a}
\end{align}
We rearrange \eqref{eq_consec_a} to obtain 
the right-hand side of \eqref{eq_consec_der}.
Moreover, 
due to Theorem~\ref{lem_Q_Q_tilde}, the function 
$\tilde{Q}:\left(\sqrt{\chi_n}/c, \infty\right) \to \Rc$ defined via
\eqref{eq_psi1_qtilde}, is 
monotonically increasing. Therefore,
\begin{align}
\sqrt{\tilde{Q}(x)} = \abrk{ \psi_n'(x) } \cdot \brk{x^2 - 1} \leq
                      \abrk{ \psi_n'(y) } \cdot \brk{y^2 - 1} =
                      \sqrt{\tilde{Q}(y)},
\end{align}
which yields the left-hand side of \eqref{eq_consec_der}.
\end{proof}
%%%%%%%%%%%
%%%%%%%%%%%%%%
\subsubsection{The Behavior of $\psi_n$ in the Upper-Half Plane}
\label{sec_upper}
The integral equation \eqref{eq_prolate_integral} provides the
analytical continuation of $\psi_n$ onto the whole complex plane.
Moreover, the same equation describes the asymptotic behavior
of $\psi_n(x+it)$ for a fixed $x$ as $t$ grows to infinity
(see Theorem~\ref{lem_q_at_infty} below).
Comparison of these asymptotics to the estimate obtained with
the help of Theorem~\ref{thm_lewis} in Section~\ref{sec_growth_ode}
yields an upper bound on $\abrk{\psi_n(x)}^{-1}$ at the roots
of $\psi_n$ (see Theorem~\ref{thm_12_10_sharp} below).
The principal result of this subsection is Theorem~\ref{thm_sharp_simple}.
%%%%%%%%%%%
\begin{thm}
Suppose that $n \geq 0$ is an integer.
Suppose also that $x$ is a root of $\psi_n$
in $(1, \infty)$. Suppose furthermore that the function
$Q:(0,\infty) \to \Rc$ is defined via the formula
\begin{align}
\label{eq_phi_q}
Q(t) = \abrk{\psi_n(x+it)}^2 + 
       \abrk{\psi_n'(x+it)}^2 
       \frac{ \abrk{(x+it)^2 - 1} }{ \abrk{c^2 (x+it)^2 - \chi_n)} },
\end{align}
where $i = \sqrt{-1}$.
Then, using the asymptotic notation
\eqref{eq_asymptotic}
of Section~\ref{sec_misc_tools},
\begin{align}
\label{eq_sqrt_qt}
\sqrt{Q(t)} \sim \frac{ e^{ct} \abrk{\psi_n(1)} \sqrt{2} }
                      { ct \abrk{\lambda_n} },
\quad t \to \infty,
\end{align}
where $\lambda_n$ is the $n$th eigenvalue of
the integral operator \eqref{eq_prolate_integral}.
\label{lem_q_at_infty}
\end{thm}
\begin{proof}
We use \eqref{eq_prolate_integral} in Section~\ref{sec_pswf} to obtain
\begin{align}
\lambda_n \psi_n(x+it)
& \; = \int_{-1}^1 \psi_n(s) e^{ics(x+it)} ds 
     = \int_{-1}^1 \psi_n(s) e^{icsx} e^{-cst} ds \nonumber \\
& \; = \int_0^2 \left[ \psi_n(s-1) e^{ic\brk{s-1}x} \right] e^{-c\brk{s-1}t} ds
       \nonumber \\
& \; = e^{ct} \int_0^2 \left[ \psi_n(s-1) e^{ic\brk{s-1}x} \right] e^{-cst} ds
       \nonumber \\
& \; = \frac{ e^{ct} }{ c }
       \int_0^{2c} \psi_n\brk{s/c-1} e^{ic\brk{s/c-1}x} e^{-st} ds
       \nonumber \\
& \; = \frac{ e^{ct} e^{-icx} }{ c }
       \int_0^{2c} \psi_n\brk{s/c-1} e^{isx} e^{-st} ds.
\label{eq_psi_expand}
\end{align}
Since $\psi_n(1) = \brk{-1}^n \psi_n(-1)$, it follows from
Theorem~\ref{thm_watson_lemma} in Section~\ref{sec_misc_tools} that
\begin{align}
\label{eq_psi_infty}
\abrk{\psi_n\brk{x+it}} \sim 
\frac{ e^{ct} \abrk{\psi_n(1)} }{ \abrk{\lambda_n} ct },
\quad t \to \infty.
\end{align}
Also, we differentiate \eqref{eq_psi_expand} 
with respect to $t$ to obtain
\begin{align}
\lambda_n \psi_n'(x+it) 
& \; = ic \int_{-1}^1 s \psi_n(s) e^{icsx} e^{-cts} ds \nonumber \\
& \; = i e^{ct} e^{-icx} 
       \int_0^{2c} \brk{s/c-1} \psi_n\brk{s/c-1} e^{ixs} e^{-st} ds.
\label{eq_dpsi_expand}
\end{align}
We combine \eqref{eq_dpsi_expand} with
Theorem~\ref{thm_watson_lemma} in Section~\ref{sec_misc_tools}
to obtain
\begin{align}
\label{eq_dpsi_infty}
\abrk{\psi_n'\brk{x+it}} \sim 
\frac{ e^{ct} \abrk{\psi_n(1)} }{ \abrk{\lambda_n} t },
\quad t \to \infty.
\end{align}
We substitute \eqref{eq_psi_infty} and \eqref{eq_dpsi_infty} into
\eqref{eq_phi_q} to obtain
\begin{align}
Q(t) 
& \; \sim \brk{ \frac{ e^{ct} \abrk{\psi_n(1)} }{ \abrk{\lambda_n}ct } }^2 +
     \frac{ \abrk{(x+it)^2 - 1} }{ \abrk{c^2 (x+it)^2 - \chi_n} }
     \brk{ \frac{ e^{ct} \abrk{\psi_n(1)} }{ \abrk{\lambda_n}t } }^2 \\
& \; \sim 2 \brk{ \frac{ e^{ct} \abrk{\psi_n(1)} }{ \abrk{\lambda_n}ct } }^2, 
     \quad t \to \infty,
\end{align}
which implies \eqref{eq_sqrt_qt}.
\end{proof}
%%%%%%%%%%%
The rest of this subsection is dedicated to establishment
of an upper bound on $|\psi_n'(x)|^{-1}$ at the roots
of $\psi_n$.
We start with
introducing the following definition.
\begin{definition}
Suppose that $x > x_0 > 1$ are real numbers. 
We define $b_c(x,x_0)$ via the formula
\begin{align}
b_c(x,x_0) = \exp\left[
\frac{\pi}{64c} \cdot \sqrt{ \frac{x^2-1}{x^2-x_0^2} } \cdot
\sum_{i,j=1}^4 \frac{1}{\delta_i(x,x_0) + \delta_j(x,x_0)}
\right],
\label{eq_bc_def}
\end{align}
with
\begin{align}
& \delta_1(x,x_0) = x-x_0, \nonumber \\
& \delta_2(x,x_0) = x+x_0, \nonumber \\
& \delta_3(x,x_0) = x-1, \nonumber \\
& \delta_4(x,x_0) = x+1.
\label{eq_deltas}
\end{align}
\label{def_bc}
\end{definition}
%%%%%%%%%%%
Next, we prove several technical theorems.
\begin{thm}
Suppose that $x>x_0>1$ are real numbers.
Then
\begin{align}
\int_0^{\infty} \left(
\sqrt{\frac{1}{2} \abrk{ \frac{(x+it)^2-x_0^2}{(x+it)^2-1} } +
\frac{1}{2} \Re\brk{ \frac{(x+it)^2-x_0^2}{(x+it)^2-1} } } -1 \right) dt = 0,
\label{eq_harm_integral}
\end{align}
where $i = \sqrt{-1}$ and, for any complex number $z$, we denote its
real part by $\Re(z)$.
\label{lem_harmonic}
\end{thm}
\begin{proof}
We fix $x_0>1$, and view the integrand in \eqref{eq_harm_integral} as a function
of $t$ and $x$. We denote this function by $u(t,x)$. In other
words, $u(t,x)$ is a real-valued function of two 
real variables, defined via the formula
\begin{align}
u(t,x) = \sqrt{\frac{1}{2} \abrk{ \frac{(x+it)^2-x_0^2}{(x+it)^2-1} } +
\frac{1}{2} \Re\brk{ \frac{(x+it)^2-x_0^2}{(x+it)^2-1} } } - 1.
\label{eq_u_tx}
\end{align}
Obviously, for fixed real $x>x_0$,
\begin{align}
\lim_{\abrk{t} \to \infty} u(t,x) = 0.
\label{eq_u_limit}
\end{align}
Next, we observe that
\begin{align}
\Re\brk{ \frac{(x+it)^2-x_0^2}{(x+it)^2-1} } =
1 +  \brk{x_0^2-1} \cdot \frac{t^2+1-x^2 }
                          { \brk{t^2-(x^2-1)}^2 + 4x^2t^2 } 
\label{eq_u_re}
\end{align}
and
\begin{align}
\abrk{ \frac{(x+it)^2-x_0^2}{(x+it)^2-1} } = 
\sqrt{1 + \brk{x_0^2-1} \cdot \frac{ 2t^2 - 2x^2 + x_0^2 + 1 }
                             { \brk{t^2 - \brk{x^2 - 1}}^2 + 4x^2 t^2 } }.
\label{eq_u_abs}
\end{align}
We combine \eqref{eq_u_tx}, \eqref{eq_u_re} and \eqref{eq_u_abs}
to conclude that
for all $x \geq x_0$ and $t\geq 0$,
\begin{align}
-1 \leq u(t,x) \leq 
\frac{x_0^2-1}{8} \cdot \frac{ 4t^2 - 4x^2 + x_0^2 + 3 }
                             { \brk{t^2 - \brk{x^2 - 1}}^2 + 4x^2 t^2 }.
\label{eq_u_ineq}
\end{align}
Therefore, $u(t,x)$ is a bounded function in the ``shifted'' upper-half plane
\begin{align}
H_{x_0} = \left\{ (t,x) \; : \; x > x_0 \right\}.
\label{eq_hx0}
\end{align}
Next, again due to \eqref{eq_u_re} and \eqref{eq_u_abs}, 
for all $x\geq x_0$ and all real $t$
satisfying the inequality $t^2 > x^2-1$, we have
\begin{align}
0 \leq u(t,x) \leq 
\frac{x_0^2-1}{8} \cdot \frac{ 4t^2 - 4x^2 + x_0^2 + 3 }
                             { \brk{t^2 - \brk{x^2 - 1}}^2 + 4x^2 t^2 }.
\label{eq_u_ineq_2}
\end{align}
In particular, the function $t \to u(t,x_0)$ belongs to
$L^1(\Rc)$. In other words,
\begin{align}
\int_{-\infty}^{\infty} |u(t,x_0)| \; dt < \infty.
\label{eq_u_lp}
\end{align}
By carrying out
tedious but straightforward calculations, one can verify that 
in $H_{x_0}$, defined via \eqref{eq_hx0}, the function $u(x,t)$
satisfies the Laplace's equation
\begin{align}
\frac{ \partial^2 u }{ \partial t^2}(t,x) +
\frac{ \partial^2 u }{ \partial x^2}(t,x) = 0.
\label{eq_u_laplace}
\end{align}
In other words, $u(t,x)$ is a bounded harmonic function 
in the shifted upper-half plane $H_{x_0}$.
We apply Theorem~\ref{thm_harmonic} 
in Section~\ref{sec_misc_tools} to 
conclude that, for all real $t$ and $x>x_0$,
\begin{align}
u(t,x) = 
\frac{1}{\pi} \int_{-\infty}^{\infty} u(s,x_0) \cdot 
\frac{x-x_0}{(t-s)^2 + (x-x_0)^2} \; ds,
\end{align}
and, moreover, for all $x > x_0$,
\begin{align}
\int_{-\infty}^{\infty} u(t,x_0) \; dt = \int_{-\infty}^{\infty} u(t,x) \; dt.
\label{eq_u_equality}
\end{align}
We integrate the right-hand side of \eqref{eq_u_ineq_2} by using
the standard complex analysis residues technique to obtain the inequality
\begin{align}
\int_{-\infty}^{\infty} u(t,x) 
& \; \leq 
\frac{x_0^2-1}{8} \cdot
\int_{-\infty}^{\infty} \frac{4t^2-4x^2+x_0^2+3}{(t^2-(x^2-1))^2+4x^2t^2} \; dt
\nonumber \\
& \; =
\frac{\pi}{16x} \cdot \frac{ (x_0^2-1)^2 }{ x^2-1 }.
\label{eq_u_bound}
\end{align}
We take the limit $x \to \infty$ in \eqref{eq_u_bound}
and use \eqref{eq_u_equality} to conclude that,
for all $x \geq x_0$,
\begin{align}
\int_{-\infty}^{\infty} u(t,x) \; dt \leq 0.
\label{eq_u_leq_0}
\end{align}
On the other hand, due to \eqref{eq_u_re} and \eqref{eq_u_abs}, $u(t,x)$ 
is a non-negative function whenever $t^2 > x^2 - 1$ and an increasing
function for $0 \leq t \leq \sqrt{x^2-1}$. Therefore,
\begin{align}
\int_{-\infty}^{\infty} u(t,x) \; dt 
& \; \geq
2 \cdot u(0,x) \cdot \sqrt{x^2-1} \nonumber \\
& \; = 
2 \cdot \brk{\sqrt{1-\frac{x_0^2-1}{x^2-1}}-1} \cdot \sqrt{x^2-1} \nonumber \\
& \; \geq -2 \cdot \frac{x_0^2-1}{x^2-1} \cdot \sqrt{x^2-1} =
     -2 \cdot \frac{x_0^2-1}{\sqrt{x^2-1}}.
\label{eq_u_below}
\end{align}
By taking the limit $x \to \infty$ in \eqref{eq_u_below}, we conclude that,
for all $x \geq x_0$,
\begin{align}
\int_{-\infty}^{\infty} u(t,x) \; dt \geq 0.
\label{eq_u_bound_2}
\end{align}
Thus \eqref{eq_harm_integral} follows from the combination
of \eqref{eq_u_bound} and \eqref{eq_u_bound_2}.
\end{proof}
%%%%%%%%%%%
\begin{thm}
Suppose that $x > x_0 > 1$ are real numbers. We define the function 
$R: \Rc \to \Rc$ via the formula
\begin{align}
R(t) = \abrk{ \brk{ (x+it)^2 - 1} \cdot \brk{ (x+it)^2 - x_0^2 } }^{-1}.
\label{eq_big_r_def}
\end{align}
Then, for all real $t$,
\begin{align}
\frac{ R'(t) }{ R(t) } = -t \cdot 
\sum_{j=1}^4 \frac{1}{t^2+\delta_j(x,x_0)^2}
\label{eq_drr}
\end{align}
where $\delta_j(x,x_0)$ are defined via \eqref{eq_deltas}
for all $j=1,2,3,4$.
Moreover,
\begin{align}
\int_0^{\infty} \brk{\frac{R'(t)}{R(t)}}^2 \; dt = 
\frac{\pi}{2} \sum_{i,j=1}^4 \frac{1}{\delta_i(x,x_0) + \delta_j(x,x_0)}.
\label{eq_drr_int}
\end{align}
\label{lem_drr}
\end{thm}
\begin{proof}
We observe that
\begin{align}
\frac{R'(t)}{R(t)} 
& \; =
\frac{d}{dt} \log R(t) = \frac{1}{2} \cdot \frac{d}{dt} \log R^2(t) 
\nonumber \\
& \; = -\frac{1}{2} \cdot \frac{d}{dt} 
  \log \prod_{j=1}^4 |\delta_j(x,x_0) + it|^2 \nonumber\\
& \; = -\frac{1}{2} \sum_{j=1}^4 \frac{d}{dt} \log |\delta_j(x,x_0) + it|^2,
\label{eq_drr_1}
\end{align}
where $\delta_1, \delta_2, \delta_3, \delta_4$ are defined via 
\eqref{eq_deltas}. We note that, for any real number $a$,
\begin{align}
\frac{d}{dt} \log |a + it|^2 = \frac{d}{dt} \log(a^2 + t^2) = 
\frac{ 2t }{ a^2 + t^2 },
\label{eq_drr_2}
\end{align}
and thus \eqref{eq_drr} follows from the combination of
\eqref{eq_drr_1} and \eqref{eq_drr_2}. Next, for any real numbers $a, b > 0$,
\begin{align}
& \int_0^{\infty} \frac{ t^2 \; dt }{ (t^2+a^2)\cdot(t^2+b^2)}
= \nonumber \\
& i\pi\brk{\text{Res}\left[ \frac{z^2}{(z^2+a^2)\cdot(z^2+b^2)} ; z=ia \right] +
\text{Res}\left[ \frac{z^2}{(z^2+a^2)\cdot(z^2+b^2)} ; z=ib \right]} 
= \nonumber \\
& i\pi \brk{ \frac{ (ia)^2 }{ 2ia (b^2-a^2)} + 
                      \frac{ (ib)^2 }{ 2ib (a^2-b^2)} }
= \frac{\pi}{2} \cdot \frac{1}{a+b},
\label{eq_drr_3}
\end{align}
and thus \eqref{eq_drr_int} follows from the combination of
\eqref{eq_drr} and \eqref{eq_drr_3}.
\end{proof}
%%%%%%%%%%%
\begin{thm}
Suppose that $x > x_0 > 1$ are real numbers, and that the function
$R: \Rc \to \Rc$ is defined via \eqref{eq_big_r_def} in Theorem~\ref{lem_drr}.
Suppose furthermore, that $c,s>0$ are real numbers. Then,
\begin{align}
& c \int_0^s
\sqrt{\frac{1}{2} \abrk{ \frac{(x+it)^2-x_0^2}{(x+it)^2-1} } +
\frac{1}{2} \Re\brk{ \frac{(x+it)^2-x_0^2}{(x+it)^2-1} } + 
\brk{\frac{R'(t)}{4cR(t)}}^2 } \; dt \; \leq \nonumber \\
& cs + \log b_c(x,x_0),
\label{eq_big_int}
\end{align}
where $b_c(x,x_0)$ is defined via \eqref{eq_bc_def} 
in Definition~\ref{def_bc}.
\label{lem_big_int}
\end{thm}
\begin{proof}
Suppose that the function $u: \Rc^2 \to \Rc$ is defined via
\eqref{eq_u_tx} in Theorem~\ref{lem_harmonic}. Then the left-hand side
of \eqref{eq_big_int} can be written as
\begin{align}
& c \int_0^s
\sqrt{\frac{1}{2} \abrk{ \frac{(x+it)^2-x_0^2}{(x+it)^2-1} } +
\frac{1}{2} \Re\brk{ \frac{(x+it)^2-x_0^2}{(x+it)^2-1} } + 
\brk{\frac{R'(t)}{4cR(t)}}^2 } \; dt = \nonumber \\
& c\int_0^s dt + 
  c\int_0^s u(t,x) \; dt + \nonumber \\
& c \int_0^s \brk{ \sqrt{ (u(t,x)+1)^2 + \brk{\frac{R'(t)}{4cR(t)}}^2 } - 
  \sqrt{ (u(t,x)+1)^2 }} dt,
\label{eq_big_int_0}
\end{align}
where the function $R: \Rc \to \Rc$ is defined via
\eqref{eq_big_r_def} in Theorem~\ref{lem_drr}. Due to 
Theorem~\ref{lem_harmonic} and \eqref{eq_u_ineq_2},
\begin{align}
c\int_0^s u(t,x) \; dt < 0.
\label{eq_big_int_1}
\end{align}
Also, due to \eqref{eq_u_re} and \eqref{eq_u_abs} in the proof of 
Theorem~\ref{lem_harmonic}, for all real $t \geq 0$,
\begin{align}
(u(t,x) + 1)^2 \geq (u(0,x) + 1)^2 = \frac{x^2-x_0^2}{x^2-1}.
\label{eq_big_int_2}
\end{align}
We combine \eqref{eq_big_int_2} with \eqref{eq_drr_int} 
in Theorem~\ref{lem_drr} to conclude that
\begin{align}
& c \int_0^s \brk{ \sqrt{ (u(t,x)+1)^2 + \brk{\frac{R'(t)}{4cR(t)}}^2 } - 
  \sqrt{ (u(t,x)+1)^2 }} dt \leq \nonumber \\
& \frac{c}{ 2 \sqrt{ (u(0,x)+1)^2 } } 
  \int_0^s \brk{\frac{ R'(t) }{ 4cR(t) }}^2 dt =
\frac{1}{32c} \cdot \sqrt{ \frac{x^2-1}{x^2-x_0^2} } 
\int_0^s \brk{\frac{ R'(t) }{ R(t) }}^2 dt < \nonumber \\
& \frac{\pi}{64c} \cdot \sqrt{ \frac{x^2-1}{x^2-x_0^2} } \cdot
  \sum_{i,j=1}^4 \frac{1}{\delta_i(x,x_0) + \delta_j(x,x_0)} = 
\log b_c(x,x_0),
\label{eq_big_int_3}
\end{align}
where $\delta_1,\delta_2,\delta_3,\delta_4$ are defined
via \eqref{eq_deltas},
and $b_c(x,x_0)$ is defined via \eqref{eq_bc_def}
in Definition~\ref{def_bc}.
Thus \eqref{eq_big_int}
follows from the combination of
\eqref{eq_big_int_0}, \eqref{eq_big_int_1} and \eqref{eq_big_int_3}.
\end{proof}
%%%%%%%%%%%
\begin{thm}
Suppose that $n \geq 0$ is an integer, and that $\chi_n > c^2$.
Suppose also that $x$ is a root of $\psi_n$ in $(1, \infty)$.
Suppose furthermore that the function $Q: \Rc \to \Rc$
is defined via the formula
\begin{align}
Q(t) = \abrk{ \psi_n(x+it) }^2 + \abrk{ \psi_n'(x+it)}^2 \cdot
   \frac{ \abrk{(x+it)^2 -1} }{ \abrk{c^2 (x+it)^2-\chi_n} }.
\label{eq_q_sharp_0}
\end{align}
Then, for all real $t>0$,
\begin{align}
\sqrt{Q(t)} \leq
\frac{ \abrk{ \psi_n'(x) } }{ ct } \cdot
\frac{ \brk{x^2-1}^{3/4} }{ \brk{x^2-(\chi_n/c^2)}^{1/4} } \cdot
e^{ct} \cdot b_c\left(x,\frac{\sqrt{\chi_n}}{c}\right),
\label{eq_phi_q_sharp_inequality}
\end{align}
where $b_c$ is defined via \eqref{eq_bc_def}.
\label{lem_q_sharp}
\end{thm}
\begin{proof}
We define the function $\varphi: \Rc \to \Cc$ via the formula
\begin{align}
\varphi(t) = \psi_n(x+it).
\label{eq_q_sharp_1}
\end{align}
Due to \eqref{eq_prolate_ode}, $\varphi$ satisfies the ODE
\begin{align}
\brk{(x+it)^2-1}\cdot\varphi''(t) + 2i(x+it)\cdot\varphi'(t)+
  (\chi_n-c^2(x+it)^2)\cdot\varphi(t) = 0.
\label{eq_q_sharp_2}
\end{align}
We define the functions $w,u: \Rc \to \Cc$ via the formulae
\begin{align}
w(t) = \varphi(t), \quad u(t) = \brk{ (x+it)^2-1 } \cdot \varphi'(t).
\label{eq_q_sharp_3}
\end{align}
Due to \eqref{eq_q_sharp_2}, the functions $w,u$ satisfy the equation
\begin{align}
\label{eq_q_sharp_4}
\begin{pmatrix} w'(t) \\ u'(t) \end{pmatrix} =
\begin{pmatrix} 0 & \beta(t) \\ \gamma(t) & 0 \end{pmatrix}
\begin{pmatrix} w(t) \\ u(t) \end{pmatrix},
\end{align}
where the functions $\beta,\gamma: \Rc \to \Cc$ are defined
via the formulae
\begin{align}
\beta(t) = \brk{ (x+it)^2-1 }^{-1}, \quad
\gamma(t) = c^2 (x+it)^2 - \chi_n.
\label{eq_q_sharp_5}
\end{align}
We combine Theorem~\ref{thm_lewis} in Section~\ref{sec_growth_ode} 
with Theorem~\ref{lem_big_int} above to conclude that, for all real $t>0$,
\begin{align}
\sqrt{ \frac{Q(t)}{Q(0)} } \leq 
\brk{ \frac{R(t)}{R(0)} }^{\frac{1}{4}} \cdot e^{ct} \cdot 
b_c\left(x, \frac{\sqrt{\chi_n}}{c}\right),
\label{eq_q_sharp_6}
\end{align}
where $b_c$ is defined via \eqref{eq_bc_def}, $Q$ is defined
via \eqref{eq_q_sharp_0}, and the function $R: \Rc \to \Rc$ is defined
via the formula
\begin{align}
R(t) = \brk{ \abrk{ (x+it)^2 - 1} \cdot \abrk{ (x+it)^2 - (\chi_n/c^2) } }^{-1}.
\label{eq_q_sharp_7}
\end{align}
Since $\psi_n(x) = 0$ by assumption, it follows that
\begin{align}
\sqrt{Q(0)} = \frac{ \abrk{\psi_n'(x)} }{ c } \cdot 
   \sqrt{\frac{x^2-1}{x^2-(\chi_n/c^2)}}.
\label{eq_q_sharp_8}
\end{align}
Moreover, for all real $t>0$,
\begin{align}
\frac{R(t)}{R(0)} & \; = 
\frac{ \brk{x^2-1} \cdot \brk{x^2-(\chi_n/c^2)} }
   { \abrk{ (x+it)^2-1 } \cdot \abrk{ (x+it)^2-(\chi_n/c^2) } } \nonumber \\
& \; \leq
\frac{ \brk{x^2-1} \cdot \brk{x^2-(\chi_n/c^2)} }{ t^4 }.
\label{eq_q_sharp_9}
\end{align}
Thus \eqref{eq_phi_q_sharp_inequality} follows 
from the combination of \eqref{eq_q_sharp_6}, \eqref{eq_q_sharp_8}
and \eqref{eq_q_sharp_9}.
\end{proof}
%%%%%%%%%%%%%%
In the following theorem, we derive a lower bound on $|\psi_n'(x)|$,
where $x$ is a root of $\psi_n$ in $(1,\infty)$. It is illustrated
in Tables~\ref{t:test84a},~\ref{t:test84b}
(see Experiment 6 in Section~\ref{sec_exp6}).
%%%%%%%%%%%%%%
\begin{thm}[A sharper bound on $\abrk{\psi_n'(x)}$ at roots]
Suppose that $n\geq 0$ is an integer, and that
$\chi_n > c^2$. Suppose also that $x$ is 
a root of $\psi_n$ in $(1,\infty)$. Then,
\begin{align}
\frac{1}{\abrk{\psi_n'(x)}} \leq 
\frac{ \abrk{\lambda_n} }{ \abrk{\psi_n(1)} \sqrt{2} } \cdot 
\frac{ (x^2-1)^{\frac{3}{4}} }{ (x^2-(\chi_n/c^2))^{\frac{1}{4}} } 
\cdot b_c\left(x,\frac{\sqrt{\chi_n}}{c} \right),
\label{eq_12_10_sharp}
\end{align}
where $b_c$ is defined via \eqref{eq_bc_def}.
\label{thm_12_10_sharp}
\end{thm}
\begin{proof}
We combine Theorem~\ref{lem_q_at_infty} with Theorem~\ref{lem_q_sharp}
and take $t \to \infty$ to conclude that
\begin{align}
\frac{ e^{ct} \abrk{\psi_n(1)} \sqrt{2} }
                      { ct \abrk{\lambda_n} }
\leq
\frac{ \abrk{ \psi_n'(x) } }{ ct } \cdot
\frac{ \brk{x^2 - 1}^{\frac{3}{4}} }{ \brk{x^2 - (\chi_n/c^2)}^{\frac{1}{4}} }
\cdot
e^{ct} \cdot b_c\left(x,\frac{\sqrt{\chi_n}}{c}\right),
\label{eq_thm_sharp_1}
\end{align}
which implies \eqref{eq_12_10_sharp}.
\end{proof}
%%%%%%%%%%%
The following theorem provides a bound on 
$b_c\left(x,\sqrt{\chi_n}/c\right)$, defined
via \eqref{eq_bc_def} in Definition~\ref{def_bc}
and used in Theorem~\ref{thm_12_10_sharp}.
\begin{thm}
Suppose that $n \geq 0$ is an integer, and that
$\chi_n > c^2$. Suppose also that $x$ is a root
of $\psi_n$ in $(1,\infty)$. Then,
\begin{align}
b_c\left(x,\frac{\sqrt{\chi_n}}{c}\right) \leq e^{1/4},
\label{eq_bc_ineq}
\end{align}
where $b_c$ is defined via \eqref{eq_bc_def}.
\label{thm_bc_big}
\end{thm}
\begin{proof}
Obviously, $b_c(x,x_0)$, defined via \eqref{eq_bc_def},
is a decreasing function of $x$ for a fixed real number $x_0>1$. Therefore,
for all real $x_0>1$,
\begin{align}
b_c(x,x_0) \leq b_c(x_1,x_0),
\label{eq_bc_0}
\end{align}
where $x_1$ is the minimal root of $\psi_n$ in $\brk{1,\infty}$
(see also Theorem~\ref{thm_four}).
We use
\eqref{eq_deltas} to conclude that
\begin{align}
\sum_{i,j=1}^4 \frac{1}{\delta_i\left(x,\frac{\sqrt{\chi_n}}{c}\right)
                       +\delta_j\left(x,\frac{\sqrt{\chi_n}}{c}\right)} < 
\frac{16}{2\cdot(x_1-(\sqrt{\chi_n}/c))} =
\frac{8}{x_1-(\sqrt{\chi_n}/c)}.
\label{eq_bc_1}
\end{align}
Also, due to \eqref{eq_x1_x0_smart} in Theorem~\ref{thm_x1_x0_good},
\begin{align}
\sqrt{\frac{x_1^2-1}{x_1^2-(\chi_n/c^2)}} <
\frac{2}{\pi} \cdot c \cdot \left(x_1 - \frac{\sqrt{\chi_n}}{c}\right).
\label{eq_bc_2}
\end{align}
We combine \eqref{eq_bc_0}, \eqref{eq_bc_1} and \eqref{eq_bc_2}
to conclude that
\begin{align}
b_c\left(x, \frac{\sqrt{\chi_n}}{c} \right) & \; \leq \exp\left[ 
   \frac{\pi}{64c} \cdot 
   \frac{8}{x_1-(\sqrt{\chi_n}/c)} \cdot
\frac{2}{\pi} \cdot c \cdot \left(x_1 - \frac{\sqrt{\chi_n}}{c}\right) \right]
= e^{1/4},
\end{align}
which implies \eqref{eq_bc_ineq}.
\end{proof}

The following theorem is a direct consequence of 
Theorems~\ref{thm_12_10_sharp}, \ref{thm_bc_big}.
This is the principal result of this subsection.

%%%%%%%%%%%%%%
\begin{thm}[A sharper bound on $\abrk{\psi_n'(x)}$ at roots]
Suppose that $n\geq 0$ is an integer, and that
$\chi_n > c^2$. Suppose also that $x$ is 
a root of $\psi_n$ in $(1,\infty)$. Then,
\begin{align}
\frac{1}{\abrk{\psi_n'(x)}} \leq 
e^{1/4} \cdot \abrk{\lambda_n} \cdot 
\frac{ (x^2-1)^{\frac{3}{4}} }{ (x^2-(\chi_n/c^2))^{\frac{1}{4}} }.
\label{eq_sharp_simple}
\end{align}
\label{thm_sharp_simple}
\end{thm}
\begin{proof}
We combine Theorems~\ref{thm_psi1_bound},
\ref{thm_12_10_sharp}, \ref{thm_bc_big} to obtain \eqref{eq_sharp_simple}.
\end{proof}

%%%%%%%%%%%%%%%%%%%%%%%%%%%%%%%%%%%%%%%%%%%%%
\subsection{Partial Fractions Expansion of $1/\psi_n$}
\label{sec_one_over_psi}
In this subsection, we analyze the
function $1/\psi_n(z)$ of the complex variable $z$.
This function is meromorphic with $n$ simple poles inside $\left(-1, 1\right)$
and infinitely many real simple poles $\pm x_1, \pm x_2, \dots$
outside $\left(-1, 1\right)$ (see
Theorems~\ref{thm_pswf_main},~\ref{thm_five} in Section~\ref{sec_pswf} and
Theorem~\ref{thm_four},
Corollary~\ref{cor_infinite} in Section~\ref{sec_first_order}).
For $-1 < t < 1$, we use Theorem~\ref{thm_cauchy} of 
Section~\ref{sec_misc_tools} to construct the
partial fractions expansion of
$1/\psi_n(t)$ (see \eqref{eq_pf_cauchy} in Section~\ref{sec_pf}).
Then, we establish that the contribution of
the poles $\pm x_1, \pm x_2, \dots$ to this expansion 
is of order $|\lambda_n|$.
This statement is made precise in Theorems~\ref{thm_complex},
\ref{thm_complex_summary},
which are the principal results of this subsection.

%%%%%%%%%%%%%
\subsubsection{Contribution of the Head of the Series 
\eqref{eq_pf_cauchy}}
\label{sec_head}
We use the results of Section~\ref{sec_oscillation} and 
Section~\ref{sec_growth} to bound the contribution
of the first few summands of the series \eqref{eq_pf_cauchy}
in Section~\ref{sec_pf}.
This is summarized in Theorem~\ref{lem_head_main} below.
In Theorem~\ref{lem_6_10}, we provide an upper bound on the contribution
of two consecutive summands of \eqref{eq_pf_cauchy}.
Theorem~\ref{lem_6_10} is illustrated in Table~\ref{t:test85}
(see Experiment 7 in Section~\ref{sec_exp7}).
%%%%%%%%%%%%%%%%%%%%%%%%%%%%%%%%%%%%%%%%%%%%%%
\begin{thm}[contribution of consecutive roots]
Suppose that $n \geq 0$ is an integer, and that
$\chi_n > c^2$. Suppose also that 
$x < y$ are two consecutive roots of $\psi_n$ in $(1,\infty)$.
Then,
\begin{align}
\label{eq_6_10}
& \abrk{ \frac{1}{\brk{t-x} \psi_n'(x)} + \frac{1}{\brk{t-y} \psi_n'(y)} }
\leq
e^{1/4} \cdot \abrk{\lambda_n} \cdot
  \int_x^y \frac{ (z+1)^2 \; dz}{\brk{z^2-\left(\chi_n/c^2\right)}^{3/2}},
\end{align}
for all real $t$ in the interval $(-1,1)$.
\label{lem_6_10}
\end{thm}
\begin{proof}
Suppose that $-1 < t < 1$ is a real number. To prove \eqref{eq_6_10},
we distinguish between two cases. In the first case,
\begin{align}
\frac{ 1 }{ \brk{x-t} \abrk{\psi_n'(x)} } \geq
\frac{ 1 }{ \brk{y-t} \abrk{\psi_n'(y)} }.
\label{eq_6_10_case_1}
\end{align}
We combine \eqref{eq_6_10_case_1} with 
Theorem~\ref{thm_consec_der} in Section~\ref{sec_two_by_two}
and Theorem~\ref{thm_four}
to obtain
\begin{align}
& \abrk{ \frac{1}{\brk{t-x} \psi_n'(x)} + \frac{1}{\brk{t-y} \psi_n'(y)} } 
= \nonumber \\
& \frac{ 1 }{ \brk{x-t} \abrk{\psi_n'(x)} } -
  \frac{ 1 }{ \brk{y-t} \abrk{\psi_n'(y)} }
\leq 
\frac{ 1 }{ \brk{x-1} \abrk{\psi_n'(x)} } -
  \frac{ 1 }{ \brk{y-1} \abrk{\psi_n'(y)} } 
\leq \nonumber \\
& \frac{ 1 }{ \abrk{\psi_n'(x)} } 
\brk{
  \frac{1}{x-1} - \frac{1}{y-1} \cdot
  \sqrt{ \frac{x^2-(\chi_n/c^2)}{x^2-1} \cdot 
         \frac{y^2-1}{y^2-(\chi_n/c^2)} } }.
\label{eq_6_10_a}
\end{align}
We substitute \eqref{eq_sharp_simple} of Theorem~\ref{thm_sharp_simple}
into \eqref{eq_6_10_a} and carry out straightforward algebraic
manipulations
to obtain
\begin{align}
& \abrk{ \frac{1}{\brk{t-x} \psi_n'(x)} + \frac{1}{\brk{t-y} \psi_n'(y)} } 
\leq \nonumber \\
& e^{1/4} \cdot |\lambda_n| \cdot
\frac{ \brk{x^2 - 1}^{\frac{3}{4}} }
     { \brk{x^2 - \left(\chi_n/c^2\right)}^{\frac{1}{4}} }
\brk{
  \frac{1}{x-t} - \frac{1}{y-t}
  \sqrt{ \frac{x^2-\left(\chi_n/c^2\right)}{x^2-1} 
  \cdot \frac{y^2-1}{y^2-\left(\chi_n/c^2\right)} } }
\leq \nonumber \\
& e^{1/4} \cdot |\lambda_n| \cdot
\brk{x^2-1}^{\frac{1}{4}} \brk{x^2-\left(\chi_n/c^2\right)}^{\frac{1}{4}}
\brk{ g(x) - g(y) },
\label{eq_6_10_b}
\end{align}
where the function $g:(\sqrt{\chi_n}/c, \infty) \to \Rc$ is defined
via the formula
\begin{align}
g(z) = \sqrt{ \frac{ z + 1 }
{ \brk{z-1} \brk{z^2 - \left(\chi_n/c^2\right)} } }.
\label{eq_6_10_c}
\end{align}
We differentiate \eqref{eq_6_10_c} with respect to $z$ to obtain
\begin{align}
g'(z) & \; = 
\frac{ \sqrt{ \brk{z-1} \brk{z^2-\left(\chi_n/c^2\right)} } }
     { 2 \sqrt{z+1} } \cdot
\frac{ -2z^3 - 2z^2 + 2z + 2\left(\chi_n/c^2\right) }
     { \brk{z-1}^2 \brk{z^2-\left(\chi_n/c^2\right)}^2 } \nonumber \\
& \; > -
\frac{ (z+1)^2 }{ \brk{z^2-\left(\chi_n/c^2\right)}^{3/2} \cdot \sqrt{z^2-1} }.
\label{eq_6_10_d}
\end{align}
We substitute \eqref{eq_6_10_d} into \eqref{eq_6_10_b} to obtain
\begin{align}
& \abrk{ \frac{1}{\brk{t-x} \psi_n'(x)} + \frac{1}{\brk{t-y} \psi_n'(y)} } 
\leq \nonumber \\
& e^{1/4} \cdot |\lambda_n| \cdot
\brk{x^2-1}^{\frac{1}{4}} \brk{x^2-\left(\chi_n/c^2\right)}^{\frac{1}{4}}
\cdot \int_x^y |g'(z)| \; dz
\leq \nonumber \\
& e^{1/4} \cdot |\lambda_n| \cdot
  \int_x^y \frac{ (z+1)^2 \; dz}{\brk{z^2-\left(\chi_n/c^2\right)}^{3/2}},
\label{eq_6_10_e}
\end{align}
which establishes \eqref{eq_6_10} under the assumption \eqref{eq_6_10_case_1}.
%%%%%%%%%%%
If, on the other hand,
\begin{align}
\frac{ 1 }{ \brk{x-t} \abrk{\psi_n'(x)} } <
\frac{ 1 }{ \brk{y-t} \abrk{\psi_n'(y)} },
\label{eq_6_10_case_2}
\end{align}
then we combine
\eqref{eq_6_10_case_2} 
with Theorem~\ref{thm_consec_der} in Section~\ref{sec_two_by_two}
to obtain
\begin{align}
& \abrk{ \frac{1}{\brk{t-x} \psi_n'(x)} + \frac{1}{\brk{t-y} \psi_n'(y)} } 
= \nonumber \\
& \frac{ 1 }{ \brk{y-t} \abrk{\psi_n'(y)} } -
  \frac{ 1 }{ \brk{x-t} \abrk{\psi_n'(x)} } 
\leq 
 \frac{ 1 }{ \brk{y+1} \abrk{\psi_n'(y)} } -
  \frac{ 1 }{ \brk{x+1} \abrk{\psi_n'(x)} } 
  \leq  \nonumber \\
& \frac{1}{\abrk{\psi_n'(y)}} \cdot \frac{1}{y+1} \cdot
\left(1 - \frac{y+1}{x+1} \cdot \frac{x^2-1}{y^2-1} \right) = 
 \frac{1}{\abrk{\psi_n'(y)}} \cdot \frac{y-x}{y^2-1}.
\label{eq_6_10_f}
\end{align}
We substitute \eqref{eq_sharp_simple} of Theorem~\ref{thm_sharp_simple}
into \eqref{eq_6_10_f} to obtain
\begin{align}
& \abrk{ \frac{1}{\brk{t-x} \psi_n'(x)} + \frac{1}{\brk{t-y} \psi_n'(y)} } 
\leq \nonumber \\
& e^{1/4} \cdot \abrk{\lambda_n} \cdot 
\frac{ (y^2-1)^{\frac{3}{4}} }{ (y^2-(\chi_n/c^2))^{\frac{1}{4}} } \cdot
\frac{y-x}{y^2-1} \leq \nonumber \\
& e^{1/4} \cdot \abrk{\lambda_n} \cdot 
\int_x^y \frac{ dz }{ (z^2-(\chi_n/c^2))^{\frac{1}{2}} },
\label{eq_6_10_g}
\end{align}
which establishes \eqref{eq_6_10} under the assumption \eqref{eq_6_10_case_2}.
\end{proof}
%%%%%%%%%%%
The following theorem is a generalization of Theorem~\ref{lem_6_10}.
\begin{thm}
Suppose that $n \geq 0$ is an integer, and that $\chi_n > c^2$.
Suppose also that $1 < x_1 < x_2 < \dots$ are the roots of $\psi_n$
in $(1,\infty)$, and that $M>0$ is an even integer.
Then, for all real $-1 < t < 1$,
\begin{align}
\left|
\sum_{k=1}^M \frac{1}{(t-x_k) \cdot \psi_n'(x_k)}
\right| < 4 e^{1/4} \cdot |\lambda_n| \cdot 
\left( \log(2 \cdot x_M) + \sqrt{1 + \frac{\sqrt{\chi_n}}{\pi}} \right).
\label{eq_head_main}
\end{align}
\label{lem_head_main}
\end{thm}
\begin{proof}
Due to Theorem~\ref{lem_6_10} above,
\begin{align}
\left|
\sum_{k=1}^M \frac{1}{(t-x_k) \cdot \psi_n'(x_k)}
\right| 
& \; \leq e^{1/4} \cdot |\lambda_n| \cdot 
\sum_{k=1}^{M/2}
\int_{x_{2k-1}}^{x_{2k}} 
\frac{ (z+1)^2 \; dz}{\brk{z^2-\left(\chi_n/c^2\right)}^{3/2}} \nonumber \\
& \; < e^{1/4} \cdot |\lambda_n| \cdot 
\int_{x_1}^{x_M}
\frac{ (z+1)^2 \; dz}{\brk{z^2-\left(\chi_n/c^2\right)}^{3/2}}.
\label{eq_head_main_a}
\end{align}
We observe that
\begin{align}
\int
\frac{ z^2 \; dz}{\brk{z^2-\left(\chi_n/c^2\right)}^{3/2}} =
\log\left( z + \sqrt{z^2-\left(\chi_n/c^2\right)}\right) -
\frac{z}{\sqrt{z^2-\left(\chi_n/c^2\right)}},
\label{eq_head_main_b}
\end{align}
and combine \eqref{eq_head_main_b} with \eqref{eq_head_main_a}
to obtain
\begin{align}
& \left|
\sum_{k=1}^M \frac{1}{(t-x_k) \cdot \psi_n'(x_k)}
\right| < 4 e^{1/4} \cdot |\lambda_n| \cdot  \left(
\log\left( 2 x_M \right) + 
\frac{x_1}{\sqrt{x_1^2-\left(\chi_n/c^2\right)}}
\right).
\label{eq_head_main_c}
\end{align}
It follows from the combination of Theorem~\ref{thm_four} and
Theorem~\ref{thm_x1_x0_good} that
\begin{align}
\frac{x_1}{\sqrt{x_1^2-\left(\chi_n/c^2\right)}}
& \; =
\sqrt{1 + \frac{\left(\chi_n/c^2\right)}{x_1^2-\left(\chi_n/c^2\right)}}
\nonumber \\
& \; \leq
\sqrt{1 + \frac{\sqrt{\chi_n}}{2c} \cdot \frac{2c}{\pi} }
= \sqrt{1 + \frac{\sqrt{\chi_n}}{\pi}},
\label{eq_head_main_d}
\end{align}
and we substitute \eqref{eq_head_main_d} into \eqref{eq_head_main_c}
to conclude the proof.
\end{proof}
%%%%%%%%%%%%%%%%%%%%%%%%%%%%%%%%%%%%%%%%%%%

%%%%%%%%%%%%%
\subsubsection{Contribution of the Tail of the Series 
\eqref{eq_pf_cauchy}}
\label{sec_tail}

In the following theorem, we establish an upper bound on $\chi_n$
in terms of $|\lambda_n|$.
\begin{thm}
Suppose that $n > 0$ is a positive integer, and that
\begin{align}
c > 30.
\label{eq_c_30}
\end{align}
Suppose also that
\begin{align}
|\lambda_n| < \frac{1}{10}.
\label{eq_lambda_10}
\end{align}
Then,
\begin{align}
\chi_n - c^2 < \frac{c^2}{|\lambda_n|}.
\label{eq_khi_lambda_bound}
\end{align}
\label{thm_khi_lambda_bound}
\end{thm}
\begin{proof}
Suppose first that
%%% it's about 0.5067
\begin{align}
n < \frac{2c}{\pi} + \frac{2}{\pi^2} \cdot \frac{10}{16} \cdot
\log\left( \frac{64e\pi}{10} \right) \cdot c.
\label{eq_klb_a}
\end{align}
We combine Theorems~\ref{thm_n_and_khi},~\ref{thm_khi_2}
in Section~\ref{sec_pswf} with \eqref{eq_c_30}, \eqref{eq_lambda_10}
to conclude that
\begin{align}
\chi_n - c^2 < 10 \cdot c^2,
\label{eq_klb_b}
\end{align}
provided that \eqref{eq_klb_a} holds. If, on the other hand,
\begin{align}
n \geq \frac{2c}{\pi} + \frac{2}{\pi^2} \cdot \frac{10}{16} \cdot
\log\left( \frac{64e\pi}{10} \right) \cdot c,
\label{eq_klb_c}
\end{align}
then we combine \eqref{eq_klb_c} with Theorem~\ref{thm_khi_1}
in Section~\ref{sec_pswf} to obtain
\begin{align}
\chi_n - c^2 > \frac{4}{\pi} \cdot \frac{10}{16} \cdot c^2
= \frac{5}{2\pi} \cdot c^2.
\label{eq_klb_d}
\end{align}
Suppose now that the function $f : (0,\infty) \times (1,\infty) \to \Rc$
is defined via the formula
\begin{align}
f(c,y) = 1195 \cdot y^{10} \cdot c \cdot 
\exp\left[ -\frac{\pi\cdot \left( y^2 - 1 \right) \cdot c}{4y}
   \right].
\label{eq_klb_e}
\end{align}
We differentiate \eqref{eq_klb_e} with respect to $c$ to obtain
\begin{align}
\frac{\partial f}{\partial c}(c,y) =
\frac{f(c,y)}{c} \cdot
\left(
1 -\frac{\pi\cdot \left( y^2 - 1 \right) \cdot c}{4y}
\right).
\label{eq_klb_f}
\end{align}
Also, we differentiate \eqref{eq_klb_e} with respect to $y$ to obtain
\begin{align}
\frac{\partial f}{\partial y}(c,y) =
\frac{f(c,y)}{y} \cdot
\left(
10 -\frac{\pi\cdot \left( y^2 + 1 \right) \cdot c}{4y}
\right).
\label{eq_klb_g}
\end{align}
We define the real number $y_0$ via the formula
\begin{align}
y_0 = \sqrt{1 + \frac{5}{2\pi}},
\label{eq_klb_h}
\end{align}
and combine \eqref{eq_klb_f}, \eqref{eq_klb_g}, \eqref{eq_klb_h} to conclude
that
\begin{align}
\frac{\partial f}{\partial c}(c,y) < 0, \quad
\frac{\partial f}{\partial y}(c,y) < 0,
\label{eq_klb_i}
\end{align}
for all $y \geq y_0$ and all $c \geq 8$.
Also, we defined the real number $c_0$ to be the solution of the
equation
\begin{align}
f(c,y_0) = 1,
\label{eq_klb_j}
\end{align}
in the unknown $c \geq 8$ (this solution is unique due to \eqref{eq_klb_i}).
We carry out elementary calculations to conclude that
\begin{align}
c_0 < 30.
\label{eq_klb_k}
\end{align}
We combine \eqref{eq_klb_h}, \eqref{eq_klb_i}, \eqref{eq_klb_j},
\eqref{eq_klb_k} to conclude that
\begin{align}
f(c,y) < 1,
\label{eq_klb_l}
\end{align}
for all $y > y_0$ and all $c > 30$. 
Suppose now that $n$ satisfies the inequality \eqref{eq_klb_c}.
We define the real number $y_n$
via the formula
\begin{align}
y_n = \sqrt{\frac{\chi_n}{c^2}},
\label{eq_klb_m}
\end{align}
and combine \eqref{eq_c_30}, \eqref{eq_klb_c}, \eqref{eq_klb_d},
\eqref{eq_klb_e}, \eqref{eq_klb_k}, \eqref{eq_klb_l},
\eqref{eq_klb_m} with Theorem~\ref{thm_lambda_khi} in Section~\ref{sec_pswf}
to conclude that
\begin{align}
\frac{\chi_n}{c^2} \cdot |\lambda_n| < f(c,y_n) < 1,
\label{eq_klb_n}
\end{align}
provided that \eqref{eq_klb_c} holds. We combine \eqref{eq_c_30},
\eqref{eq_lambda_10},
\eqref{eq_klb_a}, \eqref{eq_klb_b}, \eqref{eq_klb_c}, \eqref{eq_klb_d},
\eqref{eq_klb_h},
\eqref{eq_klb_n} to obtain \eqref{eq_khi_lambda_bound}, and thus
conclude the proof.
\end{proof}

According to Theorem~\ref{thm_spacing}, the distance between
two large consecutive roots of $\psi_n$ in $(1,\infty)$ is fairly
close to $\pi/c$. In the following theorem, we make
this observation more precise.
%%%%%%%%%%%%
\begin{thm}
Suppose that $n>0$ is a positive integer, and that
\begin{align}
n > \frac{2c}{\pi} + 1.
\label{eq_close_1}
\end{align}
Suppose also that $x, y$ are two consecutive roots of $\psi_n$
in $(1, \infty)$, and that
\begin{align}
\frac{1}{|\lambda_n|} < x < y.
\label{eq_close_2}
\end{align}
Suppose furthermore that 
\begin{align}
|\lambda_n| < \frac{1}{10},
\label{eq_close_3}
\end{align}
and that
\begin{align}
\chi_n - c^2 < \frac{c^2}{|\lambda_n|}.
\label{eq_close_4}
\end{align}
Then,
\begin{align}
\pi \leq c \cdot (y-x) \leq \pi + \frac{2}{|\lambda_n| \cdot x^2}.
\label{eq_close_roots}
\end{align}
\label{lem_close_roots}
\end{thm}
\begin{proof}
Suppose that the functions $\Psi_n, Q_n: (1,\infty) \to \Rc$
are those of Theorem~\ref{lem_trans}.
We combine Theorem~\ref{thm_khi_elem} of Section~\ref{sec_pswf},
\eqref{eq_25_09_dq} in the proof of Theorem~\ref{thm_spacing},
Theorem~\ref{thm_25_09_2} in Section~\ref{sec_oscillation_ode},
\eqref{eq_close_1}, \eqref{eq_close_2} and \eqref{eq_close_3}
to conclude that
\begin{align}
\frac{\pi}{c} \leq y-x.
\label{eq_close_a}
\end{align}
On the other hand, we combine Theorem~\ref{thm_spacing}
with \eqref{eq_close_2}, \eqref{eq_close_3}, \eqref{eq_close_4} to obtain
\begin{align}
c \cdot (y-x) & \; \leq \pi \cdot 
\sqrt{1 + \frac{\left(\chi_n/c^2\right)-1}{x^2-\left(\chi_n/c^2\right)}} \leq
\pi + \frac{\pi}{2} \cdot
      \frac{\left(\chi_n/c^2\right)-1}{x^2-\left(\chi_n/c^2\right)}
  \nonumber \\
& \; \leq \pi + \frac{\pi}{2 \cdot |\lambda_n|} \cdot
     \frac{1}{x^2 - 1 - 1/|\lambda_n|} < \pi + \frac{2}{|\lambda_n|\cdot x^2}.
\label{eq_close_b}
\end{align}
Thus \eqref{eq_close_roots} follows from the combination
of \eqref{eq_close_a} and \eqref{eq_close_b}.
\end{proof}
%%%%%%%%%%%%
The following two theorems are direct consequences of 
the integral equation \eqref{eq_prolate_integral2}
in Section~\ref{sec_pswf}.
%%%%%%%%%%%
\begin{thm}[expansion of $\psi_n(x)$]
Suppose that $n \geq 0$ is a non-negative integer, and
that $x > 1$ is a real number. 
If $n$ is even, then
\begin{align}
\psi_n(x) = \frac{2 \psi_n(1)}{cx\lambda_n}
\left[\sin(cx) + \frac{1}{\lambda_n \psi_n(1)}
   \int_{-1}^1 \frac{ \sin\brk{c(x-t)} \psi_n(t) t }{x - t} \; dt \right].
\label{eq_2_10_even1}
\end{align}
If $n$ is odd, then
\begin{align}
\psi_n(x) = \frac{2 \psi_n(1)}{icx\lambda_n}
\left[\cos(cx) + \frac{1}{i\lambda_n \psi_n(1)}
   \int_{-1}^1 \frac{ \sin\brk{c(x-t)} \psi_n(t) t }{x - t} \; dt \right].
\label{eq_2_10_odd1}
\end{align}
\label{lem_psi_for_large_x}
\end{thm}
\begin{proof}
We observe that
\begin{align}
\frac{1}{x-t} = \frac{1}{x} + \frac{t}{x \cdot (x-t)},
\label{eq_2_10_a}
\end{align}
for all real $-1 < t < 1$. We combine \eqref{eq_2_10_a}
with \eqref{eq_prolate_mu}, \eqref{eq_prolate_integral2} 
in Section~\ref{sec_pswf} to obtain
\begin{align}
\frac{ c \abrk{\lambda_n}^2 }{ 2 \pi } \psi_n(x)
& \; = 
\frac{1}{\pi} \int_{-1}^1 \frac{e^{ic(x-t)} \psi_n(t) }{2i(x-t)} dt \;-\;
\frac{1}{\pi} \int_{-1}^1 \frac{e^{-ic(x-t)} \psi_n(t)}{2i(x-t)} dt
\nonumber \\
& \; =
\frac{e^{icx}}{2\pi i} \int_{-1}^1 \frac{e^{-ict} \psi_n(t)}{x} dt \;+\;
\frac{e^{icx}}{2\pi i} \int_{-1}^1 \frac{e^{-ict} \psi_n(t)t}{x(x-t)} dt
\nonumber \\
& \; -
\frac{e^{-icx}}{2\pi i} \int_{-1}^1 \frac{e^{ict} \psi_n(t)}{x} dt \;-\;
\frac{e^{-icx}}{2\pi i} \int_{-1}^1 \frac{e^{ict} \psi_n(t)t}{x(x-t)}dt
\nonumber \\
& \; =
\frac{e^{icx} \lambda_n \psi_n(-1)}{2\pi i x} -
\frac{e^{-icx} \lambda_n \psi_n(1)}{2\pi i x} 
\nonumber \\
& \; +
\frac{1}{\pi} \int_{-1}^1 \frac{ \sin\brk{c(x-t)} }{x(x-t)} \psi_n(t) t \; dt.
\label{eq_2_10_1}
\end{align}
Due to Theorem~\ref{thm_pswf_main} in Section~\ref{sec_pswf},
\begin{align}
\psi_n(1) = (-1)^n \cdot \psi_n(-1),
\label{eq_2_10_psi1}
\end{align}
and 
\begin{align}
\abrk{\lambda_n}^2 = (-1)^n \cdot \lambda_n^2.
\label{eq_2_10_lambda}
\end{align}
Thus \eqref{eq_2_10_even1} and \eqref{eq_2_10_odd1}
follow from the combination of \eqref{eq_2_10_1},
\eqref{eq_2_10_psi1} and \eqref{eq_2_10_lambda}.
\end{proof}
%%%%%%%%%%%
\begin{thm}[expansion of $\psi_n'(x)$]
Suppose that $n \geq 0$ is a non-negative integer, and
that $x > 1$ is a real number. 
If $n$ is even, then
\begin{align}
\psi_n'(x)
=
\frac{2 \psi_n(1)}{x \lambda_n} \cdot
& \left[ \cos(cx) - \frac{\sin(cx)}{cx} +
\frac{ 1 }{ \lambda_n \psi_n(1) }
    \int_{-1}^1 \frac{ \cos\brk{c(x-t)} \psi_n(t) t }{ x - t} dt \;+
  \right. \nonumber
\\
& \quad \left. \frac{ 1 }{ c \lambda_n \psi_n(1) }
    \int_{-1}^1 \frac{ \sin\brk{c(x-t)} \psi_n(t) \brk{t^2-2xt} }
                     { x \brk{x-t}^2 } \; dt
 \right].
\label{eq_2_10_even2}
\end{align}
If $n$ is odd, then
\begin{align}
\psi_n'(x)
=
-\frac{2 \psi_n(1)}{i x \lambda_n} \cdot
& \left[
  \sin\brk{cx} + \frac{\cos(cx)}{cx} +
 \frac{ i }{ \lambda_n \psi_n(1) }
    \int_{-1}^1 \frac{ \cos\brk{c(x-t)} \psi_n(t) t }{ x - t} dt +
  \right. 
 \nonumber \\
& \quad \left. \frac{ i }{ c \lambda_n \psi_n(1) }
    \int_{-1}^1 \frac{ \sin\brk{c(x-t)} \psi_n(t) \brk{t^2-2xt} }
                     { x \brk{x-t}^2 } \; dt
\right].
\label{eq_2_10_odd2}
\end{align}
\label{lem_dpsi_for_large_x}
\end{thm}
\begin{proof}
The identities \eqref{eq_2_10_even2}, \eqref{eq_2_10_odd2} are obtained,
respectively,
via straightforward differentiation of \eqref{eq_2_10_even1},
\eqref{eq_2_10_odd1} of Theorem~\ref{lem_psi_for_large_x} with respect to $x$.
\end{proof}
%%%%%%%%%%%%%%%%%%%%%
\begin{remark}
In the rest of this subsection, we will 
assume that $n$ is even. 
The analysis for odd values of $n$ is essentially identical,
and will be omitted.
\label{rem_even_odd}
\end{remark}

\begin{thm}
Suppose that $n > 0$ is an even integer, that
\begin{align}
n > \frac{2c}{\pi} + 1,
\label{eq_sin_cos_1}
\end{align}
and that $x, y$ are two consecutive roots of $\psi_n$ in $(1,\infty)$.
Suppose also that
\begin{align}
|\lambda_n| < \frac{1}{10},
\label{eq_sin_cos_2}
\end{align}
and that 
\begin{align}
\frac{1}{|\lambda_n|^2} < x < y.
\label{eq_sin_cos_3}
\end{align}
Suppose furthermore that
\begin{align}
\chi_n - c^2 < \frac{c^2}{|\lambda_n|}, 
\label{eq_sin_cos_4}
\end{align}
and that the positive integer $K(x)$ is defined via the formula
\begin{align}
K(x) = \text{Round}\left( \frac{c}{\pi} \cdot x \right),
\label{eq_kx}
\end{align}
where, for any real number $\alpha$, $\text{Round}(\alpha)$
is the closest integer number to $\alpha$. Then,
\begin{align}
\label{eq_sc_sincx}
& |\sin(cx)| \leq \frac{2}{|\lambda_n| \cdot x}, \\
\label{eq_sc_kx}
& |cx - K(x)\cdot \pi| \leq \frac{\pi}{|\lambda_n| \cdot x}, \\
\label{eq_sc_coscx}
& (-1)^{K(x)} \cdot \cos(cx) \geq 1 - \frac{\pi}{|\lambda_n| \cdot x},
\end{align}
and, moreover, for all real $-1 < t < 1$,
\begin{align}
\label{eq_sc_sindif}
& \left| \sin( c\cdot(y-t) ) + \sin(c \cdot (x-t)) \right|
   \leq \frac{2}{|\lambda_n| \cdot x^2}, \\
\label{eq_sc_cosdif}
& \left| \cos( c\cdot(y-t) ) + \cos(c \cdot (x-t)) \right|
   \leq \frac{2}{|\lambda_n| \cdot x^2}.
\end{align}
\label{lem_sin_cos}
\end{thm}
\begin{proof}
We combine Theorems~\ref{thm_pswf_main}, \ref{thm_psi1_bound}
of Section~\ref{sec_pswf}, \eqref{eq_2_10_even1} of 
Theorem~\ref{lem_psi_for_large_x} with 
\eqref{eq_sin_cos_2}, \eqref{eq_sin_cos_3} to obtain
\begin{align}
\abrk{ \sin(c x_k) }
& \; = \abrk{ -\frac{ 1 }{ \lambda_n \psi_n(1) }
    \int_{-1}^1 
    \frac{ \sin\brk{c(x_k-t)} \cdot \psi_n(t) \cdot t \; dt}{ x_k-t } }
\nonumber \\
& \; \leq \frac{\sqrt{2}}{\abrk{\lambda_n} \brk{x_k - 1} }
  \brk{\int_{-1}^1 \psi_n^2(t) dt }^{\frac{1}{2}} \cdot
  \brk{\int_{-1}^1 t^2 dt }^{\frac{1}{2}}
\nonumber \\
& \; \leq \frac{ 2 }{ \sqrt{3} \abrk{\lambda_n} \brk{x_k - 1} },
\label{eq_2_10_sc_1_proof}
\end{align}
which implies \eqref{eq_sc_sincx}. We observe that, for all real
$-\pi/2 \leq s \leq \pi/2$,
\begin{align}
\left| s \right| \leq \frac{\pi}{2} \cdot 
 \left| \sin(s) \right|,
\label{eq_sin_nice}
\end{align}
and combine \eqref{eq_sin_nice} with \eqref{eq_sc_sincx} to 
obtain \eqref{eq_sc_kx}. The inequality \eqref{eq_sc_coscx}
follows from the combination of \eqref{eq_sc_sincx} and \eqref{eq_sc_kx}.
Finally, both \eqref{eq_sc_sindif} and \eqref{eq_sc_cosdif}
follow from the combination of \eqref{eq_sin_cos_1},
\eqref{eq_sin_cos_2}, \eqref{eq_sin_cos_3},
\eqref{eq_sin_cos_4} and Theorem~\ref{lem_close_roots}.
\end{proof}
%%%%%%%%%%%%%%%%%%%%%%
\begin{thm}
Suppose that $n>0$ is an even positive integer, and that $x,y$ are two
consecutive roots of $\psi_n$ in $(1,\infty)$. Suppose also that
the inequalities \eqref{eq_sin_cos_1},
\eqref{eq_sin_cos_2}, \eqref{eq_sin_cos_3}, \eqref{eq_sin_cos_4}
of Theorem~\ref{lem_sin_cos} hold, and that the integer $K(x)$
is defined via \eqref{eq_kx} in Theorem~\ref{lem_sin_cos}.
Suppose furthermore that
\begin{align}
c > 1.
\label{eq_c_geq_1}
\end{align}
Then, 
\begin{align}
\psi_n'(x) = 
\frac{2 \brk{-1}^{K(x)} \psi_n(1) }{ \lambda_n x } \cdot
\left[1 - D(x)\right]
\label{eq_tail_dpsix}
\end{align}
and
\begin{align}
\psi_n'(y) =
- \frac{2 \brk{-1}^{K(x)} \psi_n(1) }{ \lambda_n y } \cdot
\left[1 - D(x) + G(x)\right],
\label{eq_tail_dpsiy}
\end{align}
where the real numbers $D(x)$ and $G(x)$ satisfy,
respectively, the inequalities
\begin{align}
\left| D(x) \right| \leq \frac{6}{|\lambda_n| \cdot x}
\label{eq_tail_dx}
\end{align}
and
\begin{align}
\left| G(x) \right| \leq \frac{24}{|\lambda_n| \cdot x^2}.
\label{eq_tail_gx}
\end{align}
\label{lem_tail_dpsi}
\end{thm}
\begin{proof}
The proof is based on the identity
\eqref{eq_2_10_even2} of Theorem~\ref{lem_dpsi_for_large_x}.
First, we combine 
Theorems~\ref{thm_pswf_main}, \ref{thm_psi1_bound}
of Section~\ref{sec_pswf},
\eqref{eq_sin_cos_2} and \eqref{eq_sin_cos_3} to obtain
\begin{align}
& \left| \frac{ 1 }{ c \lambda_n \psi_n(1) }
    \int_{-1}^1 \frac{ \sin\brk{c(x-t)} \psi_n(t) \brk{t^2-2xt} }
                     { x \brk{x-t}^2 } \; dt
\right| \leq \nonumber \\
& \frac{3 \sqrt{2}}{c \cdot |\lambda_n| \cdot (x-1)^2} \cdot
\int_{-1}^1 \left| \psi_n(t) \cdot t \right| dt \leq
\frac{4}{c \cdot |\lambda_n| \cdot x^2}.
\label{eq_tail_dpsi_a}
\end{align}
By the same token,
\begin{align}
& \left| \frac{ 1 }{ c \lambda_n \psi_n(1) }
    \int_{-1}^1 \frac{ \sin\brk{c(y-t)} \psi_n(t) \brk{t^2-2yt} }
                     { y \brk{y-t}^2 } \; dt
\right| \leq 
\frac{4}{c \cdot |\lambda_n| \cdot x^2}.
\label{eq_tail_dpsi_b}
\end{align}
Also, we combine \eqref{eq_close_roots} of Theorem~\ref{lem_close_roots} and
\eqref{eq_sin_cos_2}, \eqref{eq_sin_cos_3},
\eqref{eq_sc_cosdif} of Theorem~\ref{lem_sin_cos}
to obtain, for all real $-1 < t < 1$,
\begin{align}
& \left| \frac{\cos(c\cdot(x-t))}{x-t} + \frac{\cos(c\cdot(y-t))}{y-t} \right|
  = \nonumber \\
& \left| \frac{\cos(c\cdot(x-t))+\cos(c\cdot(y-t))}{x-t} +
         \frac{\cos(c(y-t)) \cdot (x-y)}{(y-t)\cdot(x-t)} \right|
  \leq \nonumber \\
& \frac{2}{|\lambda_n| \cdot x^2} \cdot \frac{2}{x-1} + 
  \frac{2}{x^2} \cdot \left(\pi + \frac{2}{|\lambda_n| \cdot x^2}\right)
 \leq
 \frac{8}{x^2}.
\label{eq_tail_dpsi_c}
\end{align}
We combine Theorems~\ref{thm_pswf_main}, \ref{thm_psi1_bound}
of Section~\ref{sec_pswf} with \eqref{eq_tail_dpsi_c} to obtain
\begin{align}
& \left|
\frac{1}{\lambda_n \cdot \psi_n(1)} \cdot
\int_{-1}^1 \left(
\frac{\cos(c\cdot(x-t))}{x-t} + \frac{\cos(c\cdot(y-t))}{y-t}
\right) \cdot \psi_n(t) \cdot t \; dt
\right| \leq \nonumber \\
& \frac{\sqrt{2}}{|\lambda_n|} \cdot \frac{8}{x^2} \cdot
\int_{-1}^1 \left| \psi_n(t) \cdot t \right| dt
\leq \frac{10}{|\lambda_n| \cdot x^2}.
\label{eq_tail_dpsi_d}
\end{align}
We substitute \eqref{eq_sc_sincx}, \eqref{eq_sc_cosdif}
of Theorem~\ref{lem_sin_cos}, \eqref{eq_tail_dpsi_a},
\eqref{eq_tail_dpsi_b}, \eqref{eq_tail_dpsi_d} into
\eqref{eq_2_10_even2} of Theorem~\ref{lem_dpsi_for_large_x} and use
\eqref{eq_c_geq_1} to obtain
\begin{align}
& \left| 
\frac{\lambda_n}{2 \psi_n(1)} \cdot 
  \left(x \cdot \psi_n'(x) + y \cdot \psi_n'(y) \right)
\right| \leq \nonumber \\
& \left| \cos(cx)+\cos(cy) \right| + 
\left| \frac{\sin(cx)}{cx} + \frac{\sin(cy)}{cy} \right| +
\frac{4 + 4 + 10}{|\lambda_n| \cdot x^2} \leq
 \frac{24}{|\lambda_n| \cdot x^2}.
\label{eq_tail_dpsi_e}
\end{align}
In addition, we observe that, similar to \eqref{eq_tail_dpsi_a},
\eqref{eq_tail_dpsi_b}, \eqref{eq_tail_dpsi_c} above,
\begin{align}
& \left| \frac{ 1 }{ \lambda_n \psi_n(1) }
    \int_{-1}^1 \frac{ \cos\brk{c(y-t)} \cdot \psi_n(t) \cdot t }
                     { x-t } \; dt
\right| \leq \nonumber \\
& \frac{\sqrt{2}}{|\lambda_n| \cdot (x-1) } \cdot
\int_{-1}^1 \left| \psi_n(t) \cdot t \right| dt
\leq \frac{2}{|\lambda_n| \cdot x}
\label{eq_tail_dpsi_f}
\end{align}
Finally, we substitute \eqref{eq_sc_sincx}, 
\eqref{eq_sc_coscx}, \eqref{eq_tail_dpsi_e}
and \eqref{eq_tail_dpsi_f} into \eqref{eq_2_10_even2}
of Theorem~\ref{lem_dpsi_for_large_x} to
conclude the proof.
\end{proof}
%%%%%%%%%%%%%%%%%%%%%%%
In the following theorem, we provide an upper bound on
the sum of the principal parts of $1/\psi_n$ at two consecutive
roots of $\psi_n$ in $(1,\infty)$
(see \eqref{eq_pf_cauchy} in Section~\ref{sec_pf}).
%%%%%%%%%%%%%%%%%%%%%%
\begin{thm}
Suppose that $n>0$ is an even positive integer, and that $x,y$ are two
consecutive roots of $\psi_n$ in $(1,\infty)$. Suppose also that
the inequalities \eqref{eq_sin_cos_1},
\eqref{eq_sin_cos_2}, \eqref{eq_sin_cos_3}, \eqref{eq_sin_cos_4}
of Theorem~\ref{lem_sin_cos} hold.
Suppose furthermore that
\begin{align}
c > 1.
\label{eq_two_bound_c}
\end{align}
Then, for all real $-1 < t < 1$,
\begin{align}
\left| \frac{1}{\psi_n'(x) \cdot (x-t)} +
       \frac{1}{\psi_n'(y) \cdot (y-t)} \right| \leq
20 \cdot c \cdot \int_x^y \frac{ds}{s^2}.
\label{eq_two_bound}
\end{align}
\label{lem_two_bound}
\end{thm}
\begin{proof}
Suppose that the integer $K(x)$ is defined via \eqref{eq_kx}
in Theorem~\ref{lem_sin_cos}. We combine \eqref{eq_tail_dpsix},
\eqref{eq_tail_dpsiy},
\eqref{eq_tail_dx},
\eqref{eq_tail_gx} of Theorem~\ref{lem_tail_dpsi} to obtain
\begin{align}
& \left| \frac{1}{\psi_n'(x) \cdot (x-t)} +
       \frac{1}{\psi_n'(y) \cdot (y-t)} \right| = \nonumber \\
& \abrk{
  \frac{ \brk{-1}^{K(x)} \lambda_n }{ 2 \psi_n(1) } \cdot
  \left[
     \frac{x}{\brk{x-t} \brk{1-D(x)} } -
     \frac{y}{\brk{y-t} \brk{1 - D(x) + G(x)} }
  \right]} \leq \nonumber \\
& \frac{ \abrk{\lambda_n} }{ xy } \cdot
  \abrk{
     x \brk{y - t}\brk{1 - D(x) + G(x)} - 
     y \brk{x - t}\brk{1 - D(x)}
  } = \nonumber \\
& \frac{ \abrk{\lambda_n} }{ xy } \cdot
 \abrk{
   xy G(x) + t \brk{y-x}\brk{1-D(x)} - t x G(x)
 } \leq \nonumber \\
& 2 \abrk{ \lambda_n G(x) } + 
  \frac{ 2 \abrk{\lambda_n} \brk{y - x } }{ xy } = 
 2 \abrk{ \lambda_n G(x) } + 
   2 \abrk{\lambda_n} \cdot \int_x^y \frac{ds}{s^2}.
\label{eq_two_bound_a}
\end{align}
where $D(x), G(x)$ are those of Theorem~\ref{lem_sin_cos}.
We combine Theorem~\ref{lem_close_roots} and Theorem~\ref{lem_sin_cos}
to conclude that
\begin{align}
2 \abrk{ \lambda_n G(x) } \leq
\frac{48}{x^2} =
\frac{48 \dot (y-x)}{xy \cdot(y-x)} \cdot \frac{y}{x}\leq
\frac{50 \cdot c}{\pi} \cdot \int_x^y \frac{ds}{s^2}.
\label{eq_two_bound_b}
\end{align}
We substitute \eqref{eq_two_bound_b} into \eqref{eq_two_bound_a}
and use \eqref{eq_sin_cos_2} to obtain \eqref{eq_two_bound}.
\end{proof}

%%%%%%%%%%
\subsubsection{Bound on the Right-Hand Side of \eqref{eq_pf_cauchy}}
\label{sec_head_tail}

The following theorem is a consequence of
Theorem~\ref{lem_head_main} in Section~\ref{sec_head}
and Theorem~\ref{lem_two_bound} in Section~\ref{sec_tail}.

\begin{thm}
Suppose that $c > 1$ is a real number, and that $n > 0$ is a positive
integer such that
\begin{align}
n > \frac{2c}{\pi} + 1.
\label{eq_head_tail_1}
\end{align}
Suppose also that
\begin{align}
|\lambda_n| < \frac{1}{10},
\label{eq_head_tail_2}
\end{align}
and that
\begin{align}
\chi_n - c^2 < \frac{c^2}{|\lambda_n|}.
\label{eq_head_tail_3}
\end{align}
Suppose furthermore that $1 < x_1 < x_2 < \dots$ are the roots
of $\psi_n$ in $(1,\infty)$. Then, for all real $-1 < t < 1$,
\begin{align}
& \lim_{N \to \infty} \left|
\sum_{k=1}^N \left( 
 \frac{1}{\psi_n'(x_{2k-1}) \cdot (x_{2k-1}-t)} +
 \frac{1}{\psi_n'(x_{2k}) \cdot (x_{2k}-t)} 
\right) 
\right| \leq \nonumber \\
& 6 \cdot |\lambda_n| \cdot 
\left( 2 \cdot \log\left( \frac{2}{|\lambda_n|} \right)+ 
      \sqrt{1 + \frac{\sqrt{\chi_n}}{\pi}} \right)
+ 20 \cdot c \cdot |\lambda_n|^2.
\label{eq_head_tail}
\end{align}
\label{thm_head_tail}
\end{thm}
\begin{proof}
We combine \eqref{eq_head_tail_1}, \eqref{eq_head_tail_2}, 
\eqref{eq_head_tail_3} with Theorem~\ref{lem_close_roots} to 
select a positive even integer $M$ such that
\begin{align}
\frac{1}{|\lambda_n|^2} \leq x_{M+1} \leq \frac{2}{|\lambda_n|^2}.
\label{eq_head_tail_a}
\end{align}
We combine \eqref{eq_head_tail_a} with Theorem~\ref{thm_khi_elem}
in Section~\ref{sec_pswf} and Theorem~\ref{lem_head_main}
in Section~\ref{sec_head} to obtain, for all real $-1 < t < 1$,
\begin{align}
\left| \sum_{k=1}^M \frac{1}{\psi_n'(x_k) \cdot (x_k-t)} \right| \leq
6 \cdot |\lambda_n| \cdot 
\left(
\log\left( \frac{4}{|\lambda_n|^2} \right) +
      \sqrt{1 + \frac{\sqrt{\chi_n}}{\pi}}
\right)
\label{eq_head_tail_b}
\end{align}
Next, we combine \eqref{eq_head_tail_a} with
Remark~\ref{rem_even_odd} and Theorem~\ref{lem_two_bound}
in Section~\ref{sec_tail}
to obtain, for all real $-1<t<1$,
\begin{align}
& \left|
\sum_{k=(M+2)/2}^N \left( 
 \frac{1}{\psi_n'(x_{2k-1}) \cdot (x_{2k-1}-t)} +
 \frac{1}{\psi_n'(x_{2k}) \cdot (x_{2k}-t)} 
\right) 
\right| \leq \nonumber \\
& 20 \cdot c \cdot \int_{|\lambda_n|^{-2}}^{\infty} \frac{ds}{s^2}.
\label{eq_head_tail_c}
\end{align}
Thus \eqref{eq_head_tail} follows from the combination of
\eqref{eq_head_tail_b} and \eqref{eq_head_tail_c}.
\end{proof}
%%%%%%%%%%%%%%%%%%%%%%%%%%%%%%%%%%%%%
The rest of this subsection is devoted to the analysis of
the boundary term of partial fractions expansion of $1/\psi_n$
(see \eqref{eq_pf_cauchy} in Section~\ref{sec_pf}).
In the following theorem, we establish a lower bound on $|\psi_n(z)|$
for certain values of $z$.
%%%%%%%%%%%%%%%%%%%%%%
\begin{thm}
Suppose that $n > 0$ is an even positive number,
and that
\begin{align}
|\lambda_n| < \frac{1}{10}.
\label{eq_vertical_bound_0}
\end{align}
Suppose also that $k > 0$ is an integer number, and that
\begin{align}
k > \frac{8}{\pi} \cdot \frac{c+1}{|\lambda_n|}.
\label{eq_vertical_bound_1}
\end{align}
Suppose furthermore that the real number $R_k$ is defined via
the formula
\begin{align}
R_k = \frac{\pi}{c} \cdot \left(k + \frac{1}{2}\right).
\label{eq_vertical_bound_2}
\end{align}
Then, for any real number $y$, 
\begin{align}
\left| \psi_n( R_k + i \cdot y ) \right| > 
\left| \frac{\psi_n(1)}{c \cdot \lambda_n} \right| \cdot 
\frac{\cosh(cy)}{|R_k+i \cdot y|},
\label{eq_vertical_bound}
\end{align}
where $i=\sqrt{-1}$ is the imaginary unit. Moreover,
for any real number $x$,
\begin{align}
| \psi_n(x+i \cdot R_k) | >
\left| \frac{\psi_n(1)}{c \cdot \lambda_n} \right| \cdot 
\frac{\cosh(cR_k)}{|x+i \cdot R_k|}.
\label{eq_horizontal_bound}
\end{align}
\label{lem_vertical_bound}
\end{thm}
\begin{proof}
Suppose that $x, y$ are arbitrary real numbers.
We observe that
\begin{align}
|\sin(c(x+iy))|^2 & \; =
|\cosh(cy) \cdot \sin(cx) + i \cdot \cos(cx)\cdot \sinh(cy)|^2 \nonumber \\
& \; = \frac{ \cosh(2cy) - \cos(2cx) }{2}.
\label{eq_vb_a}
\end{align}
On the other hand, we combine \eqref{eq_vertical_bound_1}, 
\eqref{eq_vertical_bound_2} and \eqref{eq_vb_b} to conclude that
\begin{align}
\cos(2cR_k) = \cos(2 \pi k + \pi) = -1.
\label{eq_vb_b}
\end{align}
We combine \eqref{eq_vb_a} and \eqref{eq_vb_b} to conclude that,
for all real $-1 < t < 1$,
\begin{align}
| \sin(c\cdot(R_k+iy-t)) | \leq |\sin(c\cdot(R_k+iy)) | = \cosh(cy).
\label{eq_vb_c}
\end{align}
Next, we combine \eqref{eq_vertical_bound_0}, 
\eqref{eq_vertical_bound_1}, \eqref{eq_vertical_bound_2},
\eqref{eq_vb_c},
Theorems~\ref{thm_pswf_main},~\ref{thm_psi1_bound} in Section~\ref{sec_pswf}
to conclude that
\begin{align}
& \left|
\frac{1}{\lambda_n \psi_n(1)} \cdot 
\int_{-1}^1 \frac{ \sin\brk{c\cdot(R_k+iy-t)} \psi_n(t) t }{R_k+iy - t} \; dt
\right| \leq \nonumber \\
& \frac{\cosh(cy)}{R_k} \cdot \frac{2}{|\lambda_n|} \cdot
\int_{-1}^1 | \psi_n(t) \cdot t| \; dt 
\leq \frac{\cosh(cy)}{R_k} \cdot \frac{2}{|\lambda_n|} \leq \nonumber \\
& \cosh(cy) \cdot \frac{2}{|\lambda_n|} \cdot \frac{|\lambda_n|}{8}
  \leq \frac{\cosh(cy)}{4}.
\label{eq_vb_d}
\end{align}
We combine \eqref{eq_vb_c}, \eqref{eq_vb_d} and
\eqref{eq_2_10_even1} of 
Theorem~\ref{lem_psi_for_large_x} in Section~\ref{sec_tail} to obtain
\begin{align}
|\psi_n(R_k+iy)| >
\left| \frac{2 \cdot \psi_n(1) \cdot \sin(c\cdot(R_k+iy)) }
{ c \cdot (R_k+iy) \cdot \lambda_n}
\right| \cdot \left(1 - \frac{1}{4}\right),
\label{eq_vb_e}
\end{align}
which implies \eqref{eq_vertical_bound}. On the other hand,
due to \eqref{eq_vb_a},
\begin{align}
-1 \leq 2\cdot|\sin(c\cdot(x+iR_k))|^2 - \cosh(2cR_k) \leq 1,
\label{eq_hb_a}
\end{align}
for all real $x$. Also, due to the combination of
\eqref{eq_vertical_bound_0} and \eqref{eq_vertical_bound_1},
\begin{align}
\cosh(2cR_k) > \exp\left( \frac{16}{|\lambda_n|} \right) > e^{160}.
\label{eq_hb_b}
\end{align}
We combine \eqref{eq_hb_a}, \eqref{eq_hb_b},
\eqref{eq_vertical_bound_0}, 
\eqref{eq_vertical_bound_1}, \eqref{eq_vertical_bound_2},
Theorems~\ref{thm_pswf_main},~\ref{thm_psi1_bound} in Section~\ref{sec_pswf}
to conclude that, for all real $x$,
\begin{align}
& \left|
\frac{1}{\lambda_n \psi_n(1)} \cdot 
\int_{-1}^1 \frac{ \sin\brk{c\cdot(x+iR_k-t)} \psi_n(t) t }{x+iR_k - t} \; dt
\right| \leq \nonumber \\
& \left|
\frac{2}{\lambda_n} \cdot \frac{\sin(c\cdot(x+iR_k))}{R_k}
\right| \leq
\frac{|\sin(c\cdot(x+iR_k))|}{8}.
\label{eq_hb_c}
\end{align}
We combine \eqref{eq_hb_a}, \eqref{eq_hb_b}, \eqref{eq_hb_c} and
\eqref{eq_2_10_even1} of 
Theorem~\ref{lem_psi_for_large_x} in Section~\ref{sec_tail} to obtain,
for all real $x$,
\begin{align}
& | \psi_n(x+iR_k) | \geq 
\left| \frac{2 \cdot \psi_n(1) \cdot \sin(c\cdot(x+iR_k)) }
{ c \cdot (x+iR_k) \cdot \lambda_n}
\right| \cdot \left(1 - \frac{1}{4}\right),
%> \nonumber \\
%& \left| \frac{\psi_n(1)}{c \cdot \lambda_n} \right| \cdot 
%\frac{\cosh(cR_k)}{|x+i \cdot R_k|},
\label{eq_hb_d}
\end{align}
which implies \eqref{eq_horizontal_bound}.
\end{proof}
%%%%%%%%%%%%%%%%%%%%%
%%%%%%%%%%%%%%%%%%%%%
In the following theorem, we use
Theorem~\ref{lem_vertical_bound} to establish an upper bound on 
the absolute value of a certain
contour integral.
\begin{thm}
Suppose that $n > 0$ is an even positive number,
and that \eqref{eq_vertical_bound_0} holds.
Suppose also that $k > 0$ is an integer number that
satisfies the inequality \eqref{eq_vertical_bound_1},
and that the real number $R_k$ is defined via
\eqref{eq_vertical_bound_2}.
Suppose furthermore that $\Gamma_k$ is the boundary of the
square
\begin{align}
\left[-R_k, R_k\right] \times \left[-i\cdot R_k, i\cdot R_k\right]
\label{eq_contour_1}
\end{align}
in the complex plane,
traversed in the counterclockwise direction. In other words,
$\Gamma_k$ admits the parametrization
\begin{align}
\Gamma_k(s) = 
\begin{cases}
R_k - iR_k + 2isR_k, & 0 \leq s \leq 1, \\
R_k + iR_k - 2(s-1)R_k, & 1 \leq s \leq 2, \\
-R_k+ iR_k - 2i(s-2)R_k, & 2 \leq s \leq 3, \\
-R_k- iR_k + 2(s-3)R_k, & 3 \leq s \leq 4.
\end{cases}
\label{eq_contour_2}
\end{align}
Then, for all real $-1<t<1$,
\begin{align}
\left|
\frac{1}{2\pi i} \oint_{\Gamma_k} \frac{dz}{\psi_n(z) \cdot (z-t)}
\right| < 2\sqrt{2}\cdot |\lambda_n| \cdot
\left(1 + 2cR_k \cdot e^{-cR_k} \right).
\label{eq_contour}
\end{align}
\label{lem_contour}
\end{thm}
\begin{proof}
Suppose that $-1 < t < 1$ is a real number.
We combine Theorem~\ref{thm_psi1_bound} in Section~\ref{sec_pswf}
with \eqref{eq_vertical_bound_0}, \eqref{eq_vertical_bound_1},
\eqref{eq_vertical_bound_2}, \eqref{eq_vertical_bound}
of Theorem~\ref{lem_vertical_bound} to obtain
\begin{align}
& \left| 
\frac{1}{2\pi i} \int_{-R_k}^{R_k }
  \frac{dy}{\psi_n(R_k+iy) \cdot (R_k+iy-t)}
\right| \leq \nonumber \\
& \frac{1}{\pi} \int_{-\infty}^{\infty} 
\frac{dy}{|\psi_n(R_k+iy)| \cdot |R_k+iy|} \leq 
 \frac{\sqrt{2}}{\pi} \int_{-\infty}^{\infty} 
\frac{c |\lambda_n| \; dy}{\cosh(cy)}
= \sqrt{2} \cdot |\lambda_n|.
\label{eq_contour_a}
\end{align}
On the other hand,
we combine Theorem~\ref{thm_psi1_bound} in Section~\ref{sec_pswf}
with \eqref{eq_vertical_bound_0}, \eqref{eq_vertical_bound_1},
\eqref{eq_vertical_bound_2}, \eqref{eq_horizontal_bound}
of Theorem~\ref{lem_vertical_bound} to obtain
\begin{align}
& \left| 
-\frac{1}{2\pi i} \int_{-R_k}^{R_k }
  \frac{dx}{\psi_n(x+iR_k) \cdot (x+iR_k-t)}
\right| \leq \nonumber \\
& \frac{1}{\pi} \int_{-R_k}^{R_k} 
\frac{dx}{|\psi_n(x+iR_k)| \cdot |x+iR_k|} \leq 
\frac{c \cdot |\lambda_n| \sqrt{2}}{\pi \cdot \cosh(c R_k)}
\int_{-R_k}^{R_k} dx \leq \nonumber \\
& \frac{4 \cdot \sqrt{2} \cdot |\lambda_n| \cdot cR_k
\cdot e^{-cR_k}}{\pi}.
\label{eq_contour_b}
\end{align}
We combine \eqref{eq_contour_2}, \eqref{eq_contour_a}, \eqref{eq_contour_b}
with the observation that $|\psi_n|$ is symmetric about zero to
obtain \eqref{eq_contour}.
\end{proof}
%%%%%%%%%%%%%%%%%%%%%
We are now ready to prove the principal theorem of this section.
It is illustrated in
Table~\ref{t:test86} and in Figures~\ref{fig:test86a}, \ref{fig:test86b}
(see Experiment 8 in Section~\ref{sec_exp8}).
\begin{thm}
Suppose that $c>1$, and that $n>0$ is an even positive integer.
Suppose also that
\begin{align}
n > \frac{2c}{\pi} + 1,
\label{eq_complex_1}
\end{align}
that
\begin{align}
|\lambda_n| < \frac{1}{10},
\label{eq_complex_2}
\end{align}
and that
\begin{align}
\chi_n - c^2 < \frac{c^2}{|\lambda_n|}.
\label{eq_complex_3}
\end{align}
Suppose furthermore that $-1 < t_1 < \dots < t_n < 1$ are the roots
of $\psi_n$ in $(-1,1)$, and that the function $I:(-1,1) \to \Rc$
is defined via the formula
\begin{align}
I(t) = \frac{1}{\psi_n(t)} - 
\sum_{j=1}^n \frac{1}{\psi_n'(t_j) \cdot (t-t_j)},
\label{eq_complex_it}
\end{align}
for $-1 < t < 1$.
Then,
\begin{align}
|I(t)| \leq |\lambda_n| \cdot I_{\max}, 
\label{eq_complex}
\end{align}
where the real number $I_{max}$ is defined via the formula
\begin{align}
I_{\max} = 
24 \cdot \log\left( \frac{2}{|\lambda_n|} \right) +
13 \cdot (\chi_n)^{1/4} + 40 \cdot c \cdot |\lambda_n| + 2 \sqrt{2}.
\label{eq_complex_imax}
\end{align}
\label{thm_complex}
\end{thm}
\begin{proof}
Suppose that $1 < x_1 < x_2 < \dots$ are the roots of $\psi_n$
in $(1,\infty)$, and that $k$ is an integer satisfying the inequality
\eqref{eq_vertical_bound_1} in Theorem~\ref{lem_vertical_bound}.
Suppose also that the real number $R_k$ is defined via
\eqref{eq_vertical_bound_2} in Theorem~\ref{lem_vertical_bound},
the contour $\Gamma_k$ in the complex plane is defined via
\eqref{eq_contour_2} in Theorem~\ref{lem_contour}, and that $x_M$ is
the maximal root of $\psi_n$ in $(1, \infty)$; in other words,
\begin{align}
1 < x_1 < \dots < x_M < R_k < x_{M+1} < \dots.
\label{eq_complex_b}
\end{align}
(We observe that $\psi_n(R_k) \neq 0$ due to 
\eqref{eq_vertical_bound} in Theorem~\ref{lem_vertical_bound}.)
We combine \eqref{eq_complex_it}, \eqref{eq_complex_b} and
Theorem~\ref{thm_cauchy} of Section~\ref{sec_misc_tools} to conclude that,
for any real $-1 < t < 1$,
\begin{align}
I(t) = & \; \sum_{k=1}^M \left( \frac{1}{\psi_n'(x_k) \cdot (t-x_k)} +
                           \frac{1}{\psi_n'(-x_k) \cdot (t+x_k)} \right) +
   \nonumber \\
& \;   \frac{1}{2\pi i} \oint_{\Gamma_k} \frac{dz}{\psi_n(z)\cdot(z-t)}.
\label{eq_complex_c}
\end{align}
We combine the assumption that $c>1$ with Theorem~\ref{thm_khi_elem}
in Section~\ref{sec_pswf} to conclude that
\begin{align}
\sqrt{1 + \frac{\sqrt{\chi_n}}{\pi}} < 
(\chi_n)^{1/4} \cdot \sqrt{ \frac{1}{\sqrt{2}} + \frac{1}{\pi}}.
\label{eq_complex_d}
\end{align}
We obtain the inequality \eqref{eq_complex} by taking the limit $k \to \infty$
and using \eqref{eq_complex_c}, \eqref{eq_complex_d},
Theorem~\ref{thm_head_tail} and Theorem~\ref{lem_contour}.
\end{proof}
\begin{remark}
The conclusion of Theorem~\ref{thm_complex}
holds for odd values of $n$ as well. The proof is essentially the same,
and is based on Theorems~\ref{lem_psi_for_large_x},
\ref{thm_head_tail}, and obvious
modifications of Theorems~\ref{lem_vertical_bound},~\ref{lem_contour}.
\label{rem_complex}
\end{remark}
\begin{remark}
Suppose that the 
function $I:(-1,1) \to \Rc$ is defined via \eqref{eq_complex_it}.
If $n$ is even, then $I$ is an even function.
If $n$ is odd, then $I$ is an odd function.
\label{rem_i_parity}
\end{remark}
In the following theorem, we provide
a simple condition on $n$ that implies the inequality $|\lambda_n|<0.1$.
\begin{thm}
Suppose that $c>30$, and that $n>0$ is an integer. Suppose also that
\begin{align}
n > \frac{2c}{\pi} + 5.
\label{eq_lam10_n5}
\end{align} 
Then,
\begin{align}
|\lambda_n| < \frac{1}{10}.
\label{eq_lam10}
\end{align}
\label{thm_lam10}
\end{thm}
\begin{proof}
Suppose first that 
\begin{align}
c > 200 \cdot \pi.
\label{eq_lam10_a}
\end{align}
We combine \eqref{eq_lam10_a}
with
\eqref{eq_prolate_mu},
\eqref{eq_mu_leg_1} in Section~\ref{sec_pswf}
to conclude that, in this case,
\begin{align}
|\lambda_n| = \sqrt{ \frac{2\pi\cdot\mu_n}{c} } < 
\sqrt{ \frac{2\pi}{c} } < \frac{1}{10}.
\label{eq_lam10_b}
\end{align}
On the other hand, suppose that
\begin{align}
30 \leq c \leq 200 \cdot \pi.
\label{eq_lam10_c}
\end{align}
We observe that the interval $\left[30,200\cdot\pi\right]$ is compact,
and use this observation to
verify numerically that, if \eqref{eq_lam10_c} holds,
\begin{align}
|\lambda_{\text{floor}(2c/\pi+5)}| < \frac{1}{50},
\label{eq_lam10_d}
\end{align}
where, for a real number $a$, $\text{floor}(a)$ is the largest
integer less than or equal to $a$.
We combine Theorem~\ref{thm_pswf_main}
in Section~\ref{sec_pswf}, \eqref{eq_lam10_d} and \eqref{eq_lam10_b}
to establish \eqref{eq_lam10}.
\end{proof}
In the following theorem, we summarize 
Theorems~\ref{thm_khi_lambda_bound},
\ref{thm_complex},
\ref{thm_lam10} and Remark~\ref{rem_complex}.
\begin{thm}
Suppose that $c>0$ is a real number and $n>0$ is an integer.
Suppose also that
\begin{align}
c > 30,
\label{eq_complex_summary_0}
\end{align}
and that
\begin{align}
n > \frac{2c}{\pi} + 5.
\label{eq_complex_summary_1}
\end{align}
Suppose furthermore that the function $I:(-1,1) \to \Rc$
is defined via the formula \eqref{eq_complex_it}
in Theorem~\ref{thm_complex}.
Then,
\begin{align}
|I(t)| \leq |\lambda_n| \cdot I_{\max}, 
\label{eq_complex_summary}
\end{align}
where the real number $I_{max}$ is defined via the formula
\eqref{eq_complex_imax} in Theorem~\ref{thm_complex}.
\label{thm_complex_summary}
\end{thm}
\begin{proof}
We combine \eqref{eq_complex_summary_0}, \eqref{eq_complex_summary_1}
with Theorem~\ref{thm_lam10} to conclude that the inequality
\eqref{eq_complex_2} holds. Also, we combine
\eqref{eq_complex_summary_0},
\eqref{eq_complex_2}
with Theorem~\ref{thm_khi_lambda_bound} to conclude that
the inequality \eqref{eq_complex_3} holds. We combine these
observations with Theorem~\ref{thm_complex} and
Remark~\ref{rem_complex} to obtain \eqref{eq_complex_summary}.
\end{proof}

%%%%%%%%%%%%%%%%%%%%%%%%%%%%%%%%%%%%%%%%%%%%%%
\subsection{PSWF-based Quadrature and its Properties}
\label{sec_quad}
In this subsection, we define 
PSWF-based quadratures of order $n$, find an upper bound on their error,
and show that a prescribed absolute accuracy can be achieved by
a proper choice of $n$.

The principal result of this section is Theorem~\ref{thm_quad_eps_simple}.
\begin{definition}
Suppose that $n > 0$ is a positive integer, and that
\begin{align}
-1 < t_1 < t_2 < \dots < t_n < 1
\label{eq_quad_t}
\end{align}
are the roots of $\psi_n$ the interval in $\brk{-1, 1}$.
For each integer $j = 1, \dots, n$, we define
the function $\varphi_j: (-1,1) \to \Rc$ via the formula
\begin{align}
\varphi_j(t) = \frac{ \psi_n(t) }{ \psi_n'(t_j) \brk{ t-t_j } }.
\label{eq_quad_phi}
\end{align}
In addition, for each integer $j=1,\dots,n$, we define
the real number $W_j$ via the formula
\begin{align}
W_j = \int_{-1}^1 \varphi_j(s) \; ds =
\frac{1}{\psi_n'(t_j)} \int_{-1}^1 \frac{ \psi_n(s) \; ds }{ s - t_j }.
\label{eq_quad_w}
\end{align}
We refer to the expression of the form
\begin{align}
\sum_{j=1}^n W_j \cdot f(t_j)
\label{eq_quad_quad}
\end{align}
as the PSWF-based quadrature rule of order $n$. 
The points $t_1, \dots, t_n$ and the numbers
$W_1, \dots, W_n$ are referred to as the nodes and the 
weights of the quadrature, respectively. The purpose of \eqref{eq_quad_quad}
is to approximate the integral of a bandlimited function $f$ over
the interval $\left[-1,1\right]$.
\label{def_quad}
\end{definition}
%%%%%%%%%%%
\subsubsection{Expansion of $\varphi_j$ into a Prolate Series}
\label{sec_phi}
Suppose that $n > 0$ is a positive integer.
For every integer $j=1,\dots,n$, we define the function
$\varphi_j : (-1,1) \to \Rc$ via \eqref{eq_quad_phi}.
In the following theorem, we evaluate the inner product
$\left<\varphi_j, \psi_m\right>$ for arbitrary $m \neq n$.
This theorem is illustrated 
in Tables~\ref{t:test87a}, \ref{t:test87b},
Figure~\ref{fig:test87}
(see Experiment 9 in Section~\ref{sec_exp9}).
%%%%%%%%%%%
\begin{thm}
Suppose that $n>0$ is a positive integer,
and that $m \neq n$ is a non-negative integer. 
Suppose also that $1 \leq j \leq n$ is an integer. Then,
\begin{align}
\int_{-1}^1 \frac{ \psi_n(t) }{ t - t_j } \; \psi_m(t) \; dt =
\frac{ \abrk{\lambda_m}^2 \psi_m(t_j) }
     { \abrk{\lambda_m}^2 - \abrk{\lambda_n}^2 } \cdot
       \Big[\int_{-1}^1 \frac{ \psi_n(t) \; dt }{ t - t_j }  + 
                 i c \lambda_n \Psi_n(1,t_j) \Big],
\label{eq_quad_coef}
\end{align}
where  $t_j$ is given via \eqref{eq_quad_t} in Definition~\ref{def_quad},
and the complex-valued function $\Psi_n: (-1,1)^2 \to \Cc$
is defined via the formula
\begin{align}
\Psi_n(y,t) = \int_0^y \psi_n(x) e^{-icxt} dx.
\label{eq_quad_big_psi}
\end{align}
\label{thm_quad_coef}
\end{thm}
\begin{proof}
We combine \eqref{eq_prolate_integral} 
with \eqref{eq_quad_big_psi} to obtain, for all real $-1 < y < 1$,
\begin{align}
& \int_{t = -1}^1 \psi_n(t) 
\int_{x = 0}^y \frac{d}{dx}
\left[ \frac{e^{icx(t-t_j)} }{ic(t-t_j)} \right] dx \; dt
 =
\int_{x = 0}^y e^{-icxt_j} \int_{t = -1}^1 \psi_n(t) e^{icxt} dt \; dx =
\nonumber \\
& \lambda_n \cdot \int_0^y \psi_n(x) e^{-icxt_j} dx
= \lambda_n \cdot \Psi_n(y,t_j).
\label{eq_quad_coef_a}
\end{align}
On the other hand,
\begin{align}
& \int_{t = -1}^1 \psi_n(t) 
\int_{x = 0}^y \frac{d}{dx}
\left[ \frac{e^{icx(t-t_j)} }{ic(t-t_j)} \right] dx \; dt
 =
\frac{1}{ic} \int_{-1}^1 \frac{\psi_n(t)}{t - t_j}
  \brk{e^{icy(t-t_j)} - 1} dt =
\nonumber \\
& \frac{e^{-icyt_j}}{ic} \int_{-1}^1 \frac{\psi_n(t)}{t - t_j} e^{icyt} dt
- \frac{1}{ic} \int_{-1}^1 \frac{ \psi_n(t) \; dt }{t - t_j}.
\label{eq_quad_coef_b}
\end{align}
We combine \eqref{eq_quad_coef_a} and \eqref{eq_quad_coef_b}
to obtain, for all real $-1 < x < 1$,
\begin{align}
ic \lambda_n e^{icxt_j} \Psi_n(x,t_j) +
e^{icxt_j} \int_{-1}^1 \frac{ \psi_n(t) \; dt }{ t - t_j } =
\int_{-1}^1 \frac{ \psi_n(t) }{ t - t_j } e^{icxt} dt.
\label{eq_quad_coef_c}
\end{align}
We combine \eqref{eq_prolate_integral},
\eqref{eq_quad_big_psi} and \eqref{eq_quad_coef_c}
to obtain
\begin{align}
& \int_{-1}^1 \frac{ \psi_n(t) }{ t - t_j } \; \psi_m(t) \; dt =
 \frac{ 1 }{ \lambda_m }
\int_{x = -1}^1 \psi_m(x)
\int_{t = -1}^1 \frac{ \psi_n(t) }{ t - t_j } e^{icxt} dt \; dx =
\nonumber \\
& \frac{ ic \lambda_n}{\lambda_m}
 \int_{-1}^1 \psi_m(x) e^{icxt_j} \Psi_n(x,t_j) \; dx +
\frac{1}{\lambda_m} 
\brk{ \int_{-1}^1 \psi_m(x) e^{icxt_j} dx }
\brk{ \int_{-1}^1 \frac{ \psi_n(t) \; dt }{ t - t_j } } \nonumber \\
& \frac{ic \lambda_n}{\lambda_m}
\int_{-1}^1 \frac{\partial \Psi_m}{\partial x}(x,-t_j) \Psi_n(x,t_j) \; dx +
\psi_m(t_j) \int_{-1}^1 \frac{ \psi_n(t) \; dt }{ t - t_j }.
\label{eq_quad_coef_d}
\end{align}
We observe that $\psi_n(-t_j)=0$, and combine this observation with
\eqref{eq_prolate_integral} in Section~\ref{sec_pswf} and 
\eqref{eq_quad_big_psi} to obtain
\begin{align}
0 = \frac{\psi_n(-t_j)}{\lambda_n} = \int_{-1}^1 \psi_n(t) e^{-ictt_j} \; dt
  = \Psi_n(1,t_j)-\Psi_n(-1,t_j),
\label{eq_quad_coef_e}
\end{align}
and also
\begin{align}
\lambda_m \psi_m(t_j)
= \int_{-1}^1 \psi_m(t) e^{ictt_j} dt 
= \Psi_m(1,-t_j) - \Psi_m(-1,-t_j).
\label{eq_quad_coef_f}
\end{align}
We combine \eqref{eq_quad_coef_e}, \eqref{eq_quad_coef_f} to obtain
\begin{align}
\left[ \Psi_m(x,-t_j) \Psi_n(x,t_j) \right]_{x=-1}^1 & \; =
\Psi_n(1,t_j) \brk{ \Psi_m(1,-t_j) - \Psi_m(-1,-t_j) } \nonumber \\
& \; =
\lambda_m \psi_m(t_j) \Psi_n(1,t_j).
\label{eq_quad_coef_g}
\end{align}
Also, we combine \eqref{eq_prolate_integral},
Theorem~\ref{thm_pswf_main} in Section~\ref{sec_pswf}
and \eqref{eq_quad_big_psi}
to obtain
\begin{align}
& \int_{-1}^1 \Psi_m(x,-t_j) \frac{\partial \Psi_n}{\partial x}(x,t_j) \; dx =
\nonumber \\
& \frac{1}{\lambda_m}
    \int_{x = -1}^1 \psi_n(x) e^{-icxt_j} 
    \int_{y = 0}^x e^{icyt_j} 
    \int_{t = -1}^1 \psi_m(t) e^{icty} dt \; dy \; dx =
\nonumber \\
& \frac{1}{\lambda_m}
    \int_{t = -1}^1 \psi_m(t)
    \int_{x = -1}^1 \psi_n(x) e^{-icxt_j}
    \int_{y = 0}^x e^{ic(t_j+t)y} dy \; dx \; dt =
\nonumber \\
& \frac{1}{\lambda_m}
    \int_{t = -1}^1 \psi_m(t)
    \int_{x = -1}^1 \psi_n(x) 
    \brk{
       \frac{e^{icxt} - e^{-icxt_j}}{ic(t_j+t)}
    } dx \; dt =
\nonumber \\
& \frac{ \lambda_n}{ \lambda_m }
   \int_{-1}^1 \frac{ \psi_m(t) \psi_n(t) dt }{ic(t+t_j)} =
\frac{ \brk{-1}^{n+m+1} \lambda_n }{ ic \lambda_m }
  \int_{-1}^1 \frac{ \psi_n(t) }{ t - t_j } \; \psi_m(t) \; dt.
\label{eq_quad_coef_h}
\end{align}
We combine Theorem~\ref{thm_pswf_main} in Section~\ref{sec_pswf} with
\eqref{eq_quad_coef_g}, \eqref{eq_quad_coef_h} to obtain
\begin{align}
& \frac{ic \lambda_n}{\lambda_m}
\int_{-1}^1 \frac{\partial \Psi_m}{\partial x}(x,-t_j) \Psi_n(x,t_j) \; dx =
\nonumber \\
& \frac{ic \lambda_n}{\lambda_m} \cdot \left[
\lambda_m \psi_m(t_j) \Psi_n(1,t_j) +
\frac{ \brk{-1}^{n+m} \lambda_n }{ ic \lambda_m }
  \int_{-1}^1 \frac{ \psi_n(t) }{ t - t_j } \; \psi_m(t) \; dt
\right] = \nonumber \\
& ic\lambda_n \psi_m(t_j) \Psi_n(1,t_j) + 
\frac{|\lambda_n|^2}{|\lambda_m|^2}
\int_{-1}^1 \frac{ \psi_n(t) }{ t - t_j } \; \psi_m(t) \; dt.
\label{eq_quad_coef_i}
\end{align}
Finally, we recall that $m \neq n$ and substitute \eqref{eq_quad_coef_i}
into \eqref{eq_quad_coef_d} to obtain \eqref{eq_quad_coef}.
\end{proof}
%%%%%%%%%%%
%%%%%%%%%%%
\subsubsection{Quadrature Error}
\label{sec_quad_error}
For a positive integer $n > 0$, we define the 
PSWF-based quadrature of order $n$ via
\eqref{eq_quad_t}, \eqref{eq_quad_w} in Definition~\ref{def_quad}.
This quadrature is used to approximate the integral of an
arbitrary bandlimited
function $f:(-1,1) \to \Cc$
over the interval $(-1,1)$
(see \eqref{eq_quadrature_used} in Section~\ref{sec_outline}
and \eqref{eq_quad_quad}).
We refer to the difference
\begin{align}
\int_{-1}^1 f(t) \; dt - \sum_{j=1}^n f\left(t_j\right) \cdot W_j
\label{eq_quad_error_def}
\end{align}
as the ``quadrature error'' (for integrating $f$).
The following theorem, illustrated in 
Tables~\ref{t:test89}, \ref{t:test90}, provides an upper bound
on the absolute value
of the quadrature error (for integrating $\psi_m$ for arbitrary $m < n$).
One of the principal goals of this paper is to 
investigate this error (see
see \eqref{eq_main_goal} in Section~\ref{sec_outline}). 
The results of additional numerical experiments,
in which this quadrature is used for integration of
certain functions, are 
summarized in 
Tables~\ref{t:test90},~\ref{t:test91} and
Figures~\ref{fig:test92}, \ref{fig:test93a}, \ref{fig:test93b}
(see Experiments 11, 12 in Section~\ref{sec_exp11}).
%%%%%%%%%%%
\begin{thm}
Suppose that $n>0$ and $0 \leq m \leq n-1$ are integers.
Suppose also that $t_1, \dots, t_n$ and $W_1, \dots, W_n$ are, 
respectively, the nodes and weights of the quadrature, introduced
in Definition~\ref{def_quad} above. Suppose furthermore that
the real number $P_{n,m}$ is defined via the formula
\begin{align}
P_{n,m} = 
\sum_{j = 1}^n \frac{ \psi_m(t_j) }{ \psi_n'(t_j) } \cdot \Psi_n(1,t_j),
\label{eq_quad_p}
\end{align}
where the complex-valued function $\Psi_n:(-1,1)^2 \to \Cc$
is that of Theorem~\ref{thm_quad_coef} above.
Then,
\begin{align}
& \abrk{ \int_{-1}^1 \psi_m(s) \; ds - \sum_{j=1}^n \psi_m(t_j) W_j } \leq
\nonumber \\
& \brk{1 - \frac{\abrk{\lambda_n}^2}{\abrk{\lambda_m}^2} } \cdot 
  \| I \|_{\infty} +
\abrk{\lambda_n} \cdot \brk{ 
 \frac{ \abrk{\lambda_n} }{ \abrk{\lambda_m} } \abrk{\psi_m(0)} +
 c \abrk{P_{n,m}} },
\label{eq_quad_thm}
\end{align}
where $\| I \|_{\infty}$ is the $L^{\infty}$-norm of the function 
$I:(-1,1) \to \Rc$, 
defined via
\eqref{eq_complex_it} in Theorem~\ref{thm_complex} 
in Section~\ref{sec_head_tail},
i.e.
\begin{align}
\| I \|_{\infty} = \sup\left\{ |I(t)| \; : \; -1 < t < 1\right\}.
\label{eq_i_inf}
\end{align}
\label{thm_quad}
\end{thm}
\begin{proof}
Suppose that the function $I: (-1,1) \to \Rc$ is defined
via \eqref{eq_complex_it} in Theorem~\ref{thm_complex}
in Section~\ref{sec_head_tail}. We multiply \eqref{eq_complex_it}
by $\psi_n(t) \cdot \psi_m(t)$ to obtain, for all real $-1<t<1$,
\begin{align}
\psi_m(t) = 
\sum_{j = 1}^n \psi_m(t) \varphi_j(t) + 
\psi_m(t) \psi_n(t) I(t),
\label{eq_quad_a}
\end{align}
where, for each $j=1,\dots,n$, the function $\varphi_j: (-1,1) \to \Rc$
is that of Definition~\ref{def_quad}.
We combine \eqref{eq_prolate_integral}, 
Theorem~\ref{thm_complex}, Definition~\ref{def_quad},
Theorem~\ref{thm_quad_coef}, \eqref{eq_i_inf}, and integrate
\eqref{eq_quad_a} over the interval $(-1,1)$ to obtain
\begin{align}
& \lambda_m \psi_m(0) = \nonumber \\
& \frac{ \abrk{\lambda_m}^2 }{ \abrk{\lambda_m}^2 - \abrk{\lambda_n}^2 }
\sum_{j=1}^n \psi_m(t_j) \brk{W_j + ic\lambda_n 
                                \frac{ \Psi_n(1,t_j) }{ \psi_n'(t_j)} } +
\xi \cdot \| I \|_{\infty}, 
\label{eq_quad_b}
\end{align}
where $-1 \leq \xi \leq 1$ is a real number.
We combine \eqref{eq_quad_b} with \eqref{eq_quad_p}
to obtain
\begin{align}
& \left(1 - \frac{|\lambda_n|^2}{|\lambda_m|^2}\right) \cdot 
\lambda_m \psi_m(0) = \nonumber \\
& \sum_{j = 1}^n \psi_m(t_j) W_j + i c \lambda_n P_{n,m} +
\left(1 - \frac{|\lambda_n|^2}{|\lambda_m|^2}\right) \cdot 
\xi \cdot \| I \|_{\infty}.
\label{eq_quad_c}
\end{align}
Finally, we rearrange \eqref{eq_quad_c} to obtain \eqref{eq_quad_thm}.
\end{proof}
%%%%%%%%%%%
In the following theorem, we establish an upper bound
on $P_{n,m}$, defined via \eqref{eq_quad_p} above.
This theorem is illustrated
in Table~\ref{t:test88} and Figure~\ref{fig:test88}
(see Experiment 10 in Section~\ref{sec_exp10}).
%%%%%%%%%%%
\begin{thm}
Suppose that $n, m$ are non-negative integers,
and that $0 \leq m < n$.
Suppose also that $\chi_n > c^2$, and that
the real number $P_{n,m}$ is defined
via \eqref{eq_quad_p} in Theorem~\ref{thm_quad}. Then,
\begin{align}
c \abrk{P_{n,m}} \leq \sqrt{32} \cdot n^2.
\label{eq_quad_cp}
\end{align}
\label{lem_quad_cp}
\end{thm}
\begin{proof}
Since $\chi_n > c^2$, the inequality
\begin{align}
\psi_n^2(t) \leq \psi_n^2(1) \leq n + \frac{1}{2},
\end{align}
holds for all real $-1 \leq t \leq 1$,
due to Theorems~\ref{thm_psi1_bound}, \ref{thm_extrema}
in Section~\ref{sec_pswf}.
Therefore,
\begin{align}
\int_{1 - 1/8n}^1 \psi_n^2(t) \; dt \leq \frac{1}{8} + \frac{1}{16n} <
\frac{3}{16}.
\label{eq_quad_cp_a}
\end{align}
We combine \eqref{eq_quad_cp_a} with Theorem~\ref{thm_pswf_main}
in Section~\ref{sec_pswf} to obtain
\begin{align}
\int_0^{1 - 1/8n} \psi_n^2(t) \; dt 
= \int_0^1 \psi_n^2(t) \; dt - \int_{1 - 1/8n}^1 \psi_n^2(t) \; dt
\geq 
\frac{1}{2} - \frac{3}{16} = \frac{5}{16}.
\label{eq_quad_fourth}
\end{align}
We observe that 
\begin{align}
\int \frac{ dx }{ \brk{1 - x^2}^2 } = 
\frac{1}{2} \cdot \frac{ x }{ 1 - x^2 } + 
\frac{1}{4} \log \frac{x+1}{1-x},
\label{eq_quad_cp_b}
\end{align}
and combine \eqref{eq_quad_fourth} and \eqref{eq_quad_cp_b} to obtain
\begin{align}
& \int_0^{1-1/8n} \frac{ dx }{ \brk{1 - x^2}^2 } =
 \frac{1}{2} \cdot \frac{1 - 1/8n}{1 - \brk{1-1/8n}^2} +
  \frac{1}{4} \log \frac{2 - 1/8n}{1/8n} = \nonumber \\
& \frac{1}{2} \cdot \frac{ 8n \brk{8n-1} }{ 16n - 1} +
  \frac{1}{4} \log \brk{16n - 1} \leq
 4n + n \leq 5n.
\label{eq_quad_8n}
\end{align}
Suppose that the functions $Q(t), \tilde{Q}(t) : (-1,1) \to \Rc$ are 
defined, respectively, 
via the formulae \eqref{eq_q_old}, \eqref{eq_qtilde_old}
in
Theorem~\ref{thm_Q_Q_tilde} in Section~\ref{sec_pswf}.
We apply Theorem~\ref{thm_Q_Q_tilde}
with $t_0 = 0$ and $0 < t \leq 1$ to obtain
\begin{align}
Q(0) \cdot \chi_n
& \; = Q(0) \cdot p(0) \cdot q(0) = \tilde{Q}(0) \nonumber \\
& \; \geq \tilde{Q}(t) = c^2
\left[ \psi_n^2(t) + \frac{ \brk{t^2-1} \brk{\psi_n'(t)}^2 }
                                    { \brk{c^2 \cdot t^2 - \chi_n} } \right]
   \cdot \brk{1 - t^2} \brk{\chi_n/c^2 - t^2} \nonumber \\
& \; \geq c^2 \psi_n^2(t) \brk{1-t^2} \brk{\chi_n/c^2-t^2}
     \geq c^2 \psi_n^2(t) \brk{1-t^2}^2.
\label{eq_quad_q0_est}
\end{align}
It follows from \eqref{eq_quad_fourth}, 
\eqref{eq_quad_8n} and \eqref{eq_quad_q0_est} that
\begin{align}
& 5n \cdot Q(0) \cdot \frac{\chi_n}{c^2} \geq 
Q(0) \cdot \frac{\chi_n}{c^2} 
\int_0^{1 - 1/8n} \frac{ dx }{ \brk{1-x^2}^2 } \geq 
\int_0^{1 - 1/8n} \psi_n^2(t) \; dt \geq \frac{5}{16},
\end{align}
which, in turn, implies that
\begin{align}
\frac{1}{Q(0)} \leq 16 n \cdot \frac{\chi_n}{c^2}.
\label{eq_quad_32n}
\end{align}
Suppose now that $j \geq n/2$ is an integer, and $t_j$ is that
of Definition~\ref{def_quad}. We combine \eqref{eq_quad_32n}
with Theorem~\ref{thm_Q_Q_tilde} in Section~\ref{sec_pswf}
to obtain
\begin{align}
\frac{ \left(\psi_n'(t_j)\right)^2 }{\chi_n} \geq
\frac{ (1-t_j^2) \cdot \left(\psi_n'(t_j)\right)^2 }{ \chi_n-c^2 t_j^2} =
Q(t_j) \geq Q(0) \geq \frac{c^2}{16n \cdot \chi_n}.
\label{eq_quad_cp_c}
\end{align}
Due to Theorem~\ref{thm_all_psi_upper_bound} in Section~\ref{sec_pswf},
for all integer $0 \leq m < n$ and real $-1 < t < 1$,
\begin{align}
| \psi_m(t) | \leq 2 \sqrt{n}.
\label{eq_quad_cp_d}
\end{align}
We combine Theorem~\ref{thm_pswf_main} in Section~\ref{sec_pswf}
with \eqref{eq_quad_big_psi} of Theorem~\ref{thm_quad_coef} above
to obtain, for all real $0 \leq t \leq 1$,
\begin{align}
| \Psi_n(1,t) | =
\left| \int_0^1 \psi_n(x) e^{-icxt} \; dt \right| \leq
\frac{1}{2}\int_{-1}^1 |\psi_n(x)|\;dx \leq \frac{\sqrt{2}}{2}.
\label{eq_quad_cp_e}
\end{align}
Finally, we combine \eqref{eq_quad_p}, \eqref{eq_quad_cp_c},
\eqref{eq_quad_cp_d} and \eqref{eq_quad_cp_e} to obtain
\begin{align}
c |P_{n,m}| \leq cn \cdot 
\max_{t_j \geq 0} \left|
  \frac{\psi_m(t_j)}{\psi_n'(t_j)} \cdot \Psi_n(1,t_j)
\right| \leq 
cn \cdot \frac{\sqrt{16n}}{c} \cdot \frac{\sqrt{2}}{2} \cdot 2\sqrt{n},
\end{align}
which implies \eqref{eq_quad_cp}.
\end{proof}
%%%%%%%%%%%%
\begin{cor}
Suppose that $m$ is an odd integer. Then,
$P_{n,m}=0$.
\label{cor_p_parity}
\end{cor}
\begin{proof}
Suppose that $1 \leq j \leq n$ is an integer,
and $t_1,\dots,t_n$ are the roots of $\psi_n$ in $(-1,1)$.
We combine Theorem~\ref{thm_pswf_main} and
\eqref{eq_prolate_integral} in Section~\ref{sec_pswf}
with \eqref{eq_quad_big_psi} to obtain, for every $j=1,\dots,n$,
\begin{align}
& (-1)^n \cdot \Psi_{n, j}(1) + \Psi_{n, n+1-j}(1) = \nonumber \\
& \int_0^1 \psi_n(-x) e^{-icxt_j} \; dx +
  \int_0^1 \psi_n(x) e^{icxt_j} \; dx = \nonumber \\
& \int_{-1}^1 \psi_n(x) e^{icxt_j} \; dx = \lambda_n \psi_n(t_j) = 0.
\label{eq_cor_p_a}
\end{align}
We observe that 
$\psi_n'$ is odd for even $n$ and even for odd $n$,
and combine this observation with \eqref{eq_cor_p_a} to obtain,
for every integer $j=1,\dots,n$,
\begin{align}
\frac{ \Psi_{n, j}(1) }{ \psi_n'(t_j) } = 
\frac{ \Psi_{n, n+1-j}(1) }{ \psi_n'(t_{n+1-j}) }.
\label{eq_big_psi_2}
\end{align}
We combine \eqref{eq_big_psi_2} with \eqref{eq_quad_p} to obtain
\begin{align}
P_{n, m} & \; = \sum_{j=1}^n \psi_m(t_j) \cdot
   \frac{ \Psi_{n, j}(1) }{ \psi_n'(t_j) } \nonumber \\ 
& \; = \sum_{j \leq n/2}  
   \brk{ \psi_m(t_j) + \psi_m(-t_j) } \cdot
   \frac{ \Psi_{n, j}(1) }{ \psi_n'(t_j) } = 0.
\end{align}
\label{rem_quad_p}
\end{proof}
%%%%%%%%%%%%%%%%
% APRIL 12, 2012
%%%%%%%%%%%%
In the following theorem, we simplify the inequality \eqref{eq_quad_thm}
of Theorem~\ref{thm_quad}. It is illustrated in
Table~\ref{t:test91} and in Figure~\ref{fig:test92}
(see Experiment 12 in Section~\ref{sec_exp12}).
See also Conjecture~\ref{conj_quad_error}
and Remark~\ref{rem_conj} in Section~\ref{sec_exp12}.
\begin{thm}
Suppose that $n>0$ and $0 \leq m \leq n-1$ are integers.
Suppose also that $t_1, \dots, t_n$ and $W_1, \dots, W_n$ are, 
respectively, the nodes and weights of the quadrature, introduced
in Definition~\ref{def_quad} above. Suppose furthermore that
\begin{align}
c > 30,
\label{eq_quad_simple_c30}
\end{align}
and that
\begin{align}
n > \frac{2c}{\pi} + 5.
\label{eq_quad_simple_n2c}
\end{align}
Then,
\begin{align}
\left| \int_{-1}^1 \psi_m(s) \; ds - \sum_{j=1}^n \psi_m(t_j) W_j \right| \leq
|\lambda_n| \cdot \left(
24 \cdot \log\left( \frac{1}{|\lambda_n|} \right) +
6 \cdot \chi_n
\right).
\label{eq_quad_simple_thm}
\end{align}
\label{thm_quad_simple}
\end{thm}
\begin{proof}
We combine Theorems~\ref{thm_pswf_main}, \ref{thm_khi_elem},
\ref{thm_all_psi_upper_bound} in Section~\ref{sec_pswf},
the inequality \eqref{eq_quad_simple_n2c} and
Theorems~\ref{thm_quad}, \ref{lem_quad_cp} to conclude that
\begin{align}
\left| \int_{-1}^1 \psi_m(s) \; ds - \sum_{j=1}^n \psi_m(t_j) W_j \right|
\leq
\| I \|_{\infty} + |\lambda_n| \cdot (2 \sqrt{n} + \sqrt{32} \cdot n^2),
\label{eq_quad_simple_a}
\end{align}
where $\| I \|_{\infty}$ is defined via \eqref{eq_i_inf}
in Theorem~\ref{thm_quad}. Next, we combine
\eqref{eq_quad_simple_c30}, 
\eqref{eq_quad_simple_n2c},
Theorem~\ref{thm_khi_elem} in Section~\ref{sec_pswf},
Theorems~\ref{thm_lam10}, \ref{thm_complex_summary}
in Section~\ref{sec_head_tail}
and
\eqref{eq_i_inf}
to conclude that
\begin{align}
\| I \|_{\infty} \leq |\lambda_n| \cdot
\left(
24 \cdot \log\left( \frac{2}{|\lambda_n|}\right) +
13 \cdot (\chi_n)^{1/4} + 4\sqrt{\chi_n} + 2\sqrt{2}
\right).
\label{eq_quad_simple_b}
\end{align}
We combine \eqref{eq_quad_simple_n2c} with
Theorem~\ref{thm_n_khi_simple} in Section~\ref{sec_pswf}
to conclude that
\begin{align}
n < \sqrt{\chi_n}.
\label{eq_quad_simple_c}
\end{align}
Also, we observe that,
due to the combination of \eqref{eq_quad_simple_c30}
and Theorem~\ref{thm_khi_elem} in Section~\ref{sec_pswf},
\begin{align}
& \sqrt{32} \cdot \chi_n + 4 \sqrt{\chi_n} + 15 \cdot (\chi_n)^{1/4}
 + 2 \sqrt{2} + 24 \cdot \log(2) = \nonumber \\
& \chi_n \cdot \left(
\sqrt{32} + 4 \cdot \chi_n^{-1/2} + 15 \cdot \chi_n^{-3/4} +
 (2 \sqrt{2} + 24 \cdot \log(2)) \cdot \chi_n^{-1} 
\right) < 6 \cdot \chi_n.
\label{eq_quad_simple_d}
\end{align}
Now \eqref{eq_quad_simple_thm} follows from the combination
of \eqref{eq_quad_simple_a},
\eqref{eq_quad_simple_b},
\eqref{eq_quad_simple_c} and
\eqref{eq_quad_simple_d}.
\end{proof}
%%%%%%%%%%%%
%%%%%%%%%%%%
%%%%%%%%%%%%%
The following theorem is a conclusion of Theorem~\ref{thm_lambda_khi}
of Section~\ref{sec_pswf}
and Theorems~\ref{thm_quad_simple}, \ref{thm_lam10} above.
\begin{thm}
Suppose that $n>0$ and $0 \leq m \leq n-1$ are integers.
Suppose also that $t_1, \dots, t_n$ and $W_1, \dots, W_n$ are, 
respectively, the nodes and weights of the quadrature, introduced
in Definition~\ref{def_quad} above. Suppose furthermore that
\begin{align}
c > 30,
\label{eq_error_khi_c30}
\end{align}
and that
\begin{align}
n > \frac{2c}{\pi} + 7.
\label{eq_error_khi_n7}
\end{align}
Then,
\begin{align}
\left| \int_{-1}^1 \psi_m(s) \; ds - \sum_{j=1}^n \psi_m(t_j) W_j \right| \leq
14340 \cdot \frac{\chi_n^5}{c^7} \cdot 
   \exp\left[ -\frac{\pi}{4} \cdot 
       \frac{\chi_n-c^2}{\sqrt{\chi_n}} \right].
\label{eq_error_khi_thm}
\end{align}
\label{thm_error_khi}
\end{thm}
\begin{proof}
We combine \eqref{eq_error_khi_c30}, \eqref{eq_error_khi_n7}
with Theorem~\ref{thm_quad_simple} above to obtain
\begin{align}
\left| \int_{-1}^1 \psi_m(s) \; ds - \sum_{j=1}^n \psi_m(t_j) W_j \right| \leq
|\lambda_n| \cdot \left(
24 \cdot \log\left( \frac{1}{|\lambda_n|} \right) +
6 \cdot \chi_n
\right).
\label{eq_error_khi_a}
\end{align}
Suppose first that
\begin{align}
|\lambda_n| \leq \exp\left[ -\frac{\chi_n}{4} \right].
\label{eq_error_khi_b}
\end{align}
Then,
\begin{align}
|\lambda_n| \cdot \left(
24 \cdot \log\left( \frac{1}{|\lambda_n|} \right) +
6 \cdot \chi_n
\right) \leq 48 \cdot |\lambda_n| \cdot 
  \log\left( \frac{1}{|\lambda_n|} \right).
\label{eq_error_khi_c}
\end{align}
We combine \eqref{eq_error_khi_c30}, \eqref{eq_error_khi_b}
and Theorem~\ref{thm_n_and_khi} in Section~\ref{sec_pswf} to conclude that
\begin{align}
|\lambda_n| < \exp\left[ -\frac{c^2}{4} \right] < e^{-225} < e^{-1}.
\label{eq_error_khi_d}
\end{align}
We combine \eqref{eq_error_khi_b}, \eqref{eq_error_khi_c} and
\eqref{eq_error_khi_d} to obtain
\begin{align}
& |\lambda_n| \cdot \left(
24 \cdot \log\left( \frac{1}{|\lambda_n|} \right) +
6 \cdot \chi_n
\right) \leq 
48 \cdot |\lambda_n| \cdot \log\left( \frac{1}{|\lambda_n|} \right) \leq
\nonumber \\
& 48 \cdot \exp\left[ -\frac{\chi_n}{4} \right] \cdot \frac{\chi_n}{4} =
12 \cdot \chi_n \cdot \exp\left[ -\frac{\chi_n}{4} \right].
\label{eq_error_khi_e}
\end{align}
Suppose, on the other hand, that
\begin{align}
\exp\left[ -\frac{\chi_n}{4} \right] < |\lambda_n| < \frac{1}{10}
\label{eq_error_khi_f}
\end{align}
(note that the right-hand side inequality in \eqref{eq_error_khi_f}
follows from the combination of \eqref{eq_error_khi_c30},
\eqref{eq_error_khi_n7} and Theorem~\ref{thm_lam10}).
It follows from \eqref{eq_error_khi_f} that, in this case,
\begin{align}
|\lambda_n| \cdot \left(
24 \cdot \log\left( \frac{1}{|\lambda_n|} \right) +
6 \cdot \chi_n
\right) \leq 12 \cdot \chi_n \cdot |\lambda_n|.
\label{eq_error_khi_g}
\end{align}
We combine \eqref{eq_error_khi_n7} with Theorem~\ref{thm_lambda_khi}
to obtain
\begin{align}
|\lambda_n| < 1195 \cdot \frac{\chi_n^4}{c^7} \cdot
\exp\left[ -\frac{\pi}{4} \cdot 
       \frac{\chi_n-c^2}{\sqrt{\chi_n}} \right].
\label{eq_error_khi_h}
\end{align}
We combine \eqref{eq_error_khi_c30} with Theorem~\ref{thm_n_and_khi}
of Section~\ref{sec_pswf} to conclude that
\begin{align}
\exp\left[ -\frac{\chi_n}{4} \right] <
1195 \cdot \frac{\chi_n^4}{c^7} \cdot
\exp\left[ -\frac{\pi}{4} \cdot 
       \frac{\chi_n-c^2}{\sqrt{\chi_n}} \right].
\label{eq_error_khi_i}
\end{align}
We combine 
\eqref{eq_error_khi_b},
\eqref{eq_error_khi_e},
\eqref{eq_error_khi_f},
\eqref{eq_error_khi_g},
\eqref{eq_error_khi_h},
\eqref{eq_error_khi_i} that
\begin{align}
& |\lambda_n| \cdot \left(
24 \cdot \log\left( \frac{1}{|\lambda_n|} \right) +
6 \cdot \chi_n
\right) \leq \nonumber \\
& 12 \cdot \chi_n \cdot
1195 \cdot \frac{\chi_n^4}{c^7} \cdot
\exp\left[ -\frac{\pi}{4} \cdot 
       \frac{\chi_n-c^2}{\sqrt{\chi_n}} \right].
\label{eq_error_khi_j}
\end{align}
Now \eqref{eq_error_khi_thm} follows from the combination
of \eqref{eq_error_khi_a} and \eqref{eq_error_khi_j}.
\end{proof}
%%%%%%%%%%%%
%%%%%%%%%%%
\subsubsection{The Principal Result}
\label{sec_main_result}

In Theorem~\ref{thm_error_khi}, we established an upper bound
on the quadrature error 
for integrating $\psi_m$ (see \eqref{eq_error_khi_thm}). However, this
bound depends on $\chi_n$. In particular, it is not obvious
how large $n$ should be to make sure that the quadrature error
does not exceed given $\varepsilon>0$. In this subsection,
we eliminate this inconvenience.

The following theorem is illustrated in Table~\ref{t:test178}
(see Experiment 14 in Section~\ref{sec_exp14}).
%%%%%%%%%%%%
\begin{thm}
Suppose that $c>0$ is a positive real number, and that
\begin{align}
c > 30.
\label{eq_quad_eps_large_c30}
\end{align}
Suppose also that $\varepsilon > 0$ is a positive real number, and that
\begin{align}
0 < \log \frac{1}{\varepsilon} < 
\frac{5 \cdot \pi}{4\sqrt{6}} \cdot c - 3 \cdot \log(c) - \log(6^5 \cdot 14340).
\label{eq_quad_eps_large_eps}
\end{align}
Suppose furthermore that the real number $\alpha$ is defined
via the formula
\begin{align}
\alpha = \frac{4\sqrt{6}}{\pi} \cdot
\left(
\log \frac{1}{\varepsilon} + 3 \cdot \log(c) + \log(6^5 \cdot 14340)
\right),
\label{eq_quad_eps_large_alpha}
\end{align}
and that the real number $\nu(\alpha)$ is defined via the formula
\begin{align}
\nu(\alpha) = \frac{2c}{\pi} + \frac{\alpha}{2\pi} \cdot
   \log\left( \frac{16ec}{\alpha} \right).
\label{eq_quad_eps_large_nu}
\end{align}
Suppose, in addition, that
$n>0$ and $0 \leq m \leq n-1$ are integers,
and that
\begin{align}
n > \nu(\alpha).
\label{eq_quad_eps_large_n_nu}
\end{align}
Then,
\begin{align}
\left| \int_{-1}^1 \psi_m(s) \; ds - \sum_{j=1}^n \psi_m(t_j) W_j \right|
<
\varepsilon,
\label{eq_quad_eps_large_thm}
\end{align}
where $t_j$, $W_j$ are defined, respectively,
via \eqref{eq_quad_t}, \eqref{eq_quad_w} in Definition~\ref{def_quad}.
\label{thm_quad_eps_large}
\end{thm}
\begin{proof}
It follows from
\eqref{eq_quad_eps_large_eps} that
\begin{align}
5c > \alpha > 
\frac{4\sqrt{6}}{\pi} \cdot
\left(
3 \cdot \log(c) + \log(6^5 \cdot 14340)
\right),
\label{eq_quad_eps_large_a}
\end{align}
where $\alpha$ is defined via \eqref{eq_quad_eps_large_alpha}.
We observe that 
\begin{align}
\frac{d}{d \alpha} \left[ \alpha \cdot \log\left( \frac{16ec}{\alpha} \right)
\right] = \log\left( \frac{16c}{\alpha} \right),
\label{eq_quad_eps_large_b}
\end{align}
and hence 
the function $\nu: (0,16c) \to \Rc$, defined
via \eqref{eq_quad_eps_large_nu}, is monotonically increasing.
We combine \eqref{eq_quad_eps_large_c30},
\eqref{eq_quad_eps_large_alpha},
\eqref{eq_quad_eps_large_a},
\eqref{eq_quad_eps_large_b} to conclude that
\begin{align}
\frac{2c}{\pi}+30 < \nu(\alpha) < 
\frac{2c}{\pi} + \frac{5c}{2\pi} \cdot \log\left( \frac{16e}{5} \right)
< \frac{5c}{2}.
\label{eq_quad_eps_large_c}
\end{align}
We combine Theorem~\ref{thm_khi_1} of Section~\ref{sec_pswf} with
\eqref{eq_quad_eps_large_nu},
\eqref{eq_quad_eps_large_n_nu} and \eqref{eq_quad_eps_large_a} to obtain
the inequality
\begin{align}
\chi_n > c^2 + \alpha \cdot c.
\label{eq_quad_eps_large_d}
\end{align}
Suppose now that the function $f:(c,\infty) \to \Rc$ is
defined via the formula
\begin{align}
f(y) = y^{10} \cdot \exp\left[
-\frac{\pi}{4} \cdot \frac{y^2-c^2}{y}
\right].
\label{eq_quad_eps_large_e}
\end{align}
We differentiate \eqref{eq_quad_eps_large_e} with respect to $y$ 
and use \eqref{eq_quad_eps_large_c30} to obtain
\begin{align}
f'(y) = \frac{f(y)}{y} \cdot \left[
10 - y \cdot \frac{\pi}{4} \cdot \left(1 + \frac{c^2}{y^2}\right) 
\right] < 0,
\label{eq_quad_eps_large_f}
\end{align}
for all $y > c$. We combine
\eqref{eq_quad_eps_large_c30},
\eqref{eq_quad_eps_large_c}, 
\eqref{eq_quad_eps_large_d},
\eqref{eq_quad_eps_large_e}, \eqref{eq_quad_eps_large_f} with
Theorem~\ref{thm_error_khi} to conclude that
\begin{align}
& \left| \int_{-1}^1 \psi_m(s) \; ds - \sum_{j=1}^n \psi_m(t_j) W_j \right| 
\leq \nonumber \\
& 14340 \cdot \frac{\chi_n^5}{c^7} \cdot 
   \exp\left[ -\frac{\pi}{4} \cdot 
       \frac{\chi_n-c^2}{\sqrt{\chi_n}} \right] \leq \nonumber \\
& 14340 \cdot c^3 \cdot \left(1 + \frac{\alpha}{c}\right)^5 \cdot
\exp\left[
-\frac{\pi}{4} \cdot \frac{\alpha}{\sqrt{1+\alpha/c}}
\right].
\label{eq_quad_eps_large_g}
\end{align}
We combine \eqref{eq_quad_eps_large_a}, \eqref{eq_quad_eps_large_g}
to obtain
\begin{align}
\left| \int_{-1}^1 \psi_m(s) \; ds - \sum_{j=1}^n \psi_m(t_j) W_j \right| 
\leq 14340 \cdot 6^5 \cdot c^3 \cdot
\exp\left[
-\frac{\pi}{4} \cdot \frac{\alpha}{\sqrt{6}}
\right].
\label{eq_quad_eps_large_h}
\end{align}
Now \eqref{eq_quad_eps_large_thm} follows from
the combination of 
\eqref{eq_quad_eps_large_alpha} and
\eqref{eq_quad_eps_large_h}.
\end{proof}
%%%%%%%%%%%
The following theorem is a direct consequence of
Theorem~\ref{thm_quad_eps_large}. This theorem
is one of the principal results of the paper.
It is illustrated in Table~\ref{t:test178}
(see Experiment 14 in Section~\ref{sec_exp14}).
See also Conjecture~\ref{conj_quad_error}
in Section~\ref{sec_exp12}.
\begin{thm}
Suppose that $c>0$ is a positive real number, and that
\begin{align}
c > 30.
\label{eq_quad_eps_simple_c30}
\end{align}
Suppose also that $\varepsilon > 0$ is a positive real number, and that
\begin{align}
\exp\left[ -\frac{3}{2}\cdot(c-20) \right] < \varepsilon < 1.
\label{eq_quad_eps_simple_eps}
\end{align}
Suppose furthermore that $n>0$ and $0 \leq m < n$ are positive integers, 
and that
\begin{align}
n > \frac{2c}{\pi} +
\left(10 + \frac{3}{2} \cdot \log(c) + 
   \frac{1}{2} \cdot \log\frac{1}{\varepsilon}
\right) \cdot \log\left( \frac{c}{2} \right).
\label{eq_quad_eps_simple_n}
\end{align}
Then,
\begin{align}
\left| \int_{-1}^1 \psi_m(s) \; ds - \sum_{j=1}^n \psi_m(t_j) W_j \right|
<
\varepsilon,
\label{eq_quad_eps_simple_thm}
\end{align}
where $t_j$, $W_j$ are defined, respectively,
via \eqref{eq_quad_t}, \eqref{eq_quad_w} in Definition~\ref{def_quad}.
\label{thm_quad_eps_simple}
\end{thm}
\begin{proof}
We observe that, for all real $x > 30$,
\begin{align}
\frac{3}{2} \cdot (x-20) < 
\frac{5 \cdot \pi}{4\sqrt{6}} \cdot x - 
3 \cdot \log(x) - \log(6^5 \cdot 14340).
\label{eq_quad_eps_simple_a}
\end{align}
Also, we combine \eqref{eq_quad_eps_simple_c30}, 
\eqref{eq_quad_eps_simple_eps} to conclude that
\begin{align}
\frac{4\sqrt{6}}{2\pi^2} \cdot
\left(
\log \frac{1}{\varepsilon} + 3 \cdot \log(c) + \log(6^5 \cdot 14340)
\right) <
10 + \frac{3}{2} \cdot \log(c) + \frac{1}{2} \cdot \log \frac{1}{\varepsilon}.
\label{eq_quad_eps_simple_b}
\end{align}
Furthermore,
we combine \eqref{eq_quad_eps_simple_c30}, 
\eqref{eq_quad_eps_simple_eps} to conclude that
\begin{align}
\frac{4\sqrt{6}}{\pi} \cdot
\left(
\log \frac{1}{\varepsilon} + 3 \cdot \log(c) + \log(6^5 \cdot 14340)
\right) > 89 > 2 \cdot 16 e.
\label{eq_quad_eps_simple_c}
\end{align}
Now \eqref{eq_quad_eps_simple_thm} follows 
from the combination of
\eqref{eq_quad_eps_simple_c30},
\eqref{eq_quad_eps_simple_eps},
\eqref{eq_quad_eps_simple_n},
\eqref{eq_quad_eps_simple_a},
\eqref{eq_quad_eps_simple_b},
\eqref{eq_quad_eps_simple_c} and
Theorem~\ref{thm_quad_eps_large}.
\end{proof}
%%%%%%%%%%
The assumptions of Theorem~\ref{thm_quad_eps_simple} contain
a minor inconvenience - namely, the parameter $\varepsilon$ is
not allowed to be ``too small'' (in the sense of
\eqref{eq_quad_eps_simple_eps}). In the following theorem,
we eliminate this restriction. On the other hand, for the values
of $\varepsilon$ in the range \eqref{eq_quad_eps_simple_eps}, the resulting
inequality for $n$ is much weaker than \eqref{eq_quad_eps_simple_n}.
\begin{thm}
Suppose that $c>0$ is a positive real number, and that
\begin{align}
c > 30.
\label{eq_quad_eps_weak_c30}
\end{align}
Suppose also that $\varepsilon > 0$ is a positive real number, and that
\begin{align}
0 < \varepsilon < 1.
\label{eq_quad_eps_weak_eps}
\end{align}
Suppose furthermore that $n>0$ and $0 \leq m < n$ are positive integers, 
and that
\begin{align}
n \cdot \left(1 - \frac{40}{\pi c}\right) > 
c + \frac{12}{\pi} \cdot \log(c) + \frac{4}{\pi} \cdot
\log \frac{1}{\varepsilon}.
\label{eq_quad_eps_weak_n}
\end{align}
Then,
\begin{align}
\left| \int_{-1}^1 \psi_m(s) \; ds - \sum_{j=1}^n \psi_m(t_j) W_j \right|
<
\varepsilon,
\label{eq_quad_eps_weak_thm}
\end{align}
where $t_j$, $W_j$ are defined, respectively,
via \eqref{eq_quad_t}, \eqref{eq_quad_w} in Definition~\ref{def_quad}.
\label{thm_quad_eps_weak}
\end{thm}
\begin{proof}
We combine \eqref{eq_quad_eps_weak_n} with Theorem~\ref{thm_n_khi_simple}
in Section~\ref{sec_pswf} and \eqref{eq_E} in Section~\ref{sec_elliptic}
to conclude that
\begin{align}
c^2 < n^2 < \chi_n.
\label{eq_quad_eps_weak_a}
\end{align}
Also, we combine \eqref{eq_quad_eps_weak_c30},
\eqref{eq_quad_eps_weak_eps},
\eqref{eq_quad_eps_weak_n},
\eqref{eq_quad_eps_weak_a}
with Theorem~\ref{thm_error_khi} and
\eqref{eq_quad_eps_large_f} in the proof of Theorem~\ref{thm_quad_eps_large}
to conclude that
\begin{align}
& \left| \int_{-1}^1 \psi_m(s) \; ds - \sum_{j=1}^n \psi_m(t_j) W_j \right|
 \leq \nonumber \\
& 14340 \cdot \frac{\chi_n^5}{c^7} \cdot
   \exp\left[ -\frac{\pi}{4} \cdot \frac{\chi_n-c^2}{\sqrt{\chi_n}} \right]
 \leq \nonumber \\
& 14340 \cdot c^3 \cdot \left(\frac{n}{c}\right)^{10} \cdot
   \exp\left[ -\frac{\pi}{4} \cdot c \cdot 
   \left(\frac{n}{c}-\frac{c}{n} \right) \right].
\label{eq_quad_eps_weak_b}
\end{align}
We take the logarithm of both sides of \eqref{eq_quad_eps_weak_b}
and use \eqref{eq_quad_eps_weak_a} to obtain
\begin{align}
& \log \left| \int_{-1}^1 \psi_m(s) \; ds - \sum_{j=1}^n \psi_m(t_j) W_j \right|
 < \nonumber \\
& \log(14340) + 3\cdot\log(c) + 10\cdot\log\left(\frac{n}{c}\right) -
  \frac{\pi}{4} \cdot n + \frac{\pi}{4} \cdot c < \nonumber \\
& \log(14340) + 3\cdot\log(c) + 10\cdot\left(\frac{n}{c}\right) -10 -
  \frac{\pi}{4} \cdot n + \frac{\pi}{4} \cdot c < \nonumber \\
& \frac{\pi}{4} \cdot \left(
\frac{12}{\pi} \cdot \log(c) - n \cdot\left(1-\frac{40}{\pi c}\right)
 + c \right).
\label{eq_quad_eps_weak_c}
\end{align}
Now \eqref{eq_quad_eps_weak_thm} follows from the combination
of \eqref{eq_quad_eps_weak_n} and \eqref{eq_quad_eps_weak_c}.
\end{proof}

%%%%%%%%%%%%%%%%%%%%%%%%%%%%%%%%%%%%%%%%%%%%%%%%
\subsubsection{Quadrature Weights}
\label{sec_weights}
In this subsection, we analyze the
weights $W_1,\dots,W_n$ of the quadrature,
defined in Definition~\ref{def_quad} in Section~\ref{sec_quad}.
This analysis has two principal purposes. On the one hand,
it provides the basis for a fast algorithm for the evaluation
of the weights. On the other hand, it provides a theoretical explanation
of some empirically observed properties of the weights.

The results of this subsection are 
illustrated in Table~\ref{t:test96} and in Figure~\ref{fig:test96}
(see Experiment 15 in Section~\ref{sec_exp15}).

In the following theorem, we describe a function, whose values at the roots
$t_1, \dots, t_n$ of $\psi_n$ in $(-1,1)$ are equal to the quadrature
weights $W_1,\dots,W_n$, up to a certain scaling.
%%%%%%%%%%%
\begin{thm}
Suppose that $n$ is a non-negative integer. 
Suppose also that the function $\tilde{\Phi}_n: (-1,1) \to \Rc$
is defined via the formula
\begin{align}
\tilde{\Phi}_n(t) = \sum_{k = 0}^{\infty} \alpha_k^{(n)} Q_k(t),
\label{eq_num_tilde_phi_def}
\end{align}
where $Q_k(t)$ is the $k$th Legendre function of the second kind, defined 
in Section~\ref{sec_legendre},
and $\alpha_k^{(n)}$ is the $k$th coefficient of the Legendre
expansion of $\psi_n$, defined via
\eqref{eq_num_leg_alpha_knc}
in Section~\ref{sec_legendre}.
Suppose furthermore that $t_1 < \dots < t_n$ are the roots
of $\psi_n$ in $(-1,1)$. Then, for every integer
$j = 1, 2, \dots, n$,
\begin{align}
\tilde{\Phi}_n(t_j) = 
\frac{ 1 }{ 2 } \int_{-1}^1 \frac{ \psi_n(t) \; dt}{ t_j - t }.
\label{eq_num_tilde_phi_equality}
\end{align}
\label{lem_tilde_phi}
\end{thm}
\begin{proof}
Suppose that $1 \leq j \leq n$ is an integer, and that
$\delta > 0$ is a positive real number.
We combine \eqref{eq_num_tilde_phi_def} with
\eqref{eq_num_leg_exp},
\eqref{eq_num_leg_beta_knc},
\eqref{eq_num_leg_alpha_knc},
\eqref{eq_qn_pn} in Section~\ref{sec_legendre}
to obtain
\begin{align}
\sum_{k = 0}^{\infty} \alpha_k^{(n)} Q_k(t_j + i\delta) =
\frac{ 1 }{ 2 } \int_{-1}^1 \frac{ \psi_n(t) \; dt}{ t_j+i \delta - t },
\label{eq_lem_tilde_phi_b}
\end{align}
provided that $\delta$ is sufficiently small. 
Suppose now that 
$\varepsilon > 0$ is a real number, and that
\begin{align}
\varepsilon < \frac{1}{2} \cdot \min\left\{
|t_j-1|, \; |t_j+1|
\right\}.
\label{eq_lem_tilde_phi_c}
\end{align}
We observe that,
since $t_j$ is a root of $\psi_n$, the right-hand side
of \eqref{eq_num_tilde_phi_equality} is well defined.
We combine this observation with \eqref{eq_lem_tilde_phi_b},
\eqref{eq_lem_tilde_phi_c} to evaluate
\begin{align}
& \lim_{\delta\to 0, \; \delta>0} \left(
\frac{ 1 }{ 2 } \int_{t_j-\varepsilon}^{t_j+\varepsilon} 
  \frac{ \psi_n(t) \; dt}{ t_j+i \delta - t } -
\frac{ 1 }{ 2 } \int_{t_j-\varepsilon}^{t_j+\varepsilon} 
 \frac{ \psi_n(t) \; dt}{ t_j - t } \right) =
\nonumber \\
& \lim_{\delta\to 0, \; \delta>0}
 \frac{1}{2} \int_{-\varepsilon}^\varepsilon \psi_n(t_j+s) \cdot
 \left( \frac{1}{s+i\delta} - \frac{1}{s} \right) ds = \nonumber \\
& -\lim_{\delta\to 0, \; \delta>0} 
  \frac{i\delta\cdot\psi_n'(t_j)}{2} 
  \int_{-\varepsilon}^{\varepsilon} \frac{ds}{s+i\delta} = \nonumber \\
& \lim_{\delta\to 0, \; \delta>0} 
 \delta \cdot \psi_n'(t_j) \cdot \arctan\left( \frac{\varepsilon}{\delta}\right)
= 0.
\label{eq_lem_tilde_phi_d}
\end{align}
We combine \eqref{eq_num_tilde_phi_def},
\eqref{eq_lem_tilde_phi_b}, \eqref{eq_lem_tilde_phi_d}
to obtain \eqref{eq_num_tilde_phi_equality}.
\end{proof}
%%%%%%%%%%%
The following corollary is a direct consequence of Definition~\ref{def_quad}
and Theorem~\ref{lem_tilde_phi}.
\begin{cor}
Suppose that $n>0$ and $1 \leq j \leq n$ are positive integers.
Suppose also that
the function $\tilde{\Phi}_n:(-1,1) \to \Rc$
is defined
via \eqref{eq_num_tilde_phi_def} in Theorem~\ref{lem_tilde_phi}.
Then, 
\begin{align}
W_j = - 2 \cdot \frac{ \tilde{\Phi}_n(t_j) }{ \psi_n'(t_j) },
\label{eq_tilde_phi_w}
\end{align}
where $t_j$, $W_j$ are defined, respectively,
via \eqref{eq_quad_t}, \eqref{eq_quad_w} of Definition~\ref{def_quad}.
\label{cor_tilde_phi_w}
\end{cor}
Corollary~\ref{cor_tilde_phi_w} is illustrated in
Table~\ref{t:test96}. We observe that Theorem~\ref{lem_tilde_phi}
and Corollary~\ref{cor_tilde_phi_w} describe a connection
between the weights $W_1, \dots, W_n$ and the values
of $\tilde{\Phi}_n$ at $t_1, \dots, t_n$, where the function $\tilde{\Phi}_n$
is defined via \eqref{eq_num_tilde_phi_def}.
In the following theorem, we prove that $\tilde{\Phi}_n$ satisfies
a certain second-order non-homogeneous ODE, closely related
to the prolate ODE \eqref{eq_prolate_ode} in Section~\ref{sec_pswf}.
%%%%%%%%%%%
\begin{thm}
Suppose that $n$ is a non-negative integer, and that
the function $\tilde{\Phi}_n : (-1,1) \to \Rc$
is defined via \eqref{eq_num_tilde_phi_def} in Theorem~\ref{lem_tilde_phi}. 
Suppose also that the second-order differential operator $L_n$
is defined via the formula
\begin{align}
L_n\left[\varphi\right](t) = \brk{1-t^2} \varphi''(t) - 2t \varphi'(t)
+ \brk{\chi_n - c^2 t^2} \varphi(t).
\label{eq_ln_def}
\end{align}
Then, in the interval $(-1,1)$ the function
$\tilde{\Phi}_n$
satisfies the nonhomogeneous ODE 
\begin{align}
L_n \left[ \tilde{\Phi}_n \right] (t) = 
-c^2 \brk{ \alpha_0^{(n)} t + \alpha_1^{(n)} / 3 },
\label{eq_num_tilde_phi_ode}
\end{align}
where the coefficients $\alpha_0^{(n)}, \alpha_1^{(n)}$ are the first
two coefficients of the Legendre expansion of $\psi_n$,
defined via \eqref{eq_num_leg_alpha_knc} in Section~\ref{sec_legendre}.
\label{lem_tilde_phi_ode}
\end{thm}
\begin{proof}
We combine \eqref{eq_legendre_ode}, \eqref{eq_num_rk_rec}
of Section~\ref{sec_legendre} with
\eqref{eq_ln_def} to obtain
\begin{align}
L_n\left[Q_k\right] = \left(\chi_n - k(k+1) - c^2 t^2\right) \cdot Q_k,
\label{eq_ln_rk}
\end{align}
where $Q_k$ is the $k$th Legendre function of the second kind,
defined in Section~\ref{sec_legendre}.
We combine \eqref{eq_num_rk_rec} of Section~\ref{sec_legendre}
with \eqref{eq_ln_rk} to obtain
\begin{align}
& L_n \left[ \sum_{k=0}^{\infty} \alpha_k^{(n)} Q_k \right]
= \sum_{k=0}^{\infty} \alpha_k^{(n)} \brk{\chi_n - k(k+1) - c^2 t^2} Q_k =
\nonumber \\
& \sum_{k=0}^{\infty} \alpha_k^{(n)} \brk{\chi_n - k(k+1)} Q_k
\nonumber \\
& \; \quad  - c^2 \sum_{k=0}^{\infty} \alpha_k^{(n)} 
           \brk{A_{k-2} Q_{k-2} + B_k Q_k + C_{k+2} Q_{k+2}} = \nonumber \\
& \sum_{k = 2}^{\infty} \left[
 \brk{\chi_n - k(k+1)} \alpha_k^{(n)} - 
  c^2 \brk{\alpha_{k+2}^{(n)} A_k + \alpha_k^{(n)} B_k + 
           \alpha_{k-2}^{(n)} C_k} \right] Q_k
  \nonumber \\
& \; \quad + \left[ \brk{\chi_n - 1(1+1)} \alpha_1^{(n)} - 
                    c^2 \brk{\alpha_3^{(n)} A_1 + \alpha_1^{(n)} B_1} 
 \right] Q_1 
  \nonumber \\
& \; \quad + \left[ \brk{\chi_n - 0(0+1)} \alpha^{(n)}_0 -
                    c^2 \brk{\alpha_2^{(n)} A_0 + \alpha_0^{(n)} B_0} 
\right] Q_0
  \nonumber \\
& \; \quad - c^2 \brk{
  \alpha_1^{(n)} \brk{t^2 Q_1 - B_1 Q_1 - C_3 Q_3} + 
  \alpha_0^{(n)} \brk{t^2 Q_0 - B_0 Q_0 - C_2 Q_2}},
\label{eq_num_ln_p_expansion}
\end{align}
where $A_k, B_k, C_k$ are defined, respectively,
via
\eqref{eq_a_k_rec}, \eqref{eq_b_k_rec}, \eqref{eq_c_k_rec}
in Section~\ref{sec_legendre}.
By the same token, \eqref{eq_num_ln_p_expansion} holds, if we replace
$Q_k$'s with $P_k$'s, where $P_k$ is the $k$th Legendre polynomial
defined in Section~\ref{sec_legendre}. In other words,
\begin{align}
& L_n \left[ \sum_{k=0}^{\infty} \alpha_k^{(n)} P_k \right]
= \nonumber \\
& \sum_{k = 2}^{\infty} \left[
 \brk{\chi_n - k(k+1)} \alpha_k^{(n)} - 
  c^2 \brk{\alpha_{k+2}^{(n)} A_k + \alpha_k^{(n)} B_k + 
           \alpha_{k-2}^{(n)} C_k} \right] P_k
  \nonumber \\
& \; \quad + \left[ \brk{\chi_n - 1(1+1)} \alpha_1^{(n)} - 
                    c^2 \brk{\alpha_3^{(n)} A_1 + \alpha_1^{(n)} B_1} 
\right] P_1 
  \nonumber \\
& \; \quad + \left[ \brk{\chi_n - 0(0+1)} \alpha^{(n)}_0 -
                    c^2 \brk{\alpha_2^{(n)} A_0 + \alpha_0^{(n)} B_0} 
\right] P_0
  \nonumber \\
& \; \quad - c^2 \brk{
  \alpha_1^{(n)} \brk{t^2 P_1 - B_1 P_1 - C_3 P_3} + 
  \alpha_0^{(n)} \brk{t^2 P_0 - B_0 P_0 - C_2 P_2}},
\label{eq_num_ln_p_expansion2}
\end{align}
We combine \eqref{eq_legendre_pol_0_1},
\eqref{eq_num_rk_rec} of Section~\ref{sec_legendre}
to conclude that
\begin{align}
& t^2 \cdot P_1(t) - B_1 \cdot P_1(t) - C_3 \cdot P_3(t) = 0, \nonumber \\
& t^2 \cdot P_0(t) - B_0 \cdot P_0(t) - C_2 \cdot P_2(t) = 0.
\label{eq_num_ln_1}
\end{align}
We recall that $\left\{ P_k \right\}$ form an orthogonal
system in $L^2\left[-1,1\right]$, and combine this observation with
\eqref{eq_prolate_ode} in Section~\ref{sec_pswf}, 
\eqref{eq_num_leg_exp} in Section~\ref{sec_legendre},
\eqref{eq_ln_def}, 
\eqref{eq_num_ln_p_expansion2} and \eqref{eq_num_ln_1} to conclude that,
for every integer $k \geq 2$,
\begin{align}
\brk{\chi_n - k(k+1)} \alpha_k^{(n)} - 
  c^2 \brk{\alpha_{k+2}^{(n)} A_k + \alpha_k^{(n)} B_k + 
           \alpha_{k-2}^{(n)} C_k} = 0,
\label{eq_num_ln_2}
\end{align}
and also
\begin{align}
& \brk{\chi_n - 1(1+1)} \alpha_1^{(n)} - 
                    c^2 \brk{\alpha_3^{(n)} A_1 + \alpha_1^{(n)} B_1}  = 0, 
  \nonumber \\
& \brk{\chi_n - 0(0+1)} \alpha^{(n)}_0 -
                    c^2 \brk{\alpha_2^{(n)} A_0 + \alpha_0^{(n)} B_0} = 0.
\label{eq_num_ln_3}
\end{align}
We substitute \eqref{eq_num_ln_2}, \eqref{eq_num_ln_3}
into \eqref{eq_num_ln_p_expansion} and use
\eqref{eq_num_tilde_phi_def} to obtain
\begin{align}
L_n \left[ \tilde{\Phi}_n \right](t)
& \; = L_n \left[ \sum_{k=0}^{\infty} \alpha_k^{(n)} Q_k \right](t) 
= \nonumber \\
& \; - c^2 \alpha_1^{(n)} 
           \brk{t^2 Q_1(t) - B_1 Q_1(t) - C_3 Q_3(t)} \nonumber \\
& \; -c^2 \alpha_0^{(n)} \brk{t^2 Q_0(t) - B_0 Q_0(t) - C_2 Q_2(t)}.
\label{eq_num_ln_q_expansion}
\end{align}
We combine 
\eqref{eq_legendre_fun_0_1}, \eqref{eq_legendre_fun_2_3},
\eqref{eq_b_k_rec}, \eqref{eq_c_k_rec}
of Section~\ref{sec_legendre} to obtain
\begin{align}
& t^2 Q_0(t) - B_0 Q_0(t) - C_2 Q_2(t) = \nonumber \\
& \brk{t^2 - \frac{1}{3}} \frac{1}{2} \log \frac{1+t}{1-t}
  - \frac{2}{3} \brk{\frac{1}{4}\brk{3t^2 - 1} \log \frac{1+t}{1-t} - 
    \frac{3}{2} t } = t
\label{eq_num_ln_4}
\end{align}
and
\begin{align}
& t^2 Q_1(t) - B_1 Q_1(t) - C_3 Q_3(t) = \nonumber \\
& \brk{t^2 - \frac{3}{5}} \brk{\frac{t}{2} \log \frac{1+t}{1-t} - 1} - 
  \frac{2}{5} \brk{ \frac{1}{4} \brk{5t^3 - 3t} \log \frac{1+t}{1-t} -
   \frac{5}{2}t^2 + \frac{2}{3} } = \nonumber \\
& \brk{ \frac{t}{2} \brk{t^2 - \frac{3}{5}} - \frac{1}{10} \brk{5t^2 - 3t} }
  \log \frac{1+t}{1-t} - t^2 + \frac{3}{5} + t^2 - \frac{4}{15} = \frac{1}{3}.
\label{eq_num_ln_5}
\end{align}
Finally, we substitute \eqref{eq_num_ln_4}, 
\eqref{eq_num_ln_5} into 
\eqref{eq_num_ln_q_expansion} to obtain
\eqref{eq_num_tilde_phi_ode}.
\end{proof}
In the following corollary, we establish a recurrence
relation between the derivatives of $\tilde{\Phi}_n$ of
arbitrary order (compare to Theorem~\ref{thm_dpsi_reck} in
Section~\ref{sec_pswf}).
\begin{cor}
Suppose that the function $\tilde{\Phi}_n : (-1,1) \to \Rc$ is defined
via \eqref{eq_num_tilde_phi_def} of Theorem~\ref{lem_tilde_phi}.
Suppose also that $-1 < t < 1$ is a real number. Then,
\begin{align}
& \left(1 - t^2\right) \cdot  \tilde{\Phi}_n'''(t) - 
4t \cdot \tilde{\Phi}_n''(t) +
\left(\chi_n - c^2 t^2 - 2\right) \cdot \tilde{\Phi}_n'(t) - 
2c^2 t \cdot \tilde{\Phi}_n(t) = 
\nonumber \\
& -c^2 \alpha_0^{(n)},
\label{eq_num_dtilde_phi_3}
\end{align}
where $\alpha_0^{(n)}$ is defined via
\eqref{eq_num_leg_alpha_knc} in Section~\ref{sec_legendre}
(compare to \eqref{eq_dpsi_rec3} of
Theorem~\ref{thm_dpsi_reck} 
in Section~\ref{sec_pswf}).
Also, for every integer $k \geq 2$, 
\begin{align}
& \brk{1 - t^2} \tilde{\Phi}_n^{(k+2)}(t) 
  - 2 \brk{k+1} t \tilde{\Phi}_n^{(k+1)}(t)
  + \brk{\chi_n - k\brk{k+1} - c^2 t^2} \tilde{\Phi}_n^{(k)}(t) \nonumber \\
& \quad \quad 
  -c^2 k t \tilde{\Phi}_n^{(k-1)}(t) 
  -c^2 k \brk{k-1} \tilde{\Phi}_n^{(k-2)}(t) = 0.
\label{eq_dphi_reck}
\end{align}
In other words, the higher order derivatives of $\tilde{\Phi}_n$ 
and $\psi_n$ satisfy the same recurrence relation
\eqref{eq_dpsi_reck} (see Theorem~\ref{thm_dpsi_reck} 
in Section~\ref{sec_pswf}).
\label{cor_dtilde_phi}
\end{cor}
\begin{proof}
To prove \eqref{eq_num_dtilde_phi_3}, we differentiate both sides
of \eqref{eq_num_tilde_phi_ode} with respect to $t$. 
To prove \eqref{eq_dphi_reck},
we observe that the second derivative of the right-hand side of
\eqref{eq_num_tilde_phi_ode} is identically zero, and combine this observation
with Theorem~\ref{thm_dpsi_reck} in Section~\ref{sec_pswf}.
\end{proof}

The rest of this subsection is devoted to establishing the positivity
of the quadrature weights $W_1, \dots, W_n$, defined via \eqref{eq_quad_w}
in Definition~\ref{def_quad}.
The principal result of this part is
Theorem~\ref{thm_positive_w} 
(see also Remarks~\ref{rem_w_even_n},~\ref{rem_w_always_pos}).
%
%
% Actually it can be shown that for example for even $n$
% \begin{align}
% & \brk{1 - t_j^2} \brk{ \psi_n'(t_j) }^2 W_j = \nonumber \\
% & \brk{1 - t_{j_0}^2} \brk{ \psi_n'(t_{j_0}) }^2 W_{j_0} -
% c^2 \lambda_n \psi_n(0) \int_{t_{j_0}}^{t_j} \psi_n(t) t \; dt,
% \label{eq_num_w_explicit}
% \end{align}
% where $t_{j_0}$ is the minimal root of $\psi_n$ greater than 0 and
% $j > j_0$.
% 
%
%%%%%%%%%%%%%%%%%%%%%%%%%%%
\begin{thm}
Suppose that $c>0$ is a real number, and that $n>0$ is an odd integer.
Suppose also that $t_1, t_2, \dots, t_n$ and $W_1, W_2, \dots, W_n$
are defined, respectively, via \eqref{eq_quad_t}, \eqref{eq_quad_w}
in Definition~\ref{def_quad}. Suppose furthermore that the integer
$j_0$ is defined via the formula
\begin{align}
j_0 = \frac{n+1}{2}.
\label{eq_wj_nice_j0}
\end{align}
Then, for every integer $j = 1, \dots, n$,
\begin{align}
\frac{ \left(\psi_n'(t_j)\right)^2 \cdot (1-t_j^2) }
     { \left(\psi_n'(0)\right)^2 } \cdot W_j =
W_{j_0} + \frac{ic\lambda_n}{\psi_n'(0)} \int_0^{t_j} \psi_n(t) \; dt.
\label{eq_thm_wj_nice}
\end{align}
\label{thm_wj_nice}
\end{thm}
\begin{proof}
Suppose that the differential operator $L_n$ is defined via
\eqref{eq_ln_def} in Theorem~\ref{lem_tilde_phi_ode}. Suppose also
that the function $\Phi_n: (-1,1) \to \Rc$ is the solution
of the homogeneous second-order ODE
\begin{align}
L_n\left[ \varphi \right] = 0
\label{eq_wj_nice_a}
\end{align}
in the interval $(-1,1)$
with the initial conditions
\begin{align}
\Phi_n(0) = \frac{1}{\psi_n'(0)}, \quad \Phi_n'(0) = 0.
\label{eq_wj_nice_b}
\end{align}
Obviously, $\Phi_n$ is an even function. Moreover,
\begin{align}
\Phi_n(t) \cdot \psi_n'(t) - \Phi_n'(t) \cdot \psi_n(t) = \frac{1}{1-t^2} 
\label{eq_wj_nice_c}
\end{align}
for all real $-1 < t < 1$ (this is the classical
Abel's formula; see e.g. Theorem 3.3.2 in \cite{DiPrima}).
Suppose that the function $\tilde{\Phi}_n: (-1,1) \to \Rc$ is defined
via \eqref{eq_num_tilde_phi_def} in Theorem~\ref{lem_tilde_phi}.
We combine \eqref{eq_num_tilde_phi_ode} of Theorem~\ref{lem_tilde_phi_ode}
with \eqref{eq_num_lambda_odd} to conclude that $\tilde{\Phi}_n$
satisfies the non-homogeneous ODE
\begin{align}
L_n\left[ \tilde{\Phi}_n \right](x) = \frac{ic\lambda_n \psi_n'(0)}{2},
\label{eq_wj_nice_d}
\end{align}
for all real $-1 < x < 1$. 
We observe that $\psi_n, \Phi_n$ are two independent solutions
of the ODE \eqref{eq_wj_nice_a}, and combine this observation
with \eqref{eq_wj_nice_d} to conclude that, for all real $-1 < x < 1$,
\begin{align}
\tilde{\Phi}_n(x) = & \; C_1 \cdot \psi_n(x) + C_2 \cdot \Phi_n(x) \; + 
\nonumber \\
& \; \frac{ic\lambda_n \psi_n'(0)}{2} \cdot \left(
\psi_n(x) \int_0^x \Phi_n(t)\;dt - \Phi_n(x) \int_0^x \psi_n(t) \; dt
\right),
\label{eq_wj_nice_e}
\end{align}
for some constants $C_1, C_2$. Out of the four summands on the right-hand
side of \eqref{eq_wj_nice_e}, the function $C_1 \cdot \psi_n(x)$ is odd,
while the other three functions are even. We combine
this observation with \eqref{eq_num_tilde_phi_def} and
\eqref{eq_wj_nice_e} to conclude that
\begin{align}
C_1 = 0.
\label{eq_wj_nice_f}
\end{align}
On the other hand, we substitute $x=0$ into \eqref{eq_wj_nice_e} to conclude
that
\begin{align}
C_2 = \frac{ \tilde{\Phi}_n(0) }{ \Phi_n(0) }.
\label{eq_wj_nice_g}
\end{align}
Suppose now that $j$ is an integer between $1$ and $n$. We recall that $t_j$
is a root of $\psi_n$
due to \eqref{eq_quad_t}, and combine this observation with
\eqref{eq_wj_nice_b},
\eqref{eq_wj_nice_e}, \eqref{eq_wj_nice_f}, \eqref{eq_wj_nice_g}
to obtain
\begin{align}
\tilde{\Phi}_n(t_j) 
 = \Phi_n(t_j) \cdot \left( \tilde{\Phi}_n(0) \psi_n'(0) -
  \frac{ic \lambda_n \psi_n'(0)}{2} \int_0^{t_j} \psi_n(t) \; dt \right).
\label{eq_wj_nice_h}
\end{align}
We combine \eqref{eq_tilde_phi_w} of Corollary~\ref{cor_tilde_phi_w}
with \eqref{eq_wj_nice_j0} and \eqref{eq_wj_nice_h} to obtain
\begin{align}
W_j \cdot \psi_n'(t_j) = \Phi_n(t_j) \cdot \left(
W_{j_0} \cdot \left(\psi_n'(0)\right)^2 + 
ic\lambda_n \psi_n'(0) \int_0^{t_j} \psi_n(t) \; dt
\right).
\label{eq_wj_nice_i}
\end{align}
Finally, we combine \eqref{eq_wj_nice_c} with \eqref{eq_wj_nice_i}
to obtain \eqref{eq_thm_wj_nice}.
\end{proof}
%%%%
%%%%
\begin{thm}
Suppose that $c>0$ is a positive real number, and that
\begin{align}
c > 30.
\label{eq_pos_w_c30}
\end{align}
Suppose also that $n>0$ is an odd positive integer, and that
\begin{align}
n > \frac{2c}{\pi} + 7.
\label{eq_pos_w_n}
\end{align}
Suppose also that $t_1, \dots, t_n$ and $W_1, \dots, W_n$ are defined,
respectively, via \eqref{eq_quad_t}, \eqref{eq_quad_w}
of Definition~\ref{def_quad}. Suppose, in addition, that
\begin{align}
W_{(n+1)/2} \leq 2 \cdot |\lambda_n| \cdot \sqrt{2n}.
\label{eq_pos_w_w0}
\end{align}
Then,
\begin{align}
W_1 + \dots + W_n \leq
\frac{4 \sqrt{2} \cdot \left(\chi_n\right)^{7/4} }{\chi_n-c^2} \cdot
|\lambda_n|.
\label{eq_pos_w_thm}
\end{align}
\label{thm_pos_w}
\end{thm}
\begin{proof}
We combine \eqref{eq_pos_w_n},
Theorems~\ref{thm_n_and_khi},~\ref{thm_Q_Q_tilde} in Section~\ref{sec_pswf} and
\eqref{eq_quad_32n} in the proof of Theorem~\ref{lem_quad_cp} to conclude that
\begin{align}
\frac{c}{|\psi_n'(0)|} \leq 4 \sqrt{n}.
\label{eq_pos_w_a}
\end{align}
We combine \eqref{eq_pos_w_a} with Theorem~\ref{thm_pswf_main}
of Section~\ref{sec_pswf} to conclude that, for any $-1<x<1$,
\begin{align}
\left| \frac{ic\lambda_n}{\psi_n'(0)} \int_0^x \psi_n(t) \; dt \right| \leq
4 |\lambda_n| \sqrt{n} \int_0^1 |\psi_n(t)| \; dt \leq
2 \cdot |\lambda_n| \cdot \sqrt{2n}.
\label{eq_pos_w_b}
\end{align}
We combine \eqref{eq_pos_w_w0}, \eqref{eq_pos_w_b} with 
Theorem~\ref{thm_wj_nice} to conclude that, for every
integer $1 \leq j \leq n$,
\begin{align}
W_j \leq \frac{ \left(\psi_n'(0)\right)^2 }
{ \left(\psi_n'(t_j)\right)^2 (1-t_j^2) }
   \cdot 4 \cdot |\lambda_n| \cdot \sqrt{2n}.
\label{eq_pos_w_c}
\end{align}
We combine \eqref{eq_pos_w_n}, \eqref{eq_pos_w_c} and
Theorems~\ref{thm_n_and_khi},~\ref{thm_Q_Q_tilde} 
in Section~\ref{sec_pswf} to conclude that,
for every integer $1 \leq j \leq n$,
\begin{align}
W_j \leq \frac{\chi_n}{\chi_n-c^2 \cdot t_j^2} 
    \cdot 4 \cdot |\lambda_n| \cdot \sqrt{2n}.
\label{eq_pos_w_d}
\end{align}
We combine \eqref{eq_pos_w_n} with 
Theorems~\ref{thm_n_khi_simple} in Section~\ref{sec_pswf} to obtain
the inequality
\begin{align}
n < \sqrt{\chi_n}.
\label{eq_pos_w_e}
\end{align}
Now \eqref{eq_pos_w_thm} follows from the combination of
\eqref{eq_pos_w_d} and \eqref{eq_pos_w_e}.
\end{proof}
%%%%%%%%%%%%%%%
%%%%
\begin{thm}
Suppose that $c>0$ is a positive real number, and that
\begin{align}
c > 30.
\label{eq_sum_w_c30}
\end{align}
Suppose also that $n>0$ is a positive integer, and that
\begin{align}
n > \frac{2c}{\pi} + 7.
\label{eq_sum_w_n}
\end{align}
Suppose also that $t_1, \dots, t_n$ and $W_1, \dots, W_n$ are defined,
respectively, via \eqref{eq_quad_t}, \eqref{eq_quad_w}
of Definition~\ref{def_quad}. 
Then,
\begin{align}
W_1 + \dots + W_n >
2 - |\lambda_n| \cdot \left(24 \cdot \log \frac{1}{|\lambda_n|} +
130 \cdot \sqrt[4]{\chi_n} \right).
\label{eq_sum_w_thm}
\end{align}
\label{thm_sum_w}
\end{thm}
\begin{proof}
Suppose that the function $I(t) : (-1, 1) \to \Rc$ is defined via
\eqref{eq_complex_it} in Theorem~\ref{thm_complex}. Then,
\begin{align}
1 = \sum_{j=1}^n \frac{\psi_n(t)}{\psi_n'(t_j) \cdot (t-t_j)} +
    I(t) \cdot \psi_n(t), 
\label{eq_sum_w_a}
\end{align}
for all real $-1 < t < 1$.
We integrate \eqref{eq_sum_w_a} over the interval $(-1,1)$
and use Theorem~\ref{thm_pswf_main} in Section~\ref{sec_pswf},
Theorems~\ref{thm_khi_lambda_bound},
\ref{thm_complex},
\ref{thm_lam10} and Definition~\ref{def_quad} to obtain
\begin{align}
W_1 + \dots + W_n > 2 - 
|\lambda_n| \cdot \left(
24 \cdot \log \frac{2}{|\lambda_n|} + 
13 \sqrt[4]{\chi_n} + 40 c |\lambda_n| + 2\sqrt{2}.
\right)
\label{eq_sum_w_b} 
\end{align}
We combine \eqref{eq_prolate_mu}, \eqref{eq_mu_leg_1},
Theorem~\ref{thm_n_and_khi}
in Section~\ref{sec_pswf}
with \eqref{eq_sum_w_n}
to obtain
\begin{align}
40 c |\lambda_n| < 40 \sqrt{2\pi c} < 40 \sqrt{2\pi} \cdot \sqrt[4]{\chi_n}.
\label{eq_sum_w_c}
\end{align} 
We combine \eqref{eq_sum_w_c30}, \eqref{eq_sum_w_n},
\eqref{eq_sum_w_c} with Theorem~\ref{thm_n_and_khi}
in Section~\ref{sec_pswf} to obtain
\begin{align}
13 \sqrt[4]{\chi_n} + 40 c |\lambda_n| + 2\sqrt{2} <
130 \sqrt[4]{\chi_n}.
\label{eq_sum_w_d}
\end{align} 
Now we substitute \eqref{eq_sum_w_d} into \eqref{eq_sum_w_b}
to obtain \eqref{eq_sum_w_thm}.
\end{proof}
%%%%%%%%%%%%%%%%%%%
%%%%
\begin{thm}
Suppose that $c>0$ is a positive real number, and that
\begin{align}
c > 30.
\label{eq_final_w_c30}
\end{align}
Suppose also that the real number $\beta$ is defined via
the formula
\begin{align}
\beta = \frac{90}{\log(30)}.
\label{eq_final_w_beta}
\end{align}
Suppose furthermore that $n>0$ is a positive integer, and that
\begin{align}
n > \frac{2c}{\pi} + \frac{\beta \cdot \log(c)}{2 \pi} \cdot
\log\left( \frac{16 e c}{\beta \cdot \log(c)} \right)
\label{eq_final_w_n}
\end{align}
Then,
\begin{align}
|\lambda_n| \cdot \left(
24 \cdot \log\frac{1}{|\lambda_n|} +
130 \sqrt[4]{\chi_n} + \frac{4\sqrt{2}\left(\chi_n\right)^{7/4}}{\chi_n-c^2}
\right) < 2 \cdot e^{-10}.
\label{eq_final_w_thm}
\end{align}
\label{thm_final_w}
\end{thm}
\begin{proof}
We combine \eqref{eq_final_w_c30}, \eqref{eq_final_w_beta}, \eqref{eq_final_w_n}
with Theorem~\ref{thm_n_and_khi} in Section~\ref{sec_pswf}
to obtain the inequality
\begin{align}
130 \sqrt[4]{\chi_n} + \frac{4\sqrt{2}\left(\chi_n\right)^{7/4}}{\chi_n-c^2}
< \frac{130\cdot\left(\chi_n\right)^{5/4} + 
        4\sqrt{2}\left(\chi_n\right)^{7/4}}{\chi_n-c^2}
< \frac{10\cdot\left(\chi_n\right)^{7/4}}{\chi_n-c^2}.
\label{eq_final_w_a}
\end{align}
Also we combine \eqref{eq_final_w_c30},
\eqref{eq_final_w_beta}, \eqref{eq_final_w_n} with
Theorem~\ref{thm_khi_1} in Section~\ref{sec_pswf} to conclude that
\begin{align}
\chi_n > c^2 + \beta \cdot \log(c) \cdot c.
\label{eq_final_w_a2}
\end{align}
We combine \eqref{eq_final_w_c30}, \eqref{eq_final_w_beta}
and \eqref{eq_final_w_a2} to obtain
\begin{align}
\frac{\left(\chi_n\right)^{3/4}}{\chi_n-c^2} <
\frac{ c^{3/2} \cdot \left(1 + \beta \cdot \log(c)/c\right)^{3/4}}
{\beta \cdot\log(c) \cdot c} <
\frac{ \sqrt{8c} }{\beta \cdot \log(c)}.
\label{eq_final_w_b}
\end{align}
We substitute \eqref{eq_final_w_b} into \eqref{eq_final_w_a}
to obtain
\begin{align}
130 \sqrt[4]{\chi_n} + \frac{4\sqrt{2}\left(\chi_n\right)^{7/4}}{\chi_n-c^2}
< \frac{ 10 \sqrt{8c} \cdot \chi_n}{\beta \cdot \log(c)}.
\label{eq_final_w_c}
\end{align}
We combine \eqref{eq_final_w_c30}, \eqref{eq_final_w_beta},
\eqref{eq_final_w_n} with Theorem~\ref{thm_lambda_khi}
to obtain
\begin{align}
|\lambda_n| < 1195 \cdot \frac{\chi_n^4}{c^7} \cdot
\exp\left[ -\frac{\pi}{4} \cdot 
       \frac{\chi_n-c^2}{\sqrt{\chi_n}} \right].
\label{eq_final_w_d}
\end{align}
We combine \eqref{eq_quad_eps_large_e}, \eqref{eq_quad_eps_large_f}
in the proof of Theorem~\ref{thm_quad_eps_large} with
\eqref{eq_final_w_c30}, \eqref{eq_final_w_beta},
\eqref{eq_final_w_a2},
\eqref{eq_final_w_c}, \eqref{eq_final_w_d} to obtain
\begin{align}
& |\lambda_n| \cdot \left(
130 \sqrt[4]{\chi_n} + \frac{4\sqrt{2}\left(\chi_n\right)^{7/4}}{\chi_n-c^2}
\right) < \nonumber \\
& \frac{ 11950 \cdot c^3 \sqrt{8c} \cdot \left(1+\beta\cdot\log(c)/c\right)^5 }
{ \beta \cdot\log(c)} \cdot
\exp\left[ -\frac{\pi}{4} \cdot 
  \frac{ \beta \cdot \log(c)}{\sqrt{1 + \beta \cdot \log(c)/c}}
\right] < \nonumber \\
& \frac{ 11950 \cdot c^3 \sqrt{8c} \cdot 4^5 }
{ \beta \cdot\log(c)} \cdot
\exp\left[ -\frac{\pi \cdot \beta \cdot \log(c)}{8} 
\right].
\label{eq_final_w_e}
\end{align}
We take the logarithm of the right-hand side of \eqref{eq_final_w_e} 
and use \eqref{eq_final_w_c30}, \eqref{eq_final_w_beta} to
obtain
\begin{align}
& \log\left(
\frac{ 11950 \cdot c^3 \sqrt{8c} \cdot 4^5 }
{ \beta \cdot\log(c)} \cdot
\exp\left[ -\frac{\pi \cdot \beta \cdot \log(c)}{8} 
\right]
\right) = \nonumber \\
& \log\left( \frac{11950 \sqrt{8} \cdot 4^5}{\beta \cdot \log(c)} \right) +
\left(\frac{7}{2} - \frac{\pi \cdot \beta}{8}\right) \cdot \log(c)
 < -10.
\label{eq_final_w_f}
\end{align}
We combine \eqref{eq_final_w_e} with \eqref{eq_final_w_f} to conclude that
\begin{align}
|\lambda_n| \cdot \left(
130 \sqrt[4]{\chi_n} + \frac{4\sqrt{2}\left(\chi_n\right)^{7/4}}{\chi_n-c^2}
\right) < e^{-10}.
\label{eq_final_w_g}
\end{align}
We combine \eqref{eq_final_w_c30}, \eqref{eq_final_w_beta},
\eqref{eq_final_w_g} to conclude that
\begin{align}
|\lambda_n| < e^{-16}.
\label{eq_final_w_h}
\end{align}
It follows from \eqref{eq_final_w_h} that
\begin{align}
24 \cdot |\lambda_n| \cdot \log\frac{1}{|\lambda_n|} < 
24 \cdot 16 \cdot e^{-16} < e^{-10}.
\label{eq_final_w_i}
\end{align}
Now \eqref{eq_final_w_thm} follows from the combination
of \eqref{eq_final_w_g} and \eqref{eq_final_w_i}.
\end{proof}
%%%%%%%%%%%%%%%%%%%%%%%%%
%%%%
\begin{thm}
Suppose that $c>0$ is a positive real number, and that
\begin{align}
c > 30.
\label{eq_positive_w_c30}
\end{align}
Suppose also that $n>0$ is a positive odd integer, and that
\begin{align}
n > \frac{2c}{\pi} + 5 \cdot \log(c) \cdot \log\left(\frac{c}{2}\right).
\label{eq_positive_w_n}
\end{align}
Suppose furthermore that $W_1, \dots, W_n$ are defined,
via \eqref{eq_quad_w}
of Definition~\ref{def_quad}. 
Then, for all integer $j=1,\dots,n$,
\begin{align}
W_j > 0.
\label{eq_positive_w_thm}
\end{align}
\label{thm_positive_w}
\end{thm}
\begin{proof}
Suppose first, by contradiction, that
\begin{align}
W_{(n+1)/2} \leq 2 \cdot |\lambda_n| \cdot \sqrt{2n}.
\label{eq_positive_w_a}
\end{align}
Then we combine \eqref{eq_positive_w_c30},
\eqref{eq_positive_w_n}, \eqref{eq_positive_w_a} with
Theorems~\ref{thm_pos_w},~\ref{thm_sum_w} to conclude that
\begin{align}
\frac{4 \sqrt{2} \cdot \left(\chi_n\right)^{7/4} }{\chi_n-c^2} \cdot
|\lambda_n| & \; \geq
W_1 + \dots + W_n \nonumber \\
& \; >
2 - |\lambda_n| \cdot \left(24 \cdot \log \frac{1}{|\lambda_n|} +
130 \cdot \sqrt[4]{\chi_n} \right),
\label{eq_positive_w_b}
\end{align}
in contradiction to Theorem~\ref{thm_final_w}. Therefore,
\begin{align}
W_{(n+1)/2} > 2 \cdot |\lambda_n| \cdot \sqrt{2n}.
\label{eq_positive_w_c}
\end{align}
We combine \eqref{eq_positive_w_c} with
Theorem~\ref{thm_wj_nice} and \eqref{eq_pos_w_b} in the proof
of Theorem~\ref{thm_pos_w} to obtain,
for every $j=1,\dots,n$,
\begin{align}
\frac{ \left(\psi_n'(t_j)\right)^2 \cdot (1-t_j^2) }
     { \left(\psi_n'(0)\right)^2 } \cdot W_j & \; =
W_{(n+1)/2} + \frac{ic\lambda_n}{\psi_n'(0)} \int_0^{t_j} \psi_n(t) \; dt 
\nonumber \\
& \; > 2 \cdot |\lambda_n| \cdot \sqrt{2n} - 
 \left| \frac{c\lambda_n}{\psi_n'(0)} \int_0^{t_j} \psi_n(t) \; dt \right|
 > 0,
\label{eq_positive_w_d}
\end{align}
where $t_1, \dots, t_n$ are defined via \eqref{eq_quad_t}
in Definition~\ref{def_quad}.
Now \eqref{eq_positive_w_thm} follows directly from 
the combination of
\eqref{eq_positive_w_d}
and \eqref{eq_pos_w_b} in the proof of Theorem~\ref{thm_pos_w}.
\end{proof}
\begin{remark}
The conclusion of Theorem~\ref{thm_positive_w} holds
for even integers $n$ as well. The proof of this fact is
similar to that of Theorem~\ref{thm_positive_w}, and 
is based on
Theorems~\ref{thm_sum_w},~\ref{thm_final_w} and the obvious
modifications of Theorem~\ref{thm_wj_nice},~\ref{thm_pos_w}.
\label{rem_w_even_n}
\end{remark}
\begin{remark}
Extensive numerical experiments 
(see e.g. Table~\ref{t:test96} and Figure~\ref{fig:test96})
seem to indicate that the assumption
\eqref{eq_positive_w_n} is unnecessary. In other words,
the weights $W_1, \dots, W_n$ are always positive,
even for small values of $n$.
\label{rem_w_always_pos}
\end{remark}
\begin{remark}
It follows from Theorem~\ref{thm_wj_nice} that, if $1 \leq j,k \leq n$
are integers, then
\begin{align}
\left( \psi_n'(t_j) \right)^2 \cdot (1-t_j^2) \cdot W_j =
\left( \psi_n'(t_k) \right)^2 \cdot (1-t_k^2) \cdot W_k +
O\left( |\lambda_n| \right)
\label{eq_w_approx}
\end{align}
(see also Experiment 15 in Section~\ref{sec_exp15}).
We observe that for $c=0$
the quadrature introduced in Definition~\ref{def_quad} 
is the well known Gaussian quadrature,
whose nodes are the roots $t_1,\dots,t_n$ of the Legendre polynomial $P_n$
(see Section~\ref{sec_legendre}), and whose weights are defined via
the formula
\begin{align}
W_j = \frac{ 2 }{ P_n'(t_j)^2 \brk{1 - t_j^2} }
\label{eq_gauss_weights}
\end{align}
(see e.g. \cite{Abramovitz}, Section 25.4). 
Thus, \eqref{eq_w_approx} is not
surprising. 
\label{rem_w_approx}
\end{remark}

%%%%%%%%%%%%%%%%%%%%%%%%%%%%%%%%%%%%%%%%%%%%%
%%%%%%%%%%%%%%%%%%%%%%%%%%%%%%%%%%%%%%%%%%%%%
\section{Numerical Algorithms}
\label{sec_numerical}
In this section, we describe several numerical algorithms
for the evaluation of the PSWFs, some related quantities,
and 
the nodes and weights
of the quadrature, defined in Definition~\ref{def_quad} in
Section~\ref{sec_quad}.
Throughout this section, the band limit $c > 0$ is a real number,
and the prolate index $n \geq 0$ is
a non-negative integer.

%%%%%%%%%%%%%%%%%%%%%%%%%%%%%%%%%%%%%%%%%%%%%%%%%%%%%
\subsection{Evaluation of $\chi_n$ and 
$\psi_n(x)$, $\psi_n'(x)$ for $-1 \leq x \leq 1$}
\label{sec_evaluate_beta}

The use of the expansion of $\psi_n$ into a Legendre series
(see \eqref{eq_num_leg_exp} in Section~\ref{sec_legendre})
for the evaluation of $\psi_n$
in the interval $(-1,1)$ goes back at least to 
the classical Bouwkamp algorithm (see \cite{Bouwkamp}).
More specifically, the coefficients 
$\beta_0^{(n)}, \beta_1^{(n)}, \dots$
of the Legendre expansion are precomputed first 
(see \eqref{eq_num_leg_beta_knc},
\eqref{eq_num_leg_alpha_knc} in Section~\ref{sec_legendre}).
These coefficients decay superalgebraically; in particular,
relatively few terms of the infinite sum \eqref{eq_num_leg_exp}
are required to evaluate $\psi_n$ to essentially
machine precision (see Section~\ref{sec_legendre},
in particular Theorem~\ref{thm_tridiagonal} and Remark~\ref{rem_tridiagonal},
and also \cite{RokhlinXiaoProlate} for more details).

Suppose that $n \geq 0$, and we are interested to evaluate
the coefficients $\beta^{(m}_0, \beta^{(m)}_1, \dots$ of the Legendre
expansion of $\psi_m$, for every
integer $0 \leq m \leq n$. This can be achieved by solving two 
$N \times N$ symmetric
tridiagonal eigenproblems, where $N$ is of order $n$
(see Theorem~\ref{thm_tridiagonal} and Remark~\ref{rem_tridiagonal}
in Section~\ref{sec_legendre}, and also
\cite{RokhlinXiaoProlate} for more details about this algorithm). 
In addition, this algorithm evaluates $\chi_0, \dots, \chi_n$.
Once this precomputation 
is done, for every integer $0 \leq m \leq n$ and for every 
real $-1 < x < 1$, we can evaluate $\psi_m(x)$ in $O(n)$ operations,
by computing the sum \eqref{eq_num_leg_exp}.

Suppose, on the other hand, that we are interested in 
a single PSWF only (as opposed to all the first $n$ PSWFs).
Obviously, we can use the algorithm mentioned above; however,
its cost is $O(n^2)$ operations (see Remark~\ref{rem_tridiagonal}).
In the rest of this subsection, we describe an algorithm
for the evaluation of $\beta_0^{(n)}, \beta_1^{(n)}, \dots$ and $\chi_n$,
whose cost is only $O(n)$ operations.

This algorithm is also based on Theorem~\ref{thm_tridiagonal}
in Section~\ref{sec_legendre}. It consists of two principal steps.
First, we compute a low-accuracy approximation $\tilde{\chi}_n$
of $\chi_n$, by means of Sturm's bisection
(see Section~\ref{sec_sturm}, \eqref{eq_a_even_eig},
\eqref{eq_a_odd_eig} and Remark~\ref{rem_tridiagonal}
in Section~\ref{sec_legendre},
and also \cite{Wilkinson2}).
Second, we compute $\chi_n$ and $\beta^{(n)}$, defined
via \eqref{eq_beta_n} in Section~\ref{sec_legendre},
by means of inverse power method (see Section~\ref{sec_power},
and also \cite{Wilkinson}, \cite{Dahlquist}).
The inverse power method requires an initial approximation to
both the eigenvalue and the eigenvector;
for this purpose we use, respectively,
$\tilde{\chi}_n$ 
and a random vector of unit length.

Below is a more detailed description of these two steps.
\paragraph{Step 1 (initial approximation $\tilde{\chi}_n$ of $\chi_n$).}
Suppose that the infinite symmetric tridiagonal matrices $A^{even}$
and $A^{odd}$ are defined, respectively, via 
\eqref{eq_a_even}, \eqref{eq_a_odd} in Section~\ref{sec_legendre}.
Suppose also that $A^{(n)}$ is the $N \times N$ upper left square
submatrix of $A^{even}$, if $n$ is even, or of $A^{odd}$, if $n$ is odd. \\
{\bf Comment.} $N$ is an integer of order $n$ 
(see Remark~\ref{rem_tridiagonal}). The choice
\begin{align}
N = 1.1 \cdot c + n + 1000
\label{eq_n_choice}
\end{align}
is sufficient for all practical purposes.
\begin{itemize}
\item use Theorems~\ref{thm_n_and_khi}, \ref{thm_khi_crude}
and \ref{thm_n_khi_simple} in Section~\ref{sec_pswf}
to choose real numbers $x_0 < y_0$ such that
\begin{align}
x_0 < \chi_n < y_0.
\label{eq_khi_brackets}
\end{align}
{\bf Comment.} For a more detailed discussion of lower and upper 
bounds
on $\chi_n$, see, for example, \cite{Report}, \cite{ReportArxiv}.
See also Remark~\ref{rem_sturm_cost} below.
\item use Sturm's bisection (see Section~\ref{sec_sturm}) with
initial values $x_0, y_0$ to compute $\tilde{\chi}_n$.
On each iteration of Sturm's bisection, the Sturm sequence
(see Theorem~\ref{thm_sturm}) is computed based on
the matrix $A^{(n)}$ (see above). \\
{\bf Comment.} We only require that $\tilde{\chi}_n$ be
a low-order approximation to $\chi_n$ in the following sense:
$\tilde{\chi}_n$ is closer to $\chi_n$ than to any 
$\chi_k$ with $k \neq n$.
\end{itemize}
\begin{remark}
The use of Sturm's bisection as a tool to compute the eigenvalues of
a symmetric tridiagonal matrix goes back at least to 
\cite{Wilkinson2}; in the context of PSWFs, it seems to appear
first in \cite{Hodge}.
\label{rem_sturm_first}
\end{remark}

The cost analysis of Step 1
relies on the following observation.
This observation is based on 
Theorems~\ref{thm_prolate_ode},
\ref{thm_n_and_khi},
\ref{thm_khi_crude},
\ref{thm_n_khi_simple},
\ref{thm_khi_1},
\ref{thm_khi_2} in Section~\ref{sec_pswf},
as well as on extensive numerical experiments and asymptotic expansions
(see, for example,
\cite{RokhlinXiaoProlate},
\cite{RokhlinXiaoApprox},
\cite{SlepianAsymptotic},
\cite{Report},
\cite{ReportArxiv}).

{\bf Observation 1.} Suppose that $n \geq 0$ is an integer.

If $0 \leq n < 2c/\pi$, then (e.g. it seems $0 \leq \chi_0 \leq c$)
\begin{align}
\chi_{n+1} - \chi_n = O(c).
\label{eq_delta_khi_small}
\end{align} 

If $n > 2c/\pi$, then
\begin{align}
\chi_{n+1} - \chi_n = O(n).
\label{eq_delta_khi_large}
\end{align}
\begin{remark}
If $0 \leq n < 2c/\pi$, then we combine Theorems~\ref{thm_n_and_khi},
\ref{thm_khi_crude} in Section~\ref{sec_pswf} 
to obtain
\begin{align}
n \cdot(n+1) < \chi_n < c^2.
\label{eq_khi_small_bounds}
\end{align}
We combine \eqref{eq_delta_khi_small}, \eqref{eq_khi_small_bounds}
and Corollary~\ref{cor_sturm} in Section~\ref{sec_sturm}
to conclude that, in this case, the cost of Step 1 is
$O(n \cdot \log(c))$ operations.
If, on the other hand, $n > 2c/\pi$, then we combine
Theorems~\ref{thm_n_and_khi}, \ref{thm_n_khi_simple},
Corollary~\ref{cor_sturm} in Section~\ref{sec_sturm} and
\eqref{eq_delta_khi_large} to conclude that, in this case,
the cost of Step 1 is $O(n)$ operations.
\label{rem_sturm_cost}
\end{remark}

\paragraph{Step 2 (evaluation of $\chi_n$ and $\beta^{(n)}$).}
Suppose that $\tilde{\chi}_n$ is an approximation to $\chi_n$,
computed in Step 1 (in the sense that $\tilde{\chi}_n$ is closer
to $\chi_n$ than to any other eigenvalue $\chi_k$).
Suppose also that $N$ is that of
Remark~\ref{rem_tridiagonal} in Section~\ref{sec_legendre}
(see also Step 1 above, and, in particular, \eqref{eq_n_choice}),
and that $\beta^{(n)} \in \Rc^N$ is defined via
\eqref{eq_beta_n} in Section~\ref{sec_legendre}.
\begin{itemize}
\item generate a unit length random vector $\tilde{\beta} \in \Rc^N$. \\
{\bf Comment.}
We use $\tilde{\chi}_n$ and $\tilde{\beta}$ as initial approximations
to the eigenvalue $\chi_n$ and the corresponding eigenvector, respectively,
for the inverse power method (see Section~\ref{sec_power}).
\item conduct inverse power method iterations until $\chi_n$ is evaluated
to machine precision. The corresponding unit eigenvector
is denoted by $\hat{\beta}^{(n)}$.\\
{\bf Comment.} Each iterations costs $O(n)$ operations,
and only $O(1)$ iterations are required (see Section~\ref{sec_power}).
In practice, the number of iterations is always less than 10.
\item conduct additional $K$ iterations of inverse power method,
until the convergence of the first coordinate of $\hat{\beta}^{(n)}$. \\
{\bf Comment.} Both analysis and numerical experiments 
(to be reported at a later date) suggest that
\begin{align}
K = 1 + 
\text{ceil}\left(
\frac{\log\left( \left|\beta^{(n)}_0\right| + \left|\beta^{(n)}_1\right| 
 \right)}
{\log\left( \varepsilon \right)}
\right),
\label{eq_k_power}
\end{align}
where $\varepsilon$ is the machine precision (e.g. 
$\varepsilon \approx \mbox{\text{\rm{1D-16}}}$
for double precision calculations),
and $\text{ceil}(a)$ is the minimal integer greater than $a$,
for a real number $a$. For example, if 
$|\beta^{(n)}_0| \approx \mbox{\text{\rm{1D-99}}}$, and
$\varepsilon \approx \mbox{\text{\rm{1D-16}}}$,
then
$K = 8$.
In practice, $K$ does not to 
be known in advance; rather, we iterate until convergence.
\end{itemize}
\begin{remark}
The cost of Step 2 is $O(n)$ operations.
\label{rem_inverse_cost}
\end{remark}

\begin{remark}
It is a well known fact (see e.g. \cite{Wilkinson}, \cite{Dahlquist})
that $\chi_n$ is evaluated to essentially machine precision by
the inverse power method. In other words, suppose that $\varepsilon$
is the machine accuracy (e.g. 
$\varepsilon \approx \mbox{\text{\rm{1D-16}}}$ for
double precision calculations); then, $\chi_n$ is evaluated with
{\bf relative} accuracy $\varepsilon$.
In addition, $\hat{\beta}^{(n)}$ approximates
$\beta^{(n)}$
with relative accuracy $\varepsilon$. However, this means
that a single {\bf coordinate} of $\beta^{(n)}$ is only guaranteed to be
evaluated with
{\bf absolute} accuracy $\varepsilon$.
More specifically, for every integer $k=0,\dots,N$,
\begin{align}
\left| \frac{\beta^{(n)}_k - \hat{\beta}^{(n)}_k}{\beta^{(n)}_k} \right|
\leq
\frac{\varepsilon}{\left| \beta^{(n)}_k \right|}.
\label{eq_beta_absolute}
\end{align}
\label{rem_beta_absolute}
\end{remark}
We make the following
observation from Remark~\ref{rem_beta_absolute}.
If we use $\hat{\beta}^{(n)}$ to evaluate Legendre series
(see \eqref{eq_num_leg_exp} in Section~\ref{sec_legendre},
and also \eqref{eq_evaluate_psi}, \eqref{eq_evaluate_dpsi} below),
the result will be obtained with high accuracy. On the other hand,
the small {\it coordinates} of $\beta^{(n)}$ are only guaranteed
to be computed with low accuracy. In particular, 
due to \eqref{eq_beta_absolute}, if, for example, 
$|\beta^{(n)}_k| \leq \varepsilon/10$ for some $k$, 
then, apriori, we do not
expect $\hat{\beta}^{(n)}_k$ to
coincide with $\beta^{(n)}_k$ in any digit at all!

The following conjecture states that the situation is much better
than Remark~\ref{rem_beta_absolute} seems to suggest.
This conjecture has been confirmed by both some preliminary analysis
(see e.g. \cite{Report2}, \cite{Report2Arxiv}) 
and extensive numerical experiments.
The matter is a subject of ongoing research; the results and proofs
will be published at a later date.
\begin{conjecture}
The {\bf coordinates} of $\beta^{(n)}$ are evaluated with high
{\bf relative} accuracy. More specifically, for every $1 \leq k \leq N$,
\begin{align}
\left| \frac{\beta^{(n)}_k - \hat{\beta}^{(n)}_k}{\beta^{(n)}_k} \right|
\leq
\varepsilon \cdot \log(\sqrt{c}),
\label{eq_beta_relative}
\end{align}
where $\hat{\beta}^{(n)}$ is the numerical approximation to $\beta^{(n)}$,
computed in Step 2,
and $\varepsilon$ is the machine accuracy
(e.g. 
$\varepsilon \approx \mbox{\text{\rm{1D-16}}}$ for
double precision calculations).
\label{conj_beta_relative}
\end{conjecture}
In particular, Conjecture~\ref{conj_beta_relative} implies that,
no matter how small $\beta^{(n)}_k$ is, it coincides with
$\hat{\beta}^{(n)}_k$ in all but the last 
$\log_{10}\left( \sqrt{c} \right)$ decimal digits.

\paragraph{Evaluation of $\psi_n(x)$, $\psi_n'(x)$ for $-1<x<1$,
given $\chi_n$ and $\beta^{(n)}_0, \beta^{(n)}_1, \dots$}
Suppose $\chi_n$ and the coefficients 
$\beta^{(n)}_0, \beta^{(n)}_1, \dots$ of the Legendre expansion
of $\psi_n$,
defined via
\eqref{eq_num_leg_beta_knc},
in Section~\ref{sec_legendre}, 
have already been evaluated.
Suppose also, that the integer $N$ is that of Steps 1,2 above
(see, for example, \eqref{eq_n_choice}).

For any real $-1 < x < 1$, 
evaluate $\psi_n(x)$ via the formula
\begin{align}
\psi_n(x) = \sum_{k=0}^{2N} P_k(x) \cdot \alpha^{(n)}_k
          = \sum_{k=0}^{2N} P_k(x) \cdot \beta^{(n)}_k \cdot \sqrt{k+1/2}.
\label{eq_evaluate_psi}
\end{align} 

Also, we evaluate $\psi_n'(x)$ via the formula
\begin{align}
\psi_n(x) = \sum_{k=1}^{2N} P_k'(x) \cdot \alpha^{(n)}_k
          = \sum_{k=0}^{2N} P_k'(x) \cdot \beta^{(n)}_k \cdot \sqrt{k+1/2}.
\label{eq_evaluate_dpsi}
\end{align} 
\begin{remark}
The cost of the evaluation of $\chi_n$ and 
$\beta^{(n)}_0, \beta^{(n)}_1, \dots$ via Steps 1,2 is $O(n)$ operations
(see Remarks~\ref{rem_sturm_cost}, \ref{rem_inverse_cost} above).
Once this precomputation has been done, the cost of each
subsequent evaluation of $\psi_n(x)$, $\psi_n'(x)$, for any real
$-1<x<1$, is $O(n)$ operations,
according to \eqref{eq_evaluate_psi},
\eqref{eq_evaluate_dpsi} and Remark~\ref{rem_legendre_evaluate}
in Section~\ref{sec_legendre}.
\label{rem_total_cost}
\end{remark}

\subsection{Evaluation of $\lambda_n$}
\label{sec_evaluate_lambda}
Suppose that the coefficients $\beta_0^{(n)}, \beta_1^{(n)}, \dots$
of the Legendre expansion of $\psi_n$ 
(see \eqref{eq_num_leg_beta_knc} in Section~\ref{sec_legendre})
as well as $\psi_n(0)$, $\psi_n'(0)$ have already been evaluated
by the algorithm of Section~\ref{sec_evaluate_beta}. 
If $n$ is even, we 
compute $\lambda_n$ 
via the formula
\begin{align}
\lambda_n = \frac{1}{\psi_n(0)} \int_{-1}^1 \psi_n(t) \; dt = 
            \frac{2\alpha_0^{(n)}}{\psi_n(0)} =
            \frac{\beta_0^{(n)}  \sqrt{2}}{\psi_n(0)}.
\label{eq_num_lambda_even}
\end{align}
If $n$ is odd, we compute $\lambda_n$ via the formula
\begin{align}
\lambda_n = \frac{ic}{\psi_n'(0)} \int_{-1}^1 t \cdot \psi_n(t) \; dt =
            \frac{2}{3} \cdot \frac{ic \alpha_1^{(n)}}{ \psi_n'(0)} =
            \sqrt{\frac{2}{3}} \cdot \frac{ic \beta_1^{(n)}}{ \psi_n'(0)}
\label{eq_num_lambda_odd}
\end{align}
(see \eqref{eq_prolate_integral} in Section~\ref{sec_pswf}
and \eqref{eq_legendre_pol_0_1},
\eqref{eq_legendre_normalized},
\eqref{eq_num_leg_beta_knc},
\eqref{eq_num_leg_alpha_knc} in Section~\ref{sec_legendre}).

{\bf Observation.} According to \eqref{eq_num_lambda_even},
\eqref{eq_num_lambda_odd}, the eigenvalue $\lambda_n$ is
evaluated in $O(1)$ operations as a by-product of 
Steps 1,2 of
the algorithm of 
Section~\ref{sec_evaluate_beta}
(the cost of these steps is $O(n)$ operations, due
to Remarks~\ref{rem_sturm_cost}, \ref{rem_inverse_cost}).
Obviously, $\lambda_n$ and $\beta^{(n)}_0$, $\beta^{(n)}_1$
are evaluated to the same relative accuracy. In particular,
even though $|\lambda_n|$ can be extremely small,
$\lambda_n$ is evaluated with fairly high precision 
(see Conjecture~\ref{conj_beta_relative} in
Section~\ref{sec_evaluate_beta}).

%%%%%%%%%%%%%%%%%%%%%%%%%%%%%%%%%%%%%%%%%%%%%%%
\subsection{Evaluation of the Quadrature Nodes}
\label{sec_evaluate_nodes}
Due to Definition~\ref{def_quad} in Section~\ref{sec_quad},
the $n$ quadrature nodes $t_1, \dots, t_n$ are precisely
the roots of $\psi_n$ in $(-1,1)$. 
In this subsection, we describe a numerical algorithm for
the evaluation of the quadrature nodes.
Since $\psi_n$ is symmetric
about the origin (see Theorem~\ref{thm_pswf_main}
in Section~\ref{sec_pswf}), it suffices to evaluate the roots
of $\psi_n$ in the interval $(0,1)$.

To evaluate the quadrature nodes, we use the fast algorithm
for the calculation of the roots of special functions,
described in \cite{Glaser}. This algorithm is based
on Pr\"ufer transformation (see Section~\ref{sec_prufer}),
Runge-Kutta method (see Section~\ref{sec_runge_kutta})
and Taylor's method (see Section~\ref{sec_taylor}).
It computes all the roots of $\psi_n$ in $(-1,1)$ in only $O(n)$
operations.

A short outline of the principal steps of the algorithm is provided below.
For a more detailed description of the algorithm and its properties,
the reader is referred to \cite{Glaser}.

The following observation is a direct consequence
of Theorem~\ref{thm_prufer_old}
in Section~\ref{sec_prufer}.

{\bf Observation 1.} Suppose that the function 
$\theta: \left[t_1,t_n\right] \to \Rc$ is defined via
\eqref{eq_prufer_theta_old} in Theorem~\ref{thm_prufer_old}
in Section~\ref{sec_prufer}. Suppose also that the function
$s: \left[\pi/2, \pi\cdot(n-1/2)\right] \to \left[-t_n,t_n\right]$
is the inverse of $\theta$. Then, $s$ is well defined, monotonically
increasing and continuously differentiable. Moreover, for all
real
$\pi/2 < \eta < \pi\cdot(n-1/2)$, 
\begin{align}
s'(\eta) = \frac{1}{f\left(s(\eta)\right) - 
 v\left(s(\eta)\right)\cdot\sin(2\eta)},
\label{eq_ds_ode}
\end{align}
where the functions $f,v$ are defined, respectively,
via \eqref{eq_jan_f}, \eqref{eq_jan_v} in Section~\ref{sec_prufer}.
In addition, for every integer $i=1,\dots,n$,
\begin{align}
s\left( \left(i-\frac{1}{2}\right) \cdot \pi \right) = t_i,
\label{eq_s_at_ti}
\end{align}
and also
\begin{align}
s\left( \frac{\pi n}{2} \right) = 0.
\label{eq_s_at_0}
\end{align}

Suppose now that $t_{\min}$ is the minimal root of $\psi_n$ in
$[0,1)$. 

\paragraph{Step 1 (evaluation of $t_{\min}$).}
If $n$ is odd, then
\begin{align}
t_{\min} = t_{(n+1)/2} = 0,
\label{eq_t_min_odd}
\end{align}
and this step of the algorithm is trivial.
On the other hand, if $n$ is even, we observe that
\begin{align}
t_{\min} = t_{(n+2)/2} > 0.
\label{eq_t_min_even}
\end{align}
We numerically solve the ODE \eqref{eq_ds_ode} with the initial condition
\eqref{eq_s_at_0} in the interval $\left[\pi n/2, \pi\cdot(n+1)/2\right]$,
by using 20 steps of Runge-Kutta method (see Section~\ref{sec_runge_kutta}).
The rightmost value $\tilde{t}_{\min}$ of the solution is 
a low-order approximation of $t_{\min}$ (see
\eqref{eq_s_at_ti}, \eqref{eq_t_min_even}).

We compute $t_{\min}$ via Newton's method (see Section~\ref{sec_newton}), using
$\tilde{t}_{\min}$ as the initial approximation to $t_{\min}$.
On each Newton iteration, we evaluate $\psi_n$ and $\psi_n'$
by using the algorithm of Section~\ref{sec_evaluate_beta}.

{\bf Observation 2.} The point $\tilde{t}_{\min}$ approximates
$t_{\min}$ to roughly three-four decimal digits. Subsequently,
only several Newton iterations are required to obtain $t_{\min}$
to essentially machine precision (see \cite{Glaser} for more details).
Thus, the cost of Step 1 is $O(n)$ operations.

\paragraph{Step 2 (evaluation of $\psi_n'(t_{\min})$).}
We evaluate $\psi_n'(t_{\min})$ by using the algorithm
of Section~\ref{sec_evaluate_beta}. 

{\bf Observation 3.} The cost of Step 2 is $O(n)$ operations
(see Remark~\ref{rem_total_cost} in Section~\ref{sec_evaluate_beta}).

The remaining roots of $\psi_n$ in $(t_{\min}, 1)$ are computed 
iteratively, as follows. Suppose that $n/2 < j < n$ is an integer,
and both $t_j$ and $\psi_n'(t_j)$ have already been evaluated.

\paragraph{Step 3 (evaluation of $t_{j+1}$ and $\psi_n'(t_{j+1})$,
given $t_j$ and $\psi_n'(t_j)$).}

\begin{itemize}
\item use the recurrence relation \eqref{eq_dpsi_reck}
(see Theorem~\ref{thm_dpsi_reck} in Section~\ref{sec_pswf})
to evaluate $\psi_n^{(2)}(t_j), \dots, \psi_n^{(30)}(t_j)$.
\item use 20 steps of Runge-Kutta method (see Section~\ref{sec_runge_kutta}),
to solve the ODE \eqref{eq_ds_ode} with the initial condition
\begin{align}
s\left( \pi \cdot \left(j-\frac{1}{2}\right) \right) = t_j
\label{eq_s_eq_tj}
\end{align}
in the interval $\left[\pi \cdot(j- 1/2), \pi \cdot(j+ 1/2) \right]$,
by using 20 steps of Runge-Kutta method (see Section~\ref{sec_runge_kutta}).
The rightmost value $\tilde{t}_{j+1}$ of the solution is a low-order
approximation of $t_{j+1}$.
\item compute $t_{j+1}$ via Newton's method 
(see Section~\ref{sec_newton}), using
$\tilde{t}_{j+1}$ as the initial approximation to $t_{j+1}$.
On each Newton iteration, we evaluate $\psi_n$ and $\psi_n'$
by using Taylor's method (see Section~\ref{sec_taylor}).
The Taylor expansion of order 30 about $t_j$ is used, e.g.
\begin{align}
\psi_n(t) = \sum_{k=0}^{30} \frac{\psi_n^{(k)}(t_j)}{k!} \cdot (t-t_j)^k +
O\left( (t-t_j)^{k+1} \right).
\label{eq_t_taylor}
\end{align}
\item evaluate $\psi_n'(t_{j+1})$ by using Newton's method, i.e.
by computing the sum
\begin{align}
\sum_{k=0}^{29} \frac{\psi_n^{(k+1)}(t_j)}{k!} \cdot (t_{j+1}-t_j)^k.
\label{eq_dpsi_taylor}
\end{align}
\end{itemize}

{\bf Observation 4.} The point $\tilde{t}_{j+1}$ approximates
$t_{j+1}$ to roughly three-four decimal digits. Subsequently,
only several Newton iterations are required to obtain $t_{j+1}$
to essentially machine precision (see \cite{Glaser} for more details).
The cost of Step 3 is $O(1)$ operations. 

\paragraph{Step 4 (evaluation of $t_j$ and $\psi_n'(t_j)$ for 
all $j \leq n/2$).}
Step 3 is repeated iteratively, for every integer $n/2 < j < n$. 
To evaluate $t_j$ and $\psi_n'(t_j)$ for $-1 < t_j < 0$,
we use the symmetry
of $\psi_n$ about zero, established in Theorem~\ref{thm_pswf_main}
in Section~\ref{sec_pswf}. More specifically,
for every $1 \leq j \leq n/2$, we compute
\begin{align}
t_j = t_{n+1-j}
\label{eq_tj_symmetry}
\end{align} 
and
\begin{align}
\psi_n'(t_j) = (-1)^{n+1} \cdot \psi_n'(t_{n+1-j}).
\end{align}

\paragraph{Summary (evaluation of $t_j$ and $\psi_n'(t_j)$, for
all $j=1,\dots,n$).}
To summarize, to evaluate the roots of $\psi_n$ in $(-1,1)$
as well as $\psi_n'$ at these roots, we proceed as follows.
\begin{itemize}
\item run Step 1, to evaluate $t_{\min}$ (see \eqref{eq_t_min_odd},
\eqref{eq_t_min_even}). Cost: $O(n)$.
\item run Step 2, to evaluate $\psi_n'(t_{\min})$. Cost: $O(n)$.
\item for every integer $n/2 < j < n$, run Step 3. Cost: $O(n)$.
\item for every integer $1 \leq j \leq n/2$, run Step 4.
Cost: $O(n)$.
\end{itemize}
\begin{remark}
We observe that the algorithm of this subsection not only
computes the roots $t_1, \dots, t_n$ of $\psi_n$ in $(-1,1)$,
but also evaluates $\psi_n'$ at all these roots.
The total cost of the algorithm is $O(n)$ operations.
\label{rem_evaluate_nodes}
\end{remark}

%%%%%%%%%%%%%%%%%%%%%%%%%%%%%%%%%%%%%%%%%%%%%%%%%
\subsection{Evaluation of the Quadrature Weights}
\label{sec_evaluate_weights}
In this subsection, we describe an algorithm for the evaluation
of the weights $W_1,\dots,W_n$ of the quadrature,
defined in Definition~\ref{def_quad} in Section~\ref{sec_quad}.
The results of this subsection are 
illustrated in Table~\ref{t:test96} and in Figure~\ref{fig:test96}
(see Experiment 15 in Section~\ref{sec_exp15}).

Obviously, one way to compute $W_1, \dots, W_n$ is to evaluate
the integrals of $\varphi_1, \dots, \varphi_n$ numerically
(see Definition~\ref{def_quad}). However, each $\varphi_j$ has
$n-1$ zeros in $(-1,1)$, and this approach is unlikely to cost
less that $O(n^2)$ operations
(see also Section~\ref{sec_evaluate_beta}). In addition, each $\varphi_j$
has a singularity (albeit, removable) at $t_j$, which might be
a nuisance for numerical integration, especially if high precision
is required.

Below we describe two additional ways to evaluate the weights,
based on the results of Section~\ref{sec_weights}.
One of them, based on
Theorem~\ref{lem_tilde_phi} and Corollary~\ref{cor_tilde_phi_w}, 
is straightforward and accurate; however, its cost 
is $O(n^2)$ operations. The other way, 
based on Theorem~\ref{lem_tilde_phi_ode} and
Corollary~\ref{cor_dtilde_phi}, in addition to having high accuracy
and being easy to implement,
is also computationally efficient: its cost is only $O(n)$ operations.

We assume that the quadrature nodes $t_1, \dots, t_n$ as well as
$\psi_n'(t_1), \dots, \psi_n'(t_n)$ have already been computed
(by the algorithm of Section~\ref{sec_evaluate_nodes},
whose cost is $O(n)$ operations).

\paragraph{Algorithm 1: evaluation of $W_1, \dots, W_n$ in $O(n^2)$ operations.}
Suppose that the coefficients $\alpha_0^{(n)}, \dots, \alpha_{2N}^{(n)}$
of the Legendre expansion of $\psi_n$ 
(see \eqref{eq_num_leg_alpha_knc} in Section~\ref{sec_legendre})
have already been evaluated,
by the algorithm of Section~\ref{sec_evaluate_beta};
here $N$ is an integer of order $n$ (see 
\eqref{eq_n_choice} in Section~\ref{sec_evaluate_beta}).
We compute $W_j$
by evaluating the sum
\begin{align}
-\frac{2}{\psi_n'(t_j)} 
\sum_{k = 0}^{2N} \alpha_k^{(n)} Q_k(t_j),
\label{eq_wj_as_sum}
\end{align}
where $Q_0, Q_1, \dots$ are 
the Legendre functions of the second kind, defined 
in Section~\ref{sec_legendre}.

{\bf Observation 1.} The sum \eqref{eq_wj_as_sum} approximates
the corresponding infinite sum
to essentially machine precision,
due to the superexponential decay of $\alpha_k^{(n)}$ and
the high precision to which $Q_k(t_j)$ are evaluated
(see Sections~\ref{sec_legendre}, \ref{sec_evaluate_beta}, and also
\cite{Ryzhik},
\cite{RokhlinXiaoProlate},
\cite{Abramovitz}). In combination with
Theorem~\ref{lem_tilde_phi} and Corollary~\ref{cor_tilde_phi_w},
this implies that \eqref{eq_wj_as_sum} is an accurate formula
for the evaluation of $W_j$
(see also Experiment~15 in Section~\ref{sec_exp15}).

{\bf Observation 2.} For every integer $j$, we evaluate 
$Q_0(t_j), \dots, Q_{2N}(t_j)$ recursively, by using
\eqref{eq_legendre_fun_0_1}, \eqref{eq_legendre_fun_rec}
in Section~\ref{sec_legendre}, in $O(N)$ operations
(see Remark~\ref{rem_legendre_evaluate} in Section~\ref{sec_legendre}).
Since
$N = O(n)$ (see Section~\ref{sec_evaluate_beta}),
the overall cost of computing $W_1, \dots, W_n$ via \eqref{eq_wj_as_sum}
is $O(n^2)$ operations.

\paragraph{Algorithm 2: evaluation of $W_1, \dots, W_n$ in $O(n)$ operations.}
This algorithm consists of the following steps.

Suppose that $t_{\min}$ is the minimal root of $\psi_n$ in $[0,1)$.
In other words,
\begin{align}
t_{\min} = 
\begin{cases}
t_{(n+1)/2} & n \text{ is odd}, \\
t_{(n+2)/2} & n \text{ is even}.
\end{cases}
\label{eq_tmin_both}
\end{align}
Suppose also that the function $\tilde{\Phi}_n: (-1,1) \to \Rc$ is defined
via \eqref{eq_num_tilde_phi_def} in Theorem~\ref{lem_tilde_phi}
in Section~\ref{sec_weights}.
\paragraph{Step 1 (evaluation of 
$\tilde{\Phi}_n(t_{\min})$ and $\tilde{\Phi}_n'(t_{\min})$).}
Suppose that the coefficients $\alpha_0^{(n)}, \dots, \alpha_{2N}^{(n)}$
of the Legendre expansion of $\psi_n$ 
(see \eqref{eq_num_leg_alpha_knc} in Section~\ref{sec_legendre})
have already been evaluated
by the algorithm of Section~\ref{sec_evaluate_beta}.
We evaluate $\tilde{\Phi}_n(t_{\min})$ by computing the sum
\begin{align}
\sum_{k = 0}^{2N} \alpha_k^{(n)} Q_k(t_{\min}).
\label{eq_tilde_phi_as_sum}
\end{align}
Also, we evaluate $\tilde{\Phi}'_n(t_{\min})$ by computing the sum
\begin{align}
\sum_{k = 0}^{2N} \alpha_k^{(n)} Q_k'(t_{\min})
\label{eq_dtilde_phi_as_sum}
\end{align}
(see Algorithm 1 and Observations 1, 2 above, Theorem~\ref{lem_tilde_phi}
in Section~\ref{sec_weights} and Section~\ref{sec_legendre}).

{\bf Observation 3.} We evaluate
$Q_0'(t_{\min}), \dots, Q_{2N}'(t_{\min})$ recursively
(see Sections~\ref{sec_legendre}, \ref{sec_evaluate_beta}, and also
\cite{Ryzhik},
\cite{RokhlinXiaoProlate},
\cite{Abramovitz}). Thus both \eqref{eq_tilde_phi_as_sum} and
\eqref{eq_dtilde_phi_as_sum} approximate
$\tilde{\Phi}_n(t_{\min})$ and $\tilde{\Phi}_n'(t_{\min})$,
respectively,
to essentially machine precision, and are computed in $O(n)$ operations
(see Observations 1, 2 above and
Remark~\ref{rem_legendre_evaluate} in Section~\ref{sec_legendre}).

We evaluate $\tilde{\Phi}_n$ at all but the last four remaining roots
of $\psi_n$ in $[0,1)$ iteratively, as follows.
Suppose that $n/2 < j < n$ is an integer,
and both 
$\tilde{\Phi}_n(t_j)$ and $\tilde{\Phi}_n'(t_j)$
have already been evaluated.

\paragraph{Step 2 (evaluation of
$\tilde{\Phi}_n(t_{j+1})$ and $\tilde{\Phi}_n'(t_{j+1})$,
given $\tilde{\Phi}_n(t_j)$ and $\tilde{\Phi}_n'(t_j)$).}
\begin{itemize}
\item use the recurrence relation 
\eqref{eq_num_dtilde_phi_3}, \eqref{eq_dphi_reck}
(see Corollary~\ref{cor_dtilde_phi} in Section~\ref{sec_weights})
to evaluate $\tilde{\Phi}_n^{(2)}(t_j), \dots, \tilde{\Phi}_n^{(60)}(t_j)$.
\item evaluate $\tilde{\Phi}_n(t_{j+1})$ by using Newton's method, i.e.
by computing the sum
\begin{align}
\sum_{k=0}^{60} \frac{\psi_n^{(k)}(t_j)}{k!} \cdot (t_{j+1}-t_j)^k.
\label{eq_phi_newton}
\end{align}
\item evaluate $\tilde{\Phi}'_n(t_{j+1})$ by using Newton's method, i.e.
by computing the sum
\begin{align}
\sum_{k=0}^{59} \frac{\psi_n^{(k+1)}(t_j)}{k!} \cdot (t_{j+1}-t_j)^k.
\label{eq_dphi_newton}
\end{align}
\end{itemize}

{\bf Observation 4.} 
For each $j$, the cost of the evaluation of \eqref{eq_phi_newton},
\eqref{eq_dphi_newton} is $O(1)$ operations (i.e. does not depend on $n$). 
Also, 
\eqref{eq_phi_newton},
\eqref{eq_dphi_newton}
approximate, $\tilde{\Phi}_n(t_j)$ and $\tilde{\Phi}_n'(t_j)$,
respectively,
to essentially machine precision. For a detailed discussion
of the accuracy and stability of this step, the reader
is referred to \cite{Glaser}.

\paragraph{Step 3 (evaluation of
$\tilde{\Phi}_n(t_j)$ for $n-3 \leq j \leq n$).}
For $j=n-3, n-2, n-1, n$, we evaluate $\tilde{\Phi}_n(t_j)$
by computing the sum 
\begin{align}
\sum_{k = 0}^{2N} \alpha_k^{(n)} Q_k(t_j).
\label{eq_tilde_phi_j_as_sum}
\end{align}
(similar to \eqref{eq_tilde_phi_as_sum} in Step 1).
\begin{remark}
We compute $\tilde{\Phi}_n$ at the last four nodes via
\eqref{eq_tilde_phi_j_as_sum} rather than \eqref{eq_phi_newton},
since the accuracy of the latter deteriorates when $(1-t_j^2)$ 
becomes too small
(see \eqref{eq_dphi_reck} in Corollary~\ref{cor_dtilde_phi}).
Since this approach works in practice, is cheap in terms of
the number of operations and eliminates the above concern,
there was no need in a detailed analysis of the issue
(see also \cite{Glaser} for more details).
\label{rem_last_weights}
\end{remark}

\paragraph{Step 4 (evaluation of $\tilde{\Phi}_n(t_j)$ for
$1 \leq j \leq n/2$).}
Suppose that $1 \leq j \leq n/2$. We evaluate $\tilde{\Phi}_n(t_j)$
via the formula
\begin{align}
\tilde{\Phi}_n(t_j) = (-1)^{n+1} \cdot \tilde{\Phi}_n(t_{n+1-j})
\label{eq_tilde_phi_sym}
\end{align}
($\tilde{\Phi}$ is symmetric with respect to zero due to
the combination of 
Theorem~\ref{lem_tilde_phi} in Section~\ref{sec_weights}
and \eqref{eq_legendre_fun_rec} in Section~\ref{sec_legendre}).

\paragraph{Step 5 (evaluation of $W_1, \dots, W_n$).}
By performing Steps 1-4 of Algorithm 2,
we evaluate $\tilde{\Phi}_n$ at the roots $t_1, \dots, t_n$
of $\psi_n$ in $(-1,1)$. Now, for every $j=1,\dots,n$,
we evaluate $W_j$ via
\eqref{eq_tilde_phi_w}
of Corollary~\ref{cor_tilde_phi_w}
in Section~\ref{sec_weights}.

\begin{remark}
The overall cost of Steps 1-5 of Algorithm 2
is $O(n)$ operations.
\label{rem_weights_cost}
\end{remark}

%%%%%%%%%%%%%%%%%%%%%%%%%%%%%%%%%%%%%%%%%%%%%%%%%%%%%%%%%%%%%%%%%%
\subsection{Evaluation of $\psi_n$ and its roots outside $(-1,1)$}
\label{sec_outside}

The PSFWs provide a natural way to represent bandlimited functions
over the interval $(-1,1)$ (see Theorem~\ref{thm_pswf_main}
in Section~\ref{sec_pswf}). Therefore, even though each $\psi_n$
is defined (and holomorphic) in the whole complex plane,
in applications (construction of PSWFs, quadratures, interpolation etc.)
one is mostly interested in the properties of $\psi_n(t)$ 
for real $t$ inside $(-1,1)$
(see, for example, Section~\ref{sec_pswf},
\cite{RokhlinXiaoProlate}, \cite{Report}, \cite{ReportArxiv}, 
\cite{Report2}, \cite{Report2Arxiv}).

On the other hand, the properties of the quadrature rules studied in
this paper
(see Definition~\ref{def_quad} in Section~\ref{sec_quad})
depend, perhaps surprisingly, on the behavior of $\psi_n$
\emph{outside} the interval $(-1,1)$
(see Sections~\ref{sec_upper},~\ref{sec_one_over_psi},~\ref{sec_quad}).
Thus, while one is rarely interested in the evaluation of $\psi_n$
and related quantities outside $(-1,1)$ per se,
we do need such tools to illustrate our analysis
(see Section~\ref{sec_num_res} below).

The rest of this section is devoted to the description of numerical
algorithms for the evaluation of $\psi_n(x)$ and $\psi_n'(x)$ for
$x>1$, as well as the location of the roots of $\psi_n$ in $(1,\infty)$.
These algorithms were developed as auxiliary tools, and are not meant 
to be used in practical applications.

Throughout this subsection, we assume that $c>0$ is a positive
real number, and $n$ is a non-negative integer.

%%%%%%%%%%%%%%%%%%%%%%%%%%%%%%%%%%%%%%%%%%%%%%%%%%%
\subsubsection{Evaluation of $\psi_n(x)$ for $x>1$}
\label{sec_psi_outside}
To evaluate $\psi_n(x)$ for $x>1$, we use the 
integral equation
\eqref{eq_prolate_integral} in Section~\ref{sec_pswf}
(as opposed to using 
the Legendre series \eqref{eq_num_leg_exp} of Section~\ref{sec_legendre}
to evaluate $\psi_n(x)$ for $-1<x<1$). Namely,
we evaluate $\psi_n(x)$ via proceed as follows:
\begin{itemize}
\item Compute $\chi_n$ and the coefficients 
$\alpha^{(n)}_0, \alpha^{(n)}_1, \dots$
of the Legendre expansion of $\psi_n$ (see Section~\ref{sec_evaluate_beta}).
\item Compute $\lambda_n$ (see Section~\ref{sec_evaluate_lambda}).
\item Compute $\psi_n(x)$ via evaluating the integral
\begin{align}
\frac{1}{\lambda_n} \int_{-1}^1 \psi_n(t) \cdot e^{icxt} \; dt
\label{eq_eval_psi_outside}
\end{align}
numerically, by using $m=O(n)$ Gaussian quadrature nodes in
the interval $(-1,1)$.
\end{itemize}
We observe that the integrand in \eqref{eq_eval_psi_outside}
is oscillatory: $\psi_n$ has $n$ zeros in $(-1,1)$, and $e^{icxt}$
is periodic with period $(2\pi)/(cx)$. Moreover,
$\psi_n(x)$ itself is oscillatory with frequency
of order $n$ (unless $x$ is between
$1$ and $\sqrt{\chi_n}/c$, see Theorems~\ref{thm_four},~\ref{thm_spacing}
in Section~\ref{sec_first_order}).

Thus, we used a fairly large number of Gaussian nodes to evaluate
\eqref{eq_eval_psi_outside}. For example, for $c=100$ and $n\leq 100$ we used
the Gaussian quadrature of order $500$; for $c=1000$ and $n\leq 750$ we
used the Gaussian quadrature of order $3000$.

\begin{remark}
\label{rem_evaluate_psi}
For each of the $m$ Gaussian nodes $\tau_k$, we compute $\psi_n(\tau_k)$ via
evaluating the sum
\begin{align}
\sum_{j=0}^{2N} P_j(\tau_k) \cdot \alpha^{(n)}_j,
\label{eq_eval_psi_sum}
\end{align}
where $N$ is of order $n$ (see Section~\ref{sec_evaluate_beta}). Thus,
the resulting algorithm for the evaluation of $\psi_n(x)$ is fairly
expensive: its cost is $O(N \cdot n) = O(n^2)$ operations,
as opposed to $O(n)$ operations 
 to evaluate $\psi_n(x)$ for $-1 < x < 1$
(see Remark~\ref{rem_total_cost} in Section~\ref{sec_evaluate_beta}).
\end{remark}

%%%%%%%%%%%%%%%%%%%%%%%%%%%%%%%%%%%%%%%%%%%%%%%%%%%
\subsubsection{Evaluation of $\psi_n'(x)$ for $x>1$}
\label{sec_dpsi_outside}

We differentiate the identity \eqref{eq_prolate_integral}
in Section~\ref{sec_pswf} to obtain, for all complex $x$,
\begin{align}
\psi_n'(x) = 
\frac{ic}{\lambda_n} \int_{-1}^1 t \cdot \psi_n(t) \cdot e^{icxt} \; dt.
\label{eq_eval_dpsi}
\end{align}
We use \eqref{eq_eval_dpsi} to evaluate $\psi_n'(x)$ for $x>1$
in the same manner we use \eqref{eq_eval_psi_outside} to
evaluate $\psi_n(x)$ (see Section~\ref{sec_psi_outside}).
The resulting algorithm has the same cost as the one
of Section~\ref{sec_psi_outside} (see Remark~\ref{rem_evaluate_psi}).
%
%\begin{remark}
%\label{rem_evaluate_dpsi}
%We also use the finite differences calculation
%\begin{align}
%\psi_n'(x) = \frac{\psi_n(x+h)-\psi_n(x-h)}{2h} + O(h^2)
%\label{eq_dpsi_finite}
%\end{align}
%with $10^{-7} \leq h \leq 10^{-5}$ as a sanity test.
%\end{remark}

%%%%%%%%%%%%%%%%%%%%%%%%%%%%%%%%%%%%%%%%%%%%%%%%%%%
\subsubsection{Evaluation of the roots of $\psi_n$ in $(1,\infty)$}
\label{sec_roots_outside}

Suppose that $\chi_n > c^2$. Suppose also that $k \geq 1$ is an integer.
According to Theorem~\ref{thm_four} of Section~\ref{sec_first_order},
\begin{align}
\frac{\sqrt{\chi_n}}{c} = x_0 < x_1 < x_2 < \dots < x_k.
\label{eq_dpsi_k_outside}
\end{align}
where $x_1,\dots,x_k$ are
the $k$ minimal roots of $\psi_n$ in $(1,\infty)$.
We define the function $\theta: \left[x_0, x_k\right] \to \Rc$
via \eqref{eq_prufer_theta} in Theorem~\ref{thm_prufer}
of Section~\ref{sec_first_order}. Then, $\theta$ is monotonically increasing;
moreover,
\begin{align}
\theta(x_0) = -\frac{\pi}{2}, \quad
\theta(x_1) = \frac{\pi}{2}, \quad
\theta(x_k) = \pi \cdot \left( k - \frac{1}{2} \right).
\label{eq_theta_outside}
\end{align}
Also, $\theta$ satisfies the nonlinear first order ODE
\eqref{eq_prufer_theta_ode} (see Theorem~\ref{thm_prufer}).
Furthermore, 
for every integer $j = 0, 1, \dots, k-1$,
\begin{align}
x_{j+1} - x_j \approx \frac{\pi}{c}
\label{eq_deltax_outside}
\end{align}
(see Theorems~\ref{thm_x1_x0_good}, \ref{thm_spacing}
in Section~\ref{sec_oscillation} for a more precise statement).

Suppose now that $j$ is an integer between $0$ and $k-1$, and
$x_0, \dots, x_j$ have already been evaluated (note that to evaluate
the special point $x_0$ we only need to evaluate $\chi_n$, see
Section~\ref{sec_evaluate_beta}). We evaluate $x_{j+1}$ as follows.
\begin{itemize}
\item Define $h$ via the formula
\begin{align}
h = \frac{\pi}{100 c}.
\label{eq_h_outside}
\end{align}
\item Use Runge-Kutta method (see Section~\ref{sec_runge_kutta})
to evaluate $\theta(x_j + i \cdot h)$ numerically
(by solving the ODE \eqref{eq_prufer_theta_ode} with the initial
condition \eqref{eq_theta_outside}), for $i=1,2,3, \dots$. \\
{\bf Comment.} Due to \eqref{eq_deltax_outside}, $h$ defined via
\eqref{eq_h_outside} is a reasonable step size of the Runge-Kutta
ODE solver.
\item Stop when
\begin{align}
\theta(x_k+i\cdot h) < \pi\cdot\left(j+\frac{1}{2}\right) <
\theta(x_k+(i+1)\cdot h).
\label{eq_stop_outside}
\end{align}
\item Define $\tilde{x}_{j+1}$ via the formula
\begin{align}
\tilde{x}_{j+1} = x_k + \left(i + \frac{1}{2}\right) \cdot h,
\label{eq_tildex_outside}
\end{align}
where $i$ is as in \eqref{eq_stop_outside}. This is the initial
approximation of $x_{j+1}$. \\
{\bf Comment.} Due to \eqref{eq_deltax_outside}, \eqref{eq_h_outside},
we expect $\tilde{x}_{j+1}$ to approximate $x_{j+1}$ roughly to three-four
decimal digits.
\item Use Newton's method (see Section~\ref{sec_newton}) with the initial
point $\tilde{x}_{j+1}$ to evaluate $x_{j+1}$. \\
{\bf Comment.} For each Newton iteration, we evaluate $\psi_n(x)$, 
$\psi_n'(x)$ by using the algorithms of 
Sections~\ref{sec_psi_outside},~\ref{sec_dpsi_outside},
respectively.
\end{itemize}
\begin{remark}
\label{rem_roots_outside}
We observe that the algorithm of Section~\ref{sec_roots_outside}
is similar to that of Section~\ref{sec_evaluate_nodes}.
However, rather than solving the ODE for the inverse of 
$\theta$ (see \eqref{eq_ds_ode} in Section~\ref{sec_evaluate_nodes}), 
here we solve the ODE for $\theta$.
Also, rather than 
evaluating $\psi_n(x)$ and $\psi_n'(x)$ by Taylor's method
(see \eqref{eq_t_taylor}, \eqref{eq_dpsi_taylor} in
Section~\ref{sec_evaluate_nodes}), here we evaluate $\psi_n(x)$
and $\psi_n'(x)$ by using the algorithms of
Section~\ref{sec_psi_outside},~\ref{sec_dpsi_outside},
respectively. 
% In other words, speed and accuracy
% have been sacrificed for the sake of simplicity.
\end{remark}

\section{Numerical Results}
\label{sec_num_res}
This section has two principal purposes. First,
we illustrate the analysis of Section~\ref{sec_analytical}
by means of numerical examples. Second, we demonstrate
the performance of the algorithms presented in
Section~\ref{sec_numerical}. All the calculations were
implemented in FORTRAN (the Lahey 95 LINUX version).

In all the experiments, the principal numerical algorithms
of the paper, described in Sections~\ref{sec_evaluate_beta}--
\ref{sec_evaluate_lambda}, were run in double precision.
On the other hand, the auxiliary algorithms of Section~\ref{sec_outside}
(whose sole purpose is to illustrate the analysis)
were run in extended precision.

%%%%%%%%%%%%%%%%%%%%%%%%%%%%%%%%
\subsection{Properties of PSWFs}
In this subsection, we illustrate the analytical results
from Section~\ref{sec_oscillation}, 
Section~\ref{sec_growth} and Section~\ref{sec_one_over_psi}.

%%%%%%%%%%%%%%%%%%%%%%%%%%%%%%
\subsubsection{Illustration of Results from Section~\ref{sec_oscillation}}
\label{sec_ill_osc}
%%%%%%%%%%%%%%%%%%%%%%%%%%%%%%
\paragraph{Experiment 1.}
\label{sec_exp1}
In this experiment, we 
illustrate Theorem~\ref{thm_five} in Section~\ref{sec_pswf}
and Theorem~\ref{thm_four} in Section~\ref{sec_first_order}.
We proceed as follows.
We choose, more or less arbitrarily, the band
limit $c>0$ and the prolate index $n \geq 0$, and evaluate
$\psi_n(x)$ at 1000 equispaced points in the interval
$(-1.5, 1.5)$. To evaluate $\psi_n(x)$ for $-1 \leq x \leq 1$,
we use the algorithm of Section~\ref{sec_evaluate_beta}
(in double precision). To evaluate $\psi_n(x)$ for $|x|>1$,
we use the algorithm of Section~\ref{sec_psi_outside}
(in extended precision).

We display the results of the experiment
in Figures~\ref{fig:test75a}, \ref{fig:test75b},
corresponding to the choice $c=20$, $n=9$ and $c=20$, $n=14$,
respectively.
Each of these figures contains a plot of the corresponding $\psi_n$.

We observe that
the relations
\eqref{eq_all_tx_small} and \eqref{eq_all_tx_large} hold for
the functions in Figures~\ref{fig:test75a}, 
\ref{fig:test75b}, respectively. 
The inequality \eqref{eq_extremum_general} of
Theorem~\ref{thm_extrema} in Section~\ref{sec_pswf}
holds in both cases,
that is,
the absolute value of local extrema of $\psi_n(t)$ increases
as $t$ grows from $0$ to $1$. On the other hand,
\eqref{eq_extremum_special} holds only for Figure~\ref{fig:test75b}.
This is due to the fact that $\chi_9 < c^2$ and $\chi_{14} > c^2$
(see also Theorem~\ref{thm_n_and_khi} in Section~\ref{sec_pswf}).
Also, we observe that
the magnitude of the oscillations outside $(-1, 1)$ is roughly
inversely proportional to
$|\lambda_n|$. 

%%%%%%%%%%%%%%%%%%%%%
\begin{figure} [htbp]
\begin{center}
\includegraphics[width=12cm, bb=-130 50 760 770, clip=true]
{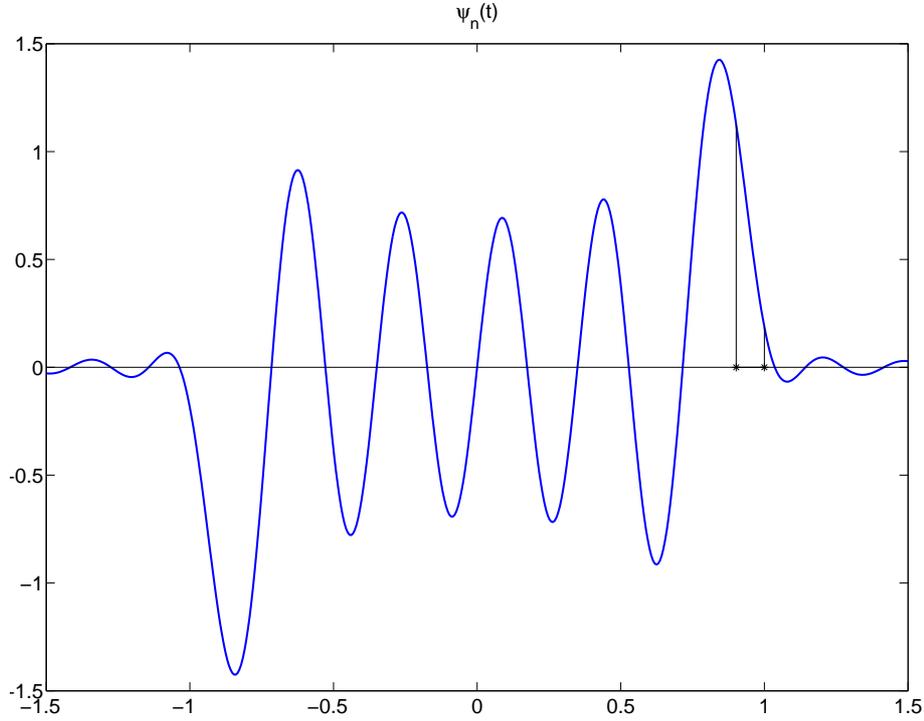}
%{\epsfxsize=450pt\epsffile{test75a.eps}}
\caption
{\it
The function $\psi_n(t)$ for $c = 20$ and $n = 9$.
Since $\chi_n \approx 325.42 < c^2$, the behavior is as asserted
in \eqref{eq_all_tx_small} of Theorem~\ref{thm_five}. The points
$\sqrt{\chi_n}/c \approx 0.90197$ and $1$ are marked with asterisks.
The eigenvalue $\abrk{\lambda_n} \approx 0.55978$ is relatively large,
and the oscillations of $\psi_n$ outside $(-1, 1)$ have
small magnitude. Compare to Figure~\ref{fig:test75b}.
Corresponds to Experiment 1.
}
\label{fig:test75a}
\end{center}
\end{figure}
%%%%%%%%%%%%
%%%%%%%%%%%%%%%%%%%%%
\begin{figure} [htbp]
\begin{center}
\includegraphics[width=12cm, bb=-130 50 760 770, clip=true]
{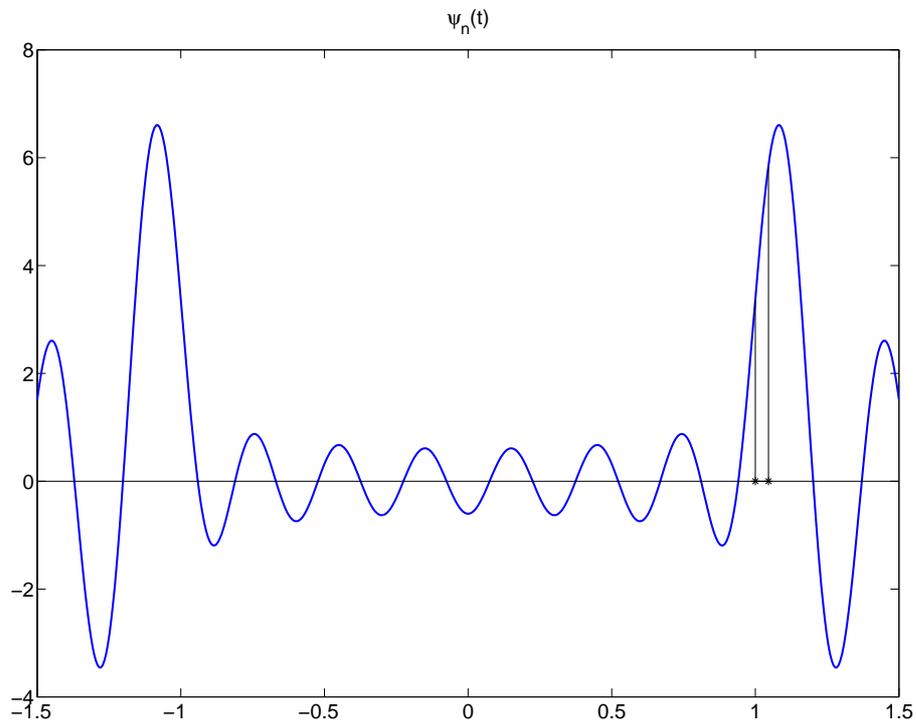}
\caption
{\it
The function $\psi_n(t)$ for $c = 20$ and $n = 14$.
Since $\chi_n \approx 437.36 > c^2$, the behavior is as asserted
in \eqref{eq_all_tx_large} of Theorem~\ref{thm_five}. The points $1$ and
$\sqrt{\chi_n}/c \approx 1.0457$ are marked with asterisks. Observe that
$\abrk{\lambda_n} \approx 0.12564$, and
the oscillations of $\psi_n$ outside $(-1, 1)$ have
relatively large magnitude (of order $\abrk{\lambda_n}^{-1}$).
Compare to Figure~\ref{fig:test75a}.
Corresponds to Experiment 1.
}
\label{fig:test75b}
\end{center}
\end{figure}
%%%%%%%%%%%%

%%%%%%%%%%%%%%%%%%%
\begin{table}[htbp]
\begin{center}
\begin{tabular}{c|c|c|c|c}
$c$   & $n$   & $x_1-\sqrt{\chi_n}/c$  &  
$\frac{\pi}{2c} \sqrt{\frac{x_1^2-1}{x_1^2-(\chi_n/c^2)}}$ & 
$(x_1-\sqrt{\chi_n}/c) \cdot \frac{2c}{\pi}
\sqrt{\frac{x_1^2-(\chi_n/c^2)}{x_1^2-1}}$\\[1ex]
\hline
10 & 15 &
  0.46561E+00 & 
  0.22542E+00 & 
  0.20655E+01 \\
10 & 19 &
  0.51090E+00 &
  0.24279E+00 &
  0.21043E+01 \\
10 & 24 &
  0.55570E+00 &
  0.26055E+00 &
  0.21328E+01 \\
\hline
100 & 76 &
  0.49260E-01 &
  0.23935E-01 &
  0.20581E+01 \\
100 & 84 &
  0.57274E-01 &
  0.27070E-01 &
  0.21158E+01 \\
100 & 92 &
  0.63570E-01 &
  0.29602E-01 &
  0.21475E+01 \\
\hline
1000 & 652 &
  0.52819E-02 &
  0.23016E-02 &
  0.22949E+01 \\
1000 & 664 &
  0.56889E-02 &
  0.27295E-02 &
  0.20843E+01 \\
1000 & 676 &
  0.63367E-02 &
  0.30338E-02 &
  0.20887E+01 \\
\end{tabular}
\end{center}
\caption{\it
The relation between 
the left-hand side and the right-hand side of the inequality
\eqref{eq_x1_x0_smart} of Theorem~\ref{thm_x1_x0_good}.
For each value of the band limit $c$, 
the three values of $n$ are chosen such that 
$|\lambda_n| \approx 10^{-5}, 10^{-9}, 10^{-13}$, respectively.
Corresponds to Experiment 2.
}
\label{t:test76}
\end{table}
%%%%%%%%%%%
%%%%%%%%%%%
\paragraph{Experiment 2.}
\label{sec_exp2}
In the following numerical experiment, we illustrate
Theorem~\ref{thm_x1_x0_good} in Section~\ref{sec_first_order}. 
We proceed as follows.
For each of the three values of band limit $c$ (namely,
$c=10,100,1000$), we pick three values of the prolate index $n$.
The values of $n$ are chosen to satisfy $n > 2c/\pi$ (which implies that
$\chi_n > c^2$, due to Theorem~\ref{thm_n_and_khi} in Section~\ref{sec_pswf}).
Then, we evaluate 
the eigenvalue $\chi_n$
of the ODE \eqref{eq_prolate_ode} of Section~\ref{sec_pswf},
by using the algorithm of Section~\ref{sec_evaluate_beta}.
Also, we evaluate
the minimal root $x_1$ of $\psi_n$ in $(1,\infty)$
(see Theorem~\ref{thm_four} 
in Section~\ref{sec_first_order}), by using
the algorithm of Section~\ref{sec_roots_outside}.

The results of this experiment are displayed
in Table~\ref{t:test76}. This table has the following structure.
The first two columns contain the band limit $c$ and PSWF index $n$.
The third column contains the difference between the
$x_1$ and the special point
$\sqrt{\chi_n}/c$ (see Theorem~\ref{thm_four} 
in Section~\ref{sec_first_order}).
This difference is 
the left-hand side of the inequality 
\eqref{eq_x1_x0_smart} of Theorem~\ref{thm_x1_x0_good}.
On the other hand, the fourth column contains the right-hand side of
\eqref{eq_x1_x0_smart} (a lower bound on this difference).
The last column contains the ratio of
the value in the third column to the value in the fourth column.

We observe that the value in the fourth column is smaller than
the value in the third column roughly by a factor of 2,
for all the choices of $c, n$. In other words,
the lower bound on $x_1-\sqrt{\chi_n}/c$,
provided by 
Theorem~\ref{thm_x1_x0_good},
is rather inaccurate, but is of correct order.

%%%%%%%%%%%%%%%%%%%
\begin{table}[htbp]
\begin{center}
\begin{tabular}{c|c|c|c|c|c}
$k$   & 
$x_{k+1} - x_k$ &
$\frac{\pi\brk{x_k^2-1}}{\sqrt{1+c^2\brk{x_k^2-1}^2}}$ &
$\frac{\pi}{c} \sqrt{\frac{x_k^2-1}{x_k^2-(\chi_n/c^2)}}$ &
lower error &
upper error \\[1ex]
\hline
 1
& 0.51496E-01 & 0.31410E-01 & 0.58023E-01 & 0.39005E+00 & 0.12676E+00  \\
 2
& 0.45166E-01 & 0.31412E-01 & 0.47546E-01 & 0.30452E+00 & 0.52703E-01  \\
 3
& 0.42078E-01 & 0.31413E-01 & 0.43379E-01 & 0.25345E+00 & 0.30936E-01  \\
 4
& 0.40179E-01 & 0.31414E-01 & 0.41019E-01 & 0.21815E+00 & 0.20908E-01  \\
 5
& 0.38872E-01 & 0.31414E-01 & 0.39466E-01 & 0.19185E+00 & 0.15285E-01  \\
 6
& 0.37908E-01 & 0.31415E-01 & 0.38354E-01 & 0.17129E+00 & 0.11754E-01  \\
 7
& 0.37164E-01 & 0.31415E-01 & 0.37512E-01 & 0.15470E+00 & 0.93670E-02  \\
 8
& 0.36570E-01 & 0.31415E-01 & 0.36851E-01 & 0.14097E+00 & 0.76646E-02  \\
 9
& 0.36084E-01 & 0.31415E-01 & 0.36315E-01 & 0.12939E+00 & 0.64016E-02  \\
 10
& 0.35678E-01 & 0.31415E-01 & 0.35872E-01 & 0.11949E+00 & 0.54352E-02  \\
 11
& 0.35334E-01 & 0.31415E-01 & 0.35499E-01 & 0.11090E+00 & 0.46772E-02  \\
 12
& 0.35038E-01 & 0.31415E-01 & 0.35180E-01 & 0.10338E+00 & 0.40703E-02  \\
 13
& 0.34780E-01 & 0.31415E-01 & 0.34905E-01 & 0.96745E-01 & 0.35761E-02  \\
 14
& 0.34554E-01 & 0.31416E-01 & 0.34664E-01 & 0.90835E-01 & 0.31677E-02  \\
 15
& 0.34354E-01 & 0.31416E-01 & 0.34451E-01 & 0.85540E-01 & 0.28261E-02  \\
 16
& 0.34176E-01 & 0.31416E-01 & 0.34263E-01 & 0.80768E-01 & 0.25372E-02  \\
 17
& 0.34016E-01 & 0.31416E-01 & 0.34094E-01 & 0.76444E-01 & 0.22905E-02  \\
 18
& 0.33872E-01 & 0.31416E-01 & 0.33942E-01 & 0.72510E-01 & 0.20780E-02  \\
 19
& 0.33741E-01 & 0.31416E-01 & 0.33805E-01 & 0.68913E-01 & 0.18937E-02  \\
\end{tabular}
\end{center}
\caption{\it
Illustration of Theorem~\ref{thm_spacing} with $c = 100$ and $n = 90$.
$|\lambda_n| \approx 10^{-10}$.
Corresponds to Experiment 3.
}
\label{t:test81a}
\end{table}
%%%%%%%%%%%
%%%%%%%%%%%%%%%%%%%
\begin{table}[htbp]
\begin{center}
\begin{tabular}{c|c|c|c|c|c}
$k$   & 
$x_{k+1} - x_k$ &
$\frac{\pi\brk{x_k^2-1}}{\sqrt{1+c^2\brk{x_k^2-1}^2}}$ &
$\frac{\pi}{c} \sqrt{\frac{x_k^2-1}{x_k^2-(\chi_n/c^2)}}$ &
lower error &
upper error \\[1ex]
\hline
 1
& 0.59672E-01 & 0.31414E-01 & 0.68077E-01 & 0.47355E+00 & 0.14086E+00  \\
 2
& 0.51323E-01 & 0.31415E-01 & 0.54472E-01 & 0.38790E+00 & 0.61363E-01  \\
 3
& 0.47161E-01 & 0.31415E-01 & 0.48918E-01 & 0.33387E+00 & 0.37253E-01  \\
 4
& 0.44558E-01 & 0.31415E-01 & 0.45710E-01 & 0.29496E+00 & 0.25858E-01  \\
 5
& 0.42740E-01 & 0.31415E-01 & 0.43566E-01 & 0.26496E+00 & 0.19329E-01  \\
 6
& 0.41383E-01 & 0.31415E-01 & 0.42010E-01 & 0.24087E+00 & 0.15150E-01  \\
 7
& 0.40325E-01 & 0.31415E-01 & 0.40820E-01 & 0.22094E+00 & 0.12275E-01  \\
 8
& 0.39472E-01 & 0.31416E-01 & 0.39874E-01 & 0.20410E+00 & 0.10193E-01  \\
 9
& 0.38767E-01 & 0.31416E-01 & 0.39102E-01 & 0.18964E+00 & 0.86267E-02  \\
 10
& 0.38174E-01 & 0.31416E-01 & 0.38457E-01 & 0.17705E+00 & 0.74127E-02  \\
 11
& 0.37667E-01 & 0.31416E-01 & 0.37910E-01 & 0.16597E+00 & 0.64491E-02  \\
 12
& 0.37229E-01 & 0.31416E-01 & 0.37440E-01 & 0.15614E+00 & 0.56691E-02  \\
 13
& 0.36845E-01 & 0.31416E-01 & 0.37030E-01 & 0.14735E+00 & 0.50273E-02  \\
 14
& 0.36506E-01 & 0.31416E-01 & 0.36670E-01 & 0.13943E+00 & 0.44920E-02  \\
 15
& 0.36204E-01 & 0.31416E-01 & 0.36350E-01 & 0.13225E+00 & 0.40401E-02  \\
 16
& 0.35933E-01 & 0.31416E-01 & 0.36065E-01 & 0.12572E+00 & 0.36546E-02  \\
 17
& 0.35690E-01 & 0.31416E-01 & 0.35808E-01 & 0.11975E+00 & 0.33228E-02  \\
 18
& 0.35469E-01 & 0.31416E-01 & 0.35576E-01 & 0.11426E+00 & 0.30349E-02  \\
 19
& 0.35267E-01 & 0.31416E-01 & 0.35365E-01 & 0.10921E+00 & 0.27833E-02  \\
\end{tabular}
\end{center}
\caption{\it
Illustration of Theorem~\ref{thm_spacing} with $c = 100$ and $n = 110$.
$|\lambda_n| \approx 10^{-25}$.
Corresponds to Experiment 3.
}
\label{t:test81b}
\end{table}
%%%%%%%%%%%
\paragraph{Experiment 3.}
\label{sec_exp3}
In the following numerical experiment,
we illustrate Theorem~\ref{thm_spacing} in Section~\ref{sec_first_order}.
We proceed as follows. We choose the band limit $c$ and the prolate
index $n$. For each such choice, 
we compute the first 20 roots
$x_1, \dots, x_{20}$ of $\psi_n$ in $(1,\infty)$, using the 
algorithm of Section~\ref{sec_roots_outside}.
Also, for each $k=1,\dots,19$, we compute the upper
and lower bound on $x_{k+1}-x_k$, established
in Theorem~\ref{thm_spacing}.

The results of the experiment are displayed in
Tables~\ref{t:test81a}, \ref{t:test81b},
that correspond to $c = 100$ and $n = 90, 110$, respectively.
These tables have the following structure.
The first column contains the index $k$ of the root $x_k$ of $\psi_n$ in 
$(1, \infty)$.
The second column contains the difference between two consecutive roots
$x_{k+1}$ and $x_k$ of $\psi_n$ in $(1,\infty)$.
The third and fourth columns contain, respectively, the lower
and upper bound on $x_{k+1} - x_k$, as in \eqref{eq_25_09_h1} of 
Theorem~\ref{thm_spacing}. The last two columns contain
the relative errors of these bounds. 

We observe that the upper bound
is more accurate in terms of relative error. Moreover, the 
relative accuracy
of both bounds improves monotonically as $k$ grows. On the other hand,
for a fixed $k$, the accuracy in Table~\ref{t:test81a} is slightly higher
than that in Table~\ref{t:test81b}, which suggests that the bounds worsen
as $n$ grows. We also observe that  
the difference $x_{k+1}-x_k$ between two consecutive
roots decreases monotonically to $\pi/c$, as $k$ grows
(see \eqref{eq_25_09_h3} in 
Theorem~\ref{thm_spacing} and Remark~\ref{rem_big_spacing}).

%%%%%%%%%%%%%%%%%%%%%%%%%%%%%%
\subsubsection{Illustration of Results from Section~\ref{sec_growth}}
%%%%%%%%%%%%%%%%%%%%%%%%%%%%%%
%%%%%%%%%%%%%%%%%%%%%
\begin{figure} [htbp]
\begin{center}
\includegraphics[width=10cm, bb=68 218 542 574, clip=true]
{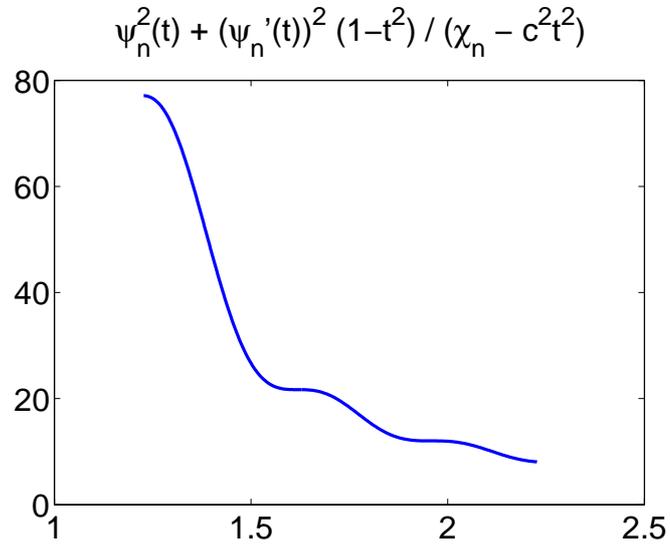}
\caption
{\it
$Q(t)$ defined via \eqref{eq_psi1_q}, with $c = 10$ and $n = 8$.
See Experiment 4.
}
\label{fig:test82c}
\end{center}
\end{figure}
%%%%%%%%%%%%
%%%%%%%%%%%%%%%%%%%%%
\begin{figure} [htbp]
\begin{center}
\includegraphics[width=10cm, bb=68 218 542 574, clip=true]
{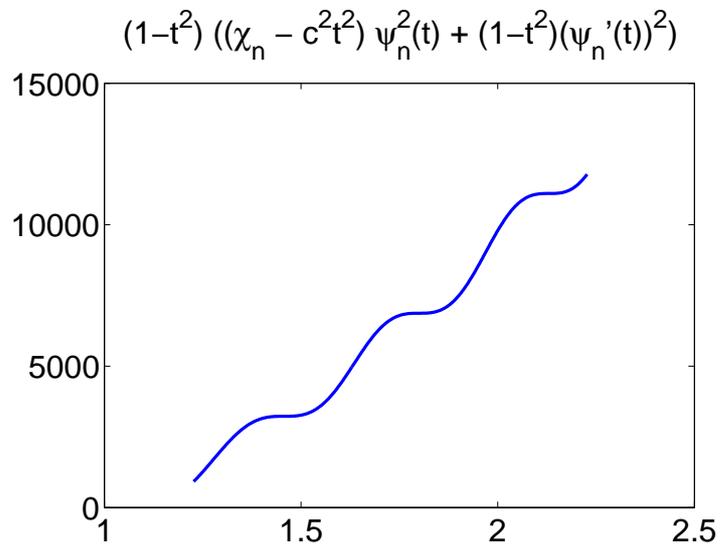}
\caption
{\it
$\tilde{Q}(t)$ defined via \eqref{eq_psi1_qtilde}, with $c = 10$ and $n = 8$.
See Experiment 4.
}
\label{fig:test82d}
\end{center}
\end{figure}
%%%%%%%%%%%%
\paragraph{Experiment 4.}
\label{sec_exp4}
In this experiment, we illustrate
Theorem~\ref{lem_Q_Q_tilde} in Section~\ref{sec_two_by_two}.
We proceed as follows. We choose the band limit $c=10$
and the prolate index $n=8$. Then, we compute $\chi_n$
by using the algorithm of Section~\ref{sec_evaluate_beta}.
Also, we evaluate $\psi_n$ and $\psi_n'$ at 500 equispaced
points in the interval 
\begin{align}
\left(\frac{\sqrt{\chi_n}+1}{c},\frac{\sqrt{\chi_n}+1}{c}+1\right).
\label{eq_qq_int}
\end{align}
For each such point $x$, we compute $Q(x)$ and $\tilde{Q}(x)$,
where the functions $Q, \tilde{Q}$ are defined, respectively,
via \eqref{eq_psi1_q}, \eqref{eq_psi1_qtilde} in Theorem~\ref{lem_Q_Q_tilde}.

In Figures~\ref{fig:test82c}, \ref{fig:test82d},
we plot, respectively, $Q$ and $\tilde{Q}$ over
the interval \eqref{eq_qq_int}. We observe that,
as expected, 
$Q$ is monotonically decreasing and $\tilde{Q}$ is 
monotonically increasing.
On the other hand, we observe that the second derivative
of each of $Q, \tilde{Q}$ does not have a constant
sign in this interval.

%%%%%%%%%%%%%%%%%%%
\begin{table}[htbp]
\begin{center}
\begin{tabular}{c|c|c|c|c|c}
$k$   & 
$\abrk{ \frac{\psi_n'(x_{k+1})}{\psi_n'(x_k)} }$ &
$\frac{ x_k^2 - 1 }{ x_{k+1}^2 - 1 }$ &
$\sqrt{ \frac{x_k^2 - 1}{c^2 x_k^2 - \chi_n} \cdot 
        \frac{c^2 x_{k+1}^2 - \chi_n}{x_{k+1}^2 - 1} }$ &
lower error &
upper error \\[1ex]
\hline
 1
& 0.93958E+00 & 0.74737E+00 & 0.11909E+01 & 0.20457E+00 & 0.26750E+00  \\
 2
& 0.93463E+00 & 0.81017E+00 & 0.10796E+01 & 0.13317E+00 & 0.15516E+00  \\
 3
& 0.93943E+00 & 0.84386E+00 & 0.10463E+01 & 0.10172E+00 & 0.11373E+00  \\
 4
& 0.94463E+00 & 0.86575E+00 & 0.10309E+01 & 0.83504E-01 & 0.91326E-01  \\
 5
& 0.94920E+00 & 0.88139E+00 & 0.10223E+01 & 0.71439E-01 & 0.77048E-01  \\
 6
& 0.95309E+00 & 0.89325E+00 & 0.10170E+01 & 0.62785E-01 & 0.67058E-01  \\
 7
& 0.95639E+00 & 0.90260E+00 & 0.10134E+01 & 0.56236E-01 & 0.59629E-01  \\
 8
& 0.95922E+00 & 0.91021E+00 & 0.10109E+01 & 0.51085E-01 & 0.53864E-01  \\
 9
& 0.96166E+00 & 0.91655E+00 & 0.10090E+01 & 0.46915E-01 & 0.49244E-01  \\
 10
& 0.96380E+00 & 0.92191E+00 & 0.10076E+01 & 0.43460E-01 & 0.45449E-01  \\
 11
& 0.96568E+00 & 0.92653E+00 & 0.10065E+01 & 0.40545E-01 & 0.42269E-01  \\
 12
& 0.96735E+00 & 0.93055E+00 & 0.10056E+01 & 0.38048E-01 & 0.39561E-01  \\
 13
& 0.96885E+00 & 0.93408E+00 & 0.10049E+01 & 0.35883E-01 & 0.37225E-01  \\
 14
& 0.97019E+00 & 0.93722E+00 & 0.10043E+01 & 0.33984E-01 & 0.35185E-01  \\
 15
& 0.97141E+00 & 0.94003E+00 & 0.10038E+01 & 0.32304E-01 & 0.33386E-01  \\
 16
& 0.97252E+00 & 0.94256E+00 & 0.10034E+01 & 0.30805E-01 & 0.31788E-01  \\
 17
& 0.97353E+00 & 0.94485E+00 & 0.10031E+01 & 0.29459E-01 & 0.30356E-01  \\
 18
& 0.97447E+00 & 0.94694E+00 & 0.10028E+01 & 0.28242E-01 & 0.29065E-01  \\
 19
& 0.97533E+00 & 0.94886E+00 & 0.10025E+01 & 0.27136E-01 & 0.27895E-01  \\
\end{tabular}
\end{center}
\caption{\it
Illustration of Theorem~\ref{thm_consec_der}, with
$c = 100$,  $n = 80$, $|\lambda_n| = \mbox{\text{\rm{0.58925E-07}}}$.
See Experiment 5.
}
\label{t:test83a}
\end{table}
%%%%%%%%%%%
%%%%%%%%%%%%%%%%%%%
\begin{table}[htbp]
\begin{center}
\begin{tabular}{c|c|c|c|c|c}
$k$   & 
$\abrk{ \frac{\psi_n'(x_{k+1})}{\psi_n'(x_k)} }$ &
$\frac{ x_k^2 - 1 }{ x_{k+1}^2 - 1 }$ &
$\sqrt{ \frac{x_k^2 - 1}{c^2 x_k^2 - \chi_n} \cdot 
        \frac{c^2 x_{k+1}^2 - \chi_n}{x_{k+1}^2 - 1} }$ &
lower error &
upper error \\[1ex]
\hline
 1
& 0.99507E+00 & 0.81420E+00 & 0.12260E+01 & 0.18177E+00 & 0.23205E+00  \\
 2
& 0.97042E+00 & 0.85769E+00 & 0.10994E+01 & 0.11618E+00 & 0.13292E+00  \\
 3
& 0.96628E+00 & 0.88122E+00 & 0.10600E+01 & 0.88030E-01 & 0.96998E-01  \\
 4
& 0.96620E+00 & 0.89669E+00 & 0.10413E+01 & 0.71944E-01 & 0.77728E-01  \\
 5
& 0.96726E+00 & 0.90789E+00 & 0.10306E+01 & 0.61380E-01 & 0.65503E-01  \\
 6
& 0.96863E+00 & 0.91647E+00 & 0.10238E+01 & 0.53843E-01 & 0.56971E-01  \\
 7
& 0.97004E+00 & 0.92332E+00 & 0.10192E+01 & 0.48159E-01 & 0.50637E-01  \\
 8
& 0.97139E+00 & 0.92894E+00 & 0.10158E+01 & 0.43700E-01 & 0.45725E-01  \\
 9
& 0.97265E+00 & 0.93365E+00 & 0.10133E+01 & 0.40096E-01 & 0.41791E-01  \\
 10
& 0.97382E+00 & 0.93768E+00 & 0.10114E+01 & 0.37115E-01 & 0.38560E-01  \\
 11
& 0.97490E+00 & 0.94116E+00 & 0.10099E+01 & 0.34603E-01 & 0.35854E-01  \\
 12
& 0.97588E+00 & 0.94421E+00 & 0.10086E+01 & 0.32453E-01 & 0.33550E-01  \\
 13
& 0.97679E+00 & 0.94691E+00 & 0.10076E+01 & 0.30590E-01 & 0.31562E-01  \\
 14
& 0.97763E+00 & 0.94932E+00 & 0.10068E+01 & 0.28957E-01 & 0.29826E-01  \\
 15
& 0.97840E+00 & 0.95148E+00 & 0.10061E+01 & 0.27514E-01 & 0.28297E-01  \\
 16
& 0.97912E+00 & 0.95344E+00 & 0.10055E+01 & 0.26228E-01 & 0.26938E-01  \\
 17
& 0.97978E+00 & 0.95522E+00 & 0.10050E+01 & 0.25073E-01 & 0.25721E-01  \\
 18
& 0.98040E+00 & 0.95684E+00 & 0.10045E+01 & 0.24030E-01 & 0.24624E-01  \\
 19
& 0.98098E+00 & 0.95834E+00 & 0.10042E+01 & 0.23082E-01 & 0.23630E-01  \\
\end{tabular}
\end{center}
\caption{\it
Illustration of Theorem~\ref{thm_consec_der},
$c = 200$, $n = 160$, $|\lambda_n| = \mbox{\text{\rm{0.17136E-13}}}$.
See Experiment 5.
}
\label{t:test83b}
\end{table}
%%%%%%%%%%%
\paragraph{Experiment 5.}
\label{sec_exp5}
In the following experiment, we illustrate 
Theorem~\ref{thm_consec_der} in Section~\ref{sec_two_by_two}.
We proceed as follows. We choose, more or less arbitrarily,
the band limit $c$ and the prolate index $n$. For each choice of
$c,n$, we evaluate $\chi_n$ by using the algorithm of
Section~\ref{sec_evaluate_beta}. Then, we evaluate the first
20 roots $x_1, \dots, x_{20}$ of $\psi_n$ in $(1,\infty)$,
by using the algorithm of Section~\ref{sec_roots_outside}
(in extended precision). For each such root $x_k$, we evaluate 
$\psi_n'(x_k)$ by using the algorithm of Section~\ref{sec_dpsi_outside}
(in extended precision).

We display the results of the experiment
in 
Tables~\ref{t:test83a}, \ref{t:test83b}, corresponding to
$c = 100$, $n = 80$ and
$c = 200$, $n = 160$, respectively. These tables have the following
structure. 
The first column contains the index $k$ of the root $x_k$ of $\psi_n$ in 
$(1,\infty)$.
The second column contains the absolute value of the ratio of 
$\psi_n'(x_{k+1})$ to $\psi_n'(x_k)$.
The third and fourth columns contain the lower and upper bound on 
that ratio, respectively, established in \eqref{eq_consec_der}
of Theorem~\ref{thm_consec_der}. The last two columns contain
the relative errors of these bounds. 

We observe that the ratio
in the second column is always less than one. Moreover, it first decreases up
to a certain $k$ and then increases as $k$ grows.
Both bounds have roughly the same relative accuracy and become sharper
as $k$ grows. Even for $k = 1$ the errors are about $20\%$, while
already at $k = 7$ they drop to about $5\%$. 
We also observe (not shown in the tables)
that the magnitude of $\left| \psi_n'(x_k) \right|$
is about $10^{8}$ for Table~\ref{t:test83a} and 
about $10^{15}$ for Table~\ref{t:test83b}
(see also Experiment 6 below).

%%%%%%%%%%%%%%%%%%%
\begin{table}[htbp]
\begin{center}
\begin{tabular}{c|c|c|c|c}
$k$   & 
$ \abrk{ \psi_n'(x_k) }^{-1} $ &
$ \frac{ |\lambda_n| \left(x_k^2 - 1\right)^{3/4} }
  {\left(x_k^2 - (\chi_n/c^2)\right)^{1/4} |\psi_n(1)| \sqrt{2} }$ &
$\varepsilon_k$ &
$b_c\left(x_k,\frac{\sqrt{\chi_n}}{c}\right)$ \\[1ex]
\hline
 1                                                                            
& 0.57349E-19 & 0.81518E-19 & 0.72340E-02 & 0.10181E+01  \\                   
 2                                                                            
& 0.56895E-19 & 0.80550E-19 & 0.15530E-02 & 0.10106E+01  \\                   
 3                                                                            
& 0.58182E-19 & 0.82319E-19 & 0.64593E-03 & 0.10081E+01  \\
 4
& 0.59907E-19 & 0.84743E-19 & 0.34935E-03 & 0.10067E+01  \\
 5
& 0.61785E-19 & 0.87390E-19 & 0.21744E-03 & 0.10059E+01  \\
 6
& 0.63718E-19 & 0.90120E-19 & 0.14767E-03 & 0.10052E+01  \\
 7
& 0.65667E-19 & 0.92874E-19 & 0.10641E-03 & 0.10048E+01  \\
 8
& 0.67615E-19 & 0.95627E-19 & 0.80051E-04 & 0.10044E+01  \\
 9
& 0.69553E-19 & 0.98367E-19 & 0.62211E-04 & 0.10041E+01  \\
 10
& 0.71477E-19 & 0.10109E-18 & 0.49596E-04 & 0.10039E+01  \\
 11
& 0.73385E-19 & 0.10379E-18 & 0.40358E-04 & 0.10036E+01  \\
 12
& 0.75277E-19 & 0.10646E-18 & 0.33400E-04 & 0.10035E+01  \\
 13
& 0.77154E-19 & 0.10911E-18 & 0.28034E-04 & 0.10033E+01  \\
 14
& 0.79015E-19 & 0.11175E-18 & 0.23815E-04 & 0.10032E+01  \\
 15
& 0.80861E-19 & 0.11436E-18 & 0.20441E-04 & 0.10030E+01  \\
 16
& 0.82693E-19 & 0.11695E-18 & 0.17703E-04 & 0.10029E+01  \\
 17
& 0.84511E-19 & 0.11952E-18 & 0.15453E-04 & 0.10028E+01  \\
 18
& 0.86317E-19 & 0.12207E-18 & 0.13584E-04 & 0.10027E+01  \\
 19
& 0.88112E-19 & 0.12461E-18 & 0.12015E-04 & 0.10026E+01  \\
% 20
%& 0.89895E-19 & 0.12713E-18 & 0.10687E-04 & 0.10026E+01  \\
\end{tabular}
\end{center}
\caption{\it
Illustration of Theorem~\ref{thm_12_10_sharp} with $c = 100$, $n = 100$.
$\lambda_n = \mbox{\text{\rm{0.94419E-18}}}$.
See Experiment 6.
}
\label{t:test84a}
\end{table}
%%%%%%%%%%%
%%%%%%%%%%%%%%%%%%%
\begin{table}[htbp]
\begin{center}
\begin{tabular}{c|c|c|c|c}
$k$   & 
$ \abrk{ \psi_n'(x_k) }^{-1} $ &
$ \frac{ |\lambda_n| \left(x_k^2 - 1\right)^{3/4} }
  {\left(x_k^2 - (\chi_n/c^2)\right)^{1/4} |\psi_n(1)| \sqrt{2} }$ &
$\varepsilon_k$ &
$b_c\left(x_k,\frac{\sqrt{\chi_n}}{c}\right)$ \\[1ex]
\hline
 1                                                                            
& 0.10723E-23 & 0.15242E-23 & 0.72932E-02 & 0.10140E+01  \\                   
 2                                                                            
& 0.10407E-23 & 0.14734E-23 & 0.15865E-02 & 0.10077E+01  \\                   
 3
& 0.10464E-23 & 0.14805E-23 & 0.67023E-03 & 0.10056E+01  \\
 4
& 0.10621E-23 & 0.15025E-23 & 0.36864E-03 & 0.10045E+01  \\
 5
& 0.10817E-23 & 0.15300E-23 & 0.23349E-03 & 0.10038E+01  \\
 6
& 0.11028E-23 & 0.15598E-23 & 0.16142E-03 & 0.10033E+01  \\
 7
& 0.11246E-23 & 0.15905E-23 & 0.11843E-03 & 0.10029E+01  \\
 8
& 0.11466E-23 & 0.16216E-23 & 0.90717E-04 & 0.10027E+01  \\
 9
& 0.11685E-23 & 0.16527E-23 & 0.71782E-04 & 0.10025E+01  \\
 10
& 0.11903E-23 & 0.16835E-23 & 0.58262E-04 & 0.10023E+01  \\
 11
& 0.12119E-23 & 0.17140E-23 & 0.48264E-04 & 0.10021E+01  \\
 12
& 0.12332E-23 & 0.17441E-23 & 0.40656E-04 & 0.10020E+01  \\
 13
& 0.12542E-23 & 0.17738E-23 & 0.34730E-04 & 0.10019E+01  \\
 14
& 0.12750E-23 & 0.18031E-23 & 0.30020E-04 & 0.10018E+01  \\
 15
& 0.12954E-23 & 0.18320E-23 & 0.26215E-04 & 0.10017E+01  \\
 16
& 0.13156E-23 & 0.18606E-23 & 0.23094E-04 & 0.10016E+01  \\
 17
& 0.13355E-23 & 0.18887E-23 & 0.20502E-04 & 0.10015E+01  \\
 18
& 0.13551E-23 & 0.19164E-23 & 0.18325E-04 & 0.10015E+01  \\
 19
& 0.13745E-23 & 0.19438E-23 & 0.16479E-04 & 0.10014E+01  \\
% 20
%& 0.13936E-23 & 0.19708E-23 & 0.14899E-04 & 0.10014E+01  \\
\end{tabular}
\end{center}
\caption{\it
Illustration of Theorem~\ref{thm_12_10_sharp} with $c = 1000$, $n = 700$.
$\lambda_n = \mbox{\text{\rm{0.12446E-21}}}$.
See Experiment 6.
}
\label{t:test84b}
\end{table}
%%%%%%%%%%%
\paragraph{Experiment 6}
\label{sec_exp6}
In this experiment, we illustrate
Theorems~\ref{thm_12_10_sharp},~\ref{thm_bc_big} in Section~\ref{sec_upper}.
We proceed as follows. We choose, more or less arbitrarily,
the band limit $c$ and the prolate index $n$.
For each such choice, we evaluate $\chi_n$ and $\lambda_n$,
by using the algorithms of 
Sections~\ref{sec_evaluate_beta},~\ref{sec_evaluate_lambda},
respectively (in double precision). Then, we compute
the first 20 roots $x_1, \dots, x_{20}$ of $\psi_n$ in $(1, \infty)$,
by using the algorithm of Section~\ref{sec_roots_outside}
(in extended precision). For each such root $x_k$, we evaluate
$\psi_n'(x_k)$ by using the algorithm of Section~\ref{sec_dpsi_outside}
(in extended precision).

We display the results of the experiment
in Tables~\ref{t:test84a},~\ref{t:test84b}, 
corresponding to $c=100, n=100$ and $c=1000, n=700$, respectively.
These tables have the following
structure.
The first column contains the index $k$ of the root $x_k$ of $\psi_n$ in 
$(1,\infty)$.
The second column contains the reciprocal of $|\psi_n'(x_k)|$.
The third column contains the quantity
\begin{align}
\frac{ |\lambda_n| \left(x_k^2 - 1\right)^{3/4} }
  {\left(x_k^2 - (\chi_n/c^2)\right)^{1/4} |\psi_n(1)| \sqrt{2} }
\label{eq_exp6_upper}
\end{align}
(see \eqref{eq_12_10_sharp} in Theorem~\ref{thm_12_10_sharp}).
The fourth column contains $\varepsilon_k$, defined via the formula
\begin{align}
|\psi_n'(x_k)| = 
\frac{2 \cdot |\psi_n(1)| \left(x_k^2 - (\chi_n/c^2)\right)^{1/4} }
     { |\lambda_n| \left(x_k^2 - 1\right)^{3/4} } \cdot
     (1 + \varepsilon_k)
\label{eq_exp6_epsilon}
\end{align}
(we observe that $\varepsilon_k$ in \eqref{eq_exp6_epsilon}
is obtained via multiplying
\eqref{eq_exp6_upper} by $|\psi_n'(x_k)|/\sqrt{2}$ and subtracting 1
from the result). The last column contains $b_c(x_k,\sqrt{\chi_n}/c)$,
defined via \eqref{eq_bc_def} of Definition~\ref{def_bc}
in Section~\ref{sec_upper}.

According to Theorem~\ref{thm_12_10_sharp}, the product of the values
in the third and fifth columns is an upper bound on $|\psi_n'(x_k)|^{-1}$
(the second column). However, \eqref{eq_exp6_upper} alone (the third column)
already overestimates $|\psi_n'(x_k)|^{-1}$ by roughly $\sqrt{2}$.
We also observe (see the fourth column) that the parameter 
$\varepsilon_k$, defined via \eqref{eq_exp6_epsilon}, is fairly
small, and decreases as $k$ grows. According to
Theorems~\ref{lem_dpsi_for_large_x},~\ref{lem_tail_dpsi} in
Section~\ref{sec_tail},
we expect $\varepsilon_k$
to tend to zero as $k$ grows to $\infty$,
since
\begin{align}
\frac{2 \cdot |\psi_n(1)| \left(x_k^2 - (\chi_n/c^2)\right)^{1/4} }
     { |\lambda_n| \left(x_k^2 - 1\right)^{3/4} } \sim
\frac{2 \cdot |\psi_n(1)|}{| \lambda_n \cdot x_k |}, \quad
k \to \infty.
\label{eq_exp6_limit}
\end{align}
On the other hand, the fact that $\varepsilon_k \approx 10^{-4}$ already
for $k=7$ is somewhat surprising. In other words, 
the left hand side of \eqref{eq_exp6_limit} is a fairly tight
estimate of $|\psi_n'(x_k)|$, even for small $k$.

We also observe that $b_c(x_k,\sqrt{\chi_n}/c)$ (see the last column)
is very close to $1$ even for $k=1$, and becomes even closer to $1$
as $k$ increases. In other words, the upper bound $e^{1/4} \approx 1.284$
on this quantity (see Theorem~\ref{thm_bc_big} in Section~\ref{sec_upper})
is somewhat overcautious.

%%%%%%%%%%%%%%%%%%%%%%%%%%%%%%
\subsubsection{Illustration of Results from Section~\ref{sec_one_over_psi}}
%%%%%%%%%%%%%%%%%%%%%%%%%%%%%%
%%%%%%%%%%%%%%%%%%%
\begin{table}[htbp]
\begin{center}
\begin{tabular}{c|c|c|c}
$k$   &
$\left| \sum_{j=k}^{k+1} \frac{1}{\brk{1-x_j} \psi_n'(x_j)} \right|$ &
%$\abrk{ \frac{1}{\brk{1-x_k} \psi_n'(x_k)} + 
%        \frac{1}{\brk{1-x_{k+1}} \psi_n'(x_{k+1})} } $ &
$e^{1/4} \cdot \abrk{\lambda_n} \cdot
  \int_x^y \frac{ (z+1)^2 \; dz}{\brk{z^2-\left(\chi_n/c^2\right)}^{3/2}}$ &
 ratio \\[1ex]
\hline
 1
& 0.29442E-19 & 0.31341E-17 & 0.10645E+03  \\
 3
& 0.99172E-20 & 0.85727E-18 & 0.86442E+02  \\
 5
& 0.57139E-20 & 0.46271E-18 & 0.80980E+02  \\
 7
& 0.39054E-20 & 0.30749E-18 & 0.78735E+02  \\
 9
& 0.29098E-20 & 0.22656E-18 & 0.77861E+02  \\
 11
& 0.22851E-20 & 0.17760E-18 & 0.77720E+02  \\
 13
& 0.18596E-20 & 0.14509E-18 & 0.78022E+02  \\
 15
& 0.15530E-20 & 0.12209E-18 & 0.78614E+02  \\
 17
& 0.13226E-20 & 0.10503E-18 & 0.79407E+02  \\
 19
& 0.11441E-20 & 0.91920E-19 & 0.80345E+02  \\
 21
& 0.10021E-20 & 0.81564E-19 & 0.81393E+02  \\
 23
& 0.88694E-21 & 0.73193E-19 & 0.82524E+02  \\
 25
& 0.79193E-21 & 0.66300E-19 & 0.83720E+02  \\
 27
& 0.71242E-21 & 0.60534E-19 & 0.84969E+02  \\
 29
& 0.64507E-21 & 0.55645E-19 & 0.86261E+02  \\
 31
& 0.58742E-21 & 0.51451E-19 & 0.87588E+02  \\
 33
& 0.53762E-21 & 0.47818E-19 & 0.88944E+02  \\
 35
& 0.49425E-21 & 0.44643E-19 & 0.90325E+02  \\
 37
& 0.45620E-21 & 0.41846E-19 & 0.91727E+02  \\
 39
& 0.42261E-21 & 0.39365E-19 & 0.93147E+02  \\
\end{tabular}
\end{center}
\caption{\it
Illustration of Theorem~\ref{lem_6_10} with $c = n = 100$.
$\lambda_n = \mbox{\text{\rm{0.94419E-18}}}$.
See Experiment 7.
}
\label{t:test85}
\end{table}
%%%%%%%%%%%
\paragraph{Experiment 7.}
\label{sec_exp7}
In this experiment, we illustrate Theorem~\ref{lem_6_10}
in Section~\ref{sec_head}. We proceed as follows.
We choose the band limit and the prolate index to be,
respectively, $c=100$ and $n=100$.
We evaluate $\chi_n$ and $\lambda_n$,
by using the algorithms of 
Sections~\ref{sec_evaluate_beta},~\ref{sec_evaluate_lambda},
respectively (in double precision). Then, we compute
the first 40 roots $x_1, \dots, x_{40}$ of $\psi_n$ in $(1, \infty)$,
by using the algorithm of Section~\ref{sec_roots_outside}
(in extended precision). For each such root $x_k$, we evaluate
$\psi_n'(x_k)$ by using the algorithm of Section~\ref{sec_dpsi_outside}
(in extended precision).

For each $k=1,3,5,\dots,39$, we evaluate 
\begin{align}
\max_{-1 \leq t \leq 1}
\left|
\frac{1}{(t-x_k)\cdot\psi_n'(x_k)} +
\frac{1}{(t-x_{k+1})\cdot\psi_n'(x_{k+1})}
\right|
\label{eq_ex7_two}
\end{align}
(it turns out that the maximum is attained
at $t=1$.) Then, we evaluate the upper bound on 
\eqref{eq_ex7_two}, provided by Theorem~\ref{lem_6_10}.

We display the results of the experiment in
Table~\ref{t:test85}.
The first column contains the index $k$ of the root 
$x_k$ of $\psi_n$ in $(1, \infty)$.
The second column contains the quantity \eqref{eq_ex7_two}.
The third column contains the upper bound on \eqref{eq_ex7_two},
provided by \eqref{eq_6_10} in Theorem~\ref{lem_6_10}.
The last column contains the ratio of the third column to the second column.

We observe that the quantity \eqref{eq_ex7_two} (in the second column)
decreases with $k$. On the other hand, the bound on \eqref{eq_ex7_two}
(in the third column) gets tighter as $k$ increases from 1 to 11,
and then deteriorates, as $k$ increases further on, roughly linearly in $k$.
The latter observation is not surprising, since
\begin{align}
|\lambda_n| \cdot
  \int_x^y \frac{ (z+1)^2 \; dz}{\brk{z^2-\left(\chi_n/c^2\right)}^{3/2}}
\sim \frac{\pi \cdot |\lambda_n|}{c \cdot x_k}, \quad
k \to \infty,
\label{eq_ex7_limit}
\end{align}
due to Theorem~\ref{thm_spacing} in Section~\ref{sec_first_order},
while, for sufficiently large $k$,
\begin{align}
\left|
\frac{1}{(t-x_k)\cdot\psi_n'(x_k)} +
\frac{1}{(t-x_{k+1})\cdot\psi_n'(x_{k+1})}
\right| \leq \frac{ 20 \cdot \pi \cdot |\lambda_n| }{ x_k^2},
\label{eq_ex7_asym}
\end{align}
due to Theorem~\ref{lem_two_bound} in Section~\ref{sec_tail}.
In other words, the upper bound on \eqref{eq_ex7_two}, provided
by Theorem~\ref{lem_6_10}, is of a wrong order ($O(x_k^{-1})$ instead
of $O(x_k^{-2})$). In particular, it can be used only to bound
the head of the convergent series 
\begin{align}
\left| \sum_{k=1}^{\infty} \frac{1}{(t-x_k) \cdot \psi_n'(x_k)}
\right|.
\label{eq_ex7_series}
\end{align}
Of course, this is precisely how Theorem~\ref{lem_6_10} is used
(see the proof of Theorem~\ref{lem_head_main} in Section~\ref{sec_head}
and the proof of Theorem~\ref{thm_head_tail} in Section~\ref{sec_head_tail}).
%%%%%%%%%%%%%%%%%%%%%
\begin{figure} [htbp]
\begin{center}
\includegraphics[width=12cm, bb=81   227   529   564, clip=true]
{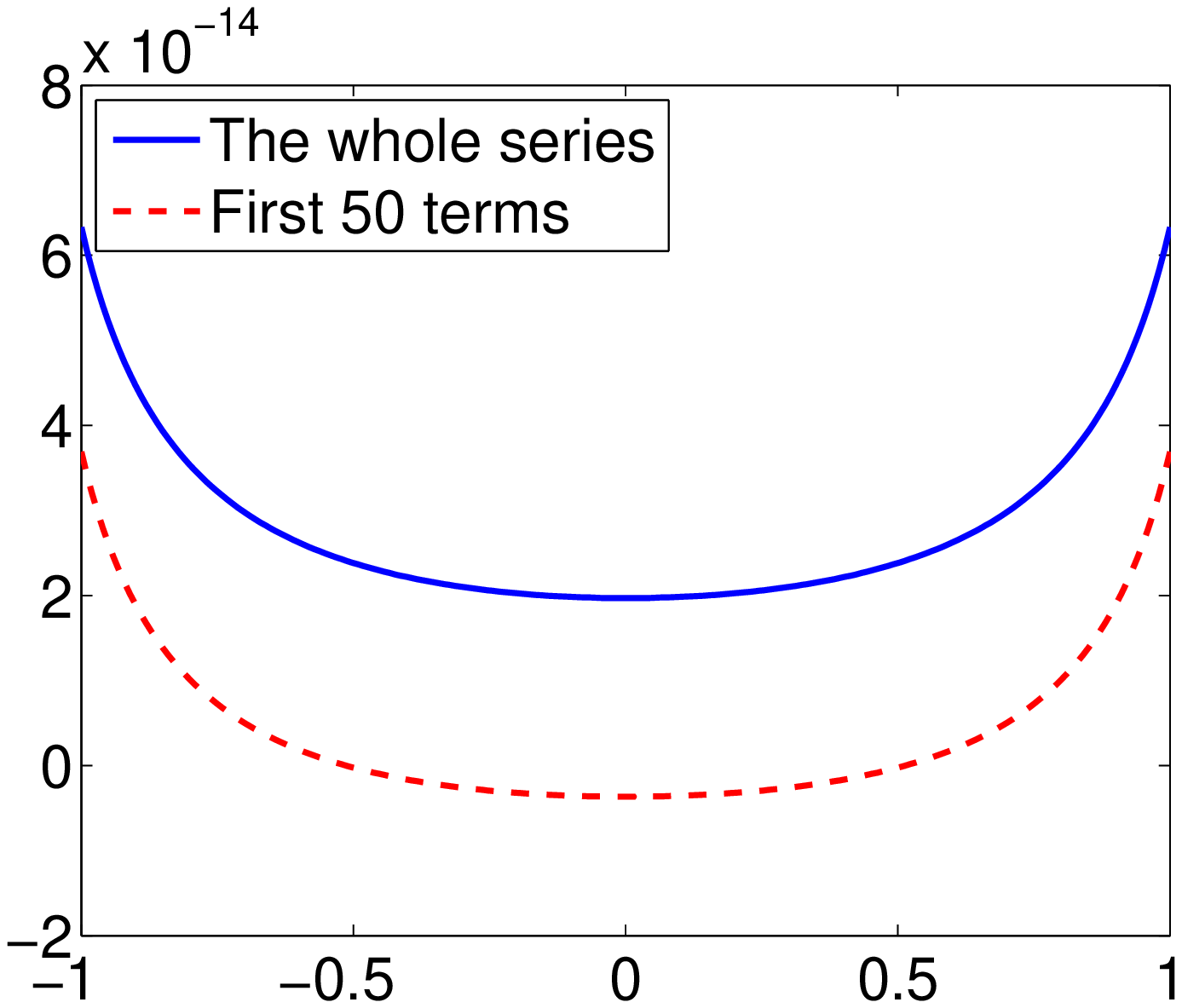}
\caption
{\it
Illustration of Theorems~\ref{thm_complex_summary} with $c = 100$, $n = 80$.
$|\lambda_n| = \mbox{\text{\rm{0.58925E-07}}}$.
}
\label{fig:test86a}
\end{center}
\end{figure}
%%%%%%%%%%%%
%%%%%%%%%%%%%%%%%%%%%
\begin{figure} [htbp]
\begin{center}
\includegraphics[width=12cm, bb=81   227   529   564, clip=true]
{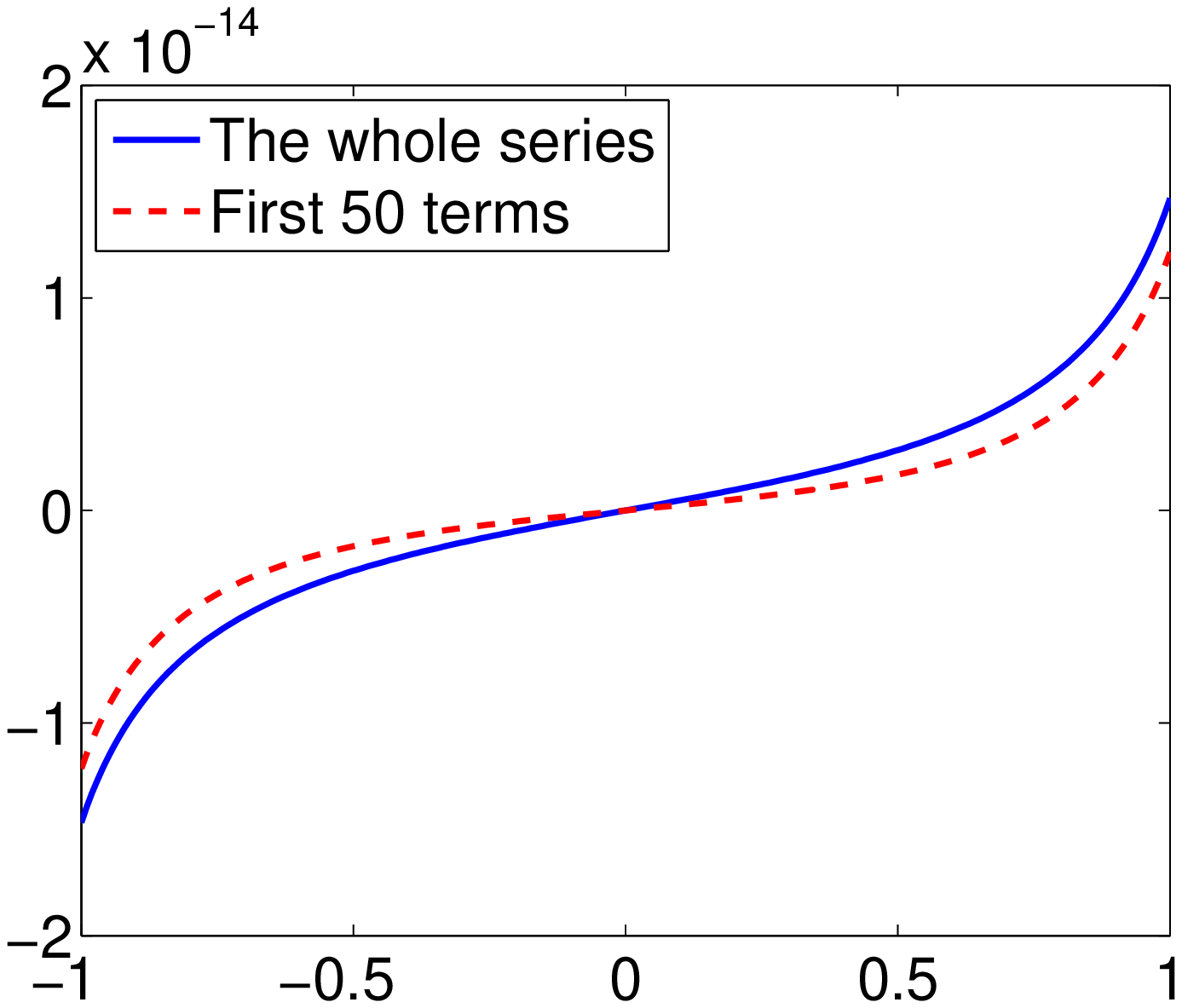}
\caption
{\it
Illustration of Theorem~\ref{thm_complex_summary} with $c = 100$, $n = 81$.
$|\lambda_n| = \mbox{\text{\rm{0.19431E-07}}}$.
}
\label{fig:test86b}
\end{center}
\end{figure}
%%%%%%%%%%%%

%%%%%%%%%%%%%%%%%%%
\begin{table}[htbp]
\begin{center}
\begin{tabular}{c|c|c|c|c|c}
$c$   &
$n$   &
$\|I\|_{\infty}$ &
$ \abrk{\lambda_n} $ &
$ \abrk{\lambda_n} / \| I \|_{\infty} $ &
$I_{\max}$ \\[1ex]
\hline
100 & 80
& 0.99408E-08 & 0.58925E-07 & 0.59276E+01 & 0.55502E+03  \\
100 & 81
& 0.28195E-08 & 0.19431E-07 & 0.68914E+01 & 0.58207E+03  \\
100 & 90
& 0.63405E-13 & 0.45487E-12 & 0.71741E+01 & 0.84186E+03  \\
100 & 91
& 0.14648E-13 & 0.12985E-12 & 0.88645E+01 & 0.87239E+03  \\
\hline
200 & 146
& 0.57204E-08 & 0.32856E-07 & 0.57436E+01 & 0.62129E+03  \\
200 & 147
& 0.19902E-08 & 0.12477E-07 & 0.62691E+01 & 0.64480E+03  \\
200 & 158
& 0.21537E-13 & 0.15123E-12 & 0.70219E+01 & 0.91959E+03  \\
200 & 159
& 0.64626E-14 & 0.51123E-13 & 0.79107E+01 & 0.94591E+03  \\
\hline
400 & 274
& 0.15108E-07 & 0.80630E-07 & 0.53369E+01 & 0.67438E+03  \\
400 & 275
& 0.61774E-08 & 0.34713E-07 & 0.56193E+01 & 0.69478E+03  \\
400 & 288
& 0.47053E-13 & 0.31193E-12 & 0.66293E+01 & 0.97598E+03  \\
400 & 289
& 0.17000E-13 & 0.12189E-12 & 0.71703E+01 & 0.99872E+03  \\
\hline
800 & 530
& 0.18269E-07 & 0.91984E-07 & 0.50351E+01 & 0.77801E+03  \\
800 & 531
& 0.83405E-08 & 0.43433E-07 & 0.52075E+01 & 0.79612E+03  \\
800 & 546
& 0.46822E-13 & 0.29701E-12 & 0.63434E+01 & 0.10833E+04  \\
800 & 547
& 0.19631E-13 & 0.12945E-12 & 0.65942E+01 & 0.11033E+04  \\
\end{tabular}
\end{center}
\caption{\it
Illustration of Theorem~\ref{thm_complex_summary}. See Experiment 8.
}
\label{t:test86}
\end{table}
%%%%%%%%%%%
\paragraph{Experiment 8.}
\label{sec_exp8}
In this experiment, we illustrate Theorem~\ref{thm_complex_summary}
in Section~\ref{sec_head_tail}.
We proceed as follows. We choose, more or less arbitrarily,
the band limit $c$ and the prolate index $n$.
Then, we evaluate $\chi_n$ and $\lambda_n$, by using
the algorithms of Section~\ref{sec_evaluate_beta},~\ref{sec_evaluate_lambda},
respectively (in double precision). 
Next, we find the roots $t_1, \dots, t_n$ of $\psi_n$ in the interval
$(-1,1)$, by using the algorithm of Section~\ref{sec_evaluate_nodes}
(in double precision). For each root $t_i$, we compute $\psi_n'(t_i)$.

Suppose now that the function $I: [-1,1] \to \Rc$ is defined
via \eqref{eq_complex_it} in Theorem~\ref{thm_complex}.
We evaluate $I$ at $3\cdot(n+1)$ points $z_1,\dots,z_{3(n+1)}$
in the interval $[-1,1]$. The points are chosen in such a way
that, if $t_k < z_j < t_{k+1}$ for some $j,k$, then
\begin{align}
\frac{1}{3} \leq \frac{z_j-t_k}{t_{k+1}-z_j} \leq 3.
\label{eq_ex8_z}
\end{align}
In other words, no point $z_j$ is ``too close'' to any root of $\psi_n$
in $(-1,1)$. For each $j=1, \dots, 3\cdot(n+1)$, we evaluate $I(z_j)$
in extended precision.
\begin{remark}
\label{rem_exp8}
For each $-1 \leq t \leq 1$, we expect $I(t)$ to be of order $|\lambda_n|$,
due to Theorems~\ref{thm_complex}, \ref{thm_complex_summary}. On the other hand,
suppose that $-1 \leq t \leq 1$, and $t_k$ is the closest
root of $\psi_n$ to $t$. Then,
\begin{align}
\frac{1}{\psi_n(t)} = \frac{1}{(t-t_k) \cdot \psi_n'(t_k)} + O(1).
\label{eq_exp8_psi}
\end{align}
Therefore, in the evaluation of $\psi_n(t)$, we expect to lose
roughly
\begin{align}
\log_{10}\left( \frac{1}{| \psi_n'(t_k) \cdot (t-t_k) \cdot \lambda_n|}\right)
\label{eq_exp8_lose}
\end{align}
decimal digits. In other words, this calculation is rather inaccurate.
However, since we need it only to illustrate the analysis,
we were satisfied when we got at least two decimal digits,
and did not make any attempts to enhance the accuracy.
\end{remark}
On the other hand, we compute the first 50 roots 
$x_1, \dots, x_{50}$ of $\psi_n$ in $(1,\infty)$, and, for each
such root $x_j$, we evaluate $\psi_n'(x_j)$. These calculations
are based on the algorithms of 
Sections~\ref{sec_dpsi_outside},~\ref{sec_roots_outside}.
Then, for each $z_j$, we evaluate the sum
\begin{align}
I_{50}(z_j) = \sum_{k=1}^{50} \left(
\frac{1}{\psi_n'(x_k) \cdot (z_j-x_k)} +
\frac{1}{\psi_n'(-x_k) \cdot (z_j+x_k)}
\right).
\label{eq_exp8_i50}
\end{align}
We display the results of the experiment
in Figures~\ref{fig:test86a},~\ref{fig:test86b},
for $c=100$, $n=90$ and $c=100$, $n=91$, respectively.
On each of these figures, we plot the function $I$,
defined via \eqref{eq_complex_it} in Theorem~\ref{thm_complex}
(blue solid line) and the function $I_{50}$,
defined via \eqref{eq_exp8_i50} (red dashed line).

We observe that, in both figures, the maximum of both $I$ and $I_{50}$
is attained at the end points of the interval.
Also, we observe that the values of $I$ and $I_{50}$ are of order
$|\lambda_n|$, as expected; also, the functions appear,
at least by eye, to be
well approximated by polynomials of order up to 3. In other words,
the reciprocal of $\psi_n$ seems to be approximated up to an
error of order $|\lambda_n|$ by a rational function with $n$ poles,
as asserted in Theorems~\ref{thm_complex}, \ref{thm_complex_summary}.

We display additional results of this experiment in
Table~\ref{t:test86}.
This table has the following structure.
The first and second column contain the band limit $c$ and
the prolate index $n$, respectively.
The third column contains the maximum of the absolute value of the
function $I$
in the interval $[-1,1]$, i.e.
\begin{align}
\| I \|_{\infty} = \max\left\{ |I(t)| \; : \; -1 \leq t \leq 1 \right\},
\label{eq_exp8_inf}
\end{align}
where $I$ is defined via \eqref{eq_complex_it} in Theorem~\ref{thm_complex}.
The fourth column contains $|\lambda_n|$.
The fifth column contains the ratio $|\lambda_n| / \|I\|_{\infty}$.
The last column contains $I_{\max}$, defined via
\eqref{eq_complex_imax} in Theorem~\ref{thm_complex}.

We make the following observations from Table~\ref{t:test86}.
First, $|\lambda_n|$ alone is already an upper bound
on $\|I\|_{\infty}$. Moreover, for a fixed band limit $c$,
the ratio $|\lambda_n|/\|I\|_{\infty}$ increases
as $n$ grows. For all the values of $c,n$ in Table~\ref{t:test86},
this ratio varies between 5 and 9.
On the other hand, $I_{\max}$ varies between 500 and 1000.
Moreover, $I_{\max}$ increases with $n$, for each fixed band limit $c$.
In other words, the upper bound $|\lambda_n| \cdot I_{\max}$
on $\|I\|_{\infty}$, established in Theorem~\ref{thm_complex},
deteriorates as $n$ increases. Moreover, the factor $I_{\max}$
in \eqref{eq_complex} of Theorem~\ref{thm_complex}
appears to be unnecessary. The main source of inaccuracy is 
Theorem~\ref{lem_6_10}
in Section~\ref{sec_head}, which provides a relatively poor upper
bound on the expressions of the form \eqref{eq_exp8_i50}
(see Figures~\ref{fig:test86a},~\ref{fig:test86b}
and Experiment 7 above).

Nevertheless, due to the fast decay of $|\lambda_n|$ with $n$,
the estimates of Theorem~\ref{thm_complex}, albeit
somewhat loose, are sufficient for the purposes of this paper
(see the analysis of the quadrature error in Section~\ref{sec_quad},
and also Experiment 14 in Section~\ref{sec_exp14} below).

%%%%%%%%%%%%%%%%%%%%%%%%%%%%%%
\subsubsection{Illustration of Results from Section~\ref{sec_quad}}
%%%%%%%%%%%%%%%%%%%%%%%%%%%%%%
%%%%%%%%%%%%%%%%%%%
\paragraph{Experiment 9.}
\label{sec_exp9}
In this numerical experiment, we illustrate 
Theorem~\ref{thm_quad_coef} 
in Section~\ref{sec_phi}. We proceed as follows.
We choose, more or less arbitrarily,
the band limit $c$, the prolate index $n$ and 
the root index $1 \leq j \leq n$.
Then, we evaluate $\lambda_n$ and the roots $t_1, \dots, t_n$
of $\psi_n$ in $(-1,1)$, by using the algorithms
of Sections~\ref{sec_evaluate_lambda},~\ref{sec_evaluate_nodes},
respectively. We use $10\cdot n$ Gaussian nodes to evaluate 
\begin{align}
A_{n,j} = \int_{-1}^1 \frac{\psi_n(t) \; dt}{t - t_j}
\label{eq_exp9_anj}
\end{align}
and
\begin{align}
B_{n,j} = ic\lambda_n \cdot \Psi_n(1,t_j),
\label{eq_exp9_bnj}
\end{align}
where $\Psi_n(1,t_j)$ is defined via \eqref{eq_quad_big_psi}
in Theorem~\ref{thm_quad_coef}.
We  observe that $A_{n,j}$ and $B_{n,j}$ appear on the right-hand
side of \eqref{eq_quad_coef} in Theorem~\ref{thm_quad_coef}.

Next, for each integer $m=0,1,\dots,n-1$, we 
use the same Gaussian quadrature to
evaluate
\begin{align}
\int_{-1}^1 \frac{\psi_n(t) \cdot \psi_m(t) \; dt}{t-t_j}.
\label{eq_exp9_lhs}
\end{align}
In addition, for each integer $m=0,1,\dots,n-1$, we compute
\begin{align}
\frac{ |\lambda_m|^2 \cdot \psi_m(t_j)}{ |\lambda_m|^2 - |\lambda_n|^2 },
\label{eq_exp9_psim}
\end{align}
by using the algorithms of 
Sections~\ref{sec_evaluate_beta},~\ref{sec_evaluate_lambda}.
All the calculations are carried out in double precision.
%%%%%%%%%%%%
\begin{table}[htbp]
\begin{center}
\begin{tabular}{c|c|c|c|c|c}
$c$ &
$n$ &
$j$ &
$|\lambda_n|$ &
$A_{n,j}$ &
$B_{n,j}$ \\[1ex]
\hline
10  & 20  & 13  & 0.11487E-09 & -.25341E+01 & -.69171E-11 \\
500 & 340 & 226 & 0.27418E-09 & -.19569E+01 & -.17690E-09 \\
\end{tabular}
\end{center}
\caption{\it
Illustration of Theorem~\ref{thm_quad_coef}.
See Experiment 9.
}
\label{t:test87}
\end{table}
%%%%%%%%%%%
%%%%%%%%%%%%%%%%%%%%%
\begin{figure} [htbp]
\begin{center}
\includegraphics[width=12cm, bb=68   218   542   574, clip=true]
{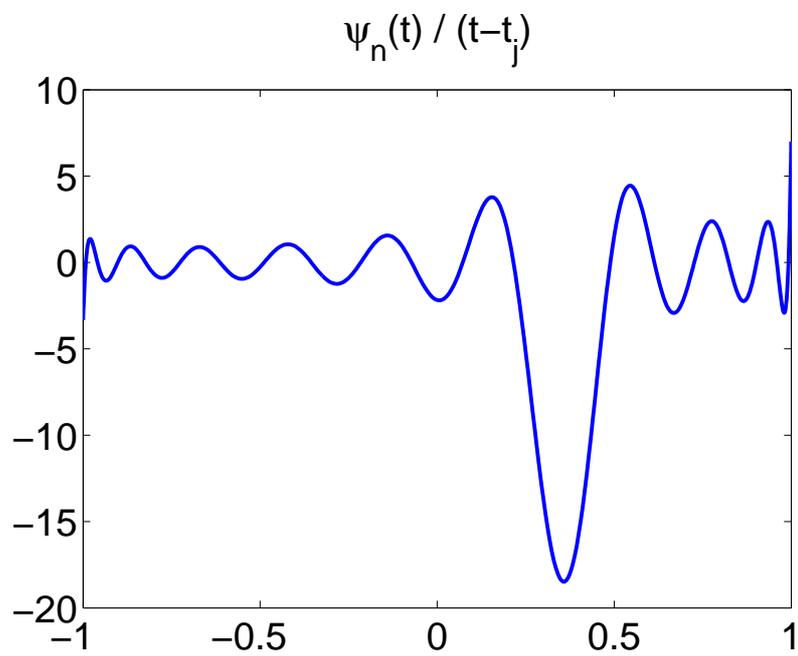}
\caption
{\it
The graph of $\psi_n(t)/(t-t_j)$ with $c = 10$, $n = 20$ and $j = 13$.
Corresponds to Table~\ref{t:test87a}. See Experiment 9.
}
\label{fig:test87}
\end{center}
\end{figure}
%%%%%%%%%%%%
\begin{table}[htbp]
\begin{center}
\begin{tabular}{c|c|c|c}
$m$   &
$\int_{-1}^1 \frac{ \psi_n(t) \psi_m(t) dt }{ t - t_j}$   &
$\frac{ \abrk{\lambda_m}^2 \psi_m(t_j) }
      { \abrk{\lambda_m}^2 - \abrk{\lambda_n}^2}$ &
$E_m$ \\[1ex]
\hline
 0
& -.18363E+01 & 0.72463E+00 & -.15543E-14  \\
 1
& -.28929E+01 & 0.11416E+01 & -.53291E-14  \\
 2
& -.18299E+01 & 0.72208E+00 & -.51070E-14  \\
 3
& 0.73457E+00 & -.28987E+00 & 0.12212E-14  \\
 4
& 0.19270E+01 & -.76041E+00 & 0.37748E-14  \\
 5
& 0.40316E+00 & -.15909E+00 & 0.21094E-14  \\
 6
& -.14464E+01 & 0.57078E+00 & -.22204E-14  \\
 7
& -.10263E+01 & 0.40498E+00 & 0.10658E-13  \\
 8
& 0.11062E+01 & -.43654E+00 & 0.95479E-14  \\
 9
& 0.17030E+01 & -.67204E+00 & 0.99920E-14  \\
 10
& -.23035E+00 & 0.90899E-01 & -.26645E-14  \\
 11
& -.19061E+01 & 0.75217E+00 & -.44409E-15  \\
 12
& -.91510E+00 & 0.36111E+00 & 0.16653E-14  \\
 13
& 0.13774E+01 & -.54355E+00 & 0.11990E-13  \\
 14
& 0.18002E+01 & -.71037E+00 & 0.37748E-14  \\
 15
& -.23786E+00 & 0.93863E-01 & -.97422E-14  \\
 16
& -.19723E+01 & 0.77830E+00 & -.11546E-13  \\
 17
& -.10566E+01 & 0.41697E+00 & -.15987E-13  \\
 18
& 0.12849E+01 & -.50705E+00 & -.19984E-14  \\
 19
& 0.19509E+01 & -.76986E+00 & -.42188E-14  \\
\end{tabular}
\end{center}
\caption{\it
Illustration of Theorem~\ref{thm_quad_coef} with
$c = 10$, $n = 20$ and $j = 13$.
See Experiment 9.
}
\label{t:test87a}
\end{table}
%%%%%%%%%%%
%%%%%%%%%%%%%%%%%%%
\begin{table}[htbp]
\begin{center}
\begin{tabular}{c|c|c|c}
$m$   &
$\int_{-1}^1 \frac{ \psi_n(t) \psi_m(t) dt }{ t - t_j}$   &
$\frac{ \abrk{\lambda_m}^2 \psi_m(t_j) }
      { \abrk{\lambda_m}^2 - \abrk{\lambda_n}^2}$ &
$E_m$ \\[1ex]
\hline
 0
& 0.60926E-12 & 0.31814E-12 & 0.33872E-15  \\
 20
& 0.43712E-01 & 0.22813E-01 & 0.52666E-14  \\
 40
& -.32804E+01 & -.17120E+01 & -.84377E-13  \\
 60
& 0.85749E+00 & 0.44751E+00 & -.60729E-13  \\
 80
& -.14190E+01 & -.74055E+00 & -.91926E-13  \\
 100
& 0.84651E+00 & 0.44178E+00 & -.28089E-13  \\
 120
& 0.35414E+00 & 0.18482E+00 & 0.47351E-13  \\
 140
& 0.53788E+00 & 0.28071E+00 & -.21316E-13  \\
 160
& -.17111E+01 & -.89302E+00 & -.35749E-13  \\
 180
& -.93523E+00 & -.48808E+00 & 0.31863E-13  \\
 200
& -.30219E+00 & -.15771E+00 & 0.48406E-13  \\
 220
& -.51322E+00 & -.26784E+00 & 0.45852E-13  \\
 240
& -.12216E+01 & -.63753E+00 & -.11546E-13  \\
 260
& -.10503E+01 & -.54811E+00 & -.82379E-13  \\
 280
& 0.93142E+00 & 0.48609E+00 & 0.84377E-14  \\
 300
& -.55310E-02 & -.28865E-02 & 0.50818E-13  \\
 320
& 0.11601E+00 & 0.60544E-01 & -.13105E-12  \\
 339
& 0.14218E+01 & 0.74200E+00 & -.94369E-13  \\
\end{tabular}
\end{center}
\caption{\it
Illustration of Theorem~\ref{thm_quad_coef} with
$c = 500$, $n = 340$ and $j = 226$.
See Experiment 9.
}
\label{t:test87b}
\end{table}
%%%%%%%%%%%

We display the results of the experiment in Figure~\ref{fig:test87}
and Tables~\ref{t:test87a},~\ref{t:test87b}.
In Figure~\ref{fig:test87}, we plot the function $\psi_n(t)/(t-t_j)$,
corresponding to $c=10$, $n=20$, and $j=13$, over the
interval $(-1,1)$.
We observe that this function has $n-1$ roots in $(-1,1)$:
all the roots of $\psi_n$ except for $t_j$. Obviously, the value
of this function at $t_j$ is $\psi_n'(t_j)$.

In Tables~\ref{t:test87},~\ref{t:test87a},~\ref{t:test87b}, 
we display the results
of the experiment, corresponding to
$c=10$, $n=20$, $j=13$ and
$c=500$, $n=340$ and $j=226$, respectively.
Table~\ref{t:test87} contains the values of the parameters
$c,n,j$, as well as the quantities $A_{n,j}$, $B_{n,j}$,
defined, respectively, via \eqref{eq_exp9_anj},
\eqref{eq_exp9_bnj} above.
Tables~\ref{t:test87a},~\ref{t:test87b} have the following structure.
The first column contains the parameter $m$
(an integer between 0 and $n-1$).
The second column contains \eqref{eq_exp9_lhs}
(the left-hand side of \eqref{eq_quad_coef} in
Theorem~\ref{thm_quad_coef}); in other words,
this is the inner product of $\psi_n(t)/(t-t_j)$ with $\psi_m$.
The third column contains \eqref{eq_exp9_psim}
(appears on the right-hand side of \eqref{eq_quad_coef}).
The last column contains the absolute error $E_m$ of the calculation
of \eqref{eq_exp9_psim}, defined via
\begin{align}
E_m =
\int_{-1}^1 \frac{\psi_n(t) \cdot \psi_m(t) \; dt}{t-t_j} - 
\frac{ |\lambda_m|^2 \cdot \psi_m(t_j)}{ |\lambda_m|^2 - |\lambda_n|^2 }
\cdot (A_{n,j}+B_{n,j})
\label{eq_exp9_em}
\end{align}
(obviously, $E_m$ would be equal to zero in exact arithmetics,
due to Theorem~\ref{thm_quad_coef}).

We make the following observations from
Tables~\ref{t:test87},~\ref{t:test87a},~\ref{t:test87b}.
As expected, $A_{n,j}$ is significantly larger than $B_{n,j}$
(by a factor of order $|\lambda_n|^{-1}$).
In other words,
\begin{align}
\int_{-1}^1 \frac{\psi_n(t) \cdot \psi_m(t) \; dt}{t-t_j} =
\frac{ |\lambda_m|^2 \cdot \psi_m(t_j)}{ |\lambda_m|^2 - |\lambda_n|^2 }
\cdot
\int_{-1}^1 \frac{\psi_n(t) \; dt}{t-t_j} \cdot
\left(1 + 
O\left( |\lambda_n| \right) \right)
\label{eq_exp9_approx}
\end{align}
(see Theorem~\ref{thm_quad_coef} and 
\eqref{eq_exp9_anj},~\eqref{eq_exp9_bnj},~\eqref{eq_exp9_lhs},~\eqref{eq_exp9_psim}
above). Also, in each of Tables~\ref{t:test87a},~\ref{t:test87b},
all the quantities in the second and third column
are roughly of the same order of magnitude
(except for the first row in Table~\ref{t:test87b}).
We also observe that the numerical evaluations of the left-hand side
and the right-hand side of \eqref{eq_quad_coef}
in Theorem~\ref{thm_quad_coef} agree up to an absolute
error of order $\approx 10^{-14}$.

%%%%%%%%%%%%%%%%%%%
\begin{table}[htbp]
\begin{center}
\begin{tabular}{c|c|c|c}
$c$   &
$n$   &
$|\lambda_n|$   &
$|P_{n,n-2}|$   \\[1ex]
\hline
1000 & 670 & 0.93659E-11 & 0.49177E-03 \\
1000 & 690 & 0.73056E-18 & 0.43907E-03 \\
1000 & 710 & 0.15947E-25 & 0.40076E-03 \\
\end{tabular}
\end{center}
\caption{\it
Illustration of Theorem~\ref{lem_quad_cp}.
Corresponds to Figure~\ref{fig:test88}.
}
\label{t:test88a}
\end{table}
%%%%%%%%%%%%%%%%%%%%%
\begin{figure} [htbp]
\begin{center}
\includegraphics[width=12cm, bb=68   218   542   574, clip=true]
{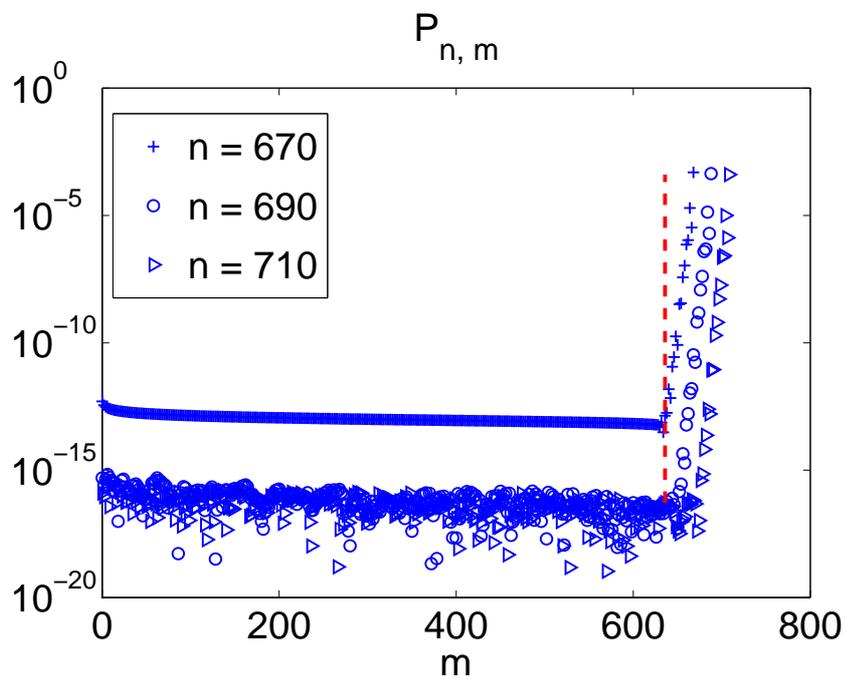}
\caption
{\it
Plot of $P_{n, m}$ \eqref{eq_quad_p} with $c = 1000$ and
$n = 670$ (crosses), $n = 690$ (circles), $n = 710$ (triangles).
The value $m = 2c/\pi$ is marked with a dashed line.
See Experiment 10.
}
\label{fig:test88}
\end{center}
\end{figure}
%%%%%%%%%%%%
\paragraph{Experiment 10.}
\label{sec_exp10}
In this experiment, we illustrate
Theorem~\ref{lem_quad_cp} in Section~\ref{sec_quad_error}.
We proceed as follows.
We choose, more or less arbitrarily,
the band limit $c$ and the prolate index $n$.
Then, we evaluate $\lambda_n$, using the algorithm
of Section~\ref{sec_evaluate_lambda}
(in double precision). Next, for each integer $0 \leq m \leq n-1$, 
we evaluate $P_{n,m}$, defined
via \eqref{eq_quad_p} in Theorem~\ref{thm_quad}
in Section~\ref{sec_quad_error},
by using the algorithms of 
Sections~\ref{sec_evaluate_beta},~\ref{sec_evaluate_nodes}
(in double precision).
We observe that, due to Corollary~\ref{cor_p_parity}
in Section~\ref{sec_quad_error},
it suffices to consider only even values of $m$
(since $P_{n,m}=0$ if $m$ is odd).

%%%%%%%%%%%%%%%%%%%
\begin{table}[htbp]
\begin{center}
\begin{tabular}{c|c|c|c|c}
$c$   &
$n$   &
$\abrk{\lambda_n}$   &
$\underset{0 \leq m < n}{\max} |P_{n,m}|$ &
$c \cdot \underset{0 \leq m < n}{\max} |P_{n,m}| $ \\[1ex]
\hline
50 & 47
& 0.26917E-07 & 0.81444E-02 & 0.40722E+00  \\
50 & 53
& 0.72096E-11 & 0.72290E-02 & 0.36145E+00  \\
50 &  57
& 0.19830E-13 & 0.67353E-02 & 0.33676E+00  \\
\hline
100 & 81
& 0.19431E-07 & 0.48065E-02 & 0.48065E+00  \\
100 &  87
& 0.18068E-10 & 0.44412E-02 & 0.44412E+00  \\
100 &  93
& 0.10185E-13 & 0.41418E-02 & 0.41418E+00  \\
\hline
250  &  179
& 0.18854E-07 & 0.22730E-02 & 0.56825E+00  \\
250 &  186
& 0.22556E-10 & 0.14014E-02 & 0.35035E+00  \\
250 &  193
& 0.17851E-13 & 0.20475E-02 & 0.51188E+00  \\
\hline
500  & 339
& 0.40938E-07 & 0.12600E-02 & 0.63000E+00  \\
500 &  348
& 0.20575E-10 & 0.85073E-03 & 0.42537E+00  \\
500 &  355
& 0.39965E-13 & 0.11550E-02 & 0.57751E+00  \\
\hline
1000 &  659
& 0.38241E-07 & 0.68143E-03 & 0.68143E+00  \\ 
1000 &  668
& 0.44256E-10 & 0.49838E-03 & 0.49838E+00  \\
1000 &  677
& 0.35933E-13 & 0.63339E-03 & 0.63339E+00  \\
\hline
2000 & 1297
& 0.41740E-07 & 0.36453E-03 & 0.72906E+00  \\
2000 & 1307
& 0.47570E-10 & 0.35192E-03 & 0.70385E+00  \\
2000 &  1317
& 0.39064E-13 & 0.34212E-03 & 0.68424E+00  \\
\hline
4000 &  2572
& 0.33682E-07 & 0.16247E-03 & 0.64987E+00  \\
4000 & 2583
& 0.37417E-10 & 0.18703E-03 & 0.74813E+00  \\
4000 & 2594
& 0.30728E-13 & 0.14902E-03 & 0.59608E+00  \\
\end{tabular}
\end{center}
\caption{\it
Illustration of Theorem~\ref{lem_quad_cp}.
See Experiment 10.
}
\label{t:test88}
\end{table}
%%%%%%%%%%%
We display the results of the experiment in 
Tables~\ref{t:test88a},~\ref{t:test88}
and in Figure~\ref{fig:test88}.
In Figure~\ref{fig:test88}, we plot $|P_{n,m}|$ as a function
of even integer $m$ on the logarithmic scale
for $c=1000$ and three choices of $n$,
namely,
$n=670$ (pluses), $n=690$ (circles), and $n=710$ (triangles).
The value $2c/\pi$ is marked with a red dashed line.
In Table~\ref{t:test88a}, we display the quantities $|\lambda_n|$
and $|P_{n,n-2}|$, corresponding to Figure~\ref{fig:test88}.

We make the following observations from Figure~\ref{fig:test88},
Table~\ref{t:test88a} and some additional numerical experiments. 
First, $|P_{n,m}| < |\lambda_n|$ for all $m<2c/\pi$
(obviously, in Figure~\ref{fig:test88} we see this phenomenon
only for $n=670$, since the calculations are carried
out in double precision; for $n=690,710$ and $m<2c/\pi$,
$|P_{n,m}| < 10^{-15}$).
On the other hand, for even $2c/\pi < m < n$,
we observe that $|P_{n,m}|$ grows roughly exponentially
with $m$, reaching its maximum at $m=n-2$. This maximum
is approximately $5 \cdot 10^{-4}$,
for all the three values of $n$ (see Table~\ref{t:test88a}).
However, Theorem~\ref{lem_quad_cp} asserts that,
for all $m < n$,
\begin{align}
|P_{n,m}| \leq \frac{ \sqrt{32} n^2 }{ c }.
\label{eq_exp10_bound}
\end{align}
In other words, Theorem~\ref{lem_quad_cp} overestimates $|P_{n,m}|$
by a factor of order $n^2$.

In Table~\ref{t:test88}, we display some additional
results of this experiment.
This table has the following structure.
The first and second column contain, respectively,
the band limit $c$ and the prolate index $n$.
The third column contains $|\lambda_n|$. The fourth column contains
\begin{align}
\max\left\{ |P_{n,m} | \; : \; 0 \leq m \leq n-1 \right\}.
\label{eq_exp10_max}
\end{align}
The last column contains the value \eqref{eq_exp10_max},
multiplied by the band limit $c$ (i.e. the product of the first
and fourth columns).

We make the following observations from Table~\ref{t:test88}
and some additional experiments. First, for each of the seven
values of $c$, the three indices $n$ were chosen in such
a way that $|\lambda_n|$ is between
$10^{-14}$ and $10^{-7}$. Even though $c$ varies between $50$
(the first three rows) and $4000$ (the last three rows),
the values in the last column are roughly of the same order,
for all the choices of $c$ and $n$. Moreover, these
values are always between $0.3$ and $0.75$.
This observation seems to indicate that Theorem~\ref{lem_quad_cp}
overestimates this quantity by $O(n^2)$
(see also \eqref{eq_exp10_bound} and Figure~\ref{fig:test88}).

Additional observations seem to indicate that the maximum
in \eqref{eq_exp10_max} is always attained at the largest even $m$
between zero and $n-1$
(as in Figure~\ref{fig:test88}).
Also, for this value of $m$, all the summands
\begin{align}
\frac{ \psi_m(t_j) \Psi_{n,j}(1) }{ \psi_n'(t_j)} 
\label{eq_exp10_summand}
\end{align}
in \eqref{eq_quad_p} have been observed to have the same sign
for all $j=1,\dots,n$.
Thus, the inaccuracy of the bound in Theorem~\ref{lem_quad_cp}
is due to overestimation of the summands \eqref{eq_exp10_summand},
rather than due to cancellation of summands with opposite signs.

%%%%%%%%%%%%%%%%%%%%%%%%%%%%%%%%%%%%%%%%%%%%%%
\subsection{Performance of the Quadrature}
In this subsection, we report the results of numerical
experiments illustrating the performance
of the quadrature, defined in Definition~\ref{def_quad},
and whose properties are studied
in Section~\ref{sec_quad}. 

%%%%%%%%%%%%%%%%%%%%%%%%%%%%%%
\subsubsection{Quadrature Error and its Relation to $\abrk{\lambda_n}$}
\label{sec_quad_err_num}
%%%%%%%%%%%%%%%%%%%
\begin{table}[htbp]
\begin{center}
\begin{tabular}{c|c|c|c|c|c}
$m$   &
$\lambda_m \psi_m(0)$ &
$S_m$ &
$\int \psi_m \brk{1-\sum \varphi_j} $ &
$\xi_m$ &
error \\[1ex]
\hline
 0
& 0.70669E+00 & 0.70669E+00 & 0.20856E-16 & 0.79599E-12 & -.55511E-15  \\
 2
& 0.49581E+00 & 0.49581E+00 & 0.77098E-15 & 0.29426E-10 & -.88818E-15  \\
 4
& 0.42581E+00 & 0.42581E+00 & 0.97200E-15 & 0.37098E-10 & -.23870E-14  \\
 6
& 0.38527E+00 & 0.38527E+00 & -.83346E-15 & -.31810E-10 & -.13323E-14  \\
 8
& 0.35695E+00 & 0.35695E+00 & -.10918E-14 & -.41671E-10 & -.99920E-15  \\
 10
& 0.33516E+00 & 0.33516E+00 & 0.25553E-15 & 0.97526E-11 & -.17208E-14  \\
 12
& 0.31730E+00 & 0.31730E+00 & -.25500E-14 & -.97326E-10 & 0.11102E-15  \\
 14
& 0.30201E+00 & 0.30201E+00 & -.35426E-14 & -.13521E-09 & 0.13878E-14  \\
 16
& 0.28844E+00 & 0.28844E+00 & -.20470E-14 & -.78128E-10 & -.16653E-15  \\
 18
& 0.27604E+00 & 0.27604E+00 & -.28733E-13 & -.10967E-08 & 0.42188E-14  \\
 20
& 0.26435E+00 & 0.26435E+00 & -.14073E-12 & -.53714E-08 & 0.90483E-14  \\
 22
& 0.25299E+00 & 0.25299E+00 & 0.26178E-11 & 0.99913E-07 & 0.94924E-14  \\
 24
& 0.24150E+00 & 0.24150E+00 & 0.15530E-10 & 0.59274E-06 & -.66613E-15  \\
 26
& 0.22919E+00 & 0.22919E+00 & -.17315E-09 & -.66085E-05 & -.72997E-14  \\
 28
& 0.21377E+00 & 0.21377E+00 & -.53359E-09 & -.20365E-04 & 0.14710E-14  \\
 30
& 0.18075E+00 & 0.18075E+00 & 0.55489E-08 & 0.21178E-03 & -.51903E-14  \\
 32
& 0.10038E+00 & 0.10038E+00 & -.62071E-08 & -.23690E-03 & -.70915E-14  \\
 34
& 0.27988E-01 & 0.27988E-01 & -.88231E-07 & -.33675E-02 & 0.10113E-13  \\
 36
& 0.49822E-02 & 0.49818E-02 & 0.40165E-06 & 0.15330E-01 & 0.29751E-14  \\
 38
& 0.70503E-03 & 0.70008E-03 & 0.49503E-05 & 0.18894E+00 & -.13444E-13  \\
\end{tabular}
\end{center}
\caption{\it
Illustration of the proof of Theorem~\ref{thm_quad}
with $c = 50$ and $n = 40$. See Experiment 11.
}
\label{t:test89}
\end{table}
%%%%%%%%%%%
\paragraph{Experiment 11.}
\label{sec_exp11}
In this experiment, we illustrate Theorem~\ref{thm_quad}
in Section~\ref{sec_quad_error}.
We proceed as follows. We choose, more or less arbitrarily,
the band limit $c$ and the prolate index $n$.
We evaluate $\lambda_n$ as well as the nodes $t_1, \dots, t_n$
and the weights $W_1, \dots, W_n$ of the quadrature,
defined in Definition~\ref{def_quad}
in Section~\ref{sec_quad}. To do so, we use the algorithms
of Sections~\ref{sec_evaluate_lambda}, 
\ref{sec_evaluate_nodes},
\ref{sec_evaluate_weights}, respectively (in double precision).

Then, we choose an even integer $0 \leq m < n$, 
and evaluate $\lambda_m$, $\psi_m(0)$ and $\psi_m(t_j)$,
for all $j=1,2,\dots,n$, by using the algorithms
of Sections~\ref{sec_evaluate_lambda}, \ref{sec_evaluate_beta}
(in double precision). 
Next, we evaluate $P_{n,m}$, defined via \eqref{eq_quad_p}
in Theorem~\ref{thm_quad} (see Experiment 10 
in Section~\ref{sec_exp10}).
Also, we compute $\| I \|_{\infty}$
(see \eqref{eq_i_inf} in Theorem~\ref{thm_quad} and
\eqref{eq_exp8_inf} in Experiment 8 in
Section~\ref{sec_exp8}). Finally, we evaluate
\begin{align}
\int_{-1}^1 \psi_m(t) \cdot \left(1 - \sum_{j=1}^n \varphi_j(t) \right) dt,
\label{eq_exp11_int1}
\end{align}
where the functions $\varphi_1, \dots, \varphi_n$ are those
of Definition~\ref{def_quad} in Section~\ref{sec_quad}.

We display the results of the experiment in Table~\ref{t:test89}.
The data in this table correspond to $c=50$ and $n=40$.
Table~\ref{t:test89} has the following structure.
The first column contains the even integer parameter $m$,
which varies between $0$ and $n-2$.
The second column contains $\lambda_m  \psi_m(0)$
(we observe that
\begin{align}
\lambda_m \psi_m(0) = \int_{-1}^1 \psi_m(t) \; dt,
\label{eq_exp11_psim}
\end{align}
due to \eqref{eq_prolate_integral} in Section~\ref{sec_pswf}).
The third column contains the quantity $S_m$, defined via
the formula
\begin{align}
S_m = 
\frac{ |\lambda_m|^2 }{ |\lambda_m|^2 - |\lambda_n|^2 } \cdot
\left(
ic\lambda_n P_{n,m} + \sum_{j=1}^n \psi_m(t_j) \cdot W_j
\right).
\label{eq_exp11_third}
\end{align}
The fourth column contains the integral \eqref{eq_exp11_int1}.
The fifth column contains the number $\xi_m$, defined via
the formula
\begin{align}
\xi_m = \frac{1}{\| I \|_{\infty}} 
\int_{-1}^1 \psi_m(t) \cdot \left(1 - \sum_{j=1}^n \varphi_j(t) \right) dt.
\label{eq_exp11_xi}
\end{align}
(We observe that,
due to \eqref{eq_quad_b} in Theorem~\ref{thm_quad},
$\xi_m$ equals to the value in the third column, divided
by $\| I \|_{\infty}$. The latter does not depend on $m$, and
is equal to 0.26201E-04, for $c=50$ and $n=40$.)
The last column contains the difference between the value
in the third column and the sum of the values in the fourth in fifth
columns (due to the combination of
\eqref{eq_exp11_int1},
\eqref{eq_exp11_psim},
\eqref{eq_exp11_third} and \eqref{eq_quad_b} in Theorem~\ref{thm_quad}, this 
quantity would be zero in exact arithmetics).

We make the following observations from Table~\ref{t:test89}.
First, 
\begin{align}
S_m = \int_{-1}^1 \psi_m(t) \cdot
\left(\varphi_1(t) + \dots + \varphi_n(t)
\right) dt,
\label{eq_exp11_obs1}
\end{align}
due to the combination of \eqref{eq_exp11_third} and
Theorem~\ref{thm_quad}. We observe that $S_m$ is close
to $\lambda_m \psi_m(0)$ for small $m$, but
coincides with the latter only in two digits for $m=38$.

Second, we observe that the value in the fourth column
(see \eqref{eq_exp11_int1}) is grows from $\approx 10^{-16}$ at $m=0$
up to $\approx 5 \cdot 10^{-6}$ at $m=38$.

Third, we observe that, due to the combination of \eqref{eq_exp11_xi}
and Theorem~\ref{thm_quad},
\begin{align}
\xi_m = \frac{1}{\| I \|_{\infty}} 
\int_{-1}^1 I(t) \cdot \psi_n(t) \cdot \psi_m(t) \; dt,
\label{eq_exp11_obs2}
\end{align}
where $I$ is that of 
Theorem~\ref{thm_complex} in Section~\ref{sec_head_tail}.
Theoretically, $|\xi_m|$ is bounded from above by 1
(see the proof of Theorem~\ref{thm_quad}). However, in fact,
$|\xi_m|$ is significantly smaller than one for small values of $m$,
though $\xi_m \approx 0.2$ for $m=38$.

Next, the value in the last column, that would be zero
in exact arithmetics, serves as a test of the accuracy of the calculation.
We observe that this value is of order $10^{-14}$, for all $m$.
Finally,
we note that $\lambda_m \psi_m(0)$ is always positive and
monotonically decreases with $m$.

%%%%%%%%%%%%%%%%%%%
\begin{table}[htbp]
\begin{center}
\begin{tabular}{c|c|c|c|c}
$m$   &
$\lambda_m \psi_m(0)$ &
$ \int \psi_m - \sum W_j \psi_m(t_j)$ &
$C_{n,m}$ &
$C_{n,m} / |\lambda_n|$ \\[1ex]
%
% $ \brk{1 - \frac{\abrk{\lambda_n}^2}{\abrk{\lambda_m}^2}} I_{\max} + 
%   \frac{\abrk{\lambda_n}}{\abrk{\lambda_m}} \abrk{\psi_m(0)} + 
%   c\abrk{P_{n,m}}$ \\[1ex]
%
\hline
 0
& 0.70669E+00 & -.44409E-15 & 0.26389E-04 & 0.20432E+00  \\
 2
& 0.49581E+00 & -.16653E-15 & 0.26333E-04 & 0.20389E+00  \\
 4
& 0.42581E+00 & -.13323E-14 & 0.26314E-04 & 0.20375E+00  \\
 6
& 0.38527E+00 & -.21649E-14 & 0.26303E-04 & 0.20366E+00  \\
 8
& 0.35695E+00 & -.22760E-14 & 0.26296E-04 & 0.20361E+00  \\
 10
& 0.33516E+00 & -.16653E-14 & 0.26290E-04 & 0.20356E+00  \\
 12
& 0.31730E+00 & -.23870E-14 & 0.26285E-04 & 0.20352E+00  \\
 14
& 0.30201E+00 & -.24980E-14 & 0.26281E-04 & 0.20349E+00  \\
 16
& 0.28844E+00 & 0.11102E-14 & 0.26277E-04 & 0.20346E+00  \\
 18
& 0.27604E+00 & -.59230E-13 & 0.26274E-04 & 0.20344E+00  \\
 20
& 0.26435E+00 & 0.83716E-12 & 0.26271E-04 & 0.20342E+00  \\
 22
& 0.25299E+00 & -.89038E-11 & 0.26268E-04 & 0.20339E+00  \\
 24
& 0.24150E+00 & 0.76862E-10 & 0.26265E-04 & 0.20337E+00  \\
 26
& 0.22919E+00 & -.65870E-09 & 0.26262E-04 & 0.20335E+00  \\
 28
& 0.21377E+00 & 0.45239E-08 & 0.26253E-04 & 0.20327E+00  \\
 30
& 0.18075E+00 & -.19826E-07 & 0.26282E-04 & 0.20350E+00  \\
 32
& 0.10038E+00 & 0.68548E-07 & 0.26276E-04 & 0.20345E+00  \\
 34
& 0.27988E-01 & -.33810E-06 & 0.26849E-04 & 0.20789E+00  \\
 36
& 0.49822E-02 & 0.27232E-05 & 0.28516E-04 & 0.22080E+00  \\
 38
& 0.70503E-03 & -.22754E-04 & 0.72700E-04 & 0.56291E+00  \\
\end{tabular}
\end{center}
\caption{\it
Illustration of Theorem~\ref{thm_quad}
with $c = 50$ and $n = 40$. See Experiment 12.
}
\label{t:test90}
\end{table}
%%%%%%%%%%%
\paragraph{Experiment 12.}
\label{sec_exp12}
In this experiment, we illustrate 
Theorems~\ref{thm_quad},~\ref{thm_quad_simple}
in Section~\ref{sec_quad_error}. We proceed as follows.
We choose, more or less arbitrarily, band limit $c$ and
prolate index $n$. We evaluate $\chi_n$, $\lambda_n$,
as well as the nodes $t_1,\dots,t_n$ and the weights
$W_1, \dots, W_n$ of the quadrature, defined in
Definition~\ref{def_quad} in Section~\ref{sec_quad}.
To do so, we use, respectively, the algorithms of
Sections~\ref{sec_evaluate_beta}, \ref{sec_evaluate_lambda},
\ref{sec_evaluate_nodes}, \ref{sec_evaluate_weights}
(in double precision).
Then, we choose an even integer $0 \leq m < n$, and
evaluate $\lambda_m$, $\psi_m(0)$, and $\psi_m(t_j)$
for all $j=1,\dots,n$, using the algorithms
of Sections~\ref{sec_evaluate_lambda}, \ref{sec_evaluate_beta}
(in double precision).

We display the results of this experiment in
Table~\ref{t:test90}.
The data in this table correspond to $c=50$ and $n=40$
(the same as for Table~\ref{t:test89} in Experiment 11).
Table~\ref{t:test90} has the following structure.
The first column contains the even integer $m$,
that varies between $0$ and $n-2$.
The second column contains $\lambda_m \psi_m(0)$.
The third column contains the difference
\begin{align}
\lambda_m \psi_m(0) - \sum_{j=1}^n \psi_m(t_j) \cdot W_j.
\label{eq_exp12_dif}
\end{align}
The fourth column contains the number $C_{n,m}$, defined via
the formula
\begin{align}
C_{n,m} = 
\left(1 - \frac{|\lambda_n|^2}{|\lambda_m|^2} \right) \cdot 
  \| I \|_{\infty} +
|\lambda_n| \cdot \left( 
 \frac{ |\lambda_n| }{ |\lambda_m| } |\psi_m(0)| +
 c \abrk{P_{n,m}} \right),
\label{eq_exp12_cnm}
\end{align}
where $\|I\|_{\infty}$ and $P_{n,m}$ are defined, respectively,
via \eqref{eq_i_inf} and \eqref{eq_quad_p} in Theorem~\ref{thm_quad}
(see also Experiment 8 in Section~\ref{sec_exp8}
and Experiment 10 in Section~\ref{sec_exp10}).
Note that \eqref{eq_exp12_cnm} is the right-hand side
of \eqref{eq_quad_thm} in Theorem~\ref{thm_quad}.
The fifth column contains $C_{n,m} / |\lambda_n|$.

We make the following observations from Table~\ref{t:test90}.
We note that \eqref{eq_exp12_dif} in the third column
is the error of the quadrature rule of Definition~\ref{def_quad},
used to integrate $\psi_m$ over $(-1,1)$
(see also \eqref{eq_prolate_integral} in Section~\ref{sec_pswf}).
The absolute value of this 
error is close to the machine precision for small $m$,
and grows up to $\approx 2 \cdot 10^{-5}$ for $m=38$.
For all values of $m$, the absolute value of \eqref{eq_exp12_dif}
is bounded by $C_{n,m}$ (the fourth column),
in agreement with Theorem~\ref{thm_quad}.
We also observe that $C_{n,m}$ is of the same order of magnitude
for all values of $m$ (as opposed to \eqref{eq_exp12_dif}).
Moreover, $C_{n,m}$ is always smaller than $|\lambda_n|$
(in this case, $|\lambda_n| =$ 0.12915E-03). More specifically,
$C_{n,m}$ is between $0.2 \cdot |\lambda_n|$ and
$0.6 \cdot |\lambda_n|$,
for all the values of $m$ (see the last column).

The behavior of the quadrature error \eqref{eq_exp12_dif}
in the third column is explained with the help
Experiment 10 and
Table~\ref{t:test89} in Experiment 11, as follows.
Due to \eqref{eq_quad_c} in the proof of Theorem~\ref{thm_quad},
\begin{align}
& \lambda_m \psi_m(0) - \sum_{j=1}^n \psi_m(t_j) \cdot W_j = \nonumber \\
& |\lambda_n|^2 \cdot \frac{\psi_m(0)}{\lambda_m} +
\left(1 - \frac{|\lambda_n|^2}{|\lambda_m|^2}\right) \cdot 
\xi_m \cdot \| I \|_{\infty}
+ i c \lambda_n P_{n,m},
\label{eq_exp12_error}
\end{align}
where $\xi_m$ is defined via \eqref{eq_exp11_xi} in Experiment 11.
The first summand in the right-hand side of \eqref{eq_exp12_error}
grows as $m$ increases.
The behavior of the second summand in the right-hand side
of \eqref{eq_exp12_error} depends on $\xi_m$, which is
close to zero for small values of $m$ and close to one
for large values of $m$ (see the fifth column
in Table~\ref{t:test89}). Finally, the last summand in
the right-hand side of \eqref{eq_exp12_error} is also
expected to grow with $m$ (compare to Figure~\ref{fig:test88}
in Experiment 10).

To conclude, $C_{n,m}$, defined via \eqref{eq_exp12_cnm},
significantly overestimates the quadrature error \eqref{eq_exp12_dif}
for small values of $m$. On the other hand, when $m$ is close to $n$,
$C_{n,m}$ is a fairly tight bound
on \eqref{eq_exp12_dif}.

In Theorem~\ref{thm_quad_simple}, we provide an upper bound on $C_{n,m}$
(and hence on the quadrature error \eqref{eq_exp12_dif}),
which is independent on $m$, namely,
\begin{align}
\left|
\lambda_m \psi_m(0) - \sum_{j=1}^n \psi_m(t_j) \cdot W_j
\right| & \leq 
C_{n,m}  \nonumber \\
& \leq |\lambda_n| \cdot
\left(
24 \cdot \log\left( \frac{1}{|\lambda_n|} \right) +
6 \cdot \chi_n
\right).
\label{eq_exp12_simple}
\end{align}
However, the logarithmic term in \eqref{eq_exp12_simple} is
due to the inaccuracy of Theorem~\ref{lem_two_bound} in Section~\ref{sec_tail}
(see Experiment 7 in Section~\ref{sec_exp7}).
Also, the term $6 \cdot \chi_n$ in \eqref{eq_exp12_simple}
is due to the inaccuracy of 
Theorem~\ref{lem_quad_cp} in Section~\ref{sec_quad_error}
(see Experiment 10 in Section~\ref{sec_exp10}).
In other words, numerical experiments seem to suggest that 
the quadrature error \eqref{eq_exp12_dif} is bounded by
$|\lambda_n|$, for all even $0 \leq m < n$.

%%%%%%%%%%%%%%%%%%%%%
\begin{figure} [htbp]
\begin{center}
\includegraphics[width=12cm, bb=36   193   574   598, clip=true]
{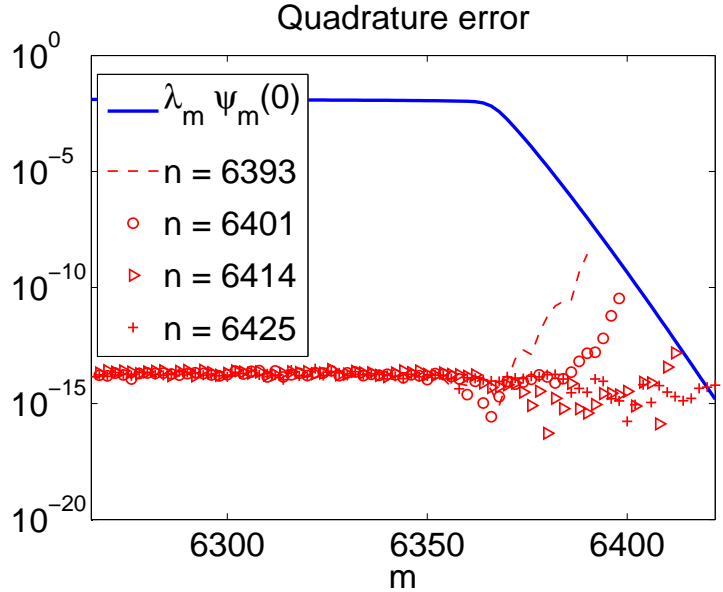}
\caption
{\it
The quadrature error $\left| \int \psi_m - \sum \psi_m(t_j) \cdot W_j \right|$
as a function of even $m < n$, for four different values
of $n$ and $c = 10000$. See Experiment 12.
}
\label{fig:test92}
\end{center}
\end{figure}
%%%%%%%%%%%%
In 
Figure~\ref{fig:test92}, we display the results of the same experiment
with different choice of parameters $c$ and $n$.
Namely, we choose $c=10000$ and plot $\lambda_m \psi_m(0)$
as a function of even $0 \leq m < 6425$, on the logarithmic scale
(solid blue line). In addition, we plot the absolute value
of the quadrature error \eqref{eq_exp12_dif},
as a function of $m$, for four
different values of $n$: $n=6393$ (red dashed line),
$n=6401$ (red circles), $n=6414$ (red triangles),
and $n=6425$ (red pluses). The corresponding values of $|\lambda_n|$
are displayed in Table~\ref{t:test92}.
%%%%%%%%%%%%%%%%%%%
\begin{table}[htbp]
\begin{center}
\begin{tabular}{c|c|c|c|c}
$n$ & 6393 & 6401 & 6414 & 6425 
\\
\hline
$|\lambda_n|$ & 0.43299E-07 & 0.54119E-09 & 0.33602E-12 & 0.52616E-15
\end{tabular}
\end{center}
\caption{\it
Values of $|\lambda_n|$ for $c=10000$ and different choices of $n$.
}
\label{t:test92}
\end{table}

We make the following observations from Figure~\ref{fig:test92}.
First, $\lambda_m \psi_m(0)$ is approximately a constant
for $m < 2c/\pi$, and decays roughly exponentially with $m$
for $m > 2c/\pi$.
Also, for each value of $n$,
the quadrature error \eqref{eq_exp12_dif}
is essentially zero for $m < 2c/\pi$,
and its absolute value
increases roughly exponentially with $m$ for $m > 2c/\pi$.
Nevertheless, the absolute error of the quadrature error
is always bounded from above by $|\lambda_n|$, for each $n$.
See also Tables~\ref{t:test90},~\ref{t:test91} and
Conjecture~\ref{conj_quad_error} below.

%%%%%%%%%%%%%%%%%%%%%%%%%%%%%%
%%%%%%%%%%%%%%%%%%%
\begin{table}[htbp]
\begin{center}
\begin{tabular}{c|c|c|c|c|c}
$c$   &
$n$   &
$m$   &
$\lambda_m \psi_m(0)$ &
$ \int \psi_m - \sum W_j \psi_m(t_j)$ &
$\abrk{\lambda_n}$ \\[1ex]
\hline
250 &  179  &  178
& 0.28699E-07 & -.52496E-08 & 0.18854E-07  \\
250 &  184  &  182
& 0.68573E-09 & -.38341E-10 & 0.16130E-09  \\
250 &  188  &  186
& 0.14108E-10 & -.68758E-12 & 0.30500E-11  \\
\hline
500 &  339  &  338
& 0.52368E-07 & -.13473E-07 & 0.40938E-07  \\
500 &  345  &  344
& 0.37412E-09 & -.86136E-10 & 0.27418E-09  \\
500 &  350  &  348
& 0.12148E-10 & -.99816E-12 & 0.35537E-11  \\
\hline
1000 & 659  &  658
& 0.42709E-07 & -.14354E-07 & 0.38241E-07  \\
1000 & 665  &  664
& 0.51665E-09 & -.15924E-09 & 0.43991E-09  \\
1000 &  671  &  670
& 0.52494E-11 & -.15024E-11 & 0.42815E-11  \\ 
\hline
2000 & 1297  &  1296
& 0.41418E-07 & -.17547E-07 & 0.41740E-07  \\
2000 &  1304  &  1302
& 0.77185E-09 & -.15036E-09 & 0.37721E-09  \\
2000 &  1311  &  1310
& 0.31078E-11 & -.11386E-11 & 0.28754E-11  \\
\hline
4000 &  2572  &  2570
& 0.54840E-07 & -.15493E-07 & 0.33682E-07  \\
4000 &  2579  &  2578
& 0.43032E-09 & -.20771E-09 & 0.46141E-09  \\
4000 &  2587  &  2586
& 0.28193E-11 & -.12805E-11 & 0.29164E-11  \\
\hline
8000 &  5119  &  5118
& 0.43268E-07 & -.26751E-07 & 0.52899E-07  \\
8000 &  5128  &  5126
& 0.50230E-09 & -.16395E-09 & 0.33442E-09  \\
8000 &  5136  &  5134
& 0.50508E-11 & -.15448E-11 & 0.32132E-11  \\
\hline
16000 &  10213  &  10212
& 0.42725E-07 & -.30880E-07 & 0.56568E-07  \\
16000 &  10222  &  10220
& 0.69663E-09 & -.28201E-09 & 0.52821E-09  \\
16000 &  10231  &  10230
& 0.34472E-11 & -.22162E-11 & 0.42902E-11  \\
\end{tabular}
\end{center}
\caption{\it
Relation between the quadrature error and $\abrk{\lambda_n}$.
See Experiment 12.
}
\label{t:test91}
\end{table}
%%%%%%%%%%%

We strengthen the observations above by repeating this experiment
with several other values of band limit $c$ and prolate index $n$.
The results are displayed in Table~\ref{t:test91}.
This table has the following structure.
The first and second column contain, respectively,
the band limit $c$ and the prolate index $n$.
The third column contains the even integer $0 \leq m < n$
(the values of $m$ were chose to be close to $n$).
The fourth column contains $\lambda_m \psi_m(0)$.
The fifth column contains the quadrature error
\eqref{eq_exp12_dif}. The last column contains $|\lambda_n|$.

We make the following observations from Table~\ref{t:test91}.
First, for each of the seven
values of $c$, the three indices $n$ were chosen in such
a way that $|\lambda_n|$ is between
$10^{-12}$ and $10^{-7}$. The values of the band limit 
$c$ vary between $250$
(the first three rows) and $16000$ (the last three rows).
For each $n$, the value of $m$ is chosen to be the largest
even integer below $n$. This choice of $m$ yields
the largest $\lambda_m \psi_m(0)$ and the largest quadrature
error \eqref{eq_exp12_dif} among all $m < n$ 
(see also Table~\ref{t:test90}). Obviously, $|\lambda_m|$ and
$|\lambda_n|$ are of the same order of magnitude, for this choice of $m$.
We also observe that, for all the values of $c,n,m$, 
the absolute error of the quadrature error \eqref{eq_exp12_dif}
is bounded from above by $|\lambda_n|$ (and is roughly equal
to $|\lambda_n|/2$).
In other words, the upper bound on the quadrature error, provided
by Theorem~\ref{thm_quad_simple} (see \eqref{eq_exp12_simple}), 
is somewhat overcautious.

We summarize these observations in the following conjecture.
\begin{conjecture}
Suppose that $c>0$ is a positive real number, and $n>2c/\pi$ is
an integer. Suppose also that $0 \leq m < n$ is an integer.
Suppose furthermore that $t_1, \dots, t_n$ and $W_1, \dots, W_n$ are, 
respectively, the nodes and weights of the quadrature, introduced
in Definition~\ref{def_quad} in Section~\ref{sec_quad}.
Then,
\begin{align}
\left| \int_{-1}^1 \psi_m(s) \; ds - \sum_{j=1}^n \psi_m(t_j) W_j \right| \leq
|\lambda_n|.
\label{eq_quad_error_conj}
\end{align}
\label{conj_quad_error}
\end{conjecture}
\begin{remark}
Conjecture~\ref{conj_quad_error} provides a stronger inequality
than that of Theorem~\ref{thm_quad_simple}. On the other hand,
Conjecture~\ref{conj_quad_error} has been only supported
by numerical evidence, while Theorem~\ref{thm_quad_simple}
has been rigorously proven.
\label{rem_conj}
\end{remark}
%%%%%%%%%%%%%%%%%%%%%%%
%%%%%%%%%%%%%%%%%%%%%
\begin{figure} [htbp]
\begin{center}
\includegraphics[width=12cm, bb=68   218   542   574, clip=true]
{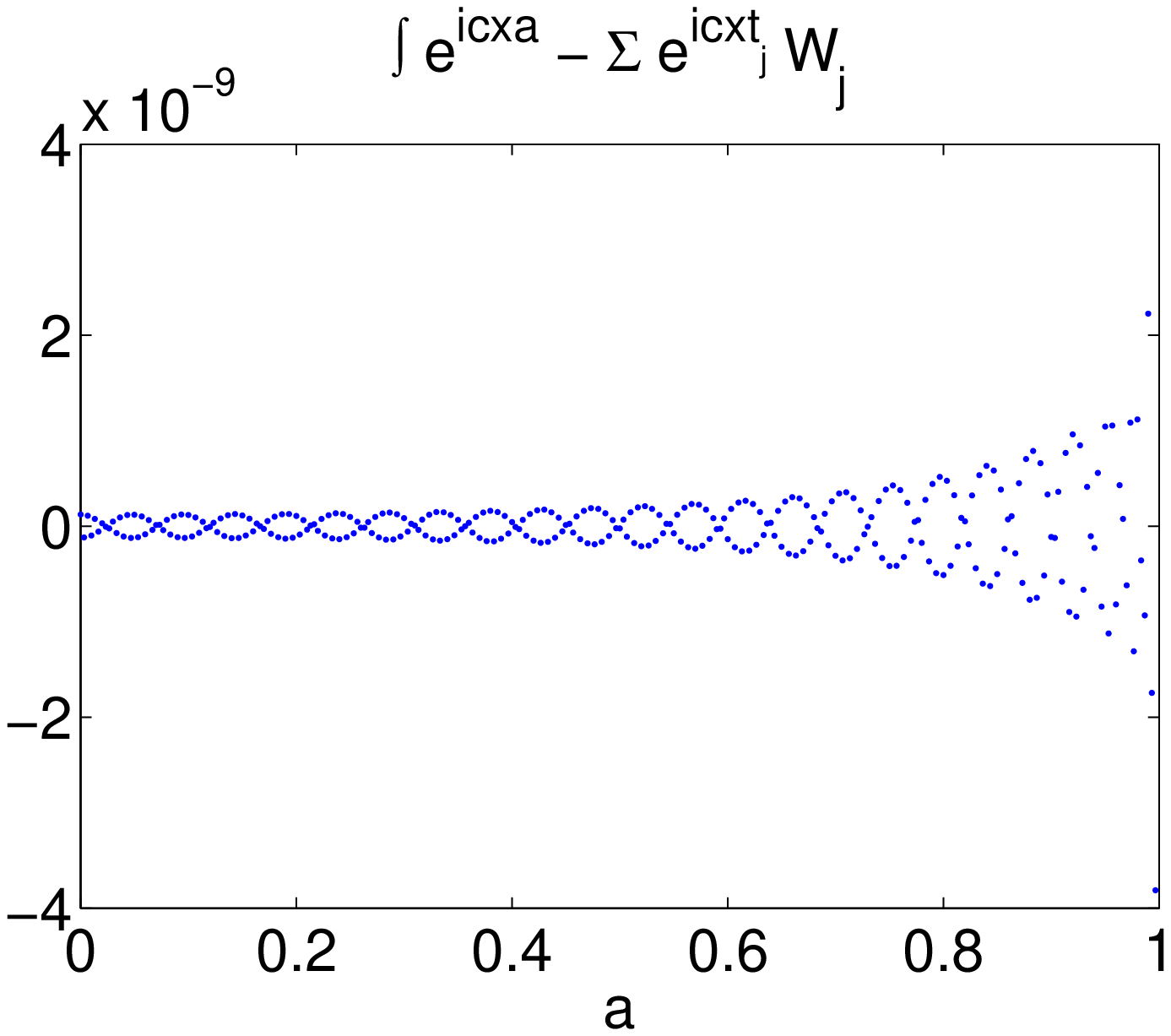}
\caption
{\it
The quadrature error \eqref{eq_exp13_error}
with $c = 1000$, $n = 650$.
See Experiment 13.
}
\label{fig:test93a}
\end{center}
\end{figure}
%%%%%%%%%%%%
%%%%%%%%%%%%%%%%%%%%%
\begin{figure} [htbp]
\begin{center}
\includegraphics[width=12cm, bb=68   218   542   574, clip=true]
{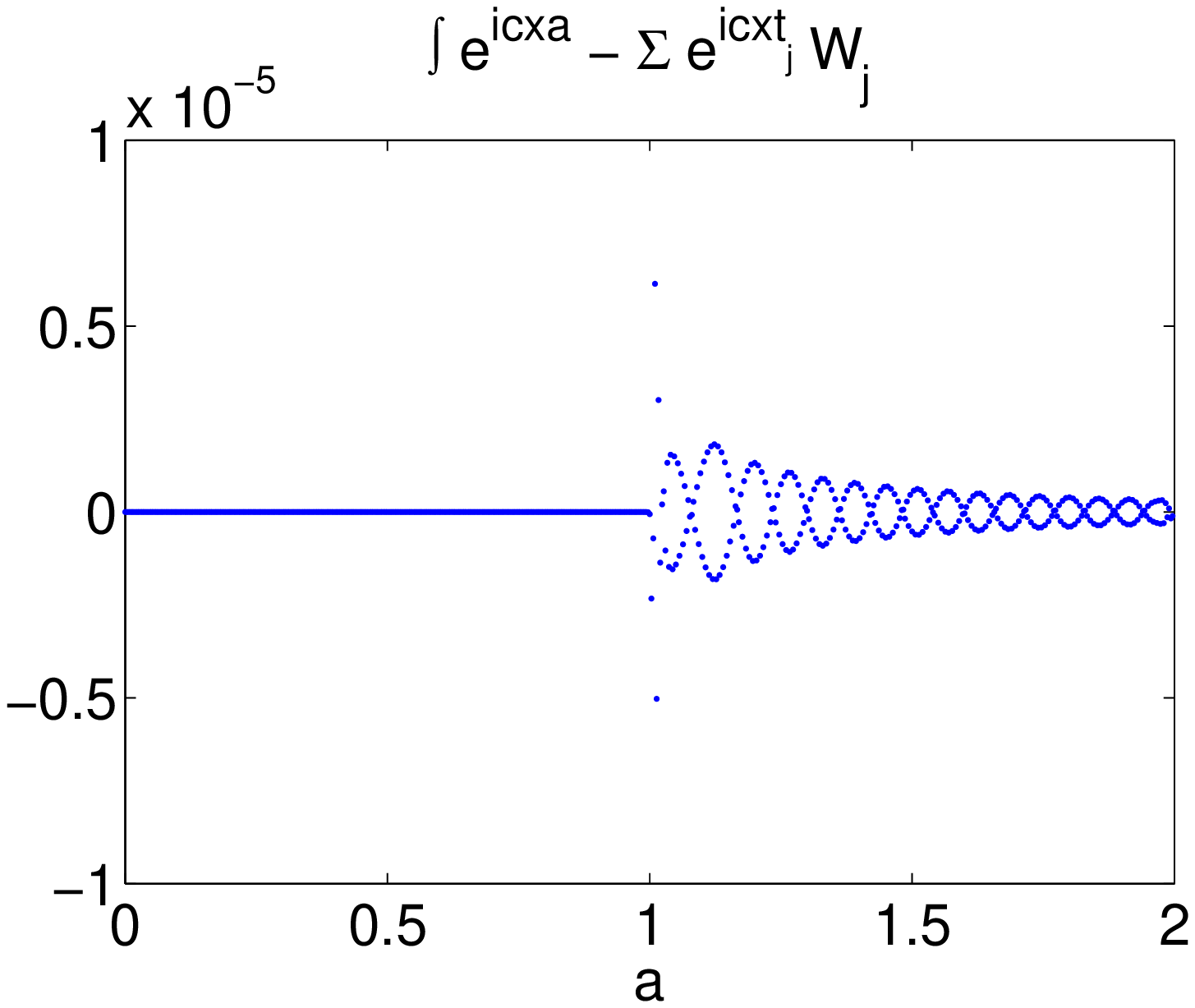}
\caption
{\it
The quadrature error \eqref{eq_exp13_error}
with $c = 1000$, $n = 650$.
See Experiment 13.
}
\label{fig:test93b}
\end{center}
\end{figure}
%%%%%%%%%%%%

\paragraph{Experiment 13.}
\label{sec_exp13}
In this experiment, we demonstrate the performance of the quadrature,
introduced in Definition~\ref{def_quad} in Section~\ref{sec_quad},
on exponential functions. We proceed as follows.
We choose, more or less arbitrarily, the band limit $c$
and the prolate index $n$. We evaluate the quadrature
nodes $t_1, \dots, t_n$ and the quadrature weights $W_1,\dots,W_n$,
by using, respectively,
the algorithms of Sections~\ref{sec_evaluate_nodes},
~\ref{sec_evaluate_weights} (in double precision). 
Also, we evaluate $|\lambda_n|$,
by using the algorithm in Section~\ref{sec_evaluate_lambda}
(in double precision).
Then, we choose a real number $a \geq 0$, and evaluate the 
integral of $e^{icax}$ over $-1 \leq x \leq 1$ via
the formula
\begin{align}
\int_{-1}^1 e^{iacx} \; dx = 
\int_{-1}^1 \cos(acx) \; dx = \frac{2 \sin(ac)}{ac}.
\label{eq_exp13_int}
\end{align}
Also, we compute an approximation to \eqref{eq_exp13_int},
by evaluating the sum
\begin{align}
\sum_{j=1}^n W_j \cdot \cos(icat_j).
\label{eq_exp13_quad}
\end{align}
Finally, we evaluate the error of this approximation, that is,
\begin{align}
\frac{2 \sin(ac)}{ac} - \sum_{j=1}^n W_j \cdot \cos(icat_j).
\label{eq_exp13_error}
\end{align}

In Figures~\ref{fig:test93a},~\ref{fig:test93b}, we display the results
of this experiment. The band limit and the prolate index were
chosen to be, respectively, $c=1000$ and $n=650$. 
This choice yields $\lambda_n =$ -.21224E-04.
In these figure, we plot the quadrature error \eqref{eq_exp13_error}
as a function of the real parameter $a$.
Figure~\ref{fig:test93a} corresponds to $0 \leq a \leq 1$,
while Figure~\ref{fig:test93b} corresponds to $0 \leq a \leq 2$.

We make the following observations
from Figures~\ref{fig:test93a},~\ref{fig:test93b}.
For $0 \leq a \leq 1$, the 
absolute value of the quadrature error \eqref{eq_exp13_error}
is bounded by $4 \cdot 10^{-9} \approx 10 \cdot |\lambda_n|^2$.
The largest quadrature error is obtained when $a$ is close to 1.
On the other hand, for $1 \leq a \leq 2$,
the absolute value of the quadrature error \eqref{eq_exp13_error}
is significantly larger, and is of order $|\lambda_n|$.
The largest quadrature error is obtained when $a$ is close to 1.

These observations admit the following (somewhat imprecise)
explanation. Suppose that $a \geq 0$ is a real number.
Due to \eqref{eq_prolate_integral} and
Theorem~\ref{thm_pswf_main} in Section~\ref{sec_pswf},
\begin{align}
e^{iacx} = \sum_{m=0}^{\infty} \lambda_m \psi_m(a) \psi_m(x),
\label{eq_exp13_exp}
\end{align}
for all real $-1 \leq x \leq 1$ (we note that while $e^{iacx}$ is
not a bandlimited function of $-1\leq x \leq 1$, 
it does belong to $L^2\left[-1,1\right]$).
Moreover,
\begin{align}
\int_{-1}^1 e^{iacx} \; dx =
\frac{2\sin(ac)}{ac} = \sum_{m=0}^{\infty} \lambda_m^2 \psi_m(a) \psi_m(0).
\label{eq_exp13_sinc}
\end{align}
We combine \eqref{eq_exp13_error}, \eqref{eq_exp13_exp}, 
\eqref{eq_exp13_sinc}, to obtain
\begin{align}
& \frac{2 \sin(ac)}{ac} - \sum_{j=1}^n W_j \cdot \cos(icat_j) = \nonumber \\
& \sum_{m = 0}^{\infty} \lambda_m \psi_m(a) 
\brk{\lambda_m \psi_m(0) - \sum_{j=1}^n W_j \psi_m(t_j) }.
\label{eq_exp13_exp_error}
\end{align}
We recall (see Experiment 12) that, for small values of $m$,
the quadrature error \eqref{eq_exp12_dif} is very small compared
to $|\lambda_n|$. On the other hand, for those values of $m < n$ that are close
to $n$, the quadrature error \eqref{eq_exp12_dif} is of order $|\lambda_n|$.
Therefore, roughly speaking, 
\begin{align}
\sum_{m = 0}^{n-1} \lambda_m \psi_m(a) 
\brk{\lambda_m \psi_m(0) - \sum_{j=1}^n W_j \psi_m(t_j) } =
O\left( |\lambda_n|^2 \cdot \psi_{n-1}(a) \right).
\label{eq_exp13_head}
\end{align}
On the other hand, due to the fast decay of $|\lambda_m|$, 
we expect
\begin{align}
\sum_{m = n}^{\infty} \lambda_m \psi_m(a) 
\brk{\lambda_m \psi_m(0) - \sum_{j=1}^n W_j \psi_m(t_j) } =
O\left( |\lambda_n|^2 \cdot \psi_{n}(a) \right).
\label{eq_exp13_tail}
\end{align}
If $0 \leq a \leq 1$, then $|\psi_n(a)| = O(\sqrt{n})$
(see 
Theorems~\ref{thm_psi1_bound},~\ref{thm_extrema},~\ref{thm_all_psi_upper_bound}
in Section~\ref{sec_pswf}).
We combine this observation with \eqref{eq_exp13_head}, \eqref{eq_exp13_tail}
to conclude that
the quadrature error \eqref{eq_exp13_error} is expected 
to be of the order
$|\lambda_n|^2 \cdot \sqrt{n}$.

If, on the other hand, $1 \leq a \leq 2$, then
$|\psi_n(a)| = O\left(|\lambda_n|^{-1}\right)$
(see, for example,
Theorem~\ref{lem_Q_Q_tilde} in Section~\ref{sec_two_by_two},
Theorem~\ref{thm_sharp_simple} in Section~\ref{sec_upper},
Theorem~\ref{lem_psi_for_large_x} in Section~\ref{sec_tail},
Experiment 1 in Section~\ref{sec_exp1},
Experiment 6 in Section~\ref{sec_exp6}).
We combine this observation with \eqref{eq_exp13_head}, \eqref{eq_exp13_tail}
to conclude that, in this case,
the quadrature error \eqref{eq_exp13_error} is expected 
to be of the order
$|\lambda_n|$.

We summarize this crude analysis, supported by the observations above,
in the following conjecture.
\begin{conjecture}
Suppose that $c>0$ is a real number, and that $n>2c/\pi$
is an integer. 
Suppose also that $t_1, \dots, t_n$ and $W_1, \dots, W_n$ are, 
respectively, the nodes and weights of the quadrature, introduced
in Definition~\ref{def_quad} in Section~\ref{sec_quad}. 
Suppose furthermore that $-1 \leq a \leq 1$ is a real number.
Then,
\begin{align}
\int_{-1}^1 e^{icax} \; dx - \sum_{j=1}^n e^{icat_j} \cdot W_j =
O\left( |\lambda_n|^2 \cdot \sqrt{n} \right).
\label{eq_exp_conj}
\end{align}
\label{conj_exp}
\end{conjecture}
\paragraph{Experiment 14.}
\label{sec_exp14}
In this experiment, we illustrate
Theorems~\ref{thm_quad_eps_large},~\ref{thm_quad_eps_simple}
in Section~\ref{sec_main_result}.
We proceed as follows. We choose, more or less arbitrarily,
the band limit $c > 0$ and the accuracy parameter $\varepsilon > 0$.
Then, we use the algorithm of Section~\ref{sec_evaluate_lambda}
to find the minimal integer $m$ 
such that $|\lambda_m| < \varepsilon$. 
In other words, we define the integer $n_1(\varepsilon)$
via the formula
\begin{align}
n_1(\varepsilon) = 
\min\left\{ m \geq 0 \; : \; |\lambda_m| < \varepsilon \right\}.
\label{eq_exp14_n1}
\end{align}
Also, we find the minimal integer such that the corresponding
bound on the quadrature error, established in
Theorem~\ref{thm_quad_simple} in Section~\ref{sec_quad_error},
is less that $\varepsilon$
(see also \eqref{eq_exp12_simple} in Experiment 12).
In other words, we defined $n_2(\varepsilon)$ via the formula
\begin{align}
n_2(\varepsilon) = 
\min\left\{ m \geq 0 \; : \; 
|\lambda_m| \cdot
\left(
24 \cdot \log\left( \frac{1}{|\lambda_m|} \right) +
6 \cdot \chi_m
\right)
< \varepsilon \right\}.
\label{eq_exp14_n2}
\end{align}
Then, we define the integer $n_3(\varepsilon)$ via the formula
\eqref{eq_quad_eps_large_nu} in Theorem~\ref{thm_quad_eps_large}.
In other words,
\begin{align}
n_3(\varepsilon) = \text{floor}\left(
\frac{2c}{\pi} + \frac{\alpha(\varepsilon)}{2\pi} \cdot 
\log\left(
\frac{16ec}{\alpha(\varepsilon)},
\right)
\right)
\label{eq_exp14_n3}
\end{align}
where $\alpha(\varepsilon)$ is defined via
\eqref{eq_quad_eps_large_alpha} in Theorem~\ref{thm_quad_eps_large}.
Finally, we define the integer $n_4(\varepsilon)$ via
the right-hand side of \eqref{eq_quad_eps_simple_n}
in Theorem~\ref{thm_quad_eps_simple}. In other words,
\begin{align}
n_4(\varepsilon) = \text{floor}\left(
\frac{2c}{\pi} +
\left(10 + \frac{3}{2} \cdot \log(c) + 
   \frac{1}{2} \cdot \log\frac{1}{\varepsilon}
\right) \cdot \log\left( \frac{c}{2} \right) \right).
\label{eq_exp14_n4}
\end{align}
In both \eqref{eq_exp14_n3} and \eqref{eq_exp14_n4}, $\text{floor}(a)$
denotes the integer part of a real number $a$.

%%%%%%%%%%%%%%%%%%%
\begin{table}[htbp]
\begin{center}
\begin{tabular}{c|c|c|c|c|c|c|c}
$c$     &
$\varepsilon$   &
$n_1(\varepsilon)$ &
$n_2(\varepsilon)$ &
$n_3(\varepsilon)$ &
$n_4(\varepsilon)$ &
$|\lambda_{n_1(\varepsilon)}|$ &
$|\lambda_{n_2(\varepsilon)}|$ \\[1ex]
\hline
   250 & $10^{-10}$ &   184&   198&   277&   303 & 0.60576E-10 & 0.86791E-16\\
   250 & $10^{-25}$ &   216&   227&   326&   386 & 0.31798E-25 & 0.14863E-30\\
   250 & $10^{-50}$ &   260&   270&   393&   525 & 0.28910E-50 & 0.75155E-56\\
\hline
   500 & $10^{-10}$ &   346&   362&   460&   488 & 0.49076E-10 & 0.60092E-16\\
   500 & $10^{-25}$ &   382&   397&   520&   583 & 0.54529E-25 & 0.19622E-31\\
   500 & $10^{-50}$ &   433&   446&   607&   742 & 0.82391E-50 & 0.38217E-56\\
\hline
  1000 & $10^{-10}$ &   666&   687&   803&   834 & 0.95582E-10 & 0.92947E-17\\
  1000 & $10^{-25}$ &   707&   725&   875&   942 & 0.97844E-25 & 0.14241E-31\\
  1000 & $10^{-50}$ &   767&   783&   981&  1120 & 0.39772E-50 & 0.56698E-57\\
\hline
  2000 & $10^{-10}$ &  1305&  1330&  1467&  1500 & 0.95177E-10 & 0.25349E-17\\
  2000 & $10^{-25}$ &  1351&  1373&  1550&  1619 & 0.86694E-25 & 0.27321E-32\\
  2000 & $10^{-50}$ &  1418&  1438&  1675&  1818 & 0.88841E-50 & 0.22795E-57\\
\hline
  4000 & $10^{-10}$ &  2581&  2610&  2768&  2804 & 0.70386E-10 & 0.64396E-18\\
  4000 & $10^{-25}$ &  2632&  2658&  2862&  2935 & 0.57213E-25 & 0.53827E-33\\
  4000 & $10^{-50}$ &  2707&  2730&  3007&  3154 & 0.56712E-50 & 0.88819E-58\\
\hline
  8000 & $10^{-10}$ &  5130&  5163&  5344&  5383 & 0.59447E-10 & 0.22821E-18\\
  8000 & $10^{-25}$ &  5185&  5216&  5450&  5526 & 0.87242E-25 & 0.16237E-33\\
  8000 & $10^{-50}$ &  5268&  5296&  5614&  5765 & 0.95784E-50 & 0.23927E-58\\
\hline
 16000 & $10^{-10}$ & 10225& 10264& 10468& 10509 & 0.63183E-10 & 0.37516E-19\\
 16000 & $10^{-25}$ & 10285& 10321& 10585& 10664 & 0.85910E-25 & 0.41416E-34\\
 16000 & $10^{-50}$ & 10377& 10409& 10769& 10923 & 0.51912E-50 & 0.56250E-59\\
\hline
 32000 & $10^{-10}$ & 20413& 20457& 20686& 20730 & 0.62113E-10 & 0.12818E-19\\
 32000 & $10^{-25}$ & 20478& 20519& 20815& 20897 & 0.78699E-25 & 0.12197E-34\\
 32000 & $10^{-50}$ & 20577& 20615& 21018& 21176 & 0.96802E-50 & 0.15816E-59\\
\hline
 64000 & $10^{-10}$ & 40786& 40837& 41092& 41139 & 0.89344E-10 & 0.28169E-20\\
 64000 & $10^{-25}$ & 40857& 40903& 41232& 41318 & 0.66605E-25 & 0.39212E-35\\
 64000 & $10^{-50}$ & 40964& 41008& 41454& 41616 & 0.85451E-50 & 0.28036E-60\\
\end{tabular}
\end{center}
\caption{\it
Illustration of
Theorems~\ref{thm_quad_eps_large},~\ref{thm_quad_eps_simple}.
See Experiment 14.
}
\label{t:test178}
\end{table}
%%%%%%%%%%%
We display the results of this experiment in Table~\ref{t:test178}.
This table has the following structure.
The first column contains the band limit $c$.
The second column contains the accuracy parameter $\varepsilon$.
The third column contains $n_1(\varepsilon)$, defined via
\eqref{eq_exp14_n1}. 
The fourth column contains $n_2(\varepsilon)$, defined via
\eqref{eq_exp14_n2}. 
The fifth column contains $n_3(\varepsilon)$, defined via
\eqref{eq_exp14_n3}. 
The sixth column contains $n_4(\varepsilon)$, defined via
\eqref{eq_exp14_n4}. 
The seventh column contains $|\lambda_{n_1(\varepsilon)}|$.
The last column contains $|\lambda_{n_2(\varepsilon)}|$.

Suppose that $c>0$ is a band limit, and $n>0$ is an integer.
We define the real number $Q(c,n)$ via the formula
\begin{align}
Q(c,n) = \max\left\{
\left| 
\int_{-1}^1 \psi_m(t) \; dt - \sum_{j=1}^n \psi_m(t_j) \cdot W_j
\right| \; : \; 0 \leq m \leq n-1
\right\},
\label{eq_exp14_qnc}
\end{align}
where $t_1,\dots,t_n$ and $W_1,\dots,W_n$ are, respectively,
the nodes and the weights of the quadrature, defined in
Definition~\ref{def_quad} in Section~\ref{sec_quad}.
In other words, this quadrature rule integrates the first $n$ PSWFs
up to an error at most $Q(c,n)$.

We make the following observations from Table~\ref{t:test178}.
We observe that $Q(c,n_1(\varepsilon)) < \varepsilon$,
due to the combination of Conjecture~\ref{conj_quad_error} 
in Section~\ref{sec_exp12} and
\eqref{eq_exp14_n1}, \eqref{eq_exp14_qnc}. In other words,
numerical evidence suggests that
the quadrature of order $n_1(\varepsilon)$ will integrate
the first $n_1(\varepsilon)$ PSWFs up to an error at most $\varepsilon$
(see Remark~\ref{rem_conj}).
On the other hand, we combine Theorem~\ref{thm_quad_simple}
in Section~\ref{sec_quad_error}
with \eqref{eq_exp14_n2}, \eqref{eq_exp14_qnc}, to conclude that
the quadrature of order $n_2(\varepsilon)$ has been
\emph{rigorously proven} to integrate the first $n_2(\varepsilon)$ PSWFs
up to an error at most $\varepsilon$. In both Theorem~\ref{thm_quad_simple}
and Conjecture~\ref{conj_quad_error}, we establish upper bounds
on $Q(c,n)$ in terms of $|\lambda_n|$. The ratio of 
$|\lambda_{n_1(\varepsilon)}|$ to $|\lambda_{n_2(\varepsilon)}|$ is quite
large: from about $10^6$ for $c=250$ and $\varepsilon = 10^{-10}, 10^{-25},
10^{-50}$ (see the first three rows in Table~\ref{t:test178}),
to about $10^{10}$ for $c=64000$ and
and $\varepsilon = 10^{-10}, 10^{-25}, 10^{-50}$ 
(see the last three rows in Table~\ref{t:test178}).
On the other hand, the difference between $n_2(\varepsilon)$ and
$n_1(\varepsilon)$ is fairly small; for example, for $\varepsilon=10^{-50}$,
this difference varies from 10 for $c=250$
to $23$ for $c=4000$, to merely $44$ for as large $c$ as $c=64000$.

As opposed to $n_1(\varepsilon)$ and $n_2(\varepsilon)$, 
the integer $n_3(\varepsilon)$,
defined via \eqref{eq_exp14_n3},
is computed via an explicit formula
that depends only on $c$ and $\varepsilon$
(rather than on $|\lambda_n|$ and $\chi_n$, that need to be
evaluated numerically). This formula is derived in 
Theorem~\ref{thm_quad_eps_large} by combining 
Theorem~\ref{thm_quad_simple} with some explicit bounds
on $|\lambda_n|$ and $\chi_n$ in terms of $c$ and $n$.
The convenience of \eqref{eq_exp14_n3} vs. \eqref{eq_exp14_n1},
\eqref{eq_exp14_n2} comes at a price: for example, 
for $\varepsilon = 10^{-50}$, the difference between $n_3(\varepsilon)$
and $n_2(\varepsilon)$ is 123 for $c=250$, and 446 for $c=64000$.
However, the difference 
$n_3(\varepsilon)-n_2(\varepsilon)$ is rather small compared to $c$:
for example, for $\varepsilon=10^{-50}$, this difference is
roughly $4 \cdot \left(\log(c)\right)^2$, for all the values
of $c$ in Table~\ref{t:test178}.

Furthermore, we observe that 
$n_4(\varepsilon)$ is also computed via an explicit
formula that depends only on $c$ and $\varepsilon$ (see \eqref{eq_exp14_n4}).
This formula is a simplification of that for $n_3(\varepsilon)$,
derived in Theorem~\ref{thm_quad_eps_simple}. Thus, not surprisingly,
$n_4(\varepsilon)$ is greater than $n_3(\varepsilon)$, for
all the values of $c$ and $\varepsilon$.

We summarize these observations as follows. 
Suppose that the band limit $c$ and 
the accuracy parameter $\varepsilon>0$ are given.
In Theorem~\ref{thm_quad_eps_large}, we prove that 
$n \geq n_3(\varepsilon)$ implies that the quadrature error
$Q(c,n)$, defined via \eqref{eq_exp14_qnc}, will be at most
$\varepsilon$ (for the quadrature of order $n$, defined
in Definition~\ref{def_quad} in Section~\ref{sec_quad}).
On the other hand, numerical evidence suggests that $Q(n,c) < \varepsilon$
also for all the values of $n$ between $n_1(\varepsilon)$ and
$n_3(\varepsilon)$ (see Experiment 12). In this experiment, we 
observed that
the difference between $n_3(\varepsilon)$ and $n_1(\varepsilon)$
is relatively small compared to $c$ 
(roughly of order $\left(\log(c)\right)^2$).

%%%%%%%%%%%%%%%%%%%%%%%%%%%%%%
\subsubsection{Quadrature Weights}
\label{sec_quad_w_num}
%%%%%%%%%%%%%%%%%%%%%%%%%%%%%%
\paragraph{Experiment 15.}
\label{sec_exp15}
In this experiment, we illustrate the results
of Section~\ref{sec_weights}
(in particular, Theorem~\ref{lem_tilde_phi},
Corollary~\ref{cor_tilde_phi_w}
and
Remark~\ref{rem_w_approx}).
We proceed as follows. We choose, more or less arbitrarily,
band limit $c$ and prolate index $n$.
Then, we compute the quadrature nodes $t_1, \dots, t_n$ as well as
$\psi_n'(t_1), \dots, \psi_n'(t_n)$,
by using the algorithm of Section~\ref{sec_evaluate_nodes}. 
We evaluate $\psi_n'(0)$, using
the algorithm of Section~\ref{sec_evaluate_beta}.
Next, we evaluate the quadrature weights $W_1,\dots,W_n$,
by using the algorithm of Section~\ref{sec_evaluate_weights}.
Also, for each $j=1,\dots,n$, we evaluate
the sum
\begin{align}
-\frac{2}{\psi_n'(t_j)} 
\sum_{k = 0}^{\infty} \alpha_k^{(n)} Q_k(t_j),
\label{eq_exp15_sum}
\end{align}
where $Q_k(t)$ is the $k$th Legendre function of the second kind, defined 
in Section~\ref{sec_legendre},
and $\alpha_k^{(n)}$ is the $k$th coefficient of the Legendre
expansion of $\psi_n$, defined via
\eqref{eq_num_leg_alpha_knc}
in Section~\ref{sec_legendre}
(see Theorem~\ref{lem_tilde_phi} and Section~\ref{sec_evaluate_beta}).
To evaluate \eqref{eq_exp15_sum} numerically, we use only $2N$ first
summands, where $N$ is an integer of order $n$
(see \eqref{eq_n_choice} in Section~\ref{sec_evaluate_beta}).
All the calculations are carried out in double precision.

%%%%%%%%%%%%%%%%%%%
\begin{table}[htbp]
\begin{center}
\begin{tabular}{c|c|c|c}
$j$     &
$W_j$   &
$W_j + 2 \cdot \tilde{\Phi}_n(t_j) / \psi_n'(t_j) $ &
$W_j - \frac{W_{21} \left( \psi_n'(0) \right)^2}
      {\left( \psi_n'(t_j) \right)^2 \cdot \left(1 - t_j^2\right) }$ \\[1ex]
\hline
 1
&  0.7602931556894E-02 & 0.00000E+00 & -.55796E-11  \\
 2
&  0.1716167229714E-01 & 0.00000E+00 & -.55504E-10  \\
 3
&  0.2563684665002E-01 & 0.00000E+00 & -.21825E-12  \\
 4
&  0.3278512460580E-01 & 0.00000E+00 & -.11959E-09  \\
 5
&  0.3863462966166E-01 & 0.16653E-15 & 0.82238E-11  \\
 6
&  0.4334940472363E-01 & 0.22204E-15 & -.16247E-09  \\
 7
&  0.4713107235981E-01 & 0.22204E-15 & 0.11270E-10  \\
 8
&  0.5016785516291E-01 & 0.19429E-15 & -.18720E-09  \\
 9
&  0.5261660773966E-01 & 0.26368E-15 & 0.10495E-10  \\
 10
&  0.5460119701692E-01 & 0.29837E-15 & -.20097E-09  \\
 11
&  0.5621699326080E-01 & 0.17347E-15 & 0.81464E-11  \\
 12
&  0.5753664411864E-01 & 0.12490E-15 & -.20866E-09  \\
 13
&  0.5861531690539E-01 & 0.10408E-15 & 0.55098E-11  \\
 14
&  0.5949490764741E-01 & 0.23592E-15 & -.21301E-09  \\
 15
&  0.6020725336886E-01 & 0.13184E-15 & 0.31869E-11  \\
 16
&  0.6077650804037E-01 & 0.18041E-15 & -.21545E-09  \\
 17
&  0.6122088420703E-01 & 0.48572E-16 & 0.14361E-11  \\
 18
&  0.6155390478472E-01 & 0.83267E-16 & -.21675E-09  \\
 19
&  0.6178529976346E-01 & 0.11102E-15 & 0.36146E-12  \\
 20
&  0.6192162112196E-01 & 0.48572E-16 & -.21732E-09  \\
 21
&  0.6196665001384E-01 & 0.00000E+00 & 0.00000E+00  \\
\end{tabular}
\end{center}
\caption{\it
Quadrature weights \eqref{eq_quad_w} with $c = 40$, $n = 41$.
$\lambda_n = \mbox{\text{\rm{i0.69857E-08}}}$.
See Experiment 15.
}
\label{t:test96}
\end{table}
%%%%%%%%%%%
We display the results of this experiment
Table~\ref{t:test96}.
The data in this table correspond to
$c = 40$ and $n = 41$.
Table~\ref{t:test96} has the following structure.
The first column contains the weight index $j$, that varies
between 1 and $21 = (n+1)/2$.
The second column contains $W_j$.
The third column contains the difference between $W_j$
and \eqref{eq_exp15_sum}.
The last column contains the difference 
\begin{align}
W_j- 
\frac{W_{21} \left( \psi_n'(0) \right)^2}
      {\left( \psi_n'(t_j) \right)^2 \cdot \left(1 - t_j^2\right) }
\label{eq_exp15_dif}
\end{align}
(see Remark~\ref{rem_w_approx}).

In Figure~\ref{fig:test96}, we plot
the weights $W_j$, displayed 
in the second column of Table~\ref{t:test96}.
For $j > 21$, the weights are computed via symmetry considerations.
Each $W_j$ is plotted as a red dot above the corresponding node
$t_j$.

We make the following observations from Table~\ref{t:test96}.
First, all the weights are positive
(see Theorem~\ref{thm_positive_w} and Remark~\ref{rem_w_always_pos}).
Moreover, $W_j$ grow
monotonically as $j$ increases to $(n+1)/2$.
Also, due to the combination of Theorems~\ref{lem_tilde_phi},
\ref{lem_tilde_phi_ode} in Section~\ref{sec_weights},
the value in the third column would be zero
in exact arithmetics. We observe that, indeed, this value
is zero up to the machine precision, which confirms
the correctness of the algorithm of Section~\ref{sec_evaluate_weights}.
(We note that, for $j=1,2,3,4$ and $j=21$, this algorithm, in fact,
does evaluate $W_j$ via \eqref{eq_exp15_sum}, and hence
this value in the corresponding rows is exactly zero).
Finally, we observe that, for all $j$,
the value \eqref{eq_exp15_dif} in the last column
is of the order $|\lambda_n|$, in correspondence with
Remark~\ref{rem_w_approx}.
%%%%%%%%%%%%%%%%%%%%%
\begin{figure} [htbp]
\begin{center}
\includegraphics[width=12cm, bb=68   218   542   574, clip=true]
{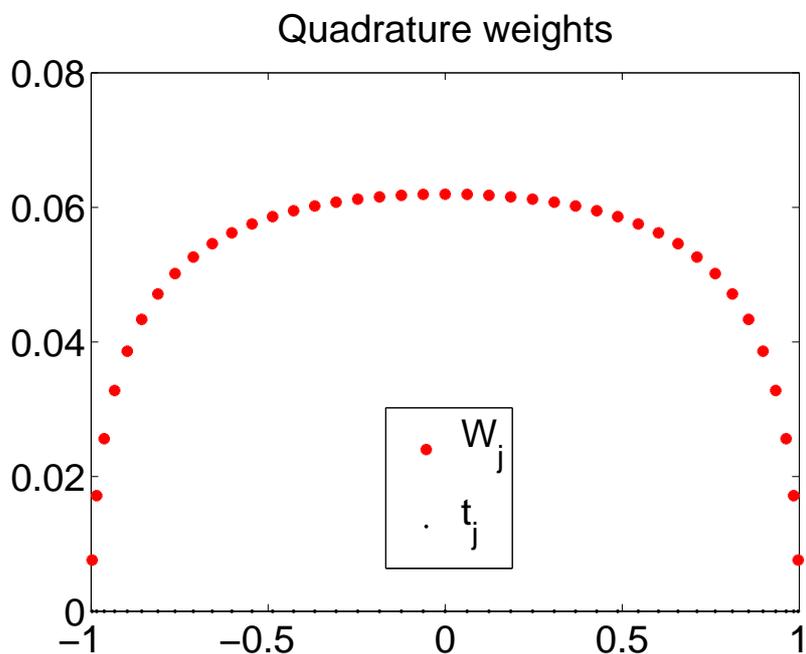}
\caption
{\it
The quadrature weights $W_1,\dots,W_n$ with $c = 40$, $n = 41$.
See Experiment 15.
}
\label{fig:test96}
\end{center}
\end{figure}
%%%%%%%%%%%%

%%%%%%%%%%%%%%%%%%%%%%%%%%%

\begin{comment}
\begin{figure}
\begin{center}
\begin{tabular}{c}
\includegraphics[width=10cm, bb=68 218 542 574, clip=true ]{test82a.eps} \\
(a) \\
\includegraphics[width=10cm, bb=68 218 542 574, clip=true ]{test82b.eps} \\
(b) \\
\includegraphics[width=8cm, bb=68 218 542 574, clip=true ]{test82c.eps} \\
\includegraphics[width=8cm, bb=68 218 542 574, clip=true ]{test82d.eps} \\
(c) \\
(d) \\
\end{tabular}
\caption{
A
}\label{fig:dist_example}
\end{center}
%\vspace{-0.6cm}
\end{figure}
\end{comment}


\begin{thebibliography}{99}

% A
\bibitem{Abramovitz}
{\sc M. Abramowitz, I. A. Stegun},
{\em Handbook of Mathematical Functions with
Formulas, Graphs and Mathematical Tables},
Dover Publications, 1964.


% B
\bibitem{Wilkinson2}
{\sc W. Barth, R. S. Martin, J. H. Wilkinson},
{\em Calculation of the Eigenvalues of
a Symmetric Tridiagonal Matrix by the Method of Bisection},
Numerische Mathematik 9, 386-393, 1967.

\bibitem{Orszag}
{\sc Carl Bender, Steven Orszag},
{\em Advanced Mathematical Methods for
Scientists and Engineers},
McGraw-Hil, Inc. 1978.

\bibitem{Bouwkamp}
{\sc C. J. Bouwkamp},
{\em On spheroidal wave functions of order zero},
J. Math. Phys. 26, 79-92, 1947.

\bibitem{DiPrima}
{\sc W. E. Boyce, R. C. DiPrima},
{\em Elementary Differential Equations and
Boundary Value Problems},
Seventh Edition, John Wiley and Sons, Inc., 2001.

% C
\bibitem{ChengRokhlin}
{\sc H. Cheng, N. Yarvin, V. Rokhlin},
{\em Non-linear Optimiuzation, Quadrature and Interpolation},
SIAM J. Optim. 9, 901-23, 1999.

% D
\bibitem{Dahlquist}
{\sc G. Dahlquist, A. Bj\"ork},
{\em Numerical Methods},
Prentice-Hall Inc., 1974.

% E
\bibitem{Evans}
{\sc L.C. Evans},
{\em Partial Differential Equations},
Graduate Studies in Mathematics Vol. 19, AMS (1998).

% F
\bibitem{Fedoryuk}
{\sc M.V. Fedoryuk}, 
{\em Asymptotic Analysis of Linear Ordinary Differential Equations}, 
Springer-Verlag, Berlin (1993).

\bibitem{Flammer}
{\sc C. Flammer},
{\em Spheroidal Wave Functions},
Stanford, CA: Stanford University Press, 1956.

% G
\bibitem{Glaser}
{\sc Andreas Glaser, Xiangtao Liu, Vladimir Rokhlin},
{\em A fast algorithm for the calculation of the roots of special functions},
SIAM J. Sci. Comput., 29(4):1420-1438 (electronic), 2007.

\bibitem{Ryzhik}
{\sc I.S. Gradshteyn,  I.M. Ryzhik},
{\em Table of Integrals, Series, and Products},
Seventh Edition, Elsevier Inc., 2007.

% H
\bibitem{Hodge}
{\sc D. B. Hodge},
{\em Eigenvalues and Eigenfunctions of the Spheroidal Wave Equation},
J. Math. Phys. 11, 2308-2312, 1970.

% I
\bibitem{Isaacson}
{\sc E. Isaacson, H. B. Keller},
{\em Analysis of Numerical Methods},
New York: Wiley, 1966.

% K
\bibitem{Karlin}
{\sc S. Karlin, W. J. Studden},
{\em Tchebycheff Systems with Applications in Analysis and Statistics},
Wiley-Interscience, New York, 1966.

\bibitem{Krein}
{\sc M. G. Krein},
{\em The Ideas of P. L. Chevyshev and A. A. Markov in the THeory of Limiting 
Values of Integrals},
AM. Math. Soc. Trans., 12, 1-122, 1959.

% L
\bibitem{ProlateLandau1}
{\sc H. J. Landau, H. O. Pollak},
{\em Prolate spheroidal wave functions, Fourier analysis,
and uncertainty - II},
Bell Syst. Tech. J. January 65-94, 1961.

\bibitem{LandauWidom}
{\sc H. J. Landau, H. Widom},
{\em Eigenvalue distribution of time and frequency limiting},
J. Math. Anal. Appl. 77, 469-81, 1980.

\bibitem{Le52}
{\sc Daniel C. Lewis}, 
{\em Inequalities for complex linear differential systems of
the second order}, Proc Natl Acad Sci U S A, 1952 January, 
38(1): 63–66. 


% M
\bibitem{MaRokhlin}
{\sc J. Ma, V. Rokhlin, S. Wandzura},
{\em Generalized Gaussian Quadratures for Systems of Arbitrary Functions},
SIAM J. Numer. Anal. 33, 971-96, 1996.

\bibitem{Markov1}
{\sc A. A. Markov},
{\em On the Limiting Values of Integrals in Connection with Interpolation},
Zap. Imp. Akad. Nauk. Fiz.-Mat. Otd. (8) 6, no 5 (in Russian), 1898.

\bibitem{Markov2}
{\sc A. A. Markov},
{\em Selected Papers on Continued Fractions and the Theory of
Functions Deviating Least From Zero},
OGIZ: Moscow (in Russian), 1948.

\bibitem{Miller}
{\sc Richard K. Miller, Anthony N. Michel},
{\em Ordinary Differential Equations},
Dover Publications, Inc., 1982.

\bibitem{PhysicsMorse}
{\sc P. M. Morse, H. Feshbach},
{\em Methods of Theoretical Physics},
New York McGraw-Hill, 1953.

% O
\bibitem{Report}
{\sc A. Osipov},
{\em Non-asymptotic Analysis of Bandlimited Functions},
Yale CS Technical Report \#1449, 2012.

\bibitem{ReportArxiv}
{\sc A. Osipov},
{\em Certain inequalities involving prolate spheroidal wave functions 
and associated quantities},
arXiv:1206.4056v1, 2012.

\bibitem{Report2}
{\sc A. Osipov},
{\em Explicit upper bounds on the eigenvalues 
associated with prolate spheroidal wave functions},
Yale CS Technical Report \#1450, 2012.

\bibitem{Report2Arxiv}
{\sc A. Osipov},
{\em Certain upper bounds on the eigenvalues 
associated with prolate spheroidal wave functions},
arXiv:1206.4541v1, 2012.

% P
\bibitem{Papoulis}
{\sc A. Papoulis},
{\em Signal Analysis},
Mc-Graw Hill, Inc., 1977.

% R
\bibitem{RokhlinXiaoApprox}
{\sc Vladimir Rokhlin, Hong Xiao},
{\em Approximate Formulae for Certain Prolate
Spheroidal Wave Functions Valid for Large Value
of Both Order and Band Limit}.

\bibitem{Rudin}
{\sc W. Rudin},
{\em Real and Complex Analysis},
Mc-Graw Hill Inc., 1970.


% S
\bibitem{Yoel}
{\sc Yoel Shkolnisky, Mark Tygert, Vladimir Rokhlin},
{\em Approximation of Bandlimited Functions}.


\bibitem{SlepianComments}
{\sc D. Slepian},
{\em Some comments on Fourier analysis, uncertainty, and modeling},
SIAM Rev.(3) 379-93, 1983.

\bibitem{ProlateSlepian1}
{\sc D. Slepian, H. O. Pollak},
{\em Prolate spheroidal wave functions, Fourier analysis,
and uncertainty - I},
Bell Syst. Tech. J. January 43-63, 1961.


\bibitem{ProlateSlepian2}
{\sc D. Slepian, H. O. Pollak},
{\em Prolate spheroidal wave functions, Fourier analysis,
and uncertainty - IV:
extensions to many dimensions, generalized
prolate spheroidal wave functions},
Bell Syst. Tech. J. November 3009-57, 1964.

\bibitem{SlepianAsymptotic}
{\sc D. Slepian},
{\em Some asymptotic expansions for prolate spheroidal wave
functions},
J. Math. Phys. 44 99-140, 1965.


% W
\bibitem{Wilkinson}
{\sc J. H. Wilkinson},
{\em Algebraic Eigenvalue Problem},
Oxford University Press, New York, 1965.


% X
\bibitem{RokhlinXiaoProlate}
{\sc H. Xiao, V. Rokhlin, N. Yarvin},
{\em Prolate spheroidal wavefunctions, quadrature and interpolation},
Inverse Problems, 17(4):805-828, 2001.


% Y
\bibitem{YarvinRokhlin}
{\sc N. Yarvin, V. Rokhlin},
{\em Generalized Gaussian Quadratures and Singular Value Decompositions
of Integral Operators},
SIAM J. Sci. Comput. 20, 699-718, 1998.






\end{thebibliography}
\end{document}